\documentclass[10pt,letterpaper]{article}
\usepackage[letterpaper, top=.95in, bottom=1in, left=1.2in, right=1.2in, marginparwidth=2.3cm, marginparsep=1mm, includefoot]{geometry}

\usepackage{latexsym,
amsthm,
color,
amssymb,
url,
bm,
cite,
amssymb,
microtype,
amsmath,
amsfonts,
mathtools}
\usepackage{verbatim,listings}
\usepackage{todonotes}
\usepackage{caption}
\usepackage{setspace}
\usepackage{subcaption}
\usepackage{setspace}
\usepackage[pdftex, plainpages = false, pdfpagelabels, 
                 bookmarks = true,
                 bookmarksopen = true,
                 bookmarksnumbered = true,
                 breaklinks = true,
                 linktocpage,
                 pagebackref,
                 colorlinks = true,  
                 linkcolor = blue,
                 urlcolor  = blue,
                 citecolor = red,
                 anchorcolor = green,
                 hyperindex = true,
                 hyperfigures
                 ]{hyperref}                   
\usepackage[inline]{enumitem}
\usepackage[mathscr]{euscript}
\usepackage{amsthm}
\usepackage{nicefrac}
\usepackage{dsfont}
\usepackage{tikz}
\usepackage[ruled,linesnumbered]{algorithm2e}
\usepackage{nicefrac}
\usepackage[mathscr]{euscript}
\graphicspath{{./figures/}}
\RequirePackage[T1]{fontenc}
\RequirePackage[utf8]{inputenc}
\usepackage[english]{babel}
\usepackage{alphabeta}
\usetikzlibrary{math}
\makeatletter
\input{colordvi}
\usepackage{aliascnt}
\hypersetup{
    colorlinks,
    linkcolor={red!75!black},
    citecolor={green!75!black},
    urlcolor={blue!75!black},
}
\definecolor{MidnightBlack}{rgb}{0.1,0.1,.34}
\definecolor{MidnightBlue}{rgb}{0.1,0.1,0.43}
\definecolor{Black}{rgb}{0,0, 0}
\definecolor{Blue}{rgb}{0, 0 ,1}
\definecolor{Red}{rgb}{1, 0 ,0}
\definecolor{White}{rgb}{1, 1, 1}
\definecolor{grey}{rgb}{.6, .6, .6}
\definecolor{Mygreen}{rgb}{.0, .7, .0}
\definecolor{Yellow}{rgb}{.55,.55,0}
\definecolor{Mustard}{rgb}{1.0, 0.86, 0.35}
\definecolor{applegreen}{rgb}{0.55, 0.71, 0.0}
\definecolor{darkturquoise}{rgb}{0.0, 0.81, 0.82}
\definecolor{celestialblue}{rgb}{0.29, 0.59, 0.82}
\definecolor{green_yellow}{rgb}{0.68, 1.0, 0.18}
\definecolor{crimsonglory}{rgb}{0.75, 0.0, 0.2}
\definecolor{darkmagenta}{rgb}{0.30, 0.0, 0.30}
\definecolor{magenta}{rgb}{0.50, 0.0, 0.50}
\definecolor{internationalorange}{rgb}{1.0, 0.31, 0.0}
\definecolor{darkorange}{rgb}{1.0, 0.55, 0.0}
\definecolor{ao}{rgb}{0.0, 0.5, 0.0}
\definecolor{awesome}{rgb}{1.0, 0.13, 0.32}
\definecolor{darkcyan}{rgb}{0.0, 0.50, 0.50}
\definecolor{violet}{rgb}{0.93, 0.51, 0.93}
\definecolor{brown}{rgb}{0.65, 0.16, 0.16}
\definecolor{orange}{rgb}{1.0, 0.65, 0.0}
\definecolor{cornflowerblue}{rgb}{0.39, 0.58, 0.93}

\newcommand{\cref}[1]{\autoref{#1}}

\newcommand{\remove}[1]{}

\newcounter{func}
\newcommand{\funref}[1]{\hyperref[#1]{f_{\ref*{#1}}}} % print a
\tikzset{black node/.style={draw, circle, fill = black, minimum size = 5pt, inner sep = 0pt}}
\tikzset{white node/.style={draw, circlternary_treese, fill = white, minimum size = 5pt, inner sep = 0pt}}
\tikzset{normal/.style = {draw=none, fill = none}}
\tikzset{lean/.style = {draw=none, rectangle, fill = none, minimum size = 0pt, inner sep = 0pt}}
\usetikzlibrary{decorations.pathreplacing}
\usetikzlibrary{arrows.meta}
\usetikzlibrary{calc}

\usetikzlibrary{shapes}
\tikzset{diam/.style={draw, diamond, fill = black, minimum size = 7pt, inner sep = 0pt}}
\tikzset{
	position/.style args={#1:#2 from #3}{
		at=($(#3)+(#1:#2)$)
	}
}

\tikzset{
  v:main/.style = {draw, circle, scale=0.8, thick,fill=black,inner sep=0.7mm},
  v:ghost/.style = {inner sep=0pt,scale=1},
  v:marked/.style = {circle, scale=1.3, fill=DarkGoldenrod,opacity=0.4},
  >={latex},
  e:main/.style = {line width=1pt}
}

\newcommand{\Acal}{\mathcal{A}}
\newcommand{\Bcal}{\mathcal{B}}
\newcommand{\Ccal}{\mathcal{C}}
\newcommand{\Dcal}{\mathcal{D}}
\newcommand{\Ecal}{\mathcal{E}}
\newcommand{\Fcal}{\mathcal{F}}
\newcommand{\Gcal}{\mathcal{G}}
\newcommand{\Hcal}{\mathcal{H}}

\newcommand{\Kcal}{\mathcal{K}}
\newcommand{\Lcal}{\mathcal{L}}
\newcommand{\Mcal}{\mathcal{M}}

\newcommand{\Ocal}{\mathcal{O}}
\newcommand{\Pcal}{\mathcal{P}}
\newcommand{\Qcal}{\mathcal{Q}}
\newcommand\Rcal{\mathcal{R}}
\newcommand{\Scal}{\mathcal{S}}
\newcommand{\Tcal}{\mathcal{T}}
\newcommand{\Ucal}{\mathcal{U}}
\newcommand{\Vcal}{\mathcal{V}}
\newcommand{\Wcal}{\mathcal{W}}

\newcommand{\Ycal}{\mathcal{Y}}
\newcommand{\Zcal}{\mathcal{Z}}

\newcommand{\Ebbb}{\mathbb{E}}

\newcommand{\Nbbb}{\mathbb{N}}
\newcommand{\Obbb}{\mathbb{O}}

\newcommand{\Rbbb}{\mathbb{R}}
\newcommand{\Sbbb}{\mathbb{S}}

\RequirePackage{stmaryrd}
\usepackage{textcomp}
\DeclareUnicodeCharacter{2286}{\subseteq}
\DeclareUnicodeCharacter{2192}{\ifmmode\to\else\textrightarrow\fi}
\DeclareUnicodeCharacter{2203}{\ensuremath\exists}
\DeclareUnicodeCharacter{183}{\cdot}
\DeclareUnicodeCharacter{2200}{\forall}
\DeclareUnicodeCharacter{2264}{\leq}
\DeclareUnicodeCharacter{2265}{\geq}
\DeclareUnicodeCharacter{8614}{\mathbin{\mapsto}}
\DeclareUnicodeCharacter{8656}{\Leftarrow}
\DeclareUnicodeCharacter{8657}{\Uparrow}
\DeclareUnicodeCharacter{8658}{\Rightarrow}
\DeclareUnicodeCharacter{8659}{\Downarrow}
\DeclareUnicodeCharacter{8669}{\rightsquigarrow}
\newcommand{\eqdef}{\stackrel{{\scriptsize\rm def}}{=}}
\DeclareUnicodeCharacter{8797}{\eqdef}
\DeclareUnicodeCharacter{8870}{\vdash}
\DeclareUnicodeCharacter{8873}{\Vdash}
\DeclareUnicodeCharacter{22A7}{\models}
\DeclareUnicodeCharacter{9121}{\lceil}
\DeclareUnicodeCharacter{9123}{\lfloor}
\DeclareUnicodeCharacter{9124}{\rceil}
\DeclareUnicodeCharacter{2208}{\in}
\DeclareUnicodeCharacter{9126}{\rfloor}
\DeclareUnicodeCharacter{9655}{\triangleright}
\DeclareUnicodeCharacter{9665}{\triangleleft}
\DeclareUnicodeCharacter{9671}{\diamond}
\DeclareUnicodeCharacter{9675}{\circ}
\DeclareUnicodeCharacter{10178}{\bot}
\DeclareUnicodeCharacter{10214}{} % needs stmaryrd
\DeclareUnicodeCharacter{10215}{} % needs stmaryrd
\DeclareUnicodeCharacter{10229}{\longleftarrow}
\DeclareUnicodeCharacter{10230}{\longrightarrow}
\DeclareUnicodeCharacter{10231}{\longleftrightarrow}
\DeclareUnicodeCharacter{10232}{\Longleftarrow}
\DeclareUnicodeCharacter{10233}{\Longrightarrow}
\DeclareUnicodeCharacter{10234}{\Longleftrightarrow}
\DeclareUnicodeCharacter{10236}{\longmapsto}
\DeclareUnicodeCharacter{10238}{\Longmapsto} % needs stmaryrd
\DeclareUnicodeCharacter{10503}{\Mapsto}    % needs stmaryrd
\DeclareUnicodeCharacter{10971}{\mathrel{\not\hspace{-0.2em}\cap}}
\DeclareUnicodeCharacter{65294}{\ldotp}
\DeclareUnicodeCharacter{65372}{\mid}

\definecolor{Red}{rgb}{1, 0 ,0}
\definecolor{Blue}{rgb}{0, 0 ,1}

\newtheorem{theorem}{Theorem}[section]

\newaliascnt{question}{theorem}

\aliascntresetthe{question}

\newaliascnt{lemma}{theorem}
\newtheorem{lemma}[lemma]{Lemma}
\aliascntresetthe{lemma}

\newaliascnt{claim}{theorem}

\aliascntresetthe{claim}

\newaliascnt{invariant}{theorem}

\aliascntresetthe{invariant}

\newaliascnt{proposition}{theorem}
\newtheorem{proposition}[proposition]{Proposition}
\aliascntresetthe{proposition}

\newaliascnt{observation}{theorem}
\newtheorem{observation}[observation]{Observation}
\aliascntresetthe{observation}

\newaliascnt{corollary}{theorem}
\newtheorem{corollary}[corollary]{Corollary}
\aliascntresetthe{corollary}

\newaliascnt{definition}{theorem}

\aliascntresetthe{definition}

\newaliascnt{conjecture}{theorem}

\aliascntresetthe{conjecture}

\newaliascnt{counterexample}{theorem}

\aliascntresetthe{counterexample}

\newcommand{\hh}{\end{document}}

\newcommand{\p}{{\sf p}}
\newcommand{\sobs}{{\sf sobs}}
\newcommand{\obs}{{\sf obs}}

\newcommand{\excl}{{\sf excl}}

\newcommand{\gall}{\mathcal{G}_{{\text{\rm  \textsf{all}}}}}

\newcommand{\tw}{{\sf tw}\xspace}%treewdith
\newcommand{\size}{{\sf size}\xspace}%size
\newcommand{\td}{{\sf td}\xspace}%tree-depth
\newcommand{\apex}{{\sf apex}\xspace}%Apex
\newcommand{\barrier}{{\sf barrier}\xspace}%Barrier
\newcommand{\ed}{{\sf ed}\xspace}%Elimination distance
\newcommand{\cupall}{{\pmb{\bigcup}}}

\newcommand*\samethanks[1][\value{footnote}]{\footnotemark[#1]}
\newcommand{\ground}{\ensuremath{\mathsf{ground}}}

\newcommand{\torso}{\ensuremath{\mathsf{torso}}}

\newcommand{\poly}{\text{$\mathsf{poly}$}\xspace}

\newcommand{\folio}{{\sf \mbox{-}folio}\xspace}%folio
\newcommand{\bd}{{\sf bd}\xspace}%Boundary
\newcommand{\rev}{\mathsf{rev}\xspace}%rev
\newcommand{\dissolve}{\mathsf{dissolve}\xspace}%dissolve
\newcommand{\rot}[1]{(#1)}% rotation society
\newcommand{\lin}[1]{\langle #1\rangle}% linear society

\newcommand{\numen}[1]{\ifthenelse{\not\equal{#1}{1}}{#1}{}}
\newcommand{\nn}[2]{\Nbbb^{\numen{#1}}\to\Nbbb^{\numen{#2}}}% rotation society

\usepackage{color}

\definecolor{vagelisColour}{RGB}{0, 65, 130}

\newcommand{\bfl}[1]{\textbf{\bf #1}}
\newcommand{\trace}{\mathsf{trace}\xspace}%

\newcommand{\defi}[1]{\emph{#1}}
\newcommand{\llabel}[1]{\label{#1}}

\title{Obstructions to Erdős-Pósa Dualities for Minors\thanks{All authors were supported by the French-German Collaboration ANR/DFG Project UTMA (ANR-20-CE92-0027). The third author was also supported by the Franco-Norwegian project PHC AURORA 2024 (Projet n° 51260WL). The last author  was also supported by   the Institute for Basic Science (IBS-R029-C1).}}

\author{Christophe Paul\thanks{LIRMM, Univ Montpellier, CNRS, Montpellier, France.} \and Evangelos Protopapas\samethanks \and Dimitrios M. Thilikos\samethanks\\ \and    Sebastian Wiederrecht\thanks{Discrete Mathematics Group, Institute for Basic Science, Daejeon, South Korea.}}

\date{\empty}

\begin{document}

\maketitle

\begin{abstract}
\medskip\medskip

\noindent
Let $\mathcal{G}$ and $\mathcal{H}$ be minor-closed graph classes.
We say that the pair $(\mathcal{H},\mathcal{G})$ is an \textsl{Erd\H{o}s-P{\'o}sa pair} (EP-pair) if there exists a function $f$ such that for every $k$ and every graph $G\in \mathcal{G},$ either $G$ has $k$ pairwise vertex-disjoint subgraphs which do not belong to $\mathcal{H},$ or there exists a set $S\subseteq V(G)$ of size at most $f(k)$ for which $G - S \in \mathcal{H}.$
The classic result of Erd\H{o}s and P{\'o}sa says that if $\mathcal{F}$ is the class of forests, then $(\mathcal{F},\mathcal{G})$ is an EP-pair \textsl{for all} graph classes $\mathcal{G}.$
A minor-closed graph class $\mathcal{G}$ is an \textsl{EP-counterexample} for $\mathcal{H}$ if $\mathcal{G}$ is minimal with the property that $(\mathcal{H},\mathcal{G})$ is \textsl{not} an EP-pair.

In this paper, we prove that for every minor-closed graph class $\mathcal{H}$ the set $\mathfrak{C}_{\mathcal{H}}$ of all EP-counterexamples for $\mathcal{H}$ is \textsl{finite}.
In particular, we provide a complete characterization of $\mathfrak{C}_{\mathcal{H}}$ for every $\mathcal{H}$ and give a constructive upper bound on its size.
We show that each class $\mathcal{G}$ in $\mathfrak{C}_{\mathcal{H}}$ can be described as the set of all minors of some, suitably defined, sequence of grid-like graphs $\langle \mathscr{W}_{k} \rangle_{k\in \mathbb{N}}.$
Moreover, each $\mathscr{W}_{k}$ admits a half-integral packing, i.e., $k$ copies of some $H \not\in \mathcal{H}$ where no vertex is used more than twice. 
This implies a complete  delineation of the half-integrality threshold of the Erd\H{o}s-P{\'o}sa property for minors and as a corollary, we obtain a constructive proof of Thomas' conjecture on the half-integral Erd\H{o}s-P{\'o}sa property for minors which was recently confirmed by Liu.
Our results are algorithmic.
Let $h=h(\mathcal{H})$ denote the maximum size of an obstruction to $\mathcal{H}.$
For every  minor-closed graph class $\mathcal{H},$  we construct an algorithm that, given a graph $G$ and an integer $k,$ either outputs a half-integral packing of $k$ copies of some $H \not\in \mathcal{H}$ or outputs a set of at most ${2^{k^{\mathcal{O}_h(1)}}}$ vertices whose deletion creates a graph in $\mathcal{H}$ in time $2^{2^{{k^{\mathcal{O}_h(1)}}}} \cdot |G|^4 \log |G|.$
Moreover, as a consequence of our results, for every minor-closed class $\mathcal{H},$ we obtain min-max-dualities, which may be seen as analogues of the celebrated Grid Theorem of Robertson and Seymour, for the recently introduced parameters $\mathcal{H}$-treewidth and elimination distance to $\mathcal{H}.$
\end{abstract}
\medskip\medskip\medskip



\newpage

\tableofcontents

\newpage

\section{Introduction}
Min-max dualities like Menger's Theorem or the related Max-Flow/Min-Cut Theorem are among the most powerful results in discrete mathematics, usually having strong algorithmic consequences.
In many cases, exact dualities like the ones above do not exist, however, allowing for some gap may yield approximate results.
Two emblematic theorems of this type in structural graph theory are the Erd\H{o}s-P{\'o}sa Theorem for hitting and packing cycles \cite{ErdosPosaOriginal} and the Grid Theorem of Robertson and Seymour which ties treewidth to the existence of large and highly connected planar minors \cite{RobertsonS86GMV}.
In this paper we generalize both results above, maintaining their deep connection, to arbitrary minor-closed graph classes $\mathcal{H}.$
Regarding the former, we determine \textsl{all} minor-closed classes $\mathcal{G}$ where, for each of its members, the absence of a small hitting set for the obstructions of $\mathcal{H}$ implies the presence of many disjoint such obstructions.
Moreover, regarding the second, we prove that any graph where the parameter ``$\mathcal{H}$-treewidth'' is large must contain a certificate for this in the form of a highly regular grid-like minor containing many ``highly entangled'' obstructions to $\mathcal{H}.$

\subsection{Erd\H{o}s-P{\'o}sa dualities for graph minors}
We begin with a general introduction to Erd\H{o}s-P{\'o}sa-type dualities together with some basic notation for the presentation of our results.

\paragraph{Minors.}
A graph $Z$ is a \defi{minor} of a graph $G$ if $Z$ can be obtained from a subgraph of $G$ by contracting edges.
This induces a partial order on the class of all graphs which we denote by $\leq.$
A graph class $\mathcal{G}$ is \defi{minor-closed} if it contains all minors of its elements and it is \defi{proper} if it does not contain all graphs.
The \defi{obstruction set} $\obs(\mathcal{G})$ of a minor-closed graph class $\mathcal{G}$ is the set of $\leq$-minimal graphs not in $\mathcal{G}.$
We refer to the graphs in $\obs(\mathcal{G})$ as the \defi{obstructions} of $\mathcal{G}.$
Due to the main result of the Graph Minors series by Robertson and Seymour \cite{robertson2004GMXX} we know that $\mathsf{obs}(\mathcal{G})$ is a \textsl{finite} set for all minor-closed graph classes $\mathcal{G}.$

\paragraph{Apices and barriers.}
Given a graph class $\Hcal,$ an \defi{$\Hcal$-modulator} of a graph $G$  is a subset of vertices of $G$ whose deletion yields a graph in $\Hcal.$
The \defi{$\Hcal$-apex} number of $G,$ denoted by $\apex_{\Hcal}(G),$ is defined as the minimum size of an $\Hcal$-modulator in $G.$
When $\Hcal$ is a minor-closed graph class, the study of this parameter includes the proof of constructive bounds, as a function in $k,$ for the number of obstructions of the graph class $\{G\mid \apex_{\Hcal}(G)≤k\}$ \cite{AdlerGK08comp,FellowsL89anan} as well as fixed-parameter algorithms for deciding whether a graph belongs to $\{G\mid \apex_{\Hcal}(G)≤k\},$ parameterized by $k$ \cite{MarxS07obta,KociumakaP19dele,JansenLS14anea,Kawarabayashi09plan,SauST20anfp,SauST21kapiII,SauST22apicesalg}.
Instances of this problem include the classic case where $\mathcal{H}$ is the class of all forests \cite{Downey1995FPT.I.,Bodlaender2007FVSKernel,Jansen2011FVSKernels,Majumdar2017FVS} as well as other classes with ``good'' algorithmic properties such as planar graphs \cite{Kawarabayashi09plan,JansenLS14anea} and graphs of bounded genus \cite{KociumakaP19dele}, that have attracted particular interest.

As an attempt to introduce an ``obstructing measure'' for a graph $G$ to belong to some  graph class $\Hcal,$ we consider a collection $\Bcal$ of subgraphs of $G,$ each  ``obstructing'' membership to $\Hcal.$ 
When $\Hcal$ is minor-closed, every member of $\Bcal$ should contain a common\footnote{One might alternatively ask for minor-containment of different elements of $\obs(\Hcal)$; however this would  not define a parameter that is functionally different, as $\mathsf{obs}(\mathcal{H})$ is finite by the Robertson-Seymour Theorem.} element of $\obs(\Hcal)$ as a minor.
If the subgraphs in $\Bcal$ are pairwise vertex disjoint we call $\Bcal$ an \defi{integral $\Hcal$-barrier} (or $\Hcal$-barrier for short).
If every vertex of $G$ belongs to at most two members of $\Bcal,$ then $\Bcal$ is a \defi{half-integral $\Hcal$-barrier}.
We may turn this notion into a maximization parameter $\barrier_{\Hcal}(G)$ (resp. $\nicefrac{1}{2}\text{-}\barrier_{\Hcal}(G)$) by defining it as the maximum size of an $\Hcal$-barrier (resp.  half-integral $\Hcal\text{-}\barrier$) of $G.$
For example, in the case where $\Hcal$ is the class of forests, an integral $\Hcal$-barrier of $G$ is a set of pairwise vertex-disjoint cycles in $G$ (notice that each cycle contains $K_3\in\mathsf{obs}(\mathcal{H})=\{K_{3}\}$ as a minor).

\paragraph{Erd\H{o}s-P{\'o}sa dualities.}
If $G$ contains an $\Hcal$-barrier $\Bcal,$ then every $\Hcal$-modulator is forced to contain at least one vertex from each member of $\Bcal.$
Therefore, $\barrier_{\Hcal}(G)\leq \apex_{\Hcal}(G).$
This holds for all possible choices of $\Hcal$ and $G.$
Similarly, one may observe that $\nicefrac{1}{2}\text{-}\barrier_{\Hcal}(G)\leq \frac{1}{2}\cdot \apex_{\Hcal}(G).$
The question of duality is whether the reverse inequalities hold approximately: 
\begin{quote}
    \textsl{Under which assumptions does the absence of a large (half-integral) $\Hcal$-barrier imply the existence of a small $\Hcal$-modulator?}
\end{quote}

\smallskip
Providing a precise and complete answer to the ``\textsl{when}'' but also to the ``\textsl{why}'' of the question above is the objective of this paper.
\smallskip

In order to discuss this objective formally, let us present a unified way to express Erd\H{o}s-P{\'o}sa type questions and results.
Let $\Hcal$ and $\Gcal$ be two minor-closed graph classes.
We refer to $\Hcal$ as the \defi{target class} and $\Gcal$ as the \defi{environment class}.
The pair $(\Hcal,\Gcal)$ is an \defi{Erd\H{o}s-P{\'o}sa pair}, or \defi{EP-pair} for short, if there is a function $f\colon\Nbbb\to\Nbbb,$ called the \defi{gap function}, such that, for every graph $G\in\Gcal,$ it holds that $\apex_{\Hcal}(G)\leq f(\barrier_{\Hcal}(G)).$
The concept of \defi{EP-pair} allows us to express \textsl{all} possible Erd\H{o}s-P{\'o}sa dualities for the minor relation and thereby formalize the ``when''-part of the question above as follows
\begin{eqnarray}
\begin{minipage}{14cm}
\centering
\textsl{Which pairs $(\Hcal,\Gcal)$ are EP-pairs and what are their gap functions?}
\end{minipage}\llabel{leading_question}
\end{eqnarray}

\noindent
Let $\mathcal{G}_{\mathsf{all}}$ denote the class of all graphs and let $\Fcal$ be the class of all forests.
The seminal theorem of Erd\H{o}s and P{\'o}sa \cite{ErdosPosaOriginal} initiated research for Erd\H{o}s-P{\'o}sa dualities and motivated their name. It can be reformulated as:
\begin{eqnarray}
\begin{minipage}{14cm}
\centering
 $(\Fcal,\gall)$ is an EP-pair whose gap function belongs to $\mathcal{O}(k\log k).$
\end{minipage}\llabel{EP_original}
\end{eqnarray}

\noindent
Robertson and Seymour, as a consequence of their celebrated Grid Theorem, proved the following: \cite{RobertsonS86GMV}
\begin{eqnarray}
\begin{minipage}{14cm}
\centering
 $(\Hcal,\gall)$ is an EP-pair if and only if $\obs(\Hcal)$ contains a planar graph.
\end{minipage}\llabel{EP_for_minors}
\end{eqnarray}

\noindent
The proof of \eqref{EP_for_minors} identifies candidates for graph families that indicate an answer to the ``why''-part of \eqref{leading_question}: they are wall-like structures representing the graphs embeddable in non-spherical surfaces such as the projective plane or the torus (see \cref{torus_and_crosscap}).
Cames van Batenburg, Huynh, Joret, and Raymond \cite{CamesVBat2019TightEP} proved that for every  EP-pair $(\mathcal{H},\mathcal{G}_{\mathsf{all}})$ the gap function\footnote{Given two functions $\chi,\psi\colon \mathbb{N}\rightarrow \mathbb{N},$ we write $\chi(n)=\mathcal{O}_{x}(\psi(n))$ to denote that there exists a computable function $f\colon\mathbb{N} \rightarrow \mathbb{N}$ such that $\chi(n)=\mathcal{O}( f(x)\cdot \psi(n)).$
Moreover, we write $\chi(n)=\poly_{x}(\psi(n))$ to denote that there exists a function $f\colon\mathbb{N}\to\mathbb{N}$ such that $\chi(n)=f(x) n^{\mathcal{O}(1)}.$} is $\mathcal{O}_h(k\log k)$ where $h\coloneqq h(\mathcal{H})$ denotes the maximum size of a member of $\mathsf{obs}(\mathcal{H}),$ thereby  matching the bound of \eqref{EP_original}.

Using $\nicefrac{1}{2}\text{-}\barrier_{\Hcal}(G)$ instead of $\barrier_{\Hcal}(G)$ in the discussion above yields the concept of an \defi{$\nicefrac{1}{2}\text{-}$Erd\H{o}s-P{\'o}sa pair}, or a \defi{$\nicefrac{1}{2}\text{-}$EP-pair} for short.
In a major breakthrough result, Liu \cite{liu2022packing} proved\footnote{A proof of Thomas' Conjecture was announced by Norin in 2010 at the SIAM Conference on Discrete Mathematics but it was never published \cite{NorinAnnouncement}. The proof we present in this paper is independent of both Liu's and Norin's work.} the conjecture of Thomas (see \cite{Kawarabayashi2007HalfIntegral}) stating the following:
\begin{eqnarray}
  \begin{minipage}{14cm}
  \centering
   $(\Hcal,\gall)$ is a $\nicefrac{1}{2}$-EP-pair for all proper minor-closed graph classes $\mathcal{H}.$
  \end{minipage}\llabel{half_integral_EP_for_minors}
\end{eqnarray} 
\noindent
A major drawback of Liu's result is that it does neither provide explicit bounds on the gap function for \eqref{half_integral_EP_for_minors}, nor does it provide a constructive way to obtain either the $\Hcal$-modulator, or the half-integral $\Hcal$-barrier.
For this reason, algorithmic applications of \eqref{half_integral_EP_for_minors} are not immediate.
Pursuing a constructive version of Liu's result appears to be a crucial step in the further development of the algorithmic aspects of Graph Minors series and its applications.

\smallskip
The Grid Theorem of Robertson and Seymour \cite{RobertsonS86GMV} implies that the target graph classes $\mathcal{H}$ which together with $\gall$ form EP-pairs are exactly those of bounded treewidth\footnote{A definition of treewidth can be found in \autoref{sec_obstructions_for_Htw}.}.
This leads to the following variant of \eqref{leading_question}:
\begin{quote}
\textsl{Given a minor-closed class $\mathcal{H}$ of unbounded treewidth, what is the most general environment class $\mathcal{G}$ such that $(\mathcal{H},\mathcal{G})$ is an EP-pair?}    
\end{quote}

\paragraph{Counterexamples to the Erd\H{o}s-P{\'o}sa property.}
To answer the question above, we define $\mathfrak{C}_{\Hcal}$ as the collection of all graph classes $\Gcal$ for which $(\Hcal,\Gcal)$ is not an EP-pair, while $(\Hcal,\Gcal')$ is an EP-pair for every minor-closed class $\Gcal'$ properly contained in $\Gcal.$
In other words, $\mathfrak{C}_{\Hcal}$ contains all environment classes that are subset-minimal counterexamples for being an EP-pair with $\Hcal$ as the target class.
Observe that, by \eqref{EP_for_minors}, $\mathfrak{C}_{\Hcal}=\emptyset$ if and only if $\Hcal$ has bounded treewidth.

\smallskip
We wish to stress the ``second order'' nature of the collection of graph classes $\mathfrak{C}_{\mathcal{H}}.$

\paragraph{An order-theoretic perspective.}

With respect to set containment, we may think of $\mathfrak{C}_{\mathcal{H}}$ as an \textsl{obstruction} for a minor-closed environment class $\mathcal{G}$ to form an EP-pair with $\mathcal{H}.$ Indeed, if $\Ebbb_{\Hcal}$
is the collection of all environment minor-closed graph classes that form EP-pairs with the target minor-closed class $\Hcal,$ then 
$\mathfrak{C}_{\mathcal{H}}$ contains the $\subseteq$-minimal minor-closed graph classes that do not belong in $\Ebbb_{\Hcal}.$ In this way, $\mathfrak{C}_{\mathcal{H}}$ can be seen as a \textsl{second-order obstruction} for the collection of classes $\Ebbb_{\Hcal}.$

Therefore, providing a characterization of $\mathfrak{C}_{\Hcal}$ is a way to describe \textsl{all} possible Erd\H{o}s-P{\'o}sa duality theorems for minor-closed graph classes.
The fact that the collection $\mathfrak{C}_{\Hcal}$ exists and is countable for every $\Hcal$ is a consequence of the fact that $(\gall,\leq)$ is a WQO, as proven by Robertson and Seymour \cite{robertson2004GMXX}.
Indeed, this follows by the fact that every graph class in $\mathfrak{C}_{\Hcal}$ is minor-closed and thereby is characterized by a finite set of obstructions.
However, since $\mathfrak{C}_{\Hcal}$ is a \textsl{collection} of graph classes, WQO for the minor relation does not imply that $\mathfrak{C}_{\Hcal}$ is finite.
In fact, the finiteness of $\mathfrak{C}_{\Hcal},$ for all $\Hcal,$ would follow immediately from the stronger conjecture that $(\gall,\leq)$ is an ``$\omega^2$-WQO'' \cite{paul2023graph}.
This conjecture, as well as the more general one that $(\gall,\leq)$ is a Better Quasi Ordering (BQO) remain central open problem in order theory \cite{Requinot2017TowardsBetter}.
Unfortunately, proving these conjectures seems to be a ``hopelessly difficult problem'', as mentioned by Robertson, Seymour, and Thomas in \cite{Robertson1995BQOisHard}.
Moreover, it is unlikely that if such a proof ever appears, it is a constructive one.
Even a constructive proof that $(\gall,\leq)$ is a WQO appears to be out of reach \cite{FriedmanRS87them,krombholz2019upper,fellows1988nonconstructive}.

An alternative way to obtain explicit bounds for the gap function would be a constructive proof for Erd\H{o}s-P{\'o}sa-type dualities.
This would constitute the basis of an algorithmic theory.
For these reasons, to tackle such theorems, it seems imperative to study the collections $\mathfrak{C}_{\Hcal}$ explicitly.

\paragraph{Our results.}
In this paper we  fully characterize the collection $\mathfrak{C}_{\Hcal}$ for every proper minor-closed graph class $\Hcal$ and prove that it is \textsl{finite}.

A \defi{parametric graph} is a minor-monotone sequence $\mathscr{G}=\langle \mathscr{G}_t \rangle_{t\in\mathbb{N}}$ of graphs, i.\@ e.\@ for every $t\in\mathbb{N},$ $\mathscr{G}_{t}\leq \mathscr{G}_{t+1}.$
As we shall see, the concept of parametric graph is central to the characterization of the collection $\mathfrak{C}_{\Hcal},$ for every proper minor-closed graph class $\Hcal.$
Indeed, for every proper minor-closed graph class $\Hcal,$ we construct a \textsl{finite} collection of parametric graphs  $\mathfrak{F}_{\Hcal}^{\nicefrac{1}{2}}=\{\mathscr{F}^{\Hcal,1},\ldots,\mathscr{F}^{\Hcal,r}\},$ $r \coloneqq r(\mathcal{H}),$ consisting of ``wall-like'' parametric graphs, hereafter called \textsl{walloids},
and prove that, every class in $\mathfrak{C}_{\Hcal}$ is the set of all minors of some parametric graph from $\mathfrak{F}_{\Hcal}^{\nicefrac{1}{2}}.$
In particular, we prove the following.

\begin{theorem}[Main result]\llabel{main_result_informal}
For every proper minor-closed graph class $\mathcal{H}$ and every minor-closed graph class $\mathcal{G},$ $(\Hcal,\Gcal)$ is an EP-pair if and only if the class $\mathcal{G}$ excludes a member of $\mathscr{F},$ for every $ \mathscr{F}\in\mathfrak{F}^{\nicefrac{1}{2}}_{\mathcal{H}}.$
\end{theorem}

Notice that \autoref{main_result_informal} indeed provides an answer to the ``when'' part of our leading question \eqref{leading_question}: the absence of a large $\mathcal{H}$-barrier in a graph $G\in\mathcal{G}$ implies the existence of a small $\mathcal{H}$-modulator if and only if $\mathcal{G}$ excludes a graph for each of the classes generated by the parametric graphs in $\mathfrak{F}_{\Hcal}^{\nicefrac{1}{2}}.$
An important feature of $\mathfrak{F}_{\Hcal}^{\nicefrac{1}{2}}$ (see the definition in \eqref{fourth_walloid_formula}) is the following.

\begin{observation}\llabel{obs_half_integral_counterexamples}
For every proper minor-closed graph class $\mathcal{H},$ every $\langle\mathscr{F}_t\rangle_{t\in\Nbbb}\in {\mathfrak{F}_{\Hcal}^{\nicefrac{1}{2}}},$ and every $t\in\mathbb{N}$ we have that $\barrier_{\Hcal}(\mathscr{F}_t)=1.$
\end{observation}

In light of \autoref{obs_half_integral_counterexamples}, the ``why'' part is answered precisely because the parametric graphs in $\mathfrak{F}_{\Hcal}^{\nicefrac{1}{2}}$ are those where it is impossible to pack two copies of some member of $\mathsf{obs}(\mathcal{H}),$ while, with increasing $t,$ they require larger and larger $\mathcal{H}$-modulators.
In this sense, they indeed act as proper counterexamples to the Erd\H{o}s-P{\'o}sa property in the exact same way as the surface-based examples of Robertson and Seymour \cite{RobertsonS86GMV}.

For each proper minor-closed class $\Hcal,$ the set $\mathfrak{F}^{\nicefrac{1}{2}}_{\Hcal}$ is constructed using the obstructions of $\Hcal$ as a starting point. 
We postpone the description of this construction to \autoref{subsec_universal_obstructions}.
We stress however that for every class $\Hcal$ of bounded treewidth, as $\mathfrak{C}_{\Hcal}=\emptyset,$ we have that $\mathfrak{F}_{\Hcal}^{\nicefrac{1}{2}}=\emptyset.$
As a first non-trivial example, let us consider $\mathcal{H}$ to be a minor-closed class such that $\mathsf{obs}(\mathcal{H})$ contains no planar graphs and at least one single-crossing minor\footnote{A graph is a \emph{single-crossing minor} if it is a minor of some graph that can be drawn in the plane with at most one crossing.}.
It turns out that $\mathfrak{F}^{\nicefrac{1}{2}}_{\Hcal}$ is composed of the two parametric graphs that are depicted in \autoref{torus_and_crosscap}\footnote{In \cite{PaulPTW24Delineating}, we have proven, among others, that the parametric graphs depicted in \autoref{torus_and_crosscap} are the two parametric graphs in $\mathfrak{F}_{\Hcal}^{\nicefrac{1}{2}}$ for every minor-closed graph class $\Hcal$ whose all obstructions are non-planar and at least one of them is single-crossing minor. The result of the first paper only applies to a restricted collection of classes $\mathcal{H}$ whose obstruction set contains at least one graph that belongs to a \textsl{specific} proper subclass of $K_8$-minor-free graphs. We explain this in more detail in \autoref{sec_conclusion}.}.

\begin{figure}[ht]
  \begin{center}
  \scalebox{.6}{\includegraphics{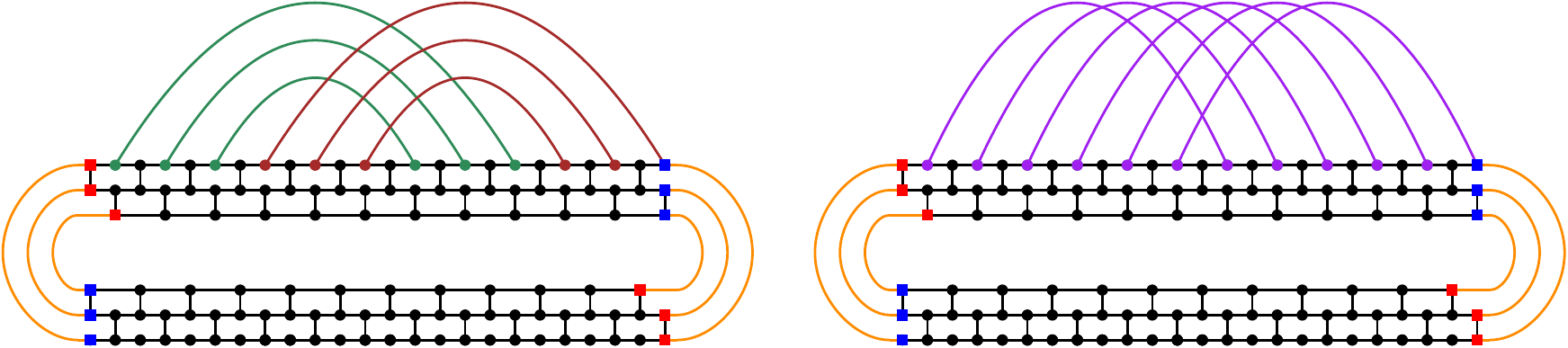}}
  \end{center}
    \caption{Two wall-like graphs in $\mathfrak{F}^{\nicefrac{1}{2}}_{\Hcal}$ when $\Hcal$ is the class of $K_{3,3}$-minor free graphs. }
  \llabel{torus_and_crosscap}
\end{figure}

This implies that $\mathfrak{C}_{\Hcal}$ is composed of the class of all toroidal graphs and the class of all projective planar graphs.
This generalizes \eqref{EP_for_minors}, by replacing the condition that the target graph class has a planar obstruction, to the condition that the target graph class $\Hcal$ has obstructions that are all non-planar and that at least one of them is a single-crossing minor, obtaining that, in this case, $(\mathcal{H},\mathcal{G})$ is an EP-pair if and only if $\mathcal{G}$ excludes some toroidal and some projective planar graph.

In fact, for every class $\mathcal{H}$ whose obstruction set does not contain a planar graph, $\mathfrak{F}^{\nicefrac{1}{2}}_{\mathcal{H}}$ contains a parametric graph representing some orientable surface other than the sphere and a parametric graph representing a non-orientable surface.
However, in general not all members of $\mathfrak{F}^{\nicefrac{1}{2}}_{\mathcal{H}}$ correspond to surfaces.
As we shall see in \autoref{subsec_universal_obstructions}, the members of $\mathfrak{F}^{\nicefrac{1}{2}}_{\mathcal{H}}$ may be arbitrarily far away from being surface-embeddable. A wide family of $\Hcal$'s where 
$\mathfrak{F}^{\nicefrac{1}{2}}_{\mathcal{H}}$ indeed corresponds to surfaces is studied in \cite{PaulPTW24Delineating}.

Finally, we wish to stress that our results yield two additional insights.
The first is that
\begin{eqnarray}
  \begin{minipage}{14cm}
    \centering
  $|\mathfrak{F}^{\nicefrac{1}{2}}_{\Hcal}|=|\mathfrak{C}_{\Hcal}|\in 2^{2^{\mathcal{O}(\ell(h^2))}}$ for every proper minor-closed graph class $\mathcal{H},$
  \end{minipage}\llabel{half_int_obs}
\end{eqnarray}
where $\ell(\cdot)$ is the so-called \defi{unique linkage function} (see \cite{Robertson2009GMXXI,Kawarabayashi2010Shorter,mazoit2013single,AdlerKKLST17Irrelevant,GolovachST22Combing} and \autoref{prop_cae}) and $h\coloneqq h(\mathcal{H})$ is the maximum size of a member of $\mathsf{obs}(\mathcal{H}).$

The second resides on following  important feature of $\mathfrak{F}_{\Hcal}^{\nicefrac{1}{2}}$: each parametric graph $\mathscr{F} = \langle \mathscr{F}_t \rangle_{t\in\Nbbb}\in \mathfrak{F}_{\Hcal}^{\nicefrac{1}{2}}$ consists of wall-like graphs that contain a half-integral $\Hcal\text{-}\barrier$ whose size increases as a function in their index $t$ (see \autoref{lemma_half_integral_Brambles}).
Likewise, the wall-like objects that ``obstruct'' the integral Erd\H{o}s-P{\'o}sa duality (see \autoref{main_result_informal}), do serve as certificates of the \textsl{half-integral} Erd\H{o}s-P{\'o}sa duality.
As a corollary of this, we obtain a constructive proof of Thomas' conjecture which implies that the gap function of \eqref{half_integral_EP_for_minors} is bounded by ${2^{k^{\mathcal{O}_{h}(1)}}}.$ 

\subsection{Min-max dualities and a high-level overview of our approach}\llabel{sec_min_max_dualities}
At its core, our approach is very close to the spirit of the Graph Minor series of Robertson and Seymour.
That is, we do not immediately obtain \autoref{main_result_informal}.
Instead, we take a detour and first describe the structure of graphs that ``locally'' exclude a large (half-integral) $\mathcal{H}$-barrier.

The Grid Theorem \cite{RobertsonS86GMV} states that there exists a function $f\colon\mathbb{N}\to\mathbb{N}$ such that for every graph $G$ and every integer $k,$ if $G$ has treewidth at least $f(k)$ then $G$ contains a \textsl{$k$-wall}\footnote{See \autoref{subsec_segments} for the definition of walls.} as a subgraph.
For this we need two ingredients.
The first is a description of what ``local'' with respect to a wall means.
Let $k$ be a positive integer, $W$ be a $k$-wall in a graph $G$ and $S\subseteq V(G)$ be a set of \textsl{strictly less} than $k$ vertices.
Then there exists at least one pair of a horizontal and a vertical path of $W$ which avoids $S.$
Notice that all possible choices of such pairs belong to the same connected component of $G-S.$
We say that this component \defi{contains the majority of $G-S$}\footnote{For the reader familiar with the theory of graph minors it is probably apparent that this notion does indeed define a tangle. However, for our purposes it suffices to deal with tangles only in this implicit fashion.}.
The second ingredient is another collection of parametric graphs whose explicit definition we postpone to \autoref{subsec_universal_obstructions}.
For every proper minor-closed graph class $\mathcal{H},$ we will describe a finite collection of parametric graphs $\mathfrak{W}_{\mathcal{H}} = \{ \mathscr{W}^{\mathcal{H},1},\dots,\mathscr{W}^{\mathcal{H},w}\},$ where $w=w(\mathcal{H})$ is a constant depending only on the class $\mathcal{H},$ consisting of ``wall-like'' parametric graphs which have, among others, the following three important properties:
\begin{enumerate}
  \item if, for some $\mathscr{W}=\langle \mathscr{W}_t\rangle_{t\in\mathbb{N}}\in\mathfrak{W}_{\mathcal{H}},$ there does not exist $q\in\mathbb{N}$ such that $\mathscr{W}_q$ contains two disjoint copies of some member of $\mathsf{obs}(\mathcal{H})$ as a minor, then $\mathscr{W}\in\mathfrak{F}^{\nicefrac{1}{2}}_{\mathcal{H}},$
  \item there exists $f^1_{\mathcal{H}}\colon\mathbb{N}\to\mathbb{N}$ such that for every $\mathscr{W}=\langle W_t\rangle_{t\in\mathbb{N}}\in\mathfrak{W}_{\mathcal{H}}\setminus\mathfrak{F}^{\nicefrac{1}{2}}_{\mathcal{H}}$ and every $k\in\mathbb{N},$ $\mathscr{W}_{{f^1_{\mathcal{H}}}(k)}$  contains an $\mathcal{H}$-barrier of size $k,$ and 
  \item there exists a function $f^2_{\mathcal{H}}\colon\mathbb{N}\to\mathbb{N}$ such that for every $i\in[w]$ and every positive integer $k,$ $\mathscr{W}^{\mathcal{H},i}_{f^2_{\mathcal{H}}(k)}$ contains a half-integral $\mathcal{H}$-barrier of size $k$ and an $f^2_{\mathcal{H}}(k)$-wall such that the wall and the barrier cannot be separated by fewer than $k$ vertices.
\end{enumerate}

At the heart of our proof sits the following ``local structure theorem''.
We present here an equivalent but slightly simplified version of \autoref{thm_local_structure} which avoids a lot of the heavier notations and technical definitions.

\begin{theorem}[Local structure theorem (simplified)]\llabel{local_structure_informal}
  For every proper minor-closed graph class $\mathcal{H}$ there exist functions $f^1,f^2\colon\mathbb{N}\to\mathbb{N}$ such that for every positive integer $k$ and every graph $G,$ if $W$ is a $f^1(k)$-wall in $G,$ then
   \begin{itemize}
      \item there exists a parametric graph $\mathscr{W}=\langle \mathscr{W}_t\rangle_{t\in\mathbb{N}}\in\mathfrak{W}_{\mathcal{H}}$ such that $G$ contains a subdivision $W'$ of $\mathscr{W}_k$ and $W'$ cannot be separated from $W$ by fewer than $k$ vertices, or
      \item there exists a set $S\subseteq V(G)$ of size at most $f^2(k)<f^1(k)$ such that the component of $G-S$ which contains the   majority of the degree-$3$-vertices of $W-S$ belong to $\mathcal{H}.$
   \end{itemize}
\end{theorem}

With \autoref{local_structure_informal} we have a powerful tool at our disposal which allows us to describe any graph $G$ which excludes a member of every parametric graph in $\mathfrak{W}_{\mathcal{H}}$ in terms of a \textsl{tree decomposition} into pieces that are either small or belong to $\mathcal{H}.$
This type of behavior is very reminiscent of the Grid Theorem of Robertson and Seymour \cite{RobertsonS86GMV}, but also of their general structure theorem for $H$-minor-free graphs \cite{Robertson2003GMXVI}, and it is similar to other, more recent results such as the one in \cite{thilikos2022killing}.
To provide a description of this intermediate result and how it leads to \autoref{main_result_informal}, we require some additional definitions.
In the following we describe, roughly, how all of our results reduce to \autoref{local_structure_informal}.

A (very) brief outline of the proof of \autoref{local_structure_informal} is given in \autoref{subsec_proof_summary}, for a slightly more in-depth discussion and overview of the proof, including several of the more technical definitions we postponed so far, we refer the reader to \autoref{subsec_universal_obstructions} and \autoref{proof_outline_extended}.

\paragraph{Parametric families.}
Let  $\mathscr{G}=\langle \mathscr{G}_t\rangle_{t\in\mathbb{N}}$ and $\mathscr{G}'=\langle \mathscr{G}_t'\rangle_{t\in\mathbb{N}}$ be two parametric graphs.
We write $\mathscr{G} \lesssim \mathscr{G}',$ if there is some function $f\colon\nn{1}{1}$ such that $\mathscr{G}_t≤ \mathscr{G}_{f(t)}'$ for all $t\in\mathbb{N}.$
If $\mathscr{G} \lesssim \mathscr{G}'$ and $\mathscr{G}' \lesssim \mathscr{G}$ we say that $\mathscr{G}$ and $\mathscr{G}'$ are \defi{equivalent}, denoted by $\mathscr{G}\equiv \mathscr{G}'.$

Let $\mathfrak{G}$ and $\mathfrak{G}'$ be parametric families.
We say that $\mathfrak{G}\lesssim^{*}\mathfrak{G}'$ if for every $\mathscr{G}'\in \mathfrak{G}'$ there is some $\mathscr{G}\in \mathfrak{G}$ such that $\mathscr{G}\lesssim \mathscr{G}'$\footnote{This ``lifting'' of the relation $\lesssim$ on parametric graphs to families of parametric graphs is called the \textsl{Smyth extension} of $\lesssim$ in the theory of WQOs.}.
Moreover, $\mathfrak{G}$ and $\mathfrak{G}'$ are said to be \defi{equivalent} if $\mathfrak{G}\lesssim^{*}\mathfrak{G}'$ and $\mathfrak{G}'\lesssim^{*}\mathfrak{G},$ this is denoted by  $\mathfrak{G} \equiv^{*} \mathfrak{G}'.$
A parametric family $\mathfrak{G}$ is an \defi{antichain} if for every pair $\mathscr{G},\mathscr{G}'\in\mathfrak{G}$ with $\mathscr{G}\neq\mathscr{G}'$ it holds that $\mathscr{G}\not\lesssim\mathscr{G}'$ and $\mathscr{G}'\not\lesssim\mathscr{G}.$
If $\mathfrak{G}$ is a parametric family, we say that a subset $\mathfrak{G}'$ of $\mathfrak{G}$ is a \defi{minimization} of $\mathfrak{G}$ if $\mathfrak{G}'\equiv^{*}\mathfrak{G}$ and $\mathfrak{G}'$ is an antichain.
Notice that all minimizations of $\mathfrak{G}$ are pairwise equivalent.
This permits us to use $\mathsf{min}(\mathfrak{G})$ to denote an arbitrarily chosen minimization of $\mathfrak{G}.$

\paragraph{Graph parameters and universal obstructions.}
A graph parameter is a function $\p\colon\gall\to\Nbbb$ mapping graphs to non-negative integers.
Given two graph parameters $\p$ and $\mathsf{q},$ we write $\p \preceq \mathsf{q}$  if there is a function $f\colon\mathbb{N} \to \mathbb{N}$ such that $\p(G) \leq f(\mathsf{q}(G))$ for every graph $G.$
We refer to $f$ as the \defi{gap function} of this relation.
It is not hard to observe that $\preceq$ is a quasi-ordering on the set of graph parameters.
We also say that $\p$ and $\mathsf{q}$ are \defi{equivalent} with gap function $f,$ denoted by $\p \sim \mathsf{q},$ if $\p \preceq \mathsf{q}$ and $\mathsf{q} \preceq \p$ where $f$ is a gap function witnessing both relations.
Let $\mathfrak{P}$ be a finite set of graph parameters. 
We define $\max(\mathfrak{P})$ as the parameter $\p(G) \coloneqq \max\{\mathsf{q}(G) \mid \mathsf{q}\in\mathfrak{P}\}.$

Given a parametric graph $\mathscr{G}=\langle \mathscr{G}_t\rangle_{t\in\mathbb{N}},$ we define the graph parameter $\p_{\mathscr{G}}\colon\gall\to\Nbbb$
to be
\begin{align*} 
\p_{\mathscr{G}}(G)\coloneqq \max\{t\mid \mathscr{G}_{t}\leq G\}.
\end{align*}
Let $\mathfrak{G}$ be a parametric family.
We define $\p_{\mathfrak{G}}(G)\coloneqq \max(\{\p_{\mathscr{G}}\mid \mathscr{G}\in\mathfrak{G}\}).$
One may now observe that $\mathfrak{G}^1\lesssim^{*}\mathfrak{G}^2$ if and only if $\p_{\mathfrak{G}^1}\preceq\p_{\mathfrak{G}^2}.$
We say that a parametric family $\mathfrak{G}$ is an \defi{obstructing set} for a graph parameter $\p$ if $\p\sim\p_\mathfrak{G}.$
Moreover, $\mathfrak{G}$ is a \defi{universal obstruction} for $\p$ if it is an obstructing set for $\p$ and a $\lesssim$-antichain.
Observe that $\mathfrak{G}_1$ and  $\mathfrak{G}_2$ are both obstructing sets for the same graph parameter if and only if $\mathfrak{G}_1\equiv^{*}\mathfrak{G}_2.$ 
\medskip

With these notions at hand, we are now ready to describe how our approach yields a far-reaching generalization of the Grid Theorem of Robertson and Seymour.
Indeed, as we will see, the parameter $\mathsf{p}_{\mathfrak{W}_{\mathcal{H}}}$ is equivalent to a generalization of treewidth which has received considerable attention in recent years, specifically by the algorithmic community.

\paragraph{$\Hcal$-treewidth.}
Given a graph $G$ and a set $X\subseteq V(G),$ we define $\mathsf{torso}(G,X)$ as the graph obtained from $G[X]$ by turning, for each component $K$ of $G-X,$ the set $N_G(V(K))$ (i.e., the neighbors of the component in $X$) into a clique. 
Given a class $\Hcal,$ the parameter \defi{$\Hcal$-treewidth}, denoted by $\Hcal$-$\tw,$ is defined as follows.
\begin{align}
\mathcal{H}\text{-}\mathsf{tw}(G) \coloneqq \min \big\{  k\in\mathbb{N} \mid&  \text{\ there exists }X\subseteq V(G)\text{ such that }\mathsf{tw}(\mathsf{torso}(G,X))\leq k\text{ and}\llabel{basic_scheme}\\
&\text{ every component of }G-X\text{ belongs to }\mathcal{H} \big\}\nonumber
\end{align}
Here $\mathsf{tw}(G)$ denotes the \textsl{treewidth} of the graph $G$ (see \autoref{sec_obstructions_for_Htw} for a definition).

The above definition was proposed by  Eiben, Ganian, Hamm, and Kwon \cite{EibenGHK21} as a natural extension of treewidth that measures the tree decomposability of a graph into pieces belonging to some class $\Hcal$ and has interesting algorithmic applications (see \cite{EibenGHK21,JansenK021verte,AgrawalKLPRSZ22delet,Agrawal2022Distance,Jansen20235Approx,inamdar2023fpt,MorelleSST23fast}).
A direct consequence of the Grid Theorem of Robertson and Seymour is that $\Hcal$-$\mathsf{tw}\sim\tw$ if and only if $\mathsf{obs}(\Hcal)$ contains a planar graph (see \autoref{obs_Htw_boundedtw}).

Our first core result is that, for every $\Hcal,$ $\mathfrak{W}_{\Hcal}$ is an obstructing set for $\Hcal$-$\tw.$
For this, we need two results.
The first is the following.

\begin{theorem}
\llabel{thm_H_tw_lowerbound}
For every proper minor-closed graph class $\Hcal,$ every parametric graph $\mathscr{W}=\langle\mathscr{W}_t\rangle_{t\in\Nbbb}$ in $\mathfrak{W}_{\Hcal},$ and every $t\in\mathbb{N}$ it holds that $\Hcal\text{-}\tw(\mathscr{W}_t)\in \Omega_{h}(t).$
\end{theorem}

The proof of \autoref{thm_H_tw_lowerbound} is based on an appropriate adaptation of the notion of \textsl{brambles} to the concept of $\mathcal{H}$-treewidth.
Brambles are known to act as lower bounds to treewidth \cite{Reed1997Treewidth}.
For the notion of $\mathcal{H}$-treewidth we adopt the following definition. 
A \defi{strict\footnote{Notice that the original notion of brambles (see for example \cite{Reed97Tree}) allows for two bramble elements to be disjoint if they are joined by an edge.
While we believe that such a (more relaxed) definition of $\mathcal{H}$-brambles could yield even tighter duality results, strict $\mathcal{H}$-brambles occur more naturally within our proofs.} $\mathcal{H}$}-bramble in a graph $G$ is a collection $\mathcal{B}$ of connected subgraphs of $G$ which pairwise intersect such that $\mathcal{B}\cap\mathcal{H}=\emptyset.$
The \defi{order} of a strict $\mathcal{H}$-bramble $\mathcal{H}$ is the minimum size of a hitting set for all elements of $\mathcal{B}.$
In \autoref{fig_intro_obstructing} we provide a sketch of the construction of a strict $\mathcal{H}$-bramble of large order in a member of $\mathfrak{W}_{\mathcal{H}}.$
The fact that the members of $\mathfrak{W}_{\mathcal{H}}$ contain large order strict $\mathcal{H}$-brambles stems from their construction.
We show that for each $\langle\mathscr{W}_t\rangle_{t\in\mathbb{N}}\in\mathfrak{W}_{\mathcal{H}}$ there exists some constant $c=c(\mathcal{H})$ such that $\mathscr{W}_{c}$ contains a member of $\mathsf{obs}(\mathcal{H})$ as a minor (see \autoref{lemma_walloid_contains_obstruction}).
In fact, for increasing $t,$ $\mathscr{W}_t$ contains a large half-integral $\mathcal{H}$-barrier consisting of such minors and this barrier also acts as the strict $\mathcal{H}$-bramble.

\begin{figure}[ht]
\begin{center}
\scalebox{.859}{\includegraphics{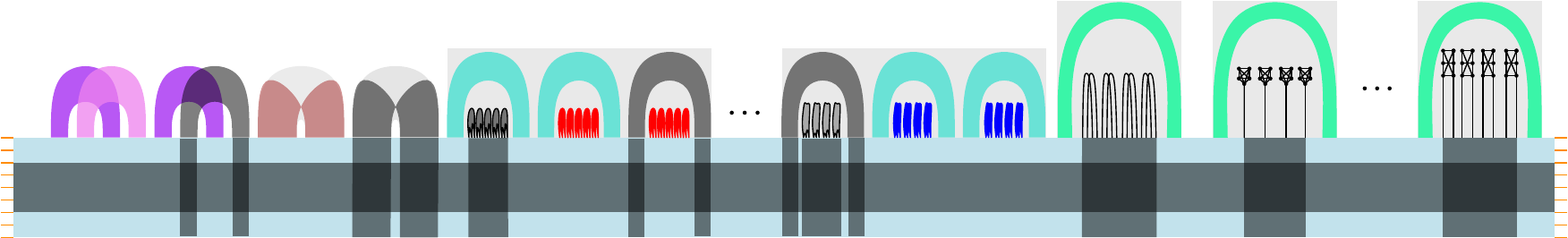}}
\end{center}
      \caption{A strict bramble of subgraphs, all obstructing membership to $\mathcal{H},$ in a member $\mathscr{W}_k$ of $\mathfrak{W}_{\mathcal{H}}.$
     The blue strip at the bottom represents the large wall that constitutes the basis of $\mathscr{W}_k.$\\
    At the top of this wall, we attach gadgets with different properties.
    These gadgets are, from left to right, two ``handle segments'' and two ``crosscap segments'' which model a surface.
    What follows are several ``flower segments''. These segments model parts of an obstruction $Z\in\mathsf{obs}(\mathcal{H})$ which attach to the rest of the graph through few (that is less than three) vertices.
    Each ``flower segment'' hosts $k$ disjoint copies of the same part.
    Finally, on the far right, there are additional ``flower segments''.
    These flowers still model parts of $Z,$ but these parts are now allowed to attach to the rest of the graph via boundaries of size larger than three (although we will show that the boundaries are still of bounded size).\\
    The gray area serves two purposes.
    First, it illustrates one possible way $Z$ might fit into $\mathscr{W}_k$ as a minor.
    Second, it indicates how to draw several copies of $Z$ in a way that no vertex of $\mathscr{W}_k$ is contained in more than two copies and that any two copies intersect. Such a family of copies of $Z,$ embedded as minors into $\mathscr{W}_k$ forms the strict bramble.
    See \autoref{fig_half_integral_packing} and \autoref{lemma_half_integral_Brambles} for a more detailed description.}
    \llabel{fig_intro_obstructing}
\end{figure}

The proof that strict $\mathcal{H}$-brambles of large order act as lower bounds to $\mathcal{H}$-treewidth follows the line of fairly standard arguments (see \autoref{lemma_bramble_lowerbound}).
The most technical part of this paper is the upper bound towards the parametric equivalence between $\mathcal{H}\text{-}\mathsf{tw}$ and $\mathsf{p}_{\mathfrak{W}_{\mathcal{H}}}.$
This proof spans \autoref{Thidplodowinga} to \autoref{sec_Htw}.
The key step, which is the proof of \autoref{local_structure_informal}, can be found in \autoref{sec_local_structure} while the most involved and novel technical parts are confined to \autoref{sec_harvest_crops} and \autoref{extr_flowers_thickets}.

Due to the wall-like nature of the parametric graphs in $\mathfrak{W}_{\Hcal},$ \autoref{main_grid_general} and \autoref{thm_H_tw_lowerbound} provide a \textsl{grid theorem} for $\Hcal$-treewidth for \textsl{every} proper minor-closed graph class $\Hcal.$
The final proof of \autoref{main_grid_general} utilizes \autoref{local_structure_informal} together with known balanced separator arguments. 

\begin{theorem}[General grid theorem]
  \llabel{main_grid_general}
  There exists a function $f_{\ref{main_grid_general}}\colon\nn{2}{1}$ such that for every proper minor-closed graph class $\Hcal,$ where $h=h(\Hcal),$ every graph $G,$ and every $k\in\Nbbb,$ if $\Hcal\text{-}\tw(G)\geq f(h,k),$ then there exists $\langle\mathscr{W}_t\rangle_{t\in\Nbbb}\in\mathfrak{W}_{\Hcal}$ such that $G$ contains a subdivision of $\mathscr{W}_k.$
   Moreover, $f_{\ref{main_grid_general}}(h,k)={2^{k^{\mathcal{O}_h(1)}}}.$
 \end{theorem}

As a direct corollary, we also obtain a similar statement for the set $\mathfrak{W}_{\mathcal{H}}^{\mathsf{min}}$ of universal obstructions.

\begin{corollary}
  \llabel{cor_grid_general}
  There exists a function $f_{\ref{cor_grid_general}}\colon\nn{2}{1}$ such that for every proper minor-closed graph class $\Hcal,$ where $h=h(\Hcal),$ every graph $G,$ and every $k\in\Nbbb,$ if $\Hcal\text{-}\tw(G)\geq f(h,k),$ then there exists $\langle\mathscr{W}_t\rangle_{t\in\Nbbb}\in\mathfrak{W}^{\mathsf{min}}_{\Hcal}$ such that $G$ contains a subdivision of $\mathscr{W}_k.$
\end{corollary}

\paragraph{From $\Hcal\text{-}\tw$ to $\apex_{\Hcal}.$}
To express the structural behavior of $\mathsf{apex}_{\mathcal{H}}$ in the form of an equivalence between parameters we must consider the family $\mathfrak{F}^{\nicefrac{1}{2}}_{\mathcal{H}}$ and all possible $\mathcal{H}$-barriers.
We define (see \eqref{second_walloid_formula} and the surrounding definitions) $\mathfrak{F}_{\mathcal{H}}$ to be the union of $\mathfrak{F}_{\mathcal{H}}^{\nicefrac{1}{2}}$ and the parametric family $\{ \mathscr{W}^{\mathbf{p}_Z} \mid Z\in\mathsf{obs}(\mathcal{H}) \}$ where $\mathscr{W}^{\mathbf{p}_Z}\coloneqq \langle t\cdot Z\rangle_{t\in\mathbb{N}}.$ 

With this, we obtain the following lower bound towards our desired parametric equivalence between the parameters $\mathsf{apex}_{\mathcal{H}}$ and $\mathsf{p}_{\mathfrak{F}_{\mathcal{H}}}.$

\begin{observation}\llabel{obs_apex_lower_bound}
For every proper minor-closed graph class $\Hcal$ and every parametric graph $\langle\mathscr{W}_t\rangle_{t\in\Nbbb}\in \mathfrak{F}_{\Hcal}$ it holds that $\apex_{\Hcal}(\mathscr{W}_t)=\Omega_{h}(t)$ where $h\coloneqq h(\mathcal{H})$ is the maximum size of a member of $\mathsf{obs}(\mathcal{H}).$
\end{observation}
 
The less trivial direction of the equivalence between $\apex_{\Hcal}$ and $\p_{\mathfrak{F}_{\Hcal}}$ and, in particular, obtaining its single exponential gap function are mostly a matter of carefully combining our results on $\mathcal{H}$-treewidth with established techniques from the study of the Erd\H{o}s-P{\'o}sa property for minors.
 
\begin{theorem}\llabel{thm_apex_upper_bound}
There exists a function $f_{\ref{thm_apex_upper_bound}}\colon\nn{2}{1}$ such that for every proper minor-closed graph class $\Hcal,$ where $h=h(\Hcal),$ every graph $G,$ and every $k\in\Nbbb,$ if $\apex_{\Hcal}(G)\geq f_{\ref{thm_apex_upper_bound}}(h,k),$ then there exists $\langle\mathscr{F}_t\rangle_{t\in\Nbbb}\in \mathfrak{F}_{\Hcal}$ such that $G$ contains $\mathscr{F}_k$ as a minor.
Moreover, $f_{\ref{thm_apex_upper_bound}}(h,k)={2^{k^{\mathcal{O}_h(1)}}}.$
\end{theorem}

The proof of \autoref{thm_apex_upper_bound} makes use of the three core properties of the family $\mathfrak{W}_{\mathcal{H}}$ which allow, by \autoref{main_grid_general} to reduce to the case of bounded $\mathcal{H}\text{-}\mathsf{tw}.$
This reduction goes as follows:
If we find a large enough member of one of the parametric graphs in $\mathfrak{F}^{\nicefrac{1}{2}}_{\mathcal{H}}$ we are done.
Moreover, by combining the first and the second core property of $\mathfrak{W}_{\mathcal{H}}$ we know that, whenever we find a large enough member of one of the parametric graphs in $\mathfrak{W}_{\mathcal{H}}\setminus \mathfrak{F}^{\nicefrac{1}{2}}_{\mathcal{H}},$ we have found a member of one of the parametric graphs in $\{ \mathscr{W}^{\mathbf{p}_Z} \mid Z\in\mathsf{obs}(\mathcal{H}) \}.$
Hence, whenever $\mathcal{H}$-treewidth of some graph $G$ is large enough, we immediately obtain that $\mathsf{p}_{\mathfrak{F}_{\mathcal{H}}}(G)$ is also large by \autoref{main_grid_general}.
It follows that the only instances we have to check are those where $\mathcal{H}$-treewidth is bounded.

From here, we are able to use standard arguments in two steps.
First, we show the theorem for every class, minor-excluding some \textsl{connected} graph $Z$  (\cref{univ_apex_con_discon}).
Roughly speaking, we employ the same recursive method as described by Diestel in his proof of \eqref{EP_for_minors} (see \cite{Diestel2010Book}).
This leaves open only cases where $\mathsf{obs}(\mathcal{H})$ contains disconnected graphs.
These cases need some additional arguments using the connected case as a basis (\cref{disc_case_obs}).
\smallskip

It is worthwhile to mention that the \textsl{mechanism} that allows us to deduce a finite obstructing set for $\mathsf{apex}_{\mathcal{H}}$ appears far more universal.
In \autoref{sec_conclusion} we discuss this approach in a more general setting and conjecture that our results are the basis for a whole family of similar results based on the knowledge of obstructing sets for certain parameters.
As a proof of concept, we apply this mechanism to the parameter \textsl{$\mathcal{H}$-treedepth} (also known as \textsl{elimination distance to $\mathcal{H}$}) in \autoref{univ_td_all} and obtain a finite obstructing set for every choice of a proper minor-closed graph class $\mathcal{H}.$
This parameter was originally introduced by Bulian and Dawar \cite{BulianD16grap,BulianD17fixe} and has inspired its own avenue of research \cite{EibenGHK21,JansenK021verte,AgrawalKLPRSZ22delet,Agrawal2022Distance,Jansen20235Approx,inamdar2023fpt,MorelleSST23fast}.

\paragraph{A sketch of the proof of \autoref{main_result_informal}.}
By the definition, $\mathfrak{F}_{\mathcal{H}}$ is the union of two parametric families, namely $\mathfrak{F}^{\nicefrac{1}{2}}_{\mathcal{H}}$ and $\{ \mathscr{W}^{\mathbf{p}_Z} \mid Z\in\mathsf{obs}(\mathcal{H}) \}$ where the latter is consisting exclusively of $\mathcal{H}$-barriers.
Now let $h=h(\mathcal{H})$ be the maximum size of a member of $\mathsf{obs}(\mathcal{H})$ and let $\mathcal{G}$ be a minor-closed graph class.
If there exists $\mathscr{F}=\langle \mathscr{F}_t\rangle_{t\in\mathbb{N}}\in\mathfrak{F}^{\nicefrac{1}{2}}_{\mathcal{H}}$ such that $\mathscr{F}_t\in \mathcal{G}$ for every $t\in\mathbb{N},$ then, by \autoref{obs_half_integral_counterexamples} and \autoref{obs_apex_lower_bound}, it follows that $(\mathcal{H},\mathcal{G})$ is not an Erd\H{o}s-P{\'o}sa pair.
For the reverse direction, we may assume that for each $\mathscr{F}=\langle \mathscr{F}_t\rangle_{t\in\mathbb{N}}\in\mathfrak{F}^{\nicefrac{1}{2}}_{\mathcal{H}},$ there exists some $t_{\mathscr{F}}\in\mathbb{N}$ such that $\mathscr{F}_t\notin\mathcal{G}$ for all $t\geq t_{\mathscr{F}}.$
Let $t^{\star}\coloneqq \max \{ t_{\mathscr{F}} \mid \mathscr{F}\in\mathfrak{F}^{\nicefrac{1}{2}}_{\mathcal{H}} \}.$
It follows that, for every $k\in\mathbb{N},$ if for any graph $G\in\mathcal{G}$ it holds that $\mathsf{apex}_{\mathcal{H}}(G)\geq f_{\ref{thm_apex_upper_bound}}(h,t^{\star}+1+k),$ then, by \autoref{thm_apex_upper_bound}, $G$ must contain an $\mathcal{H}$-barrier of size $k.$
Therefore, $(\mathcal{H},\mathcal{G})$ is indeed an Erd\H{o}s-P{\'o}sa pair.

\subsection{Algorithmic consequences}
\llabel{sec_algo_consequences}

Let us address some important algorithmic consequences of our results along with subsequent open problems.

\paragraph{A constructive (approximation) \textsf{FPT}-algorithm for half-integral packing.}
\llabel{page_Half_Integral_packing}

Our results imply, for every proper minor-closed graph class $\Hcal,$ the existence of an algorithm $\mathbf{A}_{\Hcal}$ that, given a graph $G$ and an integer $k,$ either outputs a half-integral $\Hcal$-barrier of size  $k$ or an $\Hcal$-modulator of size at most $2^{k^{\mathcal{O}(1)}}$ in time $2^{2^{{k^{\mathcal{O}_{h}(1)}}}}\cdot |G|^4\log|G|.$ 
Observe that $\mathbf{A}_{\Hcal}$ can be used as an \textsf{FPT}-approximation algorithm for the following problem, defined for every graph $Z.$
Given a graph $Z,$ we define $\mathsf{excl}(Z)$ to be the class of all graphs that exclude $Z$ as a minor.
 
\begin{center}
\begin{minipage}{12cm}
\noindent \textsc{Half-Integral $Z$-Packing}
\begin{description}
\vspace{-0.7em}
  \item[Input:] A graph $G$ and a non-negative integer $k.$
\vspace{-0.7em}
  \item[Question:] Does $G$ contain a half-integral $\mathsf{excl}(Z)$-barrier of size $k$?
\end{description}
\end{minipage}
\end{center}

Due to the algorithmic consequences of the Robertson and Seymour theorem, there \textsl{exists} an algorithm that resolves the problem above, for every $\Hcal,$ in time $f(k)\cdot |G|^2$ (see \cite{fellows1988nonconstructive}), however it is non-constructive as it is based on the knowledge of the set $\obs(\{G\mid \nicefrac{1}{2}\text{-}\barrier_{\Hcal}(G)\leq k\}).$
Due to the lack of an explicit upper bound on the size of this set and of its members, so far no constructive algorithm is known for any non-planar graph $Z.$
We observe that for the class $\mathsf{excl}({Z}),$ the algorithm $\mathbf{A}_{\mathsf{excl}(Z)}$ discussed above, either decides that $(G,k)$ is a \textsf{yes}-instance of \textsc{Half-Integral $Z$-Packing} or that $(G,f(k))$ is a \textsf{no}-instance of \textsc{Half-Integral $Z$-Packing}, where $f(k)={2^{k^{\mathcal{O}_h(1)}}}.$
This constitutes the first constructive (approximation) \textsf{FPT}-algorithm for the \textsc{Half-Integral $Z$-Packing} problem.
In \autoref{sec_conclusion} we discuss further algorithmic consequences of our results.

\paragraph{Deciding when an enviroment class has the Erd\H{o}s-P{\'o}sa property.}
Given a minor antichain $\Rcal,$ we denote by $\mathsf{excl}(\Rcal)$ the class of graphs that excludes the graphs in $\Rcal$ as minors, i.\@ e.\@, the class $\gall\cap\bigcap_{R\in\mathcal{R}}\mathsf{excl}(R).$
Consider the following problem for any given non-empty minor antichain $\Rcal.$

\begin{center}
\begin{minipage}{12cm}
\noindent \textsc{Checking EP-pairs for $\Rcal$-minor free target}
\begin{description}
    \vspace{-0.7em}
    \item[Input:] An antichain $\Zcal$ for the minor relation.
    \vspace{-0.7em}
    \item[Output:] Is $\big(\mathsf{excl}(\Rcal),\excl(\Zcal)\big)$ an EP-pair?
\end{description}
\end{minipage}
\end{center}

The antichain $\Rcal$ (resp. $\Zcal$) provides a finite representation of a target (resp. environment) class and the question is whether the fixed target class $\excl(\Rcal)$ makes an EP-pair with the input environment class $\excl(\Zcal).$

For every $\mathscr{F} = \lin{ \mathscr{F}_{t}}_{t\in\Nbbb} \in \mathfrak{F}_{\mathsf{excl}(\Rcal)},$ we define $\mathscr{F}\!\!\downarrow\coloneqq \{G\mid \text{$G$ is a minor of $\mathscr{F}_{t}$ for some $t\in\Nbbb$}\}.$
Notice that, by definition, $\mathscr{F}\!\!\downarrow$ is a minor-closed class.
Let now 
$$\Obbb_{\Rcal}\coloneqq \{\obs(\mathscr{F}\!\!\downarrow)\mid \mathscr{F}\in\mathfrak{F}_{\excl(\Rcal)}\}.$$
Observe that our results imply that the set above is \textsl{finite}. Moreover, by Robertson and Seymour theorem,  it consists of \textsl{finite} sets of graphs.
In light of the definitions above, a possible restatement of our main result (\autoref{main_result_informal}), is the following:
\begin{quote}
For any two antichains $\mathcal{R}$ and $\mathcal{Z}$ for the minor relation, where $\Rcal\neq\emptyset,$ the tuple $\big(\mathsf{excl}(\Rcal),\excl(\Zcal)\big)$ is an EP-pair if and only if for every $\Ocal\in \Obbb_{\Rcal},$ there exists some graph $Z\in\Zcal$ such that no graph in $\Ocal$ is a minor of $Z.$
\end{quote} 

The description above is given in terms of finite obstruction sets to emphasize that the entire problem
has a finite description.
The formulation above reflects that for each $\mathscr{F}\in\mathfrak{F}_{\mathsf{excl}(\mathcal{R})},$ $\mathcal{Z}\cap \mathscr{F}\!\!\downarrow\neq\emptyset,$ i.\@ e.\@, none of the $\mathscr{F}\!\!\downarrow$ is fully contained in $\mathsf{excl}(\mathcal{Z}).$

This implies that the \textsc{Checking EP-pairs for $\Rcal$-minor free target} problem can be solved by performing $\sum_{\Ocal\in\Obbb_{\Rcal}}\sum_{Z\in\Zcal}|\Ocal|$ calls to the almost-linear minor-checking algorithm of Korhonen, Pilipczuk, and  Stamoulis \cite{KorhonenPS24minorcont} (see also \cite{RobertsonS95b,KawarabayashiKR12Thedisjoint} for previous algorithms, dating back to the Graph Minor series).
Therefore,  for every non-empty antichain $\Rcal$ of the minor relation, \textsc{Checking EP-pairs for $\Rcal$-minor free target} is solvable in almost-linear time.
Making this algorithm constructive, assuming the knowledge of $\Rcal,$ is an interesting open question.

\paragraph{Erd\H{o}s-P{\'o}sa dualities for topological minors.}
A relevant question on Erd\H{o}s-P{\'o}sa dualities is, to what extent an analogous theory of half-integrality, based on obstructing sets, can be build for other partial-orders one might define on graphs.
A particularly natural cousin of the minor relation that we may consider is the topological minor relation.
Here, we already have the half-integrality result of Liu \cite{liu2022packing}.
Let $\mathcal{H}$ and $\mathcal{G}$ be graph classes that are closed under topological minors.
We say that $(\Hcal,\Gcal)$ is an \defi{Erd\H{o}s-P{\'o}sa pair for topological minors}, in short is a \defi{tEP-pair}, if there exists a function $f\colon\nn{1}{1}$ such that every graph $G\in\Gcal,$ without $k$ disjoint subdivisions of graphs that do not belong to $\mathcal{H},$ contains a set $S$ of at most $f(k)$ vertices such that $G-S\in\Hcal.$

An immediate question is: When is $(\Hcal,\gall)$ a tEP-pair? 
Is there a characterization analogous to the one of \eqref{EP_for_minors} for the topological minor relation?
This question was already asked by Robertson and Seymour in \cite{RobertsonS86GMV} and, as mentioned in \cite{liu2022packing}, Liu, Postle, and Wollan have indeed found a complete (but complicated) characterization.
Moreover, this characterization seems to be qualitatively more complex than the planarity criterion in \eqref{EP_for_minors}  for the case of minors.

Let $\mathsf{texcl}(R)$ be the class of graphs excluding the graph $R$ as a topological minor.
According to Liu \cite{Liu2024private}, the problem that, given a graph $R,$ asks whether $\big(\mathsf{texcl}(R),\gall\big)$ is a tEP-pair is \textsf{NP}-complete even if we restrict its inputs to trees.
We wish to stress the contrast to the fact that, by \eqref{EP_for_minors}, the question ``is \textsl{$\big(\mathsf{excl}(R),\gall\big)$ an EP-pair?}'' is just planarity testing and thus, can be answered in polynomial time.
This suggests that a theory of obstructing sets such as the one for minors presented in this paper is either impossible or, at least, extremely complex for the case of topological minors.
A possible source of this complexity gap seems to be the fact that graphs ordered by the topological minor relation do not form a WQO.
However, it is an interesting question whether our approach may help in the direction of finding a constructive proof of Liu's theorem with explicit bounds.
In \autoref{iceberg_metaphor} we illustrate the current landscape of Erd\H{o}s-P{\'o}sa dualities.

\begin{figure}[ht]
\begin{center}
\scalebox{.7}{\includegraphics{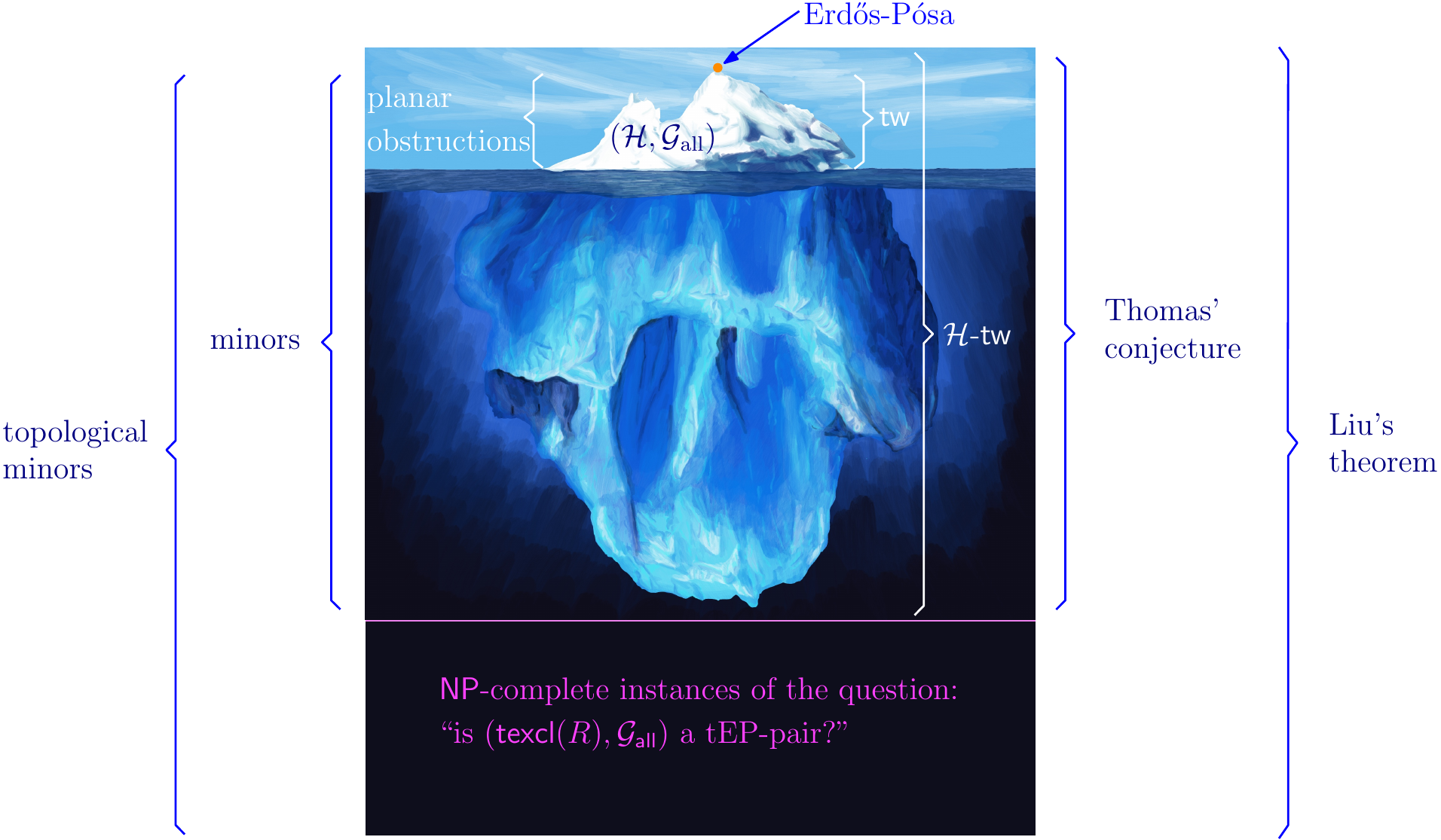}}
\end{center}
\caption{The iceberg provides a visual metaphor for the landscape on Erd\H{o}s-P{\'o}sa dualities for (topological) minor-closed graph classes.}\llabel{iceberg_metaphor}
\end{figure}

\paragraph{Improving the gap function.}
All parametric equivalences that we prove in this paper, namely $\p_{\mathfrak{W}_{\Hcal}}\sim \Hcal\text{-}\tw,$  $p_{\mathfrak{T}_{\Hcal}}\sim \Hcal\text{-}\td,$ and $p_{\mathfrak{F}_{\Hcal}}\sim \apex_\Hcal,$ have exponential gap functions, i.e., $f(h,k)\in {2^{k^{\mathcal{O}_h(1)}}}.$ 
A natural question is whether this gap can be improved to a polynomial one, that is, one of the form $k^{\mathcal{O}_h(1)}.$
This would also imply an improvement of the approximation gap of the \textsf{FPT}-approximation algorithm for the \textsc{Half-Integral $Z$-Packing} problem (mentioned on page \pageref{page_Half_Integral_packing}).
One of the two sources of exponential explosion can be found in \autoref{surfex_lemma} whose proof is given in  \cite{thilikos2023excluding} which, in turn uses \cite[Theorem 11.4]{kawarabayashi2020quickly} as a departure point where the parametric dependencies are already exponential.
This indicates that any improvement on this might require a general improvement of the dependencies of the Graph minors Structural Theorem ({GMST}) in \cite{Robertson2003GMXVI,kawarabayashi2020quickly}. 
The second level of exponentiation comes from \autoref{ripe_obt_orch} where we need thickets with nests of exponential size around their vortices in order to split those vortices.
This is because each split consumes a part of the nest to maintain the large radial linkage which leads back to the wall, and to extract new flowers.
Interestingly the reason for the exponentiality in the proof of \cite[Theorem 11.4]{kawarabayashi2020quickly} is also the vortex splitting operation.
It seems reasonable to conjecture that finding a way to bound the dependencies of the GMST by polynomials would also provide a way to avoid the first source of exponentiation.
Such a result would be a major advancement.

Finally, the degree of the polynomial in $k,$ even in the exponent of $2^{k^{\mathcal{O}_h(1)}},$ is highly dependent on the ``meta parameter'' $h=h(\mathcal{H}).$
In fact, this dependency is at least double exponential in $h.$
Therefore, another avenue of improvement could be to lower this dependency.
A bound on the gap function of the form $2^{\mathsf{poly}_h(k)},$ or, if the improvements to polynomial dependency we mention above can indeed be achieved, even $\mathsf{poly}_h(k)$ cannot be ruled out.

\subsection{A short description of the proof of \autoref{local_structure_informal}}\llabel{subsec_proof_summary}

Since this proof is by far the most technical and involved part of the entire paper, to properly explain it requires a large number of deeply nested definitions.
Several of these definitions are given in \autoref{subsec_universal_obstructions} and the majority of the remaining definitions can be found in \autoref{sec_preliminaries}.
Here we provide a very brief and informal overview of the different steps and our approach.
A second, more in-depth but also more technical overview of our proof can be found in \autoref{proof_outline_extended}.

To avoid dealing with graph classes $\mathcal{H}$ directly in our proof, we aim for the following variant of \autoref{local_structure_informal}.

\begin{theorem}\llabel{local_structure_informal_reformulated}
  There exist functions $f^1,f^2\colon\mathbb{N}^2\to\mathbb{N}$ such that for all positive integers $h$ and $k$ and every graph $G,$ if $W$ is an $f^1(h,k)$-wall in $G,$ then there exists a set $S\subseteq V(G)$ of size at most $f^2(h,k)<f^1(h,k)$ such that for every graph $Z$ on at most $h$ vertices, either
   \begin{itemize}
      \item there exists a parametric graph $\mathscr{W}=\langle \mathscr{W}_t\rangle_{t\in\mathbb{N}}\in\mathfrak{W}_{\mathsf{excl}(Z)}$ such that $G$ contains a subdivision $W'$ of $\mathscr{W}_k$ and $W'$ cannot be separated from $W$ by fewer than $k$ vertices, or
      \item the component of $G-S$ which contains the majority of $W-S$ is $Z$-minor-free.
   \end{itemize}
\end{theorem}

\paragraph{The Graph Minor Structure Theorem of Robertson and Seymour.}
To explain our approach, it is necessary to introduce the Graph Minor Structure Theorem (GMST) of Robertson and Seymour \cite{Robertson2003GMXVI,kawarabayashi2020quickly}.
We present here an abbreviated version of the local version of their theorem which avoids most of the technical statements.
Here local means again ``with respect to a large wall''.

Let $\alpha,b,d$ be non-negative integers.
We say that a graph $G$ is \defi{$\alpha$-almost embeddable} with \defi{breadth} $b$ and \defi{depth} $d$ in some surface $\Sigma$ if there exists a set $A\subseteq V(G)$ of size at most $\alpha$ and there are graphs $G^{(i)}$ such that $G-A = G^{(0)}\cup G^{(1)}\cup\dots\cup G^{(b)}$ where $G^{(0)}$ is a graph embedded in $\Sigma$ and each \defi{vortex} $G^{(i)},$ $i\in[b],$ is a graph of pathwidth at most $d$ which is ``attached along a path decomposition\footnote{Please see \autoref{subsec_almost_embeddings} for a formal definition.}'' to the vertices of some face of $G^{(0)}.$

The local version of the GMST of Robertson and Seymour now reads as follows.

\begin{center}
{\begin{eqnarray}
  \begin{minipage}{12.8cm}
  There exist functions $f^1,f^2\colon \mathbb{N}\to\mathbb{N}$ such that for every positive integer $k$ and every graph $G$ with an $f^1(k)$-wall $W,$ there either exists (1) a $K_k$-minor which cannot be separated from $W$ by deleting fewer than $k$ vertices, or (2) there exists a subgraph $G'$ of $G$ which cannot be separated from $W$ by deleting fewer than $k$ vertices and whose torso is $f^2(k)$-almost embeddable with breadth $k^2$ and depth $f^2(k)$ in a surface where $K_k$ does not embed.
  \end{minipage}
  \llabel{GMST}
\end{eqnarray}}
\end{center}
Instead of starting at \eqref{GMST} we make use of a recent refinement by Thilikos and Wiederrecht \cite{thilikos2023excluding} which allows us to obtain one of three things.
We either find: (1) a large clique minor; (2) a walloid of large order that represents a surface $\Sigma$ of large enough Euler-genus such that every graph on at most $h$ vertices can be drawn on $\Sigma;$ or (3) a subgraph $G'$ with the properties from the second outcome of \eqref{GMST} together with a walloid $W^*$ of large order which is fully contained in the graph $G^{(0)}$ of the almost embedding of the torso of $G'$ in this outcome.

From here on our proof follows along three crucial steps.

\paragraph{Step 1: Almost embeddings of a graph $Z.$}
Let us fix $G^{(0)}$ to be the part of the torso of $G'$ as above which is drawn in the surface.
Moreover, let $M$ be a minimal connected subgraph of $G$ such that $M$ contains some graph $Z$ on at most $h$ vertices as a minor and $M$ contains a vertex of the wall $W.$
Now some part of $M$ belongs to $G^{(0)},$ but other parts of $M$ may belong to the vortices $G^{(i)}$ or to the so-called ``flaps''.
Flaps are those parts of $G$ which are separated from $G'-A$ by at most three vertices.
Since these separators are so small, they do not have to be caught within the vortices.
Instead, they form facial triangles in $G^{(0)}$ (see \autoref{thm_two_paths} and the definitions surrounding it).
This induces a specific $0$-almost embedding of bounded breadth and depth for the graph $M.$
A \defi{flower} of $M$ is any maximal subgraph of $M$ which is fully contained in either a vortex or a flap\footnote{We refer to \autoref{subsec_universal_obstructions} for a formal definition.}.
Our proof introduces a procedure that either finds many duplicates of every single flower of $M,$ or finds a small hitting set for at least one of those flowers.

In \autoref{sec_harvest_crops} and \autoref{extr_flowers_thickets} we show, among others, that it suffices to repeat this procedure for only a bounded number of such graphs $M$ for every graph $Z$ on at most $h$ vertices.
Moreover, we show that we may assume that $M$ has only a bounded number, in $|Z|,$ of vertices of degree {different}  than two.

For any such $M,$ where each individual flower appears in abundance, we will then be able to construct a corresponding member of $\mathfrak{W}_{\mathsf{excl}(Z)}$ of large order.
This is made precise in \autoref{lemma_grounded_extensions_to_obstructions}.
A formal definition of flowers and the way we obtain these graphs $M$ and their $0$-almost embeddings, together with a full definition of the family $\mathfrak{W}_{\mathsf{excl}(Z)}$ is given in \autoref{subsec_universal_obstructions}.

\paragraph{Step 2: Refining an almost embedding.}
The ``gathering'' and ``hitting'' of the flowers of the graphs $M$ as above happens in two stages.
In the first stage we gather all possible flowers which can easily be found to be abundant.
Meanwhile all flowers where this is impossible, or success is at least ``not as apparent'', are confined within a small number of disks described with reference to the embedding of $G^{(0)}.$
Dealing with these disks, we call them \textsl{thickets}, is the second stage and the only part where we actually find the hitting set.

The first stage is exclusively concerned with flowers of $M$ that are produced via the flaps of $G^{(0)}.$ 
Recall that we have the walloid $W^*$ within $G^{(0)}$ which fully represents our surface $\Sigma.$
With the help of $W^*$ we now produce a tiling of $\Sigma$ using only hexagons. 
We may assume that there are only a bounded number of $M$'s we have to consider and each $M$ can only produce a bounded number of different flowers.
This allows us to classify, using the pigeonhole principle, those flowers that appear (a) in many different and far apart hexagons of our tiling and those (b) that appear either clustered closely together or a bounded number of times.
This substep is done in \autoref{Thidplodowinga} and \autoref{sec_gardening}.

Let us mark all hexagons which contain flowers of type (b).
By imitating a technique developed by Kawarabayashi, Thomas, and Wollan \cite{kawarabayashi2020quickly} in their new proof of \eqref{GMST} we now gather all marked hexagons in a small number of disks within the drawing of $G^{(0)}$ and then further refine our distinction between abundant and scarce flowers until each of the disks that still contains marked hexagons can be seen as a vortex of bounded depth itself.
This step is done in \autoref{ripe_obt_orch}.

\paragraph{Step 3: Harvesting the flowers of a ``thicket''.}
The vortices created in the previous steps are all equipped with additional infrastructure.
This allows us to route anything ``of value'' we might find inside, back to $W^*.$
We call these objects \textsl{thickets}.
The second purpose of \autoref{sec_harvest_crops} and \autoref{extr_flowers_thickets} is to apply a 
vast  generalization of the ``vortex killing'' argument of Thilikos and Wiederrecht \cite{thilikos2022killing}.
The outcome of this argument is, for any possible flower $F$ of any $M$ we have to consider, to either find many disjoint copies of $F$ pasted along a single face of $G^{(0)}$ and highly connected to $W^*,$ or find a small hitting set for all occurrences of $F$ within this thicket.
Let $S$ be the union of all of these hitting sets for all thickets.
Finally, we prove that whenever some $M$ avoids the set $S,$ then this must mean that there is an equivalent $M'$ whose flowers have been found in abundance in \textbf{Step 2} and this one, and thus we are able to produce the corresponding member of $\mathfrak{W}_{\mathsf{excl(Z)}}$ (see \autoref{lemma_grounded_extensions_to_obstructions}).

The proof of \autoref{local_structure_informal_reformulated}, or its more technical incarnation \autoref{thm_local_structure}, is now only a matter of carefully combining the previous steps.

\subsection{Defining the families $\mathfrak{F}^{\nicefrac{1}{2}}_{\mathcal{H}}$ and $\mathfrak{W}_{\mathcal{H}}$: A constructive description of $\mathfrak{C}_{\mathcal{H}}$}\llabel{subsec_universal_obstructions}

In order to state \autoref{main_result_informal} in full formality, we first need to make our definitions of the graphs $M$ that appear in \textbf{Step 1} of \autoref{subsec_proof_summary} and then describe how the families $\mathfrak{W}_{\mathcal{H}}$ and ultimately $\mathfrak{F}^{\nicefrac{1}{2}}_{\mathcal{H}}$ are defined.
Once this is done, the collection $\mathfrak{C}_{\mathcal{H}}$ of counterexamples is fully described, in accordance with \autoref{main_result_informal}, as the collection of every graph class generated through the minor closure of a single member of $\mathfrak{F}^{\nicefrac{1}{2}}_{\mathcal{H}}.$
Some definitions presented in this section still hide further technicalities which are expanded upon in \autoref{sec_preliminaries}.

As a first step, we make the $0$-almost embeddings of our ``minor models''\footnote{Later we will call these graphs the \textsl{extensions} of $Z.$} explicit.

\paragraph{Hypergraphs from graph partitions.}

A \defi{subgraph partition} of a graph $Z$ is a collection $\Pcal$ of subgraphs of $Z$ such that\footnote{Given a set $\Scal$
of objects where $\cup$ has been defined, we use $\cupall\Scal$ as a shorthand for $\bigcup_{S\in\Scal}S.$ Here the union of two graphs $G_{1}=(V_{1},E_{1})$ and $G_{2}=(V_{2},E_{2})$ is defined as $G_{1}\cup G_{2}=(V_{1}\cup V_{2}, E_{1}\cup E_{2}).$} $\cupall\Pcal=Z$ and $\{ E(P) \mid P\in\mathcal{P} \}$ is a  partition of $E(Z).$
Given a subgraph $P\in \Pcal,$ we define the \defi{boundary} of $P$ in $Z$ to be the set $\partial_{Z}(P)$ containing every vertex of  $V(P)$ that is also a vertex of some other member of $\Pcal\setminus\{P\}.$
We define $\Pcal(Z)$ to be the set containing all subgraph partitions $\mathcal{P}$ of $Z$ such that for every $P\in\mathcal{P},$ $\partial_Z(P)\neq\emptyset.$ 

Given a subgraph partition $\Pcal\in\Pcal(Z),$ we define the hypergraph $\Kcal_{\Pcal}=(V_{\Pcal},\Ecal_{\Pcal})$ where $\Ecal_{\Pcal}=\{\partial_{Z}(P) \mid P\in \Pcal\}$ and $V_{\Pcal}=\cupall \Ecal_{\Pcal}.$
It follows that each hyperedge $e$ of $\Kcal_{\Pcal}$ is the boundary of some member, say $G_e,$ of the graph partition (notice that $\Ecal_{\Pcal}$ may be a multi-set, as in the case of the example of \autoref{ex_obst_univ_walloid}).
\medskip

In \autoref{subsec_proof_summary} we spoke a lot about ``pasting graphs along the vertices of a face'' of ``confining'' parts of a graph within a disk.
The following provides a formal framework to deal with notions like those in a way that is, at first, independent of topology.

\paragraph{Linear societies and their decompositions.}
A \defi{linear society} is a pair $\langle H,\Lambda \rangle$ where $H$ is a {non-empty} graph and $\Lambda$ is a linear ordering of some {non-empty} subset of vertices of $H.$
We use $V(\Lambda)$ to denote the set of vertices of $\Lambda$ and we refer to $V(\Lambda)$ as the \defi{boundary} of the linear society $\langle H,\Lambda\rangle.$
Moreover, we \textsl{always} demand that every connected component of $H$ contains some vertex of $V(\Lambda).$

Let $\langle H,\Lambda\rangle$ be a linear society where $\Lambda=\lin{y_{1},\ldots,y_{r}}.$
Consider a drawing, possibly with crossings, of $H$ in a closed disk $\Delta$ such that the vertices drawn on the boundary, denoted by  $\bd(\Delta),$ are exactly the vertices of $\Lambda$ drawn with respect to the order $\Lambda.$
We also demand that all edges of $H$ are drawn in the interior of $\Delta.$
For simplicity, we do not distinguish between the vertices of $V(\Lambda)$ and the points of $\bd(\Delta)$ where these vertices are drawn.
\smallskip

Whenever different components of $H$ are drawn within $\Delta$ in the way above such that there exists a curve $\gamma$ in $\Delta$ with exactly its endpoints in $\bd(\Delta)$ where $\gamma$ does not intersect the drawing of $H$ in any way, we are able to ``split'' $\langle H,\Lambda\rangle$ into smaller linear societies using $\gamma$ while also respecting the original linear ordering $\Lambda$ up to shifts.
We formalize this intuition with the definition below to obtain a notion of ``components'' of a linear society.

We see the set $\bd(\Delta)\setminus V(\Lambda)$ as a set of open lines between consecutive vertices in $\Lambda.$
For each $i\in[r],$ let $p_{i}$ be a point of the line with endpoints $y_{i-1}$ and $y_{i}$ (where $y_{0}=y_{r}$).
A closed subdisk $\Delta'\subseteq\Delta$ is \defi{splitting} for $\Delta$ if 
\begin{itemize}
  \item $H$ can be drawn (possibly with crossings) inside $\Delta\setminus \Delta',$
  \item $\{p_{1}\}\subseteq \bd(\Delta)\cap \bd(\Delta')\subseteq\{p_{1},\ldots,p_{r}\},$ and
  \item the cardinality of $|\bd(\Delta)\cap \bd(\Delta')|$ is maximized.
\end{itemize}
Let  $\bd(\Delta)\cap \bd(\Delta')=\{p_{i}\mid i\in I\}$ for some $[1]\subseteq I\subseteq [r]$ where $I=\{i_1=1,i_2,\ldots,i_{q}\}$ and observe that $I$ is the same for every choice of a splitting subdisk $\Delta'$ for $\Delta.$
For each $j\in[q],$ let $\Lambda_{j}\coloneqq \lin{y_{i_{j}},\ldots,y_{i_{j+1}-1}}$ and let $H_{j}$ be the union of all connected components of $H$ that contain vertices of $\Lambda_{j}.$
We denote the \defi{decomposition} of $\langle H,\Lambda\rangle$ into the linear societies $\langle H_i,\Lambda_i\rangle$ as above by $\mathsf{dec}({H,\Lambda})\coloneqq \lin{\lin{H_{1},\Lambda_{1}},\ldots,\lin{H_{q},\Lambda_{q}}}.$

Next, we reintroduce the topology to fully understand how linear societies may be used to encode the ``almost'' part of our embeddings.

\paragraph{Surfaces and linear societies.}
By \defi{surface} we mean a two-dimensional manifold (possibly with boundary).
We say that a surface $\Sigma$ is \defi{contained} in some surface $\Sigma'$ if $\Sigma'$ can be obtained from $\Sigma$ by adding handles and crosscaps.
Given a pair $(\mathsf{h},\mathsf{c})\in\mathbb{N}\times[0,2]$ we define $\Sigma^{(\mathsf{h},\mathsf{c})}$ to be the surface without boundary created from the sphere by adding $\mathsf{h}$ handles and $\mathsf{c}$ crosscaps.
If $\mathsf{c}=0,$ the surface $\Sigma^{(\mathsf{h},\mathsf{c})}$ is \defi{orientable}, otherwise it is \defi{non-orientable}.
By Dyck's theorem \cite{Dyck1888Beitrage,Francis99ConwayZIP}, two crosscaps are equivalent (via homeomorphisms) to a handle in the presence of a (third) crosscap. 
This implies that the notation $\Sigma^{(\mathsf{h},\mathsf{c})}$ is sufficient to capture all surfaces without boundary.

Given a graph $Z$ we denote by $\Sbbb^{Z}$ the set containing all surfaces without boundary where $Z$ cannot be embedded.
We define $\sobs(\Sbbb^{Z})$ to be the set of all containment-wise minimal surfaces not in $\Sbbb^{Z}.$ 
It follows that $\sobs(\Sbbb^{Z})$ contains one or two surfaces, see \cite{thilikos2023excluding}.
We define  $\Sbbb^{(Z)}$ as the set containing every surface that is contained in some surface of $\sobs(\Sbbb^{Z}).$
\smallskip

Let now $\Sigma$ be some surface and let $M$ be a graph, let $\Pcal\in \Pcal(M)$ be a subgraph partition of $M,$ and assume that $\Gamma$ is a $\Sigma$-embedding.
In such an embedding all vertices of a hyperedge $e\in\Ecal_{\Pcal}$ are drawn on the boundary of some closed disk $\Delta_{e}$ of $\Sigma$ and two disks $\Delta_{e}$  and $\Delta_{e'}$  intersect only on the common vertices of $e$ and $e'.$
Given a choice of starting vertex $y\in e$ and some rotation of the boundary of $\Delta_{e},$ we obtain a linear ordering $\Lambda_{e,y}$ of the vertices of $e$ starting from $y.$
The rotation of each hyperedge is chosen to be counterclockwise (that is, it has $\Delta_{e}$ on the left) in case $\Sigma$ orientable, and it is chosen to be an arbitrary rotation if $\Sigma$ is non-orientable.
Each $\langle G_{e},\Lambda_{e,y}\rangle,$ for $y\in e,$ can now be seen as a linear society. 
We say that the linear society $\langle G_{e},\Lambda_{e,y}\rangle$ is \defi{disk embeddable} if the graph $G_{e}$ can be embedded in a closed disk $\Delta$ such that exactly the vertices of $\Lambda_{e,y}$ are drawn on its boundary.
\medskip

We are finally ready to give a definition of those $0$-almost embeddings of graphs that sit at the core of our description of $\mathfrak{C}_{\mathcal{H}}.$

\paragraph{Embedding pairs.}
Let $Z$ be a graph.
An \defi{embedding pair} of $Z$ is a pair $(\Sigma,\mathbf{B}),$ where $\Sigma\in \Sbbb^{(Z)}$ is a surface, and $\mathbf{B}$ is the set of all linear societies which are generated, as described below, by a fixed triple of the form $(M, \Pcal,\Gamma)$ where
\begin{enumerate}
\item $M$ is a connected graph containing $Z$ as a minor such that $|M| ≤ \mathsf{ext}(|Z|),$ 
\item $\Pcal\in \Pcal(M),$ and
\item $\Gamma$ is a $\Sigma$-embedding of $\Kcal_{\Pcal},$
\end{enumerate}
where $\mathsf{ext} \colon \mathbb{N} \to \mathbb{N}$ is a function of order $2^{\Ocal(\ell(|Z|^{2}))}$ that will be determined in the proof of \autoref{lemma_grounded_extensions_to_obstructions} in \autoref{sec_local_structure} and $\ell(\cdot)$ is the \textsl{unique linkage function} (see \cite{Robertson2009GMXXI,Kawarabayashi2010Shorter,AdlerKKLST17Irrelevant,mazoit2013single,GolovachST22Combing}).

Given a triple as above, we further consider some function $\eta\colon\Ecal_{\Pcal}\to V(M)$ that assigns a starting vertex $\eta(e)\in e$ to every $e\in\Ecal_{\Pcal}.$
Thereby, $\eta$ generates the following set of linear societies
$$\{\lin{G_{e},\Lambda_{e,\eta(e)}}\mid e\in \Ecal_{\Pcal} \text{~and $\langle G_{e},\Lambda_{e,\eta(e)}\rangle$ is not disk-embeddable}\}.$$

Notice that ${\bf B}$ may be empty.
We denote by $\mathfrak{P}_{\mathsf{excl}(Z)}$ the set containing all embedding pairs of $Z.$
Notice that a graph on $h$ vertices has at most $2^{\mathcal{O}((\mathsf{ext}(|Z|))^2)}$ many subgraph partitions.
Moreover, there exist at most $2|Z|^2+1$ many different surfaces in $\Sbbb^{(Z)}$ and for each choice of $\mathcal{P}\in\mathcal{P}(Z)$ we have $|\mathcal{P}|\leq \mathsf{ext}(|Z|)^2$ and $|P|\leq \mathsf{ext}(|Z|)$ for every $P\in\mathcal{P}$ which implies that there are at most $\mathsf{ext}(|Z|)^3$ many choices for $\eta.$
It follows that $\mathfrak{P}_{\mathsf{excl}(Z)}$ has $2^{(\mathsf{ext}(|Z|))^2}$ elements.
Given a proper minor-closed graph class $\mathcal{H}$ we define $\mathfrak{P}_{\mathcal{H}}\coloneqq \bigcup_{Z\in\mathsf{obs}(\mathcal{H})}\mathfrak{P}_{\mathsf{excl}(Z)}.$
We now obtain, in particular, that $|\mathfrak{P}_{\Hcal}|$ is finite for every choice of $\Hcal.$

\medskip
The walloids that make up our families $\mathfrak{F}^{\nicefrac{1}{2}}_{\mathcal{H}}$ and $\mathfrak{W}_{\mathcal{H}}$ are constructed from several smaller pieces that can be circularly arranged.
Such a modular definition allows for a nicer presentation, but it also facilitates our proofs as it allows us in many places to reduce to the analysis of single pieces of the walloid rather than its entire global structure.

\paragraph{$t$-segments.} 
Each member $\mathscr{W}$ of $\mathfrak{F}_{\Hcal}^{\nicefrac{1}{2}}$ is a wall-like parametric graph to which, from now on, we refer as an \defi{elementary} \defi{guest walloid}.
These parametric graphs $\mathscr{W}$ are built using the {embedding pairs} from $\mathfrak{P}_{\Hcal}.$
The main building components of elementary guest walloids are called \defi{elementary} \defi{$t$-segments}.
While the formal definitions are explained in detail in \autoref{subsec_segments}.
Here, we provide an intuitive description which will be sufficient to understand the family $\mathfrak{F}_{\Hcal}^{\nicefrac{1}{2}}$ and the more detailed summary of the proof of \autoref{local_structure_informal}.

The three basic components are the elementary \textsl{$(t',t)$-wall segment}, the \textsl{elementary $t$-handle segment}, and the 
\textsl{elementary $t$-crosscap segment}.
The \defi{elementary $(t',t)$-wall segment} is an elementary $(t',t)$-wall and acts as the starting point of our construction of all other elementary $t$-segments.
The leftmost drawing of \autoref{fdsjuy456uryufd} depicts an elementary $(12,3)$-wall segment. 
We depict the $t$ ``leftmost'' vertices of an elementary $(t',t)$-wall segment in red and its $t$ ``rightmost'' vertices in blue. Later, we will use these vertices in order to glue together distinct $t$-segments, using new edges depicted in orange. 
The $t'$ ``top'' vertices of the elementary $(t',t)$-wall segment are drawn with white interiors.

The \defi{elementary $t$-handle segment} is an elementary $(4t, t)$-wall segment where we added two collections of $t$ edges, each between the top vertices of the $(4t,t)$-wall, as indicated in the second drawing of \autoref{fdsjuy456uryufd}.
The \defi{elementary $t$-crosscap segment} is an elementary $(4t, t)$-wall segment obtained by adding a collection of $2t$ edges between the top vertices of the elementary $(4t,t)$-wall in a pairwise crossing manner, as indicated in the third drawing of \autoref{fdsjuy456uryufd}. 

\begin{figure}[ht]
  \begin{center}
  \scalebox{.56}{\includegraphics{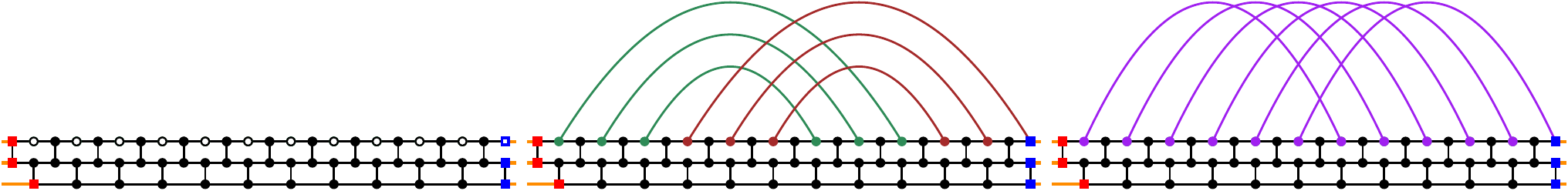}\hspace{1mm}\includegraphics{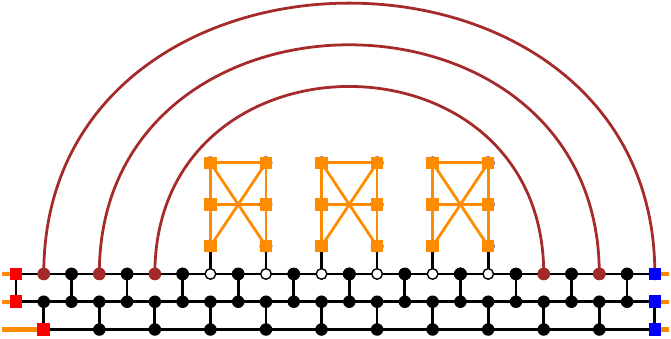}}
  \end{center}
    \caption{From left to right: an elementary $3$-wall segment, an elementary $3$-handle segment, an elementary $3$-crosscap segment, and an elementary $(3,3,K_{3,3}^-,\langle y_1,y_2\rangle)$-flower segment.}
  \llabel{fdsjuy456uryufd}
\end{figure}

The fourth and most complicated, and somehow central, elementary $t$-segment is parametrized by an integer $b\in\Nbbb_{≥1}$ and a linear society $\langle H,\Lambda\rangle$ where $\mathsf{dec}(H,\Lambda)=\lin{\lin{H_{1},\Lambda_{1}},\ldots,\lin{H_{q},\Lambda_{q}}}.$

An \defi{elementary $(t,b,H,\Lambda)$-flower segment} is constructed by starting from an elementary $(t+b|\Lambda|+t, t)$-wall segment.
As a first step, we add a collection of $t$ edges joining its $t$ leftmost top and its $t$ rightmost top vertices, as indicated by the brown edges of the rightmost drawing of \autoref{fdsjuy456uryufd}. 
The construction is completed by taking, for every $j\in[q],$ $b$ copies $\langle H^1_j,\Lambda^1_j\rangle,$\ldots,$\langle H^b_j,\Lambda^b_j\rangle$ of $\langle H_j,\Lambda_j\rangle$ and connecting the $p$-th vertex of $\Lambda^{i}_j$ with the $(t+(\Sigma_{j'\in[j-1]}|\Lambda_{j-1}|)b+(i-1)|\Lambda_{j}|+p)$-th top vertex for every $j\in[q],$ every $i\in[b],$ and every $p\in |\Lambda_j|.$

The rightmost drawing in \autoref{fdsjuy456uryufd} depicts an elementary $(3,3,K_{3,3}^-,\langle y_1,y_2\rangle)$-flower segment, where $K_{3,3}^-$ is the graph obtained from $K_{3,3}$ by deleting the edge $y_{1}y_{2}.$
The 3 copies of $K_{3,3}^-$ are depicted in orange (see also \autoref{ex_obst_univ_walloid}).
\medskip

\paragraph{Cylindrical concatenations and the construction of $\mathfrak{W}_{\Hcal}.$}
Given a sequence $\lin{W_{1},\ldots,W_{q'}}$  of $t$-segments, we define their \defi{cylindrical concatenation} as the graph obtained by adding edges (depicted in orange) between same-height rightmost (depicted in blue) vertices of $W_{i}$ and leftmost (depicted in red) vertices of $W_{i+1}$ for every $i\in[q'-1].$
The construction is then completed by adding edges between same height rightmost vertices of $W_{1}$ and leftmost vertices of $W_{q'}$ (see \autoref{ex_obst_univ_walloid} for an example).

\begin{figure}[ht]
  \begin{center}
  \scalebox{1}{\includegraphics{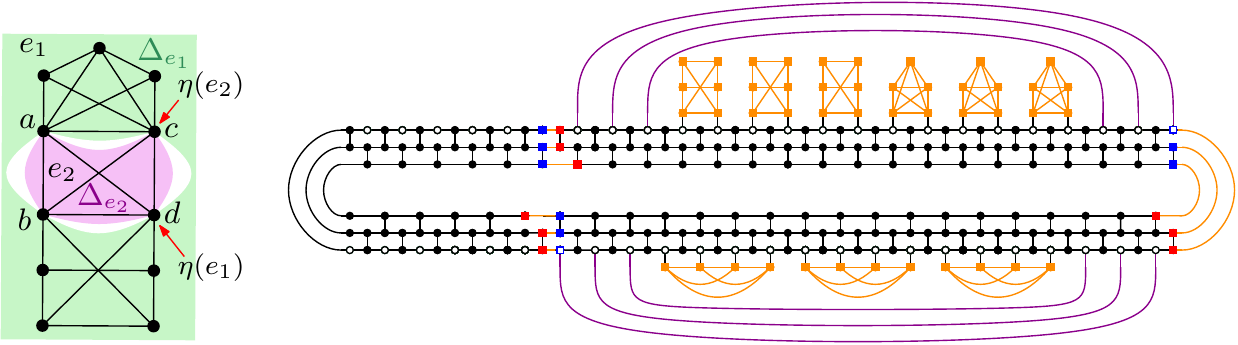}}
  \end{center}
    \caption{Left: A graph $M=Z$ and a subgraph partition $\Pcal$ of $M$ containing two graphs $H_{1}$ and $H_{2}.$ $H_{1}$ is drawn in the green disk while $H_{2}$ is drawn in the violet disk. The hypergraph $\Kcal_{\Pcal}$ contains the vertices $a,b,c,$ and $d$ and two hyper-edges $e_{1}=e_{2}=\{a,b,c,d\}.$ The figure shows an embedding of $\Kcal_{\Pcal}$ in the sphere.
    The choices for $Z,$ $M,$ $\Pcal,$ and $\eta,$ generate the embedding pair $\mathbf{p}=(\Sigma,\mathbf{B})$ for the class $\mathsf{excl}(Z)$ where $\mathbf{B}=\{\lin{H_{1},\lin{d,b,a,c}},\lin{H_{2},\lin{c,a,b,d}}\}.$ Right: The guest walloid $\mathscr{W}^{\mathbf{p}}_3.$}
  \llabel{ex_obst_univ_walloid}
\end{figure}

Let now $\mathbf{p}=(\Sigma,\mathbf{B})$ be some embedding pair of $\mathfrak{P}_{\Hcal},$ where $\Sigma=\Sigma^{(\mathsf{h},\mathsf{c})}$ for some $\mathsf{h}\in\Nbbb$ and $\mathsf{c}\in[0,2].$
Let also $\mathbf{B}=\{(H_{1},\Lambda_{1}),\ldots,(H_{q},\Lambda_{q})\}.$
We define the parametric graph $\mathscr{W}^{\mathbf{p}}=\langle \mathscr{W}^{\mathbf{p}}_t\rangle_{t\in \Nbbb},$ where $\mathscr{W}^{\mathbf{p}}_t$ is the cylindric concatenation of $\lin{W_{1},\ldots,W_{q'}},$ where $q'=1+\mathsf{h}+\mathsf{c}+q$ and such that $W_{1}$ is the elementary $(4t,t)$-wall segment, $W_{2},\ldots,W_{1+\mathsf{h}}$ are elementary $t$-handle segments, $W_{1+\mathsf{h}+1},\ldots,W_{1+\mathsf{h}+\mathsf{c}}$ are  elementary $t$-crosscap segments, and for $i\in[q],$ $W_{1+\mathsf{h}+\mathsf{c}+i}$ is the elementary $(t,t,H_i,\Lambda_i)$-flower segment.

Notice that the above construction of $\mathscr{W}^{\mathbf{p}}$ depends on the ordering of the elements of $\mathbf{B}$ considered.
This is not a problem for our definition as all alternative parametric graphs defined by different orderings are pairwise equivalent (see for example the proof of \autoref{lem_garden_creation} and \autoref{subsec_swapping}). 
So we may choose an arbitrary ordering for for defining $\mathscr{W}^{\mathbf{p}}.$

For the purposes of our construction, for every graph $Z,$ we also set up a special symbol ${\bf p}_{Z}$ (treated as a separate ``embedding pair'') and we set $\mathscr{W}^{{\bf p}_{Z}}=\lin{t\cdot Z\mid t\in\Nbbb}$ and $\widehat{\mathfrak{P}}^{(Z)}\coloneqq \mathfrak{P}^{(Z)}\cup\{{\bf p}_{Z}\}.$ Let now $\Hcal$ be a proper minor-closed class.
We define $\mathfrak{F}_{\Hcal}^{\varnothing}\coloneqq \{\mathscr{W}^{{\bf p}_{Z}}\mid Z\in \obs(\Hcal)\}.$

Let $Z_{1},\ldots,Z_{r},$ be the connected components of our graph $Z.$
We define the following four parametric families.
\begin{eqnarray}
\mathfrak{F}^{(Z)} & \coloneqq & \big\{\mathscr{W}^{\mathbf{p}_1}+\cdots+{\mathscr{W}}^{\mathbf{p}_r}\mid (\mathbf{p}_1,\ldots,\mathbf{p}_r)\in\widehat{\mathfrak{P}}^{(Z_1)}\times\cdots\times \widehat{\mathfrak{P}}^{(Z_r)}\big\}, \llabel{first_walloid_formula}\\
\mathfrak{F}_{\Hcal}^{\nicefrac{1}{2}} & \coloneqq & \big\{ \mathscr{W} \in \cupall\{ \mathfrak{F}^{(Z)}\mid {Z\in\obs(\Hcal)}\}\mid \text{for every $\mathscr{W}' \in \mathfrak{F}_{\Hcal}^{\varnothing}$ it holds that $\mathscr{W}' \not\lesssim \mathscr{W}$} \big\},\llabel{fourth_walloid_formula}\\
{\mathfrak{F}}_{\Hcal} & \coloneqq & \mathfrak{F}^{\nicefrac{1}{2}}_{\mathcal{H}} \cup \mathfrak{F}_{\Hcal}^{\varnothing} \llabel{second_walloid_formula}, \text{~and}\\
{\mathfrak{F}}_{\Hcal}^{\mathsf{min}} & \coloneqq & \mathsf{min}({\mathfrak{F}}_{\Hcal}). \llabel{third_walloid_formula}
\end{eqnarray}
As we already mentioned in the beginning of this section, all parametric families defined by the minimization in \eqref{second_walloid_formula} are $\lesssim$-antichains.
All of these minimizations, for a fixed class $\mathcal{H},$ are pairwise equivalent, and, in particular, they are equivalent to the set that they minimize.
Therefore, the definitions for $\mathfrak{F}_{\Hcal}$ and $\mathfrak{F}_{\Hcal}^{\nicefrac{1}{2}}$ are unique up to equivalence and we may take an arbitrary representative.

\paragraph{Examples and the obstructing sets for $\mathcal{H}$-treewidth.}
An important special case in the construction above is any choice $\Hcal=\excl(Z)$ where $Z$ is a connected planar graph.
Then 
$\mathfrak{P}^{(Z)}$ contains (among others) the pair $(\Sigma^{(0,0)},\emptyset).$
This implies that $\mathsf{min}(\{\mathscr{W}^{\mathbf{p}}\mid \mathbf{p}\in \widehat{\mathfrak{P}}^{(Z)})$ (as in \eqref{second_walloid_formula}) contains only the cylindrical concatenation of the $(4t,t)$-wall segment. 
Moroever, the minimization in $\eqref{second_walloid_formula}$ then ``absorbs'' this wall into $\mathscr{W}^{{\bf p}_{Z}}=\lin{k\cdot P}_{k\in\Nbbb}.$
This means that $\mathfrak{F}_{\Hcal}=\mathfrak{F}_{\Hcal}^{\varnothing}$ and therefore $\mathfrak{F}_{\Hcal}^{\nicefrac{1}{2}}=\emptyset.$ 

We now define the two new parametric families which act as our obstructing sets for $\mathcal{H}$-treewidth.
\begin{align}
\mathfrak{W}_{\Hcal} &\coloneqq \{\mathscr{W}^{\mathbf{p}}\mid \mathbf{p}\in \mathfrak{P}^{(Z)}\  \text{and}\ Z\in\obs(\Hcal)\}\text{ and}\\
\mathfrak{W}_{\Hcal}^{\mathsf{min}} &\coloneqq \mathsf{min}(\mathfrak{W}_{\Hcal}) 
\end{align}
Notice that the definition of $\mathfrak{W}_{\Hcal}$ above and the definition of ${\mathfrak{F}}_{\Hcal}$ in \eqref{first_walloid_formula} and \eqref{second_walloid_formula} differ in two aspects:
The first is that we now use $\mathfrak{P}^{(Z)}$ instead of the more enhanced $\widehat{\mathfrak{P}}^{(Z)}$ and the second is that we do not have to go through an additional construction step such as \eqref{first_walloid_formula}.
This readily implies that $\mathfrak{F}_{\Hcal}\lesssim^{*}\mathfrak{W}_{\Hcal}.$
\medskip

Let us give another simple, but less trivial, example for the construction of $\mathfrak{F}_{\Hcal}.$
This time we discuss the case where $\Hcal=\excl(K_{3,3}).$
Notice that $\Sbbb^{(\{K_{3,3}\})}=\{\Sigma^{(0,0)},\Sigma^{(1,0)},\Sigma^{(0,1)}\},$ i.e., $\Sbbb^{(K_{3,3})}$ contains the sphere, the torus, and the projective plane.
Next, observe that for each $\Sigma\in\{\Sigma^{(1,0)},\Sigma^{(0,1)}\}$ its holds that $(\Sigma,\emptyset)\in\mathfrak{P}^{(K_{3,3})}$ and that for every other $(\Sigma,\mathbf{B})\in\mathfrak{P}^{(K_{3,3})}$ we have that $\mathscr{W}^{{(\Sigma,\emptyset)}}\lesssim\mathscr{W}^{(\Sigma,\mathbf{B})}.$
Therefore $\mathfrak{F}_{\mathsf{excl}(K_{3,3})}$ should contain the cylindrical concatenation of the elementary $t$-wall segment and the elementary $t$-handle segment as well as the cylindrical concatenation of the elementary $t$-wall segment and the elementary $t$-crosscap segment.
For $t=3,$ these graphs are depicted in \autoref{torus_and_crosscap}.
Let us now examine the parametric graphs in $\mathfrak{W}_{\Hcal}$ that are generated by choosing the sphere as the surface $\Sigma.$
As $K_{3,3}$ cannot be embedded in $\Sigma^{(0,0)}$ we have to consider some non-trivial subgraph partition $\Pcal$ of $K_{3,3}.$
It is easy to observe that for every $(\Sigma^{(0,0)},\mathbf{B})\in\mathfrak{P}^{(K_{3,3})}$ it holds that  $\mathscr{W}^{{(\Sigma^{(0,0)},\mathbf{B}')}}\lesssim\mathscr{W}^{(\Sigma^{(0,0)},\mathbf{B})}$ where $\mathbf{B}'=\{(K_{3,3},\lin{y})\}$ and $y$ is an arbitrary vertex of $K_{3,3}.$

Thus, $\mathscr{W}^{{(\Sigma^{(0,0)},\mathbf{B}')}}$ is the only parametric graph that has some chance of surviving in $\mathfrak{F}^{\mathsf{min}}_{\mathsf{excl}(K_{3,3})}$ among those generated by the sphere.
But this chance is further lost because $\mathscr{W}^{{\bf p}_{K_{3,3}}}=\lin{t\cdot Z\mid t\in\Nbbb}\lesssim \mathscr{W}^{{(\Sigma^{(0,0)},\mathbf{B}')}}.$
So all parametric graphs generated by the sphere are absorbed by $\mathfrak{F}_{\mathsf{excl}(K_{3,3})}^{\varnothing}.$
Therefore $\mathfrak{F}_{\mathsf{excl}(K_{3,3})}^{\nicefrac{1}{2}}$ is as announced in \autoref{torus_and_crosscap}. 
The same analysis holds if, instead of $K_{3,3},$ we consider any minor-closed graph class excluding some single-crossing minor.

As asserted by our main result, \autoref{main_result_informal}, $\mathfrak{F}_{\Hcal}^{\nicefrac{1}{2}}$ provides a precise and finite delineation of the half-integrality for $\Hcal.$
In the next subsection, we give some further intuition why this is the case.

\subsection{A (more) detailed outline of the proof of \autoref{local_structure_informal}.}\llabel{proof_outline_extended}

The proof of \autoref{main_result_informal} we presented at the end of \autoref{sec_min_max_dualities} utilized as a core tool \autoref{thm_apex_upper_bound} which, in turn, can be obtained from \autoref{main_grid_general}.
Let us begin by giving a short description of how we obtain \autoref{main_grid_general} from the local structure theorem \autoref{local_structure_informal}.

\paragraph{Balanced $\mathcal{H}$-separators and wall(oid)s.}
A core concept from the (algorithmic) theory of tree decompositions is the notion of ``well-linked'' sets and balanced separators \cite{reed1992finding}.
We define analogues of these concepts with respect to a graph class $\mathcal{H}.$
A similar concept has been used by Jansen, de Kroon, and W{\l}odarczyk \cite{Jansen20235Approx} in their approximation algorithm for $\mathcal{H}$-treewidth.

A set $X\subseteq V(G)$ is $k$-$\mathcal{H}$-linked if for every set $S\subseteq V(G)$ of size at most $k-1$ there exists a component $K$ of $G-S$ such that $|V(K)\cap X|\geq\frac{2}{3}|X|,$ and $K\notin \mathcal{H}.$
This means that whenever we delete few vertices, most of $X$ still belongs to some subgraph of $G$ which contains an obstruction of $\mathcal{H}.$
It is straightforward to verify that any graph with a $k$-$\mathcal{H}$-linked set must have $\mathcal{H}$-treewidth at least $k.$
Moreover, under the assumption that there are no $k$-$\mathcal{H}$-linked sets in $G$ one can construct a tree-decomposition that witnesses small $\mathcal{H}$-treewidth using standard arguments.

Since the existence of a highly $\mathcal{H}$-linked set $X$ implies large $\mathcal{H}$-treewidth, it, in particular, implies large treewidth.
By using the Grid Theorem \cite{RobertsonS86GMV,chuzhoy2021towards} together with some arguments from \cite{kawarabayashi2020quickly} we are able to find a large wall which is ``associated''\footnote{Formally speaking and to help a reader familiar with the terminology of graph minors, the tangle of the wall agrees with the tangle of the highly linked set.} with our highly $\mathcal{H}$-linked set.
This wall may now be used as the input for an application of \autoref{local_structure_informal}.
Finally, since the wall and $X$ are ``associated'', it is impossible that the second outcome of \autoref{local_structure_informal} holds.
Therefore, \autoref{main_grid_general} follows from the parametric equivalence between the largest number $k$ for which a graph has a $k$-$\mathcal{H}$-linked set and $\mathcal{H}$-treewidth together with \autoref{local_structure_informal}.
\medskip

In the remainder of this section, we give a more detailed overview of our proof for \autoref{local_structure_informal} and the techniques involved.
As mentioned earlier, we prove a more general theorem, that is \autoref{local_structure_informal_reformulated}, which yields the outcome of \autoref{local_structure_informal} for all proper minor-closed graph classes $\mathcal{H}$ with $h(\mathcal{H})\leq h$ simultaneously.
In most of the discussion that follows, we focus on a single graph $Z$ on at most $h$ vertices for ease of presentation.

Let $\Sigma=\Sigma^{(\mathsf{h},\mathsf{c})}$ be some surface.
We define the parametric graph $\mathscr{W}^{(\mathsf{h},\mathsf{c})}=\langle \mathscr{W}^{(\mathsf{h},\mathsf{c})}_t\rangle_{t\in\mathbb{N}}$ where $\mathscr{W}^{(\mathsf{h},\mathsf{c})}_t$ is obtained from the cylindrical concatenation of a single $t$-wall segment, $\mathsf{h}$ $t$-handle segments, and $\mathsf{c}$ $t$-crosscap segments.
A key tool we borrow from \cite{thilikos2023excluding} is an algorithm that takes as an input some graph $G$ and a large wall $W$ together with integers $t,$ $k,$ and $g.$
This algorithm then returns one of three things.
\begin{enumerate}
    \item a $K_k$-minor,
    \item a subdivision of $\mathscr{W}_t^{(\mathsf{h},\mathsf{c})}$ where $g=2\mathsf{h}+\mathsf{c},$ or
    \item a small ``apex'' set $A\subseteq V(G)$ together with a so-called \textsl{$\Sigma$-decomposition} $\delta$ of $G-A$ which ``respects'' the wall $W,$ where the Euler-genus of the surface $\Sigma$ is $2\mathsf{h}'+\mathsf{c}'<g$ for appropriate choices of $\mathsf{h}'$ and $\mathsf{c}',$ and a subdivision $W_1$ of $\mathscr{W}_t^{(\mathsf{h}',\mathsf{c}')}$ which is also ``respected'' by $\delta.$
    Moreover, the number and depth of the vortices of $\delta$ is bounded in some function in $t$ and $k.$
\end{enumerate}
For most intents and purposes, one may think of a $\Sigma$-decomposition as an almost-embedding where the depth of the vortices is not bounded.
The way $\delta$ ``respects'' $W$ and $W_1$ may be understood as: they appear as subdivisions in $G^{(0)}$ when $\delta$ is understood as a $0$-almost embedding.
See \autoref{subsec_almost_embeddings} for a full definition of $\Sigma$-decompositions.

An alternative way to understand a $\Sigma$-decomposition $\delta$ is through subgraph partitions and their hypergraphs as follows.
On an intuitive level, one may see $\delta$ as an embedding of the hypergraph $\Kcal_{\Pcal}=(V_{\Pcal},\Ecal_{\Pcal})$  in $\Sigma^{(\mathsf{h}',\mathsf{c}')}$ where $\Pcal\in \Pcal(G)$ is a subgraph partition of $G$ where only a bounded number of hyperedges have size $\geq 4.$
Moreover, each hyperedge $e\in\Ecal_{\Pcal}$ is seen as a disk $c,$ called \defi{cell} of $\delta$ where the graph $G_{c}\coloneqq G_{e}$ is drawn (possibly with crossings) inside $c$ and the vertices of $e$ are drawn on the boundary $c$ in some liner ordering $\Lambda_{c}.$
We call a cell with at least $4$ vertices drawn on their boundary a \defi{vortex} while we call the others \defi{flaps}.
This defines a linear society $(G_{c},\Lambda_c)$ that has bounded depth in the sense that there is no big flow in $G_{c}$ from some set $Q\subseteq e$ of vertices appearing consecutively on $\bd(\Delta)$ to  $e\setminus Q$ (we call such a flow a \defi{transaction} of the linear society $(G_{c},\Lambda_c)$). 
Part of the outcome of the algorithm from \cite{thilikos2023excluding} (originally proven in \cite{kawarabayashi2020quickly}) ensures also that all vortices are situated ``inside'' the biggest cycle of $W_1.$
We refer to this cycle as the \defi{exceptional cycle} of $W_1$ (see \autoref{thickets_break_pattern_sp} for a visualization).
We refer to the set $V_{\delta}\coloneqq V(V_{\Pcal})$ as the set of \defi{ground vertices} of $\delta.$

By editing the $\Sigma$-decomposition $\delta$ and collecting everything that is drawn in the unique disk defined by the exceptional cycle of $W_1$ which avoids all of the other cycles of $W_1$ into a single cell, we obtain a new $\Sigma$-decomposition $\delta_1$ of $G_1\coloneqq G-A.$
This produces a triple $(G_1,\delta_1,W_1)$ which represents the arable land mapped out by $W_1.$
At this stage, we call the triple $(G_1,\delta_1,W_1)$ a \defi{$(t,\Sigma)$-meadow}.

From here on, we begin a sequence of cultivation steps that turn the wildly grown \textsl{meadow} into a neatly organized \textsl{orchard} whose ``flowers'' are ready for the \textsl{harvest}.
Below we provide a description of the individual steps and provide a sketch of how these ultimately lead to \autoref{local_structure_informal}.

For the purpose of our proofs, it is helpful to consider a specific type of witness for the existence of a graph $Z$ as a minor.
We call these objects \textsl{extensions} of $Z$ and they are defined in \autoref{def_extensions}.
For the purpose of this introduction, it suffices to think of an extension as a subgraph $M$ of $G$ which is minimal with respect to the following properties: 1) $M$ contains $Z$ as a minor, 2) $M$ is connected, 3) $M$ contains a vertex of $W_1,$ and 4) the total number of vertices of $M$ whose degree is different from two is bounded by some function in $|Z|$ which will be determined later.
In our proofs we do not actually require property 3), but our arguments are tailored to specifically deal with extensions of this type.
Moreover, 4) is not part of the actual definition of an extension of $Z,$ but an integral part of the content of \autoref{sec_harvest_crops} and \autoref{extr_flowers_thickets} is to prove that it suffices to only consider extensions of this type.

\paragraph{Plowing the meadow.}
Notice that all but one cell of $\delta_1$ are flaps.
For each flap we define a set $\mathsf{folio}(G_c,\Delta_c)$ that encodes all possible ways an extension of $Z$ may ``invade'' the linear society $(G_c,\Lambda_c).$
We also define the set $\mathsf{folio}(\delta_1)$ as the union of all possible folios of the flaps of $\delta_1.$
We will show that the size of this set is bounded by a function in $|Z|.$
We now divide the non-vortex part of $\Sigma=\Sigma^{(\mathsf{h}',\mathsf{c})'}$ into $2\mathsf{h}'+\mathsf{c}'+1$ \defi{enclosures}, which are areas that may be treated separately throughout our cultivation process.
In \autoref{thickets_break_pattern_sp} we provide an example of such a division.
Each face $F,$ with respect to $\delta_1,$ of an enclosure is called a \defi{brick} and we define $\mathsf{folio}(F)$ to be the union of all folios of the flaps that are either fully inside or on its boundary. 

Our first step is to plow the meadow and thereby prepare it for further cultivation.
That is we homogenize it by updating the triple $(G_1,\delta_1,W_1)$ to a new triple $(G_1,\delta_2,W_2)$ such that for every enclosure $C$ of the new walloid $W_2,$ all bricks of $C$ have the same folio.
This is done by both ``merging'' bricks of $(G_1,\delta_1,W_1)$ and by ``pushing'' some of the flaps of $\delta_1$ into the exceptional face of $W_2.$
Notice that such a process comes at a great price in terms of the order of the walloid $W_1$ and thus $W_1$ must be chosen to be enormous from the beginning.
This completes the plowing procedure.
 
\begin{figure}[ht]
\scalebox{.91}{\includegraphics{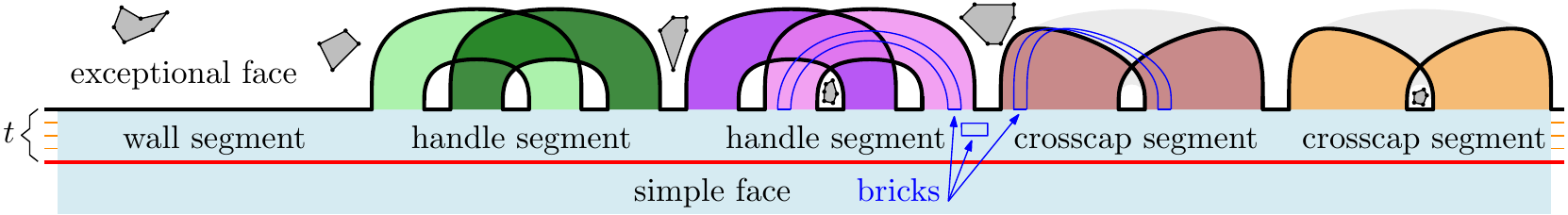}}
    \caption{A $(t,\Sigma^{(2,2)})$-meadow with two handle segments and two crosscap segments. The black line is the boundary of the exceptional face and the red line is the boundary of the simple face.
    The colored territories are the enclosures of $W.$
    The vortices are represented by the gray polygons and are all inside the exceptional face.
    }
  \llabel{thickets_break_pattern_sp}
\end{figure}

\paragraph{Sowing a garden.}
So far we have prepared our meadow and turned it into a field ready to receive the seeds that will eventually grow into the flowers and trees of our \textsl{orchard}.
Towards this goal, the next step on our agenda is to plant the first \textsl{flowers}.
This means that we are going to represent every linear society $(H,\Lambda)$ in $\mathsf{folio}(\delta_2)$ by ``planting'' a sequence of $h\coloneqq|Z|$ many consecutive $(t,b,H,\Lambda)$-flower segments in our meadow.
We refer to the newly created ``super segment'' consisting of all these flower segments as a $(t,b,h,H,\Lambda)$-\defi{parterre segment} where $b$ is a function in $k.$
The key feature of our plowed meadow $(G_1,\delta_2,W_2)$ that allows for this procedure is the homogeneity of its enclosures which implies that every society in $\mathsf{folio}(\delta_2)$ appears in abundance.
This allows us to rearrange $W_2$ in a way that such parterre appears on the boundary of the exceptional face.
Once this process is complete, we have obtained a new triple $(G_1,\delta_2,W_3)$ called a (blooming) \emph{$(t,b,h,\Sigma)$-garden}.
This garden now contains, for each member $F$ of $\mathsf{folio}(\delta_2)$ a parterre segment whose flowers represent this particular $F.$
Notice that $\delta_2$ still has only one vortex of (possibly) unbounded depth which sits in the exceptional face of $W_3.$

\begin{figure}[ht]
\begin{center}
\scalebox{0.85}{\includegraphics{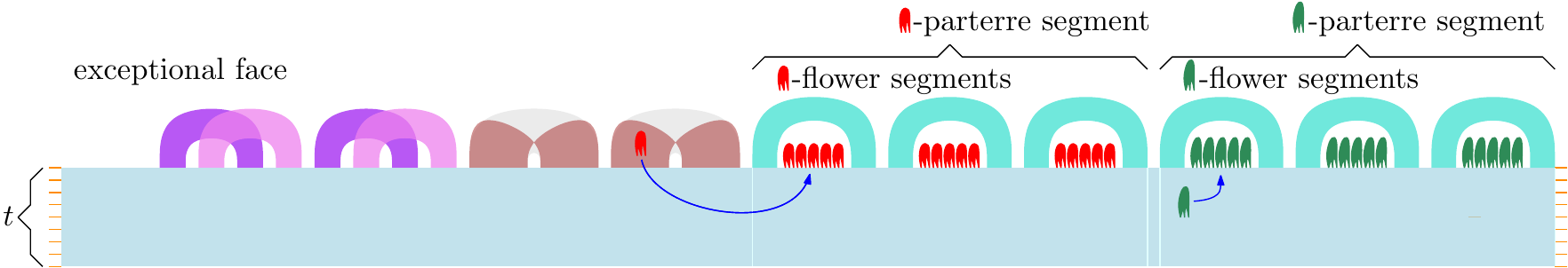}}
\end{center}
    \caption{A $(t,b,h,\Sigma^{(2,2)})$-garden with two handle segments, two crosscap segments, two parterre segments each containing three flower segments of two distinct linear societies (depicted in red and green).
  Each society appears 4 times inside each flower.}
  \llabel{thickets_break_pattern_before}
\end{figure}

\paragraph{Growing an orchard.}
With $(G_1,\delta_2,W_3)$ we have reached a first milestone in the cultivation of our farmland.
What is left is the wilderness in the form of the huge vortex that has been created in the exceptional face of $W_3$ during the previous steps.
The time has come to conquer this jungle and expand the arable land until the remaining wilderness is confined to a small number of rich biotopes that allow for other means of exploration.

To put our plan for the next step into more technical terms:
In the previous steps, we have given up on our vortices of bounded depth in order to achieve a form of homogenization with respect to the folio of the enclosures of $W_3.$
We now want to expand the homogeneous area and recreate (a possibly larger number of) the vortices of bounded depth.
This might create a new type of vortex: Before a vortex was a cell that could not be embedded in a disk in a way that agrees with our $\Sigma$-decomposition $\delta_2.$
Now, some vortices might actually have such a ``flat rendition'', but for every such rendition we choose, the folio of one of its flaps will contain a linear society $\langle H,\Lambda\rangle$ that is not abundant.
This means that, apriori, we do not have enough copies of $\langle H,\Lambda\rangle$ available to create a large enough parterre segment for this kind of flower.

To fully formalize the new target structure, we need to introduce a new type of $t$-segment: a \textsl{$(t,s)$-thicket segment}.
While before we always attached flowers or edges in a very structured manner to the top vertices of a $t$-wall segment, now we have to be a bit more lax.
We select a collection of $s$ concentric cycles $\mathcal{C}=\{ C_1,\dots,C_s\},$ called the \textsl{nest}, all of which are drawn in $\delta_2,$ and extend the vertical paths of our $t$-wall segment beyond the top vertices, letting the resulting paths traverse the cycles of $\mathcal{C}$ and reaching the ``inner'' area, that is the disk $\Delta$ of $\Sigma$ which is enclosed by the closed curve that represents $C_1$ in $\delta_2.$
These paths are called the \textsl{rails} of the thicket.
We then require that within this disk $\Delta$ there exists a single vortex of $\delta_2.$
A triple $(G',\delta',W')$ is now called a \textsl{$(t,b,h,s,\Sigma)$-orchard} if it can be seen as a $(t,b,h,\Sigma)$-garden when pushing all of its thicket segments back into a single vortex, and there exists a $(t,s)$-thicket segment for every vortex of $\delta'.$
See \autoref{thickets_break_pattern_more} for an illustration.
Notice that each thicket segment induces its own linear society based on the outer-most cycle $C_s$ of its nest.
This allows us to define the \textsl{depth} of a thicket in the same way we defined the depth of a vortex above.
Our goal now is to create more and more thicket segments, until all thickets are of bounded depth or we have created a parterre segment for each possible flower (recall that there are only finitely many possible flowers because we are only interested in flowers of bounded ``detail'').

Since $\delta_2$ has a single vortex, we may immediately interpret $(G_1,\delta_2,W_3)$ as an $(t,s)$-orchard with a single thicket\footnote{This transformation actually requires to sacrifice a bit of the infrastructure of $W_3$ as observed in \autoref{obs_single_thicket_orchard} but for ease of presentation we skip this very minor transformation step.}.
Before we describe the process of ``splitting'' a single thicket segment let us quickly discuss why there exists an upper bound on the total number of thicket we have to consider at any point in time.
For this, observe first that the original $\Sigma$-decomposition $\delta_1$ had a bounded number, say $x,$ of vortices that contained an actual obstruction to disk-embeddability.
Let us call these vortices the \textsl{real} vortices for now.
These vortices were already of bounded depth and we may use this fact to show that there can never be more than $x$ real vortices at any point in time.
A parterre segment requires us to have $h\cdot b$ many isomorphic flowers.
Moreover, since the ``detail'' of the flowers we are interested in is bounded by some number, say $d,$ there are at most $\mathsf{d}=2^{2^{\mathcal{O}(d^2)}}$ many possible folios (see \autoref{obs_size_dfolio}).
This means, if we ever have more than $h\cdot b\cdot \mathsf{d}^2$ thicket segments whose vortices are not real, then there must be $h\cdot b\cdot \mathsf{d}$ many, all with the same folio.
Let $\mathcal{F}$ be this folio.
This gives us the right, using the infrastructure provided by the rails and the nests of these thickets, to extract a large enough parterre segment for each member of $\mathcal{F}$ and afterward, the thicket segments with folio $\mathcal{F}$ can be discarded.
This means that we have an upper bound on the total number of thickets present at any given point in time of our procedure.

\begin{figure}[ht]
\begin{center}
\scalebox{.823}{\includegraphics{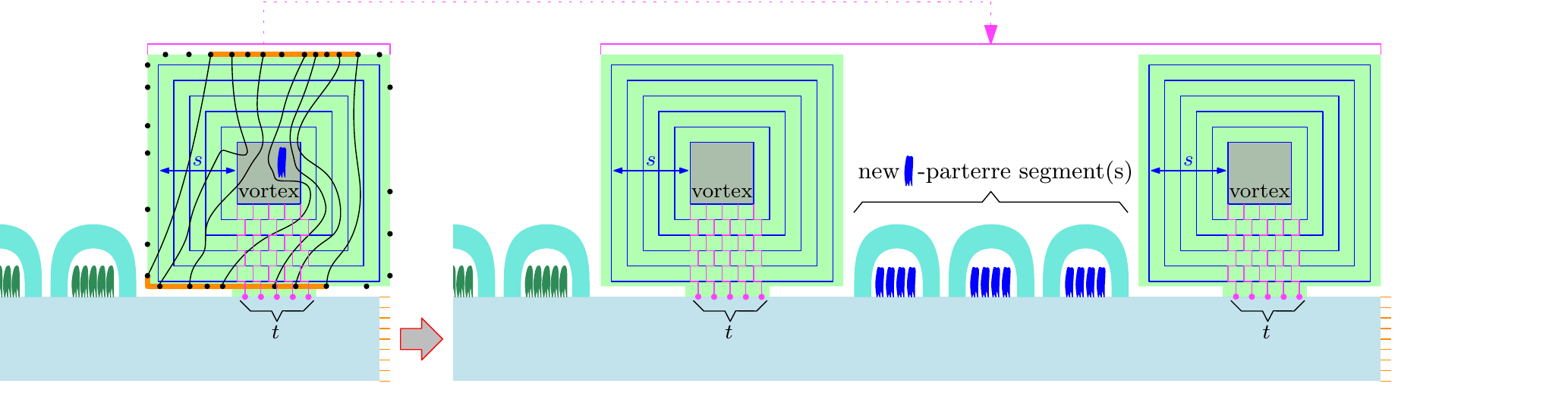}}
\end{center}
    \caption{Left: a $(t,b,h,s,\Sigma^{(2,2)})$-orchard with one $(s,t)$-thicket and a big transaction traversing it (in black).
    The black vertices mark the boundary of the thicket society.
    Right: the result of splitting a thicket along this transaction and creating a new parterre.
    The magenta vertices form the base of the thicket.}
  \llabel{thickets_break_pattern_more}
\end{figure} 

We are now ready to describe the operation of ``splitting'' a thicket.
For this, we consider a large transaction on the linear society defined by the outer-most cycle of the nest.
We may assume that such a transaction exists since otherwise the thicket is of bounded depth and there is nothing to be done.
Moreover, since there are at most $x$ real vortices and all of them are of bounded depth, which is already certified by $\delta_1,$ we may assume that the ``area'' defined by the transaction can be used to extend $\delta_2$ by reintegrating a part of $\delta_1$ that was forgotten earlier.
Such a transaction is depicted in the left part of \autoref{thickets_break_pattern_more}.
When we have found this new ``flat'' area, we may use the same arguments as those used for the plowing step above to further refine the transaction into a homogenized area with respect to the folios of its cells.
In case we find a linear society in these folios that has not been represented in a parterre segment of $W_3,$ we may now use the infrastructure provided by the rails and the nest of the thicket, together with the paths of the transaction, in order to extract additional parterre segments.
An illustration of this can be found in the right part of \autoref{thickets_break_pattern_more}.
Finally, the paths of the transaction, together with the cycles of the nest $\mathcal{C}$ now define up to two new disks in the area that was a single vortex before.
By making use of the paths in the transaction, the nest, and the rails, this allows to ``split' the thicket into up to two new thicket segments as depicted in \autoref{thickets_break_pattern_more}.
Please notice that all of the operations described here require the sacrifice of some part of our infrastructure.
However, as discussed above, the total number of thickets and the total number of folios, and therefore parterre segments, is bounded.
This means that any of these steps can only occur a bounded number of times and thus, by choosing the infrastructure provided by $(G_1,\delta_2,W_3)$ to be large enough in the beginning, we make sure to always have enough ``space'' available.

Eventually, the extraction operations above will not be available anymore.
We might still find more large transactions on the societies of our thickets, but these transactions can now only create new nests that are increasingly ``tight'' around the vortices.
Finally, the entire process must come to a halt and when this happens, all of our thickets are of bounded depth.
We call the resulting triple $(G_1,\delta_3,W_4)$ a \textsl{ripe} $(t,b,h,s,\Sigma)$-orchard.
See \autoref{fig_half_integral_packing} for an illustration. 

\begin{figure}[ht]
\begin{center}
\scalebox{.834}{\includegraphics{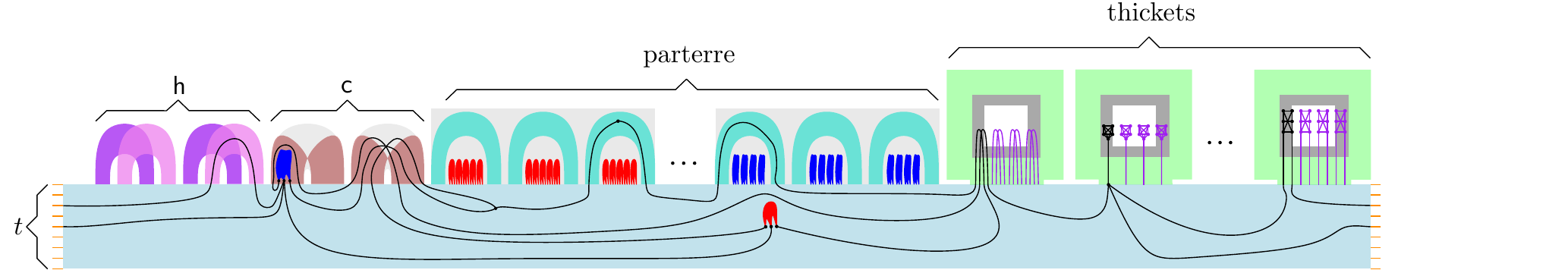}}
\end{center}
    \caption{A ripe $(t,b,h,s,\Sigma^{(2,2)})$-orchard: All vortices (one for each thicket) are of bounded depth together with some extension $Z'$ of a graph $Z$ (black) invading (some of) the thickets.
    This extension goes through flaps and orchards.
    Each flap is $k$-fold represented in a flower and each orchard can also be seen as a flower.
    In the figure, the red and the blue linear societies of $Z'$ are located inside flaps and are represented by blue and red flowers inside the corresponding parterre.}
  \llabel{fig_half_integral_packing}
\end{figure} 

\paragraph{Harvesting.}
We next focus our attention on how the {ripe} $(t,b,h,s,\Sigma)$-orchard $(G_1,\delta_3,W_4)$ can be traversed by an extension $M$
of some graph $Z$ where $|Z|≤h.$ 
As we already mentioned, $M$ has a bounded (by some quadratic function in $|Z|$) number of vertices of degree different from two.
We call these vertices \defi{terminals} of $M.$
As we argued before, the ways $M$ may visit the flaps of $\delta$ are already represented by the flowers of the parterres.
So we should fix our attention on how $M$ may invade the non-flat cells of $\delta,$ i.e., its vortices.
recall that each vortex sits deep inside the thickets and consider $\Delta$ to be a disk around each vortex $c$ defined as the are enclosed by the closed curve that is the outermost cycle of a given thicket segment of $(G_1,\delta_3,W_4).$
Since $(G_1,\delta_3,W_4)$ is ripe, there are only few vortices and therefore few disks as above, and the depth of both the given vortex as well as the society defined by $\Delta$ is bounded by some function in $h$ and $k.$

Next, we fix our attention to the part of $M$ which is drawn inside the disk $\Delta.$
Let us denote this part of $M$ by $M_{\Delta}.$
Also, we use $T_{\Delta}$ for the terminals of $M$ that are drawn inside $\Delta.$
This defines the notion of a \textit{fiber} $\mu = (M_{\Delta}, T_{\Delta},\widehat{\Lambda})$ where $\widehat{\Lambda}$ is a cyclic ordering of the vertices of $M$ that are drawn on the boundary of $\Delta$ in the order they appear in $\widehat{\Lambda}.$
While the number of terminals of any fiber is bounded, we cannot say the same for its boundary vertices, i.e., for the vertices in $\widehat{\Lambda}.$
For this, we consider fibers whose boundary has size at most $b,$ let us call them \textit{$b$-fibers}, and we prove a local Erd\H{o}s-P{\'o}sa duality on how they may be covered or packed inside $\Delta$ with respect to the boundary of $\Delta.$
This is done by the so-called \textsl{harvesting procedure} which either finds a set that intersects any $b$-fiber or finds a sequence of a big enough number of copies of such $b$-fibers which are well arranged on the boundary of $\Delta$ in the form of a ``primitive flower'' (\cref{sec_harvest_crops}). 

We see a $b$-fiber as being drawn inside the union of some collection of sub-disks of $\Delta,$ called \defi{pseudo-disk}.
Here, the number of vertices drawn in the boundary of this pseudo-disk is bounded by some function in the depth of the vortex $c$ and the depth of the disk $\Delta$ (see \cref{pse_cylindr_rend}).
While the depth of both is bounded in a function in $h$ and $k,$ we are able to show that, for a single $b$-fiber, the ``depth'' of its drawing now only depends on $h.$
For the harvesting procedure, we define the notion of a \textsl{critical pseudo-disk} that  ``critically'' contains a $b$-fiber $\mu$ in the sense that $\mu$ is not contained in any of its proper pseudo-disks (\cref{Critical_pseudo_disks}).
Then we pick a $b$-fiber $\mu$ and we detect a packing of critical pseudo-disks containing $\mu$ (\cref{harvesting_fibers}). 
If ``enough'', where the value for enough depends on a specific function in $k$ and $h,$ of them are found we can organize them into \textit{packages} each containing many, say $k,$ copies of each of the sub-fibers corresponding to the elements of $\mathsf{dec}(M_{\Delta},\widehat{\Lambda}).$
This collection of packages will be the primitive form of a flower, parts of which will constitute some of the flower segments of our parametric graphs.
If only few critical pseudo-disks are found, we may consider the vertices of the boundaries of their disks in order to find a covering of all $b$-fibers of $\Delta$ that are isomorphic to $\mu.$
The size of this set is bounded by a function in $k$ and $h.$ 
The union of all these covering sets for all possible fibers of boundary at most $d$ is the $(d,k)$-\textsl{cover} of the disk $\Delta.$ 
Recall that fibers were defined as ways that an extension $M$ enters the disk $\Delta$ and thus, this $(d,k)$-cover should now constitute a hitting set for all such fibers which are not abundant on the boundary of $\Delta.$
We next prove that we may choose $d$ so that if $S$ is a $(d,k)$-\textsl{cover} of the disk $\Delta,$ then every restriction $M_{\Delta}$ of an extension $M$ in $\Delta$ that defines a fiber with more than $d$ boundary vertices and which is not covered by $S$ can be rearranged to one that defines a $d$-fiber and still is not intersected by $S$ (\cref{main_harvesting_lemmata}).
The proof of this applies a global rerouting of $M,$ based on the unique linkage theorem and some suitable adaptation of it to our needs (\cref{tamming_more_sec}).
The proof also uses the fact that $S$ resides on the union of disk boundaries and strongly uses the topological properties of these disks as well as the cycle collection $\Ccal$ insulating the vortex $c$ from $\Delta.$
This means that whenever $S$ is not a covering of our model $M$ some part of this model, in the form of a collection of packages of $k$ fibers 
can be found inside $\Delta$ as desired.
Moreover, the above choice of $d$ is bounded by some function in $h.$
This implies that it is enough to ``pack'' or ``cover'' $b$-fibers, for his particular value of $b,$ instead of fibers in general.
Therefore, the $(d,k)$-\textsl{cover} of the disk $\Delta$ is bounded by some function in $k$ and $h.$
We now define $S_{\Delta}$ as the union of all  the $(d,k)$\textsl{covers} of all the disks $\Delta,$ as above, corresponding to the thickets of the $(t,b,h,s,\Sigma)$-orchard $(G_1,\delta_3,W_4).$
We denote this collection of disks by $\mathbf{\Delta}$ and define $S_{\mathsf{all}}$ to be the union of the (already bounded) apex set $A$ from the very beginning and $\bigcup_{\Delta\in\mathbf{\Delta}}S_{\Delta}.$
Clearly, $|S_{\mathsf{all}}|$ is bounded by a function in $h$ and $k$ and by the number of thickets of $(G_1,\delta_3,W_4)$ which is also bounded by some function in $k$ and $h.$

\paragraph{Plowing flowers from thickets.}

Our next step is to analyze how the Erd\H{o}s-P{\'o}sa duality for objects ``anchored'' on the society of a vortex from the previous step interacts with the overall way an extension $M$ of $Z$ may be situated within our $(t,b,h,s,\Sigma)$-orchard $(G_1,\delta_3,W_4)$ and its thickets.
It follows that if $S_{\mathsf{all}}$ intersects every extension of $Z$ that is ``bound'' to $\delta$ (i.e., it contains a vertex of $W_4$; for the formal definition, see \cref{bound_def}), then the component of $G-S_{\mathsf{all}}$ which contains the majority of $W_{4}-S_{\mathsf{all}}$ is $Z$-minor-free and this is the second outcome of the (simplified) local structure theorem (\cref{local_structure_informal}).

\begin{figure}[ht]
\begin{center}
\scalebox{.83}{\includegraphics{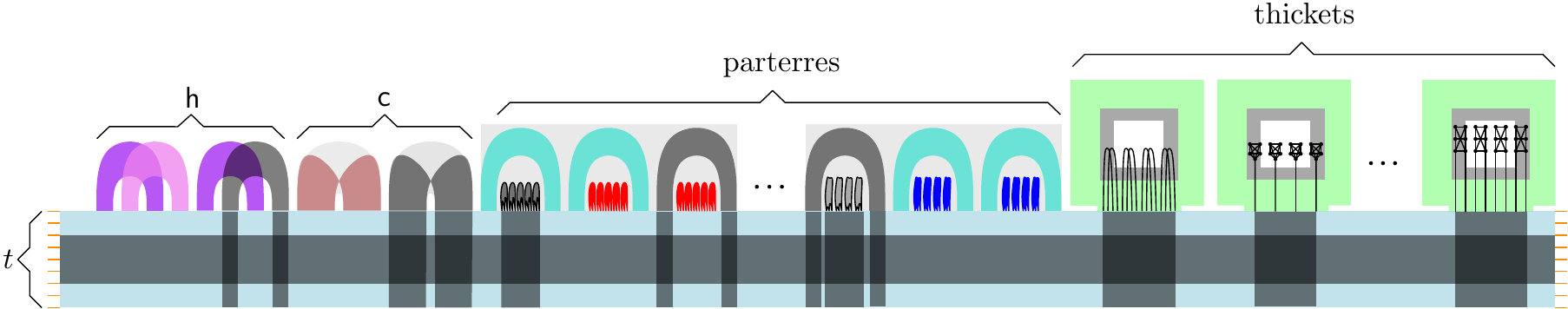}}
\end{center}
\caption{A subdivision of a walloid in $\mathscr{W}^{\mathbf{p}}$ (in black) inside an orchard.}
\llabel{fig_half_integral_packing_2}
\end{figure}

Assume now that there is an extension $M$ of $Z$ which is not intersected by $S_{\mathsf{all}}.$
Then, due to the Erd\H{o}s-P{\'o}sa property above, the parts of $M$ corresponding to the disks $\Delta\in\mathbf{\Delta'}$ are organized as a sequence $\Mcal_{1}^{\Delta},\ldots,\Mcal_{p}^{\Delta}$ of consecutive ``packages'', each containing $k$ vertex disjoint and pairwise isomorphic 
sub-fibers emerging from the decomposition of some $\Delta$-invading fiber. 
At this point, we observe that among these sub-fibers, only those that invade the corresponding vortices are of interest as all others, being drawn in the flat parts of $\delta,$ are already represented in the parterre of $(G_1,\delta_3,W_4).$
Also, it follows from the results in \cref{tamming_more_sec} than the number of disks intersected by $M,$ we denote the set of these disks by $\mathbf{\Delta}'\subseteq\mathbf{\Delta},$ is also bounded by some function in $h.$
Let $\Sigma^{\mathsf{c}}$ be the closure of $\Sigma\setminus\bigcup_{\Delta\in\mathbf{\Delta}'}\Delta.$
We now observe that the part of $M,$ call it $M^{\mathsf{c}},$ that is drawn outside the disks of $\mathbf{\Delta}'$ is drawn in $\Sigma$ in a ``flat way'' as all invaded vortices are inside the disks. This drawing is encoded as a $\Sigma^{\mathsf{c}}$-decomposition of $M^{\mathsf{c}}$ and constitutes a certificate of the flatness of $M$ outside the disks of $\mathbf{\Delta}'.$
In \cref{certif_flat} we formalize this flatness certificate via the concept of a \textit{$\Sigma$-flat $Z$-certifying complement} of  $\{\mu_{\Delta}\mid \Delta\in\mathbf{\Delta}'\}.$ 
Observe now that the fibers in the packages $\Mcal_{1}^{\Delta},\ldots,\Mcal^{\Delta}_{p^{\Delta}}$ are close to the structure of a flower based on $M_{\Delta}.$
{However, these fibers are rooted on the boundary of $\Delta$ and we need their boundaries to be in the base of the thicket corresponding to $\Delta$ in order to see them as a flower. 
For this, we prove a special variant of the combing lemma from \cite{GolovachST22Combing} in \cref{Variant_combing}.
This lemma permits a rerouting that, after sacrificing a considerable amount of the fibers, revises $\Mcal_{1}^{\Delta},\ldots,\Mcal^{\Delta}_{p^{\Delta}}$ so that it is indeed rooted at the base of the thicket corresponding to $\Delta,$ for every $\Delta\in\mathbf{\Delta}'$ (\cref{the_main_extr_flowers_final}). }
Later we additionally use the surrounding infrastructure of the corresponding thickets in order to interpret all thickets whose vortices are invaded by $M$ as flowers.
All of these transformations have as a starting point a sequence of packages $\Mcal_{1}^{\Delta},\ldots,\Mcal^{\Delta}_{p^{\Delta}},$ for each $\Delta\in\mathbf{\Delta}',$ corresponding to a collection of fibers that admit some $\Sigma$-flat $Z$-certifying complement.
In our proofs, this flat complement certification property stays invariant for all revised packages.
We now arrived in a situation where all parts of $M$ that are embedded outside of the vortices of $(G_1,\delta_3,W_4)$ are represented by the flowers of its parterres while the parts of $M$ that are invading the vortices, i.e., the vortices of the disks in $\mathbf{\Delta}',$ are fibers corresponding to sequences of packages that permit us see the thickets as flowers of these fibers.
Let $M'$ be the graph obtained by considering all boundary vertices of the fiber $\mu_{\Delta}$ generated by $M$ as terminal vertices.
This defines a subdivision of a graph whose size is bounded by some function in $h$ that contains 
$Z$ as a minor.
As we prove in \cref{lemma_grounded_extensions_to_obstructions}, this function is the function \textsf{ext} used in the definition of embedding pairs.
Next, using the graph $M',$ we interpret the subgraphs of $M'$ drawn in the cells of the flat complement $\delta^{\mathsf{c}}$ and the graphs drawn in the disks of $\mathbf{\Delta}'$ as a subgraph partition $\Pcal$ of $M'$ and this, in turn, defines an embedding pair $\mathbf{p}=(\Sigma,\mathbf{B})$ where $\mathbf{B}$ contains societies that are either the parts of $M'$ drawn in the flaps of $\delta^{\mathsf{c}}$ (which are also flaps of $\delta$) or the parts of $M'$ that are the fibers in $\{\mu_{\Delta}\mid \Delta\in\mathbf{\Delta}'\}.$
As we already noticed, the first category of societies is already present in the parterres of $(G_1,\delta_3,W_4)$ and the second category of societies is present as flowers inside the thickets.
As a consequence of this, a subdivision of some big enough walloid in $\mathscr{W}^{\mathbf{p}}$
can be found as a subgraph of $G$ (see \cref{fig_half_integral_packing_2}).
This gives the first outcome of the (simplified) local structure theorem (\cref{local_structure_informal}).

\section{Preliminaries}\llabel{sec_preliminaries}

In this section, we present the bulk of definitions that will be necessary for the development of the proofs in the following sections.

\subsection{Basic concepts}

In this first subsection, we present the basic concepts we require about sets, integers, and graphs.

\paragraph{Sets and integers.} We denote by $\mathbb{N}$ the set of non-negative integers, by $\Nbbb_{\geq n},$ $n > 1,$ to be the set $\Nbbb \setminus \{ m \in \Nbbb \mid m < n \},$ and by $\Nbbb^\mathsf{even}$ the set of even numbers in $\Nbbb.$
Given two integers $p, q,$ where $p \leq q,$ we denote by $[p, q]$ the set $\{p, \dots, q\}.$ For an integer $p \geq 1,$ we set $[p] = [1, p]$ and $\mathbb{N}_{\geq p} = \mathbb{N} \setminus [0, p - 1].$
For a set $S,$ we denote by $2^{S}$ the set of all subsets of $S$ and by $\binom{S}{2}$ the set of all subsets of $S$ of size $2.$
If $\mathcal{S}$ is a collection of objects where the operation $\cup$ is defined, then we denote $\cupall S = \bigcup_{X \in \mathcal{S}} X.$
Also, given a function $f\colon A\to B$ we always consider its extension $f:2^A\to B$ such that for every $X\subseteq A,$ $f(X)=\{f(a)\mid a \in X\}.$
  
\paragraph{Basic concepts on graphs.} A graph $G$ is a pair $(V, E)$ where $V$ is a finite set and $E \subseteq \binom{V}{2},$ i.e., all graphs in this paper are undirected, finite, and without loops or multiple edges. When we denote 
an edge $\{x,y\},$ we use instead the simpler notation $xy$ (or $yx$). 
We write $\gall$ for the set of all graphs.
We also define $V(G) = V$ and $E(G) = E.$

We say that a pair $(A, B) \in 2^{V(G)} \times 2^{V(G)}$ is a \defi{separation} of $G$ if $A \cup B = V(G)$ and there is no edge in $G$ between $A \setminus B$ and $B \setminus A.$
We call $|A \cap B|$ the \defi{order} of $(A, B).$

If $H$ is a subgraph of $G,$ that is $V(H)\subseteq V(G)$ and $E(H)\subseteq E(G),$ we denote this by $H\subseteq G.$
Given two graphs $G_{1}$ and $G_{2},$ we denote $G_{1}\cup G_{2}=(V(G_{1})\cup V(G_{2}),E(G_{1})\cup E(G_{2})).$ We also use $G_{1}+G_{2}$ to denote the disjoint union of $G_{1}$ and $G_{2}.$
Also, given a $t \in \Nbbb,$ we denote by $t\cdot G$ the disjoint union of $t$ copies of $G.$

Given a vertex $v \in V(G),$ we denote by $N_{G}(v)$ the set of vertices of $G$ that are adjacent to $v$ in $G.$
Also, given a set $S \subseteq V(G),$ we set $N_{G}(S) = \bigcup_{v \in S} N_{G}(v) \setminus S.$

For $S \subseteq V(G),$ we set $G[S] = (S, E \cap \binom{S}{2})$ and use $G - S$ to denote $G[V(G) \setminus S].$ We say that $G[S]$ is an \defi{induced (by $S$) subgraph} of $G.$

Given an edge $e = uv \in E(G),$ we define the \defi{subdivision} of $e$ to be the operation of deleting $e,$ adding a new vertex $w,$ and making it adjacent to $u$ and $v.$ 
Given two graphs $H$ and $G,$ we say that $H$ is  \defi{a subdivision} of $G$ if $H$ can be obtained from $G$ by subdividing edges.

The \defi{contraction} of an edge $e = uv \in E(G)$ results in a graph $G'$ obtained from $G \setminus \{ u, v \}$ by adding a new vertex $w$ adjacent to all vertices in the set $(N_{G}(u) \cup N_{G}(v)) \setminus \{ u, v \}.$
A graph $H$ is a \defi{minor} of a graph $G$ if $H$ can be obtained from a subgraph of $G$ after a series of edge contractions.
We denote this relation by $\leq.$
Given a set $\Hcal$ of graphs, we use $\Hcal≤G$ in order 
to denote that at least one of the graphs in $\Hcal$ is a minor of $G.$
We refer the reader to~\cite{Diestel2010Book} for any undefined terminology on graphs.

\subsection{Surfaces, decompositions, and renditions}\llabel{subsec_almost_embeddings}

In this subsection, we present all definitions we require that are relevant to drawing and embedding graphs in surfaces as well as more involved definitions such as surface decompositions, societies, and more.

\paragraph{Surfaces.}
Given a pair $(\mathsf{h},\mathsf{c}) \in \mathbb{N} \times [0,2]$ we define $\Sigma^{(\mathsf{h}, \mathsf{c})}$ to be the two-dimensional surface without boundary created from the sphere by adding $\mathsf{h}$ handles and $\mathsf{c}$ crosscaps (for a more detailed definition see~\cite{MoharT01Graphs}).
If $\mathsf{c} = 0,$ the surface $\Sigma^{(\mathsf{h},\mathsf{c})}$ is an \defi{orientable} surface, otherwise it is a \defi{non-orientable} one.
By Dyck's theorem \cite{Dyck1888Beitrage,Francis99ConwayZIP}, two crosscaps are equivalent to a handle in the presence of a (third) crosscap.
This implies that the notation $\Sigma^{(\mathsf{h}, \mathsf{c})}$ is sufficient to denote all two-dimensional surfaces without boundary.

\paragraph{Societies.}
Given a graph $G$ and a set $X = \{x_{1},\ldots, x_{n}\} \subseteq V(G),$ we consider two types of orderings of $X.$
\begin{itemize}
\setlength\itemsep{0pt}
\item a \defi{rotation ordering} $\Omega = \rot{x_{1}, \ldots, x_{n}}$ where we assume that $\rot{x_{1}, \ldots, x_{n}} = \rot{x_{1+i}, \ldots, x_{n+i} },$ for every $i \in [0,n-1].$
Note that $n+i$ is interpreted as $(n-1 \mod n) + i + 1.$

\item a \defi{linear ordering} $\Lambda = \lin{x_{1},\ldots,x_{n}}.$
Given a linear ordering $\Lambda = \lin{x_{1}, \ldots, x_{n}}$ we define its \defi{$i$-position shift} as the linear ordering $\lin{x_{1+i}, \ldots, x_{n+i}}.$
Note that, as before, $n+i$ is interpreted as $(n-1 \mod n) + i + 1.$
\end{itemize}

Given a rotation (resp. linear) ordering $\Omega = \rot{x_{1},\ldots, x_{n}}$ (resp. $\Lambda = \lin{x_{1}, \ldots, x_{n}}$) we define its reverse as $\rev(\Omega) \coloneqq \rot{x_{n},\ldots, x_{1}}$ (resp. $\Lambda \coloneqq \lin{x_{n},\ldots, x_{1}}$).
Given a linear ordering $\Lambda,$ we define $\mathsf{next}_{\Lambda}(x_{i}) = x_{i+1},$ for every $i \in [n-1]$ and $\mathsf{next}_{\Lambda}(x_{n}) \coloneqq x_{1}.$
We also define $\mathsf{previous}_{\Lambda}(x_i) \coloneqq \mathsf{next}_{\mathsf{rev}(\Lambda)}(x_{i}).$
When $\Lambda$ is clear from the context we simply write $\mathsf{next}(x_{i})$ and $\mathsf{previous}(x_{i}).$
Moreover, given two distinct linear orderings $\Lambda = \lin{x_{1}, \ldots, x_{n}}$ and $\Lambda' = \lin{y_{1}, \ldots, y_{m}},$ we define their \defi{concatenation} as the linear ordering $\Lambda \oplus \Lambda' = \lin{x_{1}, \ldots, x_{n}, y_{1}, \ldots, y_{m}}.$

We define the \defi{shifts} of a linear ordering $\Lambda$ as the set containing every $i$-position shift of $\Lambda,$ for $i \in [n].$
Observe that, given a rotation ordering $\Omega = \rot{x_{1},\ldots, x_{n}}$ we can extract $n$ 
different linear orderings depending on the choice of the first vertex.
Let $\Lambda$ be any of these linear orderings.
Then all these linear orderings, which are the shifts of $\Lambda,$ correspond to a single rotation ordering that we call the \defi{rotation ordering of $\Lambda$}.

\medskip
A \defi{linear} (resp. \defi{rotation}) \defi{society} is a pair $\lin{G,\Lambda }$ (resp. $\rot{G,\Omega}$) where $G$ is a graph and $\Lambda $ (resp. $\Omega$) is a linear (resp. rotation) ordering of a subset of the vertices of $G.$
The \defi{rotation society} of a linear society $\lin{G, \Lambda}$ is the rotation society $\rot{G, \Omega},$ where $\Omega$ is the rotation ordering of $\Lambda.$
We denote by $V(\Lambda)$ (resp. $V(\Omega)$) the set of vertices on that correspond to elements of $\Lambda$ (resp. $\Omega$).
We also use $|\Lambda| \coloneqq |V(\Lambda)|$ (resp. $|\Omega| \coloneqq |V(\Omega)|$).
From this point forward we always use the notation $\lin{G, \Lambda}$ (resp. $\rot{G, \Omega}$) to denote linear (resp. rotation) societies.
Whenever we use the term \defi{society} we will be referring to a linear society.
On the other hand, we will explicitly mention when we are dealing with a rotation society.

A \defi{subdivision} of a linear (resp. rotation) society $\lin{G,\Lambda }$ (resp. $\rot{G,\Omega}$) is every linear (resp. rotation) society $\lin{M, \Lambda}$ (resp. $\rot{M, \Omega}$) where $M$ is a subdivision of $G.$

Two societies $\lin{G,\Lambda }$ and $\lin{G',\Lambda '}$ are \defi{isomorphic} if there is an isomorphism $\psi \coloneqq V(G) \to V(G')$ from $G$ to $G'$ such that $\psi(\Lambda) = \Lambda',$ i.e., if $\Lambda = \lin{x_{1}, \ldots, x_{r}},$ it holds that $\Lambda' = \lin{\psi(x_{1}), \ldots, \psi(x_{r})}.$

Given a linear ordering $\Lambda$ and a set $S \subseteq V(\Lambda),$ we define $\Lambda \cap S$ (resp. $\Lambda \setminus S$) as the linear ordering $\Lambda'$ obtained from $\Lambda$ if we remove all elements that do not (resp. do) belong in $S.$
Given a linear society $\lin{G, \Lambda}$ and a set $S \subseteq V(G),$ we define $\lin{G, \Lambda} - S$ as the linear ordering $\lin{G - S, \Lambda \setminus S}.$\llabel{society_minus}

A \defi{cross} in a society $\lin{G, \Lambda}$ is a pair $(P_1,P_2)$ of disjoint paths\footnote{Two paths are \emph{disjoint} when their vertex sets are disjoint.} in $G$ such that $P_i$ has endpoints $s_i, t_i \in V(\Lambda)$ and is otherwise disjoint from $V(\Lambda),$ and the vertices $s_1, s_2, t_1, t_2$ occur in $\Lambda $ in the order listed.

\paragraph{Drawing a graph in a surface.}\llabel{many_draw}
Let $\Sigma$ be a surface. 
If $\Sigma$ has a boundary we denote it by $\bd(\Sigma)$ and we refer to $\Sigma \setminus \bd(\Sigma)$ as the \defi{interior} of $\Sigma.$
A \defi{$\Sigma$-drawing} (with crossings) is a triple $\Gamma=(U,V,E)$ such that
\begin{itemize}[itemsep=-2pt] 
\setlength\itemsep{0em}
\item $V$ and $E$ are finite, 
\item $V\subseteq U\subseteq \Sigma$ and $\bd(\Sigma)\cap \Gamma\subseteq V,$ 
\item $V\cup\bigcup_{e\in E}e=U$ and $V\cap (\bigcup_{e\in E}e)=\emptyset,$ 
\item for every $e\in E,$ either $e=h((0,1)),$ where $h\colon[0,1]_{\mathbb{R}}\to U$ is a homeomorphism onto its image with $h(0),h(1)\in V,$ or $e = h(\Sbbb^{2} - (1, 0)),$ where $h : \Sbbb^{2} \to U$ is a homeomorphism onto its image with $h(0,1) \in V$ (where $\Sbbb^{2}=\{(x,y)\in\Rbbb^2\mid x^2+y^2=1\}$), and
\item if $e,e'\in E$ are distinct, then $|e\cap e'|$ is finite.
\end{itemize}
We call the set $V$ the \defi{nodes} and the set $E$ the \defi{arcs} of $\Gamma.$
If $G$ is a graph and $\Gamma = (U,V,E)$ is a $\Sigma$-drawing with crossings such that the nodes $V$ and the arcs $E$ correspond to the vertices $V(G)$ and the edges $E(G)$ respectively, we say that $\Gamma$ is a \defi{drawing} of $G$ in $\Sigma$ (possibly with crossings) or alternatively a $\Sigma$-drawing of $G.$
Two distinct arcs in $E$ \defi{cross} if they have a common point.
If no two arcs in $E$ cross, we say that $\Gamma$ is a $\Sigma$-\defi{drawing} of $G$ \defi{without crossings} or, alternatively, a $\Sigma$-\defi{embedding} of $G.$
In this last case, the connected components of $\Sigma\setminus U,$ are the \defi{faces} of $\Gamma.$
  
\paragraph{Surface decompositions.}
Let $\Sigma$ be a surface, possibly with boundary.
A \defi{closed} (resp. \defi{open}) \defi{disk} of $\Sigma$ is any closed (resp. open) disk $\Delta$ that is a subset of $\Sigma.$
For brevity, whenever we use the term \defi{disk} we mean a closed disk.
A $\Sigma$-\defi{decomposition} of a graph $G$ is a pair $\delta = (\Gamma, \mathcal{D}),$ where $\Gamma = (U,V,E)$ is a $\Sigma$-drawing of $G$ and $\mathcal{D}$ is a collection of disks of $\Sigma$ such that 
\begin{enumerate}
\item the disks in $\mathcal{D}$ have pairwise disjoint interiors, 
\item for every disk $\Delta \in \Dcal,$ $\Delta \cap U \subseteq V,$ i.e., the boundary of $\Sigma$ and of each disk in $\mathcal{D}$ intersects $U$ only in vertices, 
\item if $\Delta_1, \Delta_2 \in \mathcal{D}$ are distinct, then $\Delta_1 \cap \Delta_2 \subseteq V(\Gamma),$ and
\item every edge $\Gamma$ is a subset of the interior of one of the disks in $\mathcal{D}.$
\end{enumerate}
Let $N$ be the set of all nodes of $\Gamma$ that do not belong to the interior of any disk in $\mathcal{D}.$
We refer to the elements of $N$ as the \defi{nodes} of $\delta.$ 
If $\Delta \in \mathcal{D},$ then we refer to the set $\Delta - N$ as a \defi{cell} of $\delta.$
We denote the set of nodes of $\delta$ by $N(\delta)$ and the set of cells of $\delta$ by $C(\delta).$ 

Given a cell $c \in C(\delta),$ we define the \defi{disk of $c$} as the disk $\Delta_{c} \coloneqq \bd(c) \cup c.$
For a cell $c \in C(\delta),$ the set of nodes of $\delta$ that belong to $\Delta_{c}$ is denoted by $\widetilde{c}.$
Thus the cells $c$ of $\delta$ such that $\widetilde{c} \neq \emptyset,$ form the hyperedges of a hypergraph with vertex set $N(\delta),$ where $\widetilde{c}$ is the set of vertices incident with $c.$
For a cell $c \in C(\delta),$ we define the graph $\sigma_{\delta}(c),$ or $\sigma(c)$ if $\delta$ is clear from the context, to be the subgraph of $G$ consisting of all vertices and edges that are drawn in $\Delta_{c}.$
We define $\pi_{\delta} \colon N(\delta) \to V(G)$ to be the mapping that assigns to every node in $N(\delta)$ the corresponding vertex of $G.$
We also define the set $\ground(\delta) \coloneqq \pi_{\delta}(N(\delta))$ of \defi{ground vertices} in $\delta.$
If $\pi_{\rho}(p) = x$ then we say that $p$ is \defi{the point of} $x$ in $\delta.$
Similarly, if $X$ is a set of ground vertices then we refer to the points of the vertices of $X$ as the points of $X$ in $\delta.$

A cell $c \in C(\delta)$ is called a \defi{vortex} if $|\widetilde{c}| \geq 4$ and a \defi{flap} otherwise.
We call a cell $c$ \defi{trivial} if $\sigma_{\delta}(c)$ is the graph on two vertices and one edge.
We also partition $C(\delta)$ into the sets $C_{\mathsf{f}}(\delta)$ and $C_{\mathsf{v}}(\delta),$ containing the flaps and the vortices of $\delta$ respectively. 
We say that $\delta$ is \defi{vortex-free} if $C_{\mathsf{v}}(\delta) = \emptyset,$ i.e., if no cell in $C(\delta)$ is a vortex.

\paragraph{Surface decompositions of subgraphs.}

Let $\Sigma$ be a surface and $\delta = (\Gamma, \mathcal{D})$ be a \defi{$\Sigma$-decomposition} of a graph $G,$ where $\Gamma = (U,V,E).$
Also let $H$ be a subgraph of $G.$
We define $\delta \cap H$ to be the following $\Sigma$-decomposition $(\Gamma', \Dcal')$ of $H.$
\begin{itemize}
\item We obtain $\Gamma' = (U', V', E')$ from $\Gamma = (U, V, E)$ by defining $V'$ (resp. $E'$) from $V$ (resp. from $E$) by removing every node (resp. arc) that does not correspond to a vertex (resp. edge) in $H.$
Also we define $U' \coloneqq V' \cup \bigcup_{e \in E'} e.$
\item Let $N^-$ be the set of all nodes of $\delta$ that have been removed.
For every $x \in N^-$ let $D_{x}$ be an open disk containing $x$ that does not intersect $U'.$ We define $\Dcal'$ as follows. From every $D \in \Dcal,$ define $D' \in \Dcal'$ as $\Dcal' \coloneqq D \setminus \bigcup_{x\in N^{^-}} D_{x}.$
\end{itemize}

Notice that the above update to the disks of $\delta$ is necessary in order to obtain the third property of the definition of a $\Sigma$-decomposition.
In this way, we certify the existence of an injection between the disks of $\delta$ and those of $\delta \cap H,$ which only become smaller.

\paragraph{$\delta$-aligned disks.}
Let $\delta = (\Gamma, \mathcal{D}),$ $\Gamma = (U,V,E),$ be a $\Sigma$-decomposition of a graph $G.$
We call a disk $\Delta$ of $\Sigma$ $\delta$-\defi{aligned} if $U \cap \bd(\Delta) \subseteq N(\delta).$
Clearly, for each $c \in C(\delta),$ $\Delta_c$ is a $\delta$-aligned disk. Given a $\delta$-aligned disk $\Delta,$ we define $B^{\delta}_{\Delta} = \pi_{\delta}(\bd(\Delta)\cap N(\delta)).$
When $\delta$ is clear from the context, we may write $B_{\Delta}$ instead of $B^{\delta}_{\Delta}.$

Given an arc-wise connected set $\Delta \subseteq \Sigma$ such that $U \cap \Delta \subseteq N,$ we use $G \cap \Delta$ to denote the subgraph of $G$ consisting of the vertices and edges of $G$ that are drawn in $\Delta.$
Note that, in particular, the above definition applies in the case that $\Delta$ is a $\delta$-aligned disk of $\Sigma.$
We also define $\Gamma\cap \Delta \coloneqq (U \cap \Delta, V \cap \Delta, \{ e \in E \mid e \subseteq \Delta \})$ and observe that $\Gamma \cap \Delta$ is a drawing of $G \cap \Delta$ in $\Delta.$

Given that $\Delta$ is a $\delta$-aligned disk we define the rotation ordering  $\Omega_{\Delta}$ of $B_\Delta$ as follows.
If $\Sigma$ is an orientable surface, we define $\Omega_{\Delta}$ as the rotation ordering of the vertices of $B_{\Delta},$ when traversing along $\bd(\Delta)$ in the counter-clockwise direction.
In case $\Sigma$ is a non-orientable surface, there is no distinction between the counter-clockwise or clockwise ordering of the disks of $\Sigma.$
Therefore, we always assume that $\delta$ is accompanied with a \defi{rotation tie-breaker} that picks one of the two possible rotations of $\Delta.$
Then the rotation ordering $\Omega_{\Delta}$ is defined by following the rotation direction indicated by the rotation tie-breaker.
Given some $y\in B_{\Delta},$ we define $\Lambda _{\Delta,y}$ as the linear ordering obtained by $\Omega_{\Delta}$ by fixing $y$ as the starting vertex.
Given a $\delta$-aligned disk $\Delta,$ a \defi{$\Delta$-society} of $\delta$ is a society $\lin{G \cap \Delta, \Lambda _{\Delta,y}},$ for some $y \in B_{\Delta}.$
For every $c \in C(\delta),$ we denote $G_{c} \coloneqq G \cap \Delta_{c}$ and $\Omega_c = \Omega_{\Delta_{c}}.$ 
A \defi{$c$-society} in $\delta$ is any $\Delta_{c}$-society, where $c \in C(\delta).$
Moreover, in the case that $c$ is a vortex cell, we call any $c$-society of $\delta$ a \defi{vortex society of $c$} in $\delta.$

\paragraph{Renditions.}\llabel{renditions_def}
A \defi{rendition} of a society $\lin{G, \Lambda}$ in a disk $\Delta$ is a $\Delta$-decomposition $\rho$ of $G$ such that $\lin{G, \Lambda}$ is a $\Delta$-society of $\rho.$
A \defi{cylindrical rendition} $(\Gamma, \Dcal, c_{0})$ of a society $\lin{G, \Lambda}$ around a cell $c_{0}$ is a rendition $\rho = (\Gamma, \Dcal)$ of $G$ such that all cells of $\rho$
that are different than $c_{0}$ are flap cells.

Given a $\Sigma$-decomposition $\delta$ of a graph $G$ and a $\delta$-aligned disk $\Delta$ of $\Sigma,$ we denote by $\delta \cap \Delta$ the rendition $(\Gamma \cap \Delta, \{ \Delta_{c} \in \mathcal{D} \mid c \subseteq \Delta \})$ of the society $\lin{G \cap {\Delta}, \Lambda}$ in $\Delta.$
Moreover, we assume that the rotation tie-breaker of $\delta \cap \Delta$ is chosen so that the rotation of all disk in $\Delta$ agree with the orientation of $\Delta.$

\paragraph{Nests.} 
Let $\rho = (\Gamma, \mathcal{D})$ be a rendition of a society $\lin{G, \Lambda}$ in a disk $\Delta$ and let $\Delta^* \subseteq \Delta$ be an arc-wise connected set.
An \defi{$s$-nest in $\rho$ around $\Delta^*$} is a set $\mathcal{C} = \{C_1, \dots, C_s \}$ of pairwise disjoint $\rho$-grounded cycles where, for $i \in [s],$ the trace of $C_i$ bounds a closed disk $\Delta_i$ in such a way that $\Delta^* \subseteq \Delta_1 \subsetneq \Delta_2 \subsetneq \dots \subsetneq \Delta_s \subseteq \Delta,$ $\bd(\Delta_1) \cap \Delta^* = \emptyset,$ and $\bd(\Delta_s) \cap \Delta = \emptyset.$
We call $\lin{\Delta_{1},\ldots,\Delta_{s}}$ the \defi{disk sequence} of the $s$-nest $\Ccal$ and we say that, for $i \in[s]$ $\Delta_{i}$ \defi{corresponds} to the cycle $C_{i}.$

\paragraph{Rural societies.} We say that a society $\lin{G, \Lambda}$ is a \defi{vortex society} if $|\Lambda| ≥ 4$ and a \defi{simple society} otherwise.
The next proposition is a restatement of the celebrated ``two paths theorem'' that gives an alternative characterization of cross-free societies.
Several proofs of this result have appeared, see \cite{Jung70eine,RobertsonS90GMIX,shiloach1980polynomial,thomassen19802}. 
For a recent proof see \cite[Theorem 1.3]{KawarabayashiTW18anew}.

\begin{proposition}\llabel{thm_two_paths} A linear society $\lin{G, \Lambda}$ has no cross if and only if it has a vortex-free rendition in a disk.
\end{proposition}

A linear society $\lin{G,\Lambda }$ that has no cross is called \defi{rural}.  
Observe that simple societies are trivially rural.
Given a $\Sigma$-decomposition $\delta$ of a graph $G,$ we say that a cell $c$ of $\delta$ (resp. disk $\Delta$ of $\Sigma$) is \defi{flat} if some --actually any-- $c$-society (resp. $\Delta$-society) of $\delta$ is rural.
 
\paragraph{Grounded and bound graphs.}
Let $\delta$ be a $\Sigma$-decomposition of a graph $G.$ 
Let $Q \subseteq G$ be either a cycle or a path that uses no edge of $\cupall \{ \sigma(c) \mid c \in C_{\mathsf{v}}(\delta) \}.$
We say that $Q$ is $\delta$-\defi{grounded} if either $Q$ is a non-trivial path with both endpoints in $\pi_{\delta}(N(\delta))$ or $Q$ is a cycle that contains edges of $\sigma(c_{1})$ and $\sigma(c_{2})$ for at least two distinct cells $c_{1}, c_{2} \in C(\delta).$
A $2$-connected subgraph $H$ of $G$ is said to be $\delta$-\defi{grounded} if every cycle in $H$ is grounded in $\delta.$
We also say that a subgraph $G'$ of $G$ is $\delta$-\defi{bound}\llabel{bound_def} if there is a path in $G$ between a vertex of $G'$ and a ground vertex of $\delta.$

\paragraph{Disk-embeddable societies.}
A linear society $\lin{G, \Lambda}$ is \defi{disk-embeddable} if it has a rendition in some disk $\Delta$ where all cells are trivial.

\paragraph{Traces.}
Let $\delta$ be a $\Sigma$-decomposition of a graph $G$ in a surface $\Sigma.$ 
For every cell $c \in C(\delta)$ with $|\widetilde{c}| = 2$ we select one of the components of $\bd(c) - \widetilde{c}.$
This selection is called a \defi{trace tie-breaker} in $\delta,$ and we assume every $\Sigma$-decomposition to come equipped with a tie-breaker.
If $Q$ is $\delta$-grounded, we define the \defi{trace} of $Q$ as follows.
Let $P_{1}, \dots, P_{t}$ be distinct maximal subpaths of $Q$ such that $P_{i}$ is a subgraph of $\sigma(c)$ for some cell $c.$
Fix an index $i.$
The maximality of $P_{i}$ implies that its endpoints are $\pi_\delta(n_{1})$ and $\pi_\delta(n_{2})$ for distinct nodes $n_{1}, n_{2} \in N(\delta).$
If $|\widetilde{c}| = 2,$ define $L_{i}$ to be the component of $\bd(c) - \{ n_{1}, n_{2} \}$ selected by the tie-breaker, and if $|\widetilde{c}| = 3,$ define $L_{i}$ to be the component of $\bd(c) - \{ n_{1}, n_{2} \}$ that is disjoint from $\widetilde{c}.$
Finally, we define $L'_{i}$ by slightly pushing $L_{i}$ to make it disjoint from all cells in $C(\delta).$
We define such a curve $L'_{i}$ for all $i$ while ensuring that the curves intersect only at a common endpoint.
The \defi{trace} of $Q,$ denoted by $\trace_{\delta}(Q)$ or simply $\trace(Q)$ when $\delta$ is clear from the context, is defined to be $\bigcup_{i \in [t]} L'_{i}.$
So the trace of a cycle is the homeomorphic image of the unit circle and the trace of a path is an arc in $\Sigma$ with both endpoints in $N(\delta).$

We say that a cycle $Q$ of $G$ is \defi{$\delta$-contractible} if $\Sigma \setminus \trace(Q)$ contains at least one open disk.
Let $B \subseteq N(\delta).$
Given a $\delta$-contractible cycle $C$ of $G$ whose trace is disjoint from $B,$ we call the disk of $\Sigma,$ if it exists, that is bounded by $\trace(C)$ and is disjoint from $B,$ the $B$-\defi{avoiding disk} of $C.$
Notice that the $B$-avoiding disk of $C$ does not always exist, but if it exists, it is unique.

\paragraph{Linkages.}\llabel{linkages_def}
Let $G$ be a graph.
A \defi{linkage} in $G$ is a set $\Lcal = \{L_{1}, \ldots, L_{r}\}$ of pairwise vertex-disjoint paths.
In a slight abuse of notation, if $\mathcal{L}$ is a linkage, we use $V(\mathcal{L})$ and $E(\mathcal{L})$ to denote $\bigcup_{L\in\mathcal{L}} V(L)$ and $\bigcup_{L\in\mathcal{L}}E(L)$ respectively.
Given two sets $A$ and $B,$ we say that a linkage $\mathcal{L}$ is an $A$-$B$-\defi{linkage} if every path in $\mathcal{L}$ has one endpoint in $A$ and one endpoint in $B.$

\paragraph{Segments in a society.}
Let $\lin{G, \Lambda}$ be a society where $\Lambda = \lin{y_{1}, \ldots, y_{r}}.$
A \defi{segment} of $\Lambda$ is a set $S \subseteq V(\Lambda)$ such that there do not exist $s_1, s_2 \in S$ and $t_1,t_2 \in V(\Lambda ) \setminus S$ such that $s_1,t_1,s_2,t_2$ occur in the rotation ordering of $\Lambda$ in the order listed.
A vertex $s \in S$ is an \defi{endpoint} of the segment $S$ if there is a vertex $t \in V(\Lambda) \setminus S$ which immediately precedes or immediately succeeds $s$ in the rotation ordering of $\Lambda.$
Notice that for each $\lin{i,j} \in |\Lambda| \times |\Lambda|$ we may define a segment of the society $\lin{G,\Lambda}$ that starts with $y_{i}$ and finishes with $y_{j}.$
We call this segment the \defi{$(i,j)$-segment} of $\lin{G, \Lambda}$ and denote it by $y_i \Lambda y_j.$

\paragraph{Transactions in a society.}
Let $\lin{G,\Lambda }$ be a society. 
A \defi{transaction} in $\lin{G, \Lambda}$ is an $A$-$B$-linkage for disjoint segments $A, B$ of $\Lambda.$ 
We define the \defi{depth} of $\lin{G, \Lambda}$ as the maximum order of a transaction in $\lin{G, \Lambda}.$

Given a $\Sigma$ decomposition $\delta$ of a graph $G,$ if $\Delta$ is a $\delta$-aligned disk of $\Sigma$ (resp. a cell $c$ of $\delta$), then the \defi{depth} of $\Delta$ (resp. $c$) is the depth of some --actually any-- $\Delta$-society (resp. $c$-society).

Let $\mathcal{T}$ be a transaction in a society $\lin{G, \Lambda}.$ 
We say that $\mathcal{T}$ is \defi{planar} if no two members of $\mathcal{T}$ form a cross in $\lin{G,\Lambda }.$ 
An element $P\in\mathcal{T}$ is \defi{peripheral} if there exists a segment $X$ of $\Lambda $ containing both endpoints of $P$ and no endpoint of another path in $\mathcal{T}.$ 
A transaction is \defi{crooked} if it has no peripheral element.

Let $\rho$ be a rendition of a society $\langle G, \Lambda \rangle$ in a disk $\Delta,$ let $\Delta^* \subseteq \Delta$ be an arcwise connected set and let $\mathcal{C} = \{C_1, \dots, C_s\}$ be an $s$-nest in $\rho$ around $\Delta^*.$
We say that a family of pairwise vertex-disjoint paths $\Pcal$ in $G$ is a \defi{radial linkage} if each path in $\mathcal{P}$ has at least one endpoint in $V(\Lambda)$ and the other endpoint of $\mathcal{P}$ is drawn in $\Delta^*.$
Moreover, we say that $\mathcal{P}$ is \defi{orthogonal} to $\mathcal{C}$ if for every $P \in \mathcal{P}$ and every $i \in [s],$ $C_i \cap P$ consists of a single component.
Similarly, if $\mathcal{P}$ is a transaction in $\langle G, \Lambda \rangle$ then $\mathcal{P}$ is said to be \defi{orthogonal} to $\mathcal{C}$ if for every $i \in [s]$ and every $P \in \mathcal{P},$ $C_i \cap P$ consists of exactly two components.

\subsection{Fibers and folios of linear societies}

In this subsection, we define the notion of fibers as a means to formalize a form of ``rooted'' subdivision of a given society.
Moreover, we define the notion of $d$-folio which is a set that contains all different non-isomorphic fibers that represent some society.

\paragraph{Fibers.}
A triple $\mu = (M,T,\Lambda )$  is a \defi{fiber} if  

\begin{enumerate}
\item $\lin{M,\Lambda}$ is a linear society,
\item $T \subseteq V(M),$
\item for every vertex $x$ of $V(M) \setminus T$ it holds that 
\begin{itemize}
\item if $x \not \in V(\Lambda),$ then $x$ has degree two in $M$ and 
\item if $x\in V(\Lambda),$ then $x$ has degree one or two in $M.$
\end{itemize}
\end{enumerate}

The \defi{detail} of $\mu$ is defined as $\mathsf{detail}(\mu) \coloneqq |T|$ and the \defi{boundary size} of $\mu$ is $|B|.$

We define $\mathsf{dissolve}(\mu)$ as the society $\lin{H,\Lambda}$ where $H$ is the graph obtained from $M$ after contracting all edges with at least one endpoint not in $T \cup V(\Lambda )$ (while contracting, vertices in $T \cup V(\Lambda)$ prevail).
Notice that, in this case, the society $\lin{M, \Lambda}$ is a subdivision of the society $\mathsf{dissolve}(\mu).$

We say that two fibers $\mu_1, \mu_2$ are \defi{equivalent} if the societies $\mathsf{dissolve}(\mu_1)$ and $\mathsf{dissolve}(\mu_2)$ are isomorphic.

A fiber $\mu = (M, T, \Lambda)$ is a \defi{fiber of} $\lin{G, \Lambda'}$ if $M$ is a subgraph of $G$ where $V(M) \cap V(\Lambda') = V(\Lambda)$ and $\Lambda \subseteq \Lambda'$ (i.e., $\Lambda$ is a sub-sequence of $\Lambda'$).

\paragraph{Folio of a society.} 

Let $d$ be a positive integer.
We define the \defi{$d$-folio} of a linear society $\lin{G, \Lambda}$ as 
\begin{eqnarray*}
d\text{-}\mathsf{folio}(G,\Lambda )& = &\{ \mathsf{dissolve}(\mu) \mid \mu \text{~is a fiber of $\lin{G,\Lambda'},$ of detail at most $d,$ where $\Lambda'$ is a shift of $\Lambda$} \}.
\end{eqnarray*}

Let $\delta$ be a $\Sigma$-decomposition of a graph $G$ in a surface $\Sigma.$
We define the \defi{$d$-folio} of $\delta$ as $$d\folio(\delta) \coloneqq \cupall \big\{ d\text{-}\mathsf{folio}(G_c, \Lambda _{c,x}) \mid \text{$c \in C_{\mathsf{f}}(\delta)$ and $x \in B_{\Delta_{c}}$} \big\}.$$

We note that the size of the $d$-folio of $\delta$ is upper bounded by an exponential function that depends on the detail $d,$ which is essentially the number of different graphs we can have on $d$ vertices.

\begin{observation}\llabel{obs_size_dfolio} There exists a function $f_{\ref{obs_size_dfolio}}\colon\Nbbb \to \Nbbb$ such that, for every $d \in \Nbbb_{\geq 1},$ for every $\Sigma$-decomposition $\delta$ of a graph in a surface $\Sigma,$ $$|d\text{-}\mathsf{folio}(\delta)| \leq f_{\ref{obs_size_dfolio}}(d).$$
Moreover,
$$f_{\ref{obs_size_dfolio}}(d) = 2^{\Ocal(d^{2})}.$$
\end{observation}

\subsection{$t$-segments}\llabel{subsec_segments}

In this subsection, we introduce various types of wall-like structures that we uniformly call $t$-segments.
The different types of $t$-segments will collectively define the notion of a $t$-walloid which will serve as the underlying infrastructure, in its various evolving forms, throughout the refinement steps that our $\Sigma$-decomposition will go through in the proofs spanning the next sections.

\paragraph{Grids and (cylindrical) walls.}
An \defi{$(n \times m)$-grid} is the graph with vertex set $[n]\times[m]$ and edge set \[\{(i,j)(i,j+1) \mid i\in[n],j\in[m-1]\}\cup\{(i,j)(i+1,j) \mid i\in[n-1],j\in[m]\}.\]
We call the path where vertices appear as $(i,1),(i,2),\dots, (i,m)$ the \defi{$i$th row} and the path where vertices appear as $(1,j),(2,j),\dots, (n,j)$ the \defi{$j$th column} of the grid.

Let  $z, t \in \Nbbb_{≥3}.$
An \defi{elementary} $(z,t)$-\defi{wall} is obtained from the $(2z \times t)$-grid
by removing a maximal induced matching $M$ that does not contain any edge of its horizontal paths.
A $(z,t)$-\defi{wall} is any graph obtained from an elementary $(z,t)$-wall after arbitrarily subdividing edges.
Notice that a $(z,t)$-wall $W$ can be seen as the union of $t$ horizontal paths and $z$ vertical paths that are defined as follows.

The $t$ \defi{horizontal paths} $P_{1}, \dots, P_{t},$ ordered (and visualized) from top to bottom
are the subdivisions of the $t$ rows of the original $(2z \times t)$-grid.
We call $i,$ the \defi{index} of the path $P_{i}.$
We refer to $P_{1}$ (resp. $P_{t}$) as the \defi{top} (resp. \defi{bottom}) path of $W.$
The \defi{top} (resp. \defi{bottom}) \defi{boundary vertices} of $W$ are vertices of the original $(2z \times t)$-grid that belong in the top (resp. bottom) path of $W$ and are incident to edges of the maximal induced matching $M.$
Notice that there are $z$ top and $z$ bottom boundary vertices.
The $z$ \defi{vertical paths} $Q_{1}, \dots, Q_{z},$ ordered (and visualized) from left to right are the pairwise disjoint paths whose union meets all vertices of the original $(2z \times t)$-grid, and where each $Q_{i}$ has one endpoint in the top and the other in the bottom boundary vertices of $W.$

For $i \in [z]$ the $i$-th \defi{top} (resp. \defi{bottom}) \defi{boundary vertex} of $W$ is the single vertex of $P_{1} \cap Q_{i}$ (resp. $P_{t} \cap Q_{i}$).
For $i \in [t]$ the $i$-th \defi{leftmost} (resp. \defi{rightmost}) \defi{boundary vertex} of $W$ is the single vertex of $P_{i} \cap Q_{1}$ (resp. $P_{i} \cap Q_{z}$).
Notice that all top, bottom, leftmost, and rightmost boundary vertices are vertices of degree at most two in $W.$
Also, notice that $P_{1} \cup Q_{1} \cup P_{t} \cup Q_{z}$ is a cycle, which we call the \defi{perimeter} of $W.$

Let $s \in \Nbbb_{\geq 3}.$ 
 An $(t, s)$-cylindrical wall is obtained from a $(t, s)$-wall after adding $s$ pairwise disjoint non-trivial paths $P_{1}^*, \ldots, P_{s}^*$ whose internal vertices are not in the $(t, s)$-wall we start from and where for each $i \in [t],$ $P_{i}^*$ connects the $i$-th leftmost boundary vertex with the $i$-th rightmost boundary vertex of $(t, s)$-wall we start from.
Notice that a $(t, s)$-cylindrical wall, has neither leftmost nor rightmost boundary vertices, however, it has top (resp. bottom) boundary vertices that are the top (resp. bottom) boundary vertices of the original $(t,s)$-wall plus the path $P_{1}^{*}$ (resp. $P_{s}^{*}$).
Also, the $s$ horizontal paths of the $(t,s)$-wall together with the newly added paths, define $s$ vertex disjoint cycles $C_{1}, \dots, C_{s},$ ordered from top to bottom.
We call the cycle $C_{1}$ (resp. $C_{s}$) the \defi{inner} (resp. \defi{outer}) \defi{cycle} of the $(t,s)$-cylindrical wall and they are visualized as such.

\paragraph{Segments.}

We refer to a graph as a $t$-\defi{segment} if it belongs to one of the following categories.

\medskip
\noindent

\noindent 
\textbf{$(z,t)$-wall segment.} Given $z$ and $t \in \Nbbb_{≥3},$ an (elementary) $(z,t)$-\defi{wall segment} $W$ is an (elementary) $(z, t)$-wall $\widetilde{W}.$

\medskip
\noindent 
\textbf{$t$-handle segment.} Given a $t \in \Nbbb_{≥3},$ we define an (elementary)  $t$-\defi{handle segment} $W$ by considering an (elementary)  $(4t, t)$-wall $\widetilde{W}$ whose top boundary vertices are $v_{1}, \ldots, v_{t}, v'_{1},\ldots, v_{t}', u_{1}, \ldots, u_{t}, u'_{1}, \ldots, u_{t}'$ in left to right order, adding the edges in $\{v_i u_{t-i+1}, v'_i u'_{t-i+1} \mid i \in [t] \},$ and, (in the non-elementary version) subdividing them an arbitrary number of times.
Notice that the paths created by possibly subdividing these edges form a $\{v_{1}, \ldots, v_{t}\}$-$\{ u_{1},\ldots, u_{t}\}$-linkage and a $\{v_{1}', \ldots, v_{t}'\}$-$\{u_{1}', \ldots, u_{t}\}'$-linkage.
We call the former linkage the \defi{leftmost rainbow}, while the latter the \defi{rightmost rainbow} of the handle segment.
We define the \defi{leftmost} (resp. \defi{rightmost}) \defi{enclosure} of $W$ as the cycle of $W$ whose degree-3 vertices in $W$ appear in left to right order as
$$\text{$v_{1},\ldots,v_{t},u_{1},\ldots,u_{t}$ (resp. $v_{1}',\ldots,v_{t}',u_{1}',\ldots,u_{t}'$}).$$

\medskip
\noindent
\textbf{$t$-crosscap segment.} Given a $t \in \Nbbb_{≥3},$ we define an (elementary) $t$-\defi{crosscap segment} $W$ by considering an (elementary) $(4t, t)$-wall $\widetilde{W}$ whose top boundary vertices are $v_{1}, \ldots, v_{2t}, u_{1}, \ldots, u_{2t}$ in left to right order, adding the edges in $\{v_i u_{2t+i} \mid i \in [2t] \},$ and (in the non-elementary version) subdividing these edges an arbitrary number of times.
Notice that the paths created by possibly subdividing these edges form a $\{v_{1}, \ldots, v_{2t}\}$-$\{u_{1}, \ldots, u_{2t}\}$-linkage and we refer to it as the \defi{rainbow} of the handle segment.
We define the \defi{enclosure} of $W$ as the cycle of $W$ whose degree-3 vertices in $W$
appear in left-to-right order as
$$v_{1}, \ldots, v_{2t}, u_{2t}, \ldots, u_{1}.$$

\medskip
\noindent
\textbf{$(t,b)$-flower segment.} Given a society $\lin{H,\Lambda }$ where $\mathsf{dec}({H, \Lambda}) \coloneqq \lin{\lin{H_{1},\Lambda_{1}}, \ldots, \lin{H_{q},\Lambda_{q}}}$ and a $b \in \Nbbb_{\geq 1},$ an (elementary) $(t,b,H,\Lambda)$-\defi{flower segment} is defined as follows.
Consider an (elementary) $(2t + b \cdot |\Lambda|, t)$-wall $\widetilde{W}$ whose top boundary vertices are, in left to right order,
$$t_{1},\ldots,v_{t}, T_{1}^{1},\ldots,T_{1}^{b}, ~\ldots~, T_{q}^{1},\ldots,T_{q}^{b},t_{1},\ldots,t_{t}$$ 
where for $(j,i) \in [q] \times [b],$ $T_{i}^{j} = \lin{t_{1}^{i,j},\ldots, t_{m_{i,j}}^{i,j}}$ where $m_{i,j}=|T_{i}^{j}|.$
Then, for each $j \in [q],$ consider $b$ societies $\lin{M_j^{i}, \Lambda_j^{i}}, i \in [b],$ such that each $\lin{M_j^{i}, \Lambda_j^{i}}$ is either equal to $\lin{H_j, \Lambda_j}$ (in the case of an elementary flower) or is a subdivision of $\lin{H_j, \Lambda_j}.$
Then, for $j \in [q],$ for $i \in [b],$ and for every $p \in [m_{i,j}]$ add an edge between $t_{p}^{i,j}$ and the $p$-th vertex of $\Lambda_{j}^{i}.$

Notice that the paths created by (possibly) subdividing these edges form a $\{v_{1}, \ldots, v_{t}\}$-$\{u_{1}, \ldots, u_{t}\}$-linkage and we refer to it as the \defi{rainbow} of the $(t, b, H, \Lambda)$-flower segment.
The \defi{internal vertices of} the $(t, b, H, \Lambda)$-flower segment are those that are not vertices of $\widetilde{W}.$
We refer to $\lin{H, \Lambda}$ as the \defi{society} of the $(t,b,H,\Lambda)$-flower segment we just defined.
Also, its \defi{inner cycle} consists of the subdivided edge $v_t u_1$ and the $(v_t, u_1)$-path that is a subpath of the top path of $\widehat{W}.$

\medskip
\noindent
\textbf{$(t,b,h)$-parterre segment.} Given a simple society $\lin{H,\Lambda }$ and $b,h\in \Nbbb_{\geq 1},$ a $(t,b,h,H,\Lambda )$-\defi{parterre segment} is defined as follows. Consider a sequence $F_{1}, \ldots, F_{h}$ of disjoint $(t, b, H, \Lambda )$-flower segments. 
Then, for $i \in [h]$ and $j \in [t],$ add an edge between the $j$-th rightmost boundary vertex of $S_{i}$ and the $j$-th leftmost boundary vertex of $S_{i}$ and subdivide this edge an arbitrary number of times. 
We refer to $F_{1}, \ldots, F_{h}$ as the \defi{flowers} and to $\lin{H,\Lambda}$ as the \defi{society} of the defined $(t,b,h,H,\Lambda)$-parterre segment.
Also, we refer to the rainbows and inner cycles of its flowers as the \defi{rainbows}, and \defi{inner cycles} of the parterre segment.
Notice that a $(t,b,h,H,\Lambda)$-parterre segment has $h$ many flowers, rainbows, and inner cycles.

\medskip
\noindent
\textbf{$(t,s)$-Thicket segment.} Given  $t, s \in \Nbbb_{\geq 3}$ we define a \defi{$(t,s)$-thicket segment} by considering the disjoint union of an $(t, t)$-wall $\widetilde{W},$ whose top boundary vertices are $v_{1},\ldots,v_{t}$ in left to right order, and a $(t, s+2)$-cylindrical wall $\widehat{W}$ whose bottom boundary vertices are $u_{1}, \ldots, u_{t}$ in cyclic order.
Then, we add the edges in $\{ v_i u_i \mid i \in [t] \}$ and we subdivide them an arbitrary number of times.
We denote by $P_{1}, \ldots, P_{t}$ the paths obtained by the previous subdivisions.

We refer to the inner (resp. outer) cycle of the cylindrical wall $\widehat{W}$ as the \defi{inner} (resp. \defi{outer}) \defi{cycle} of the $(t,s)$-thicket segment.
Let $\mathcal{R} = \{R_{1}, \ldots, R_{t}\}$ be the linkage of $\widehat{W}$ that consists of its vertical paths in left to right order and let $\Ccal = \{C_{1}, \ldots, C_{s} \}$ be the cycles of $\widehat{W}$ minus its outer and inner cycle ordered so that $C_{1}$ is the cycle succeeding the inner cycle of $\widehat{W}$ while $C_{s}$ is the cycle preceding the outer cycle of $\widehat{W}.$
Notice that each path $R_{i}$ in $\Rcal$ is a path connecting the $i$-th top boundary vertex of $\widehat{W}$ with the $t$-th bottom boundary vertex of $\widehat{W},$ say $u_{i}.$
We now enhance the paths $\mathcal{R} = \{R_{1}, \ldots, R_{t}\}$ by updating each $R_{i} \coloneqq R_{i} \cup P_{i}.$
We call the pair $(\Ccal, \mathcal{R})$ the \defi{railed nest} of the $(t,s)$-thicket segment.

\medskip
Notice that in the definition of any of the different types of a $t$-segment $W,$ we start by considering a wall $\widetilde{W}$ with $t$ horizontal paths.
We call this wall the \defi{base wall} of $W.$
The leftmost boundary vertices as well as the rightmost boundary vertices of $\widetilde{W}$ form the \defi{left boundary} and the \defi{right boundary} of the $t$-segment $W.$

\paragraph{Host and guest walloids.}
Given a sequence $\Wcal = \lin{W_{1}, \ldots, W_{l}}$ of $t$-segments, we define the \defi{cylindrical concatenation} of $\Wcal$ as the graph $W$ obtained from $\cupall \Wcal,$ by adding, for $i \in [l]$ and $j \in [t],$ an edge between the $j$-th rightmost boundary vertex of $W_{i}$ and the $j$-th leftmost boundary vertex of $W_{i+1},$ where $W_{l+1} = W_{1},$ and then subdividing these edges an arbitrary number of times.
We define two types of cylindrical concatenations of $t$-segments and we refer to them as \defi{$t$-walloids}.
Note that $W$ contains as a subgraph the cylindrical concatenation of the base walls of all segments from $\mathcal{W}.$
We refer to this cylindrical wall as the \defi{base cylinder} of $W.$

Let $t, s \in \Nbbb_{\geq 3},$ $b, h \in \Nbbb_{\geq 1},$ $\Sigma = \Sigma^{(\mathsf{h}, \mathsf{c})}$ be a surface, and $\ell, \ell' \in \Nbbb.$
Let $\Wcal = \rot{W_{1}, \ldots, W_{\mathsf{h}+\mathsf{c}+\ell+\ell'+1}}$ be a sequence of $t$-segments where 
\begin{itemize}
\item $W_{1}$ is a $(t, t)$-wall segment,
\item $W_{2}, \ldots, W_{\mathsf{h} + 1}$ are $t$-handle segments, 
\item $W_{\mathsf{h} + 2}, \ldots, W_{\mathsf{h}+\mathsf{c}+1}$ are $t$-crosscap segments, 
\item $W_{\mathsf{h} + \mathsf{c} + 2}, \ldots, W_{\mathsf{h} + \mathsf{c} + \ell + 1}$ are $(t, b, h)$-parterre segments, and 
\item $W_{\mathsf{h} + \mathsf{c} + \ell + + 2}, \ldots, W_{\mathsf{h} + \mathsf{c} + \ell + \ell' + 1}$ are $(t, s)$-thicket segments.
\end{itemize}

We call the wall-like graph obtained by the cylindrical concatenation $W$ of $\Wcal$ a $(t, b, h, s, \Sigma)$-\defi{host walloid}.
In the above construction, if instead of $(t, b, h)$-parterre segments and/or $(t, s)$-thicket segments we only consider (elementary) $(t,b)$-flower segments, then we define an (elementary) $(t, b, \Sigma)$-\defi{guest walloid}.
Notice that for host walloids we insist that the societies of their parterre segments are\textsl{simple} societies.
This is not the case for an (elementary) guest walloid as the societies of their flowers are assumed to be simple.
Moreover, we refer to the $t$-segments of $\Wcal$ as the \defi{$t$-segments} of $W.$
We denote by $\overline{W}$ the graph obtained from $W$ if we remove all internal vertices of its flowers.
Clearly, if $W$ does not have any parterre or flower segments, then $W = \overline{W}.$

Notice that the essential difference between the elementary and non-elementary version of a $(t, b, \Sigma)$-guest walloid is that the elementary version is uniquely defined while the non-elementary one is a subdivision of the former.

Let $W$ be a $(t, b, h, s, \Sigma)$-host walloid (resp. $(t,b,\Sigma)$-guest walloid) as above that has $\ell$ $(t,b,h)$-parterre (resp. $(t,b)$-flower) segments and $\ell'$ (resp. $\ell'=0$) $(t,s)$-thicket segments. 
For $i \in [\ell]$ (resp. $i \in [\ell']$) we refer to $W_{\mathsf{h} + \mathsf{c} + 1 + i}$ (resp. $W_{\mathsf{h} + \mathsf{c} + \ell + + 1 + i}$) as the $i$-th $(t,b,h)$-parterre (resp. $(t,b)$-flower) segment (resp. $i$-th $(t,s)$-thicket segment) of $W.$

\paragraph{Classification of facial cycles.}
Let $W$ be a $(t, b, h, s, \Sigma)$-host walloid or an (elementary) $(t,b,\Sigma)$-guest walloid.
Notice that $\overline{W}$ is (the subdivision of) a 3-regular $\Sigma^{(\mathsf{h}, \mathsf{c})}$-embeddable graph.
Consider a $\Sigma^{(\mathsf{h}, \mathsf{c})}$-embedding of $\overline{W}.$
We call \defi{bricks} the facial cycles that have exactly six degree-3 vertices of $\overline{W}.$

Among the facial cycles of the $\Sigma^{(\mathsf{h},\mathsf{c})}$-embedding of $\widetilde{W}$ that are not bricks, $2 \mathsf{h}$ many (resp. $\mathsf{c}$ many) of them correspond to the enclosures of the $t$-handle (resp. $t$-crosscap) segments, $h\ell$ many of them correspond to the inner cycles of its $(t,b)$-flower segments, which we call the \defi{flower cycles} of $W,$ and $\ell'$ of them correspond to the inner cycles of the $(t,s)$-cylindrical walls of the $\ell'$ many $(t, s)$-thicket segments, which we call the \defi{thicket cycles} of $W.$

In the $\Sigma^{(\mathsf{h}, \mathsf{c})}$-embedding of $\overline{W},$ two facial cycles remain unclassified.
One of the two corresponds to the outer cycle of the base cylinder of $W$ which we call the \defi{simple cycle} of $W,$ denoted by $C^{\mathsf{si}}.$
We call the only remaining unclassified facial cycle, the \defi{exceptional cycle} of $W,$ denoted by $C^{\mathsf{ex}},$ which is the only facial cycle that contains the endpoints of every top path of all base walls of the segments of $\Wcal.$

Let $\widetilde{W}$ denote the base cylinder of $W.$
A cycle of $\overline{W}$ is an \defi{enclosure} if it is an enclosure of some of its $t$-handle or its $t$-crosscap segments or if it is the inner cycle of $\widetilde{W},$ which we call the \defi{big enclosure} of $W.$ 
Notice that $W$ has $2\mathsf{h} + \mathsf{c} + 1$ enclosures.
A cycle $C$ of $\overline{W}$ is a \defi{fence} if it is either a facial cycle of $\overline{W}$ or if there exists an edge separation $(E_1, E_2)$ of $\overline{W},$ such that
\begin{itemize}
\item $V(C)$ is the set of vertices that are endpoints of edges in both $E_{1}$ and $E_{2},$  
\item $E_2 \setminus E_1 \neq \emptyset,$ $E_1 \setminus E_2 \neq \emptyset,$ and  
\item $E(C^{\mathsf{ex}}) \subseteq E_2.$
\end{itemize}

Notice that for every fence $C$ such an edge separation $(E_{1}, E_{2})$ is uniquely defined.
We say that a fence $C'$ of $\overline{W}$ is \defi{inside} the fence $C$ if $E(C') \subseteq E_{1}.$
A brick of $W$ is a \defi{brick} of a fence $C$ if it is inside $C.$
It is easy to see that, given a $\Sigma^{(\mathsf{h}, \mathsf{c})}$-embedding of $\overline{W},$ the fences of $\overline{W}$ are exactly the non-contractible cycles of $\overline{W}.$

\subsection{Meadows, gardens, and orchards.}

In this subsection, we build on the definitions of the previous subsection and present various notions that aim to formalize all the different intermediate states and all properties that our $\Sigma$-decomposition will attain throughout the refinement steps in the proofs of the sections that follow.

\medskip
Let $t, s \in \Nbbb_{≥3},$ $b, h \in \Nbbb_{\geq 1},$ $\Sigma$ be a surface, and $\mathsf{h} \in \Nbbb$ and $\mathsf{c} \in [0, 2]$ such that $\Sigma = \Sigma^{(\mathsf{h}, \mathsf{c})}.$
Additionally, let $\delta$ be a $\Sigma$-decomposition of a graph $G$ and $W$ be a $\delta$-grounded $(t, b, h, s, \Sigma)$-host walloid such that
\begin{itemize}
\item for every $\delta$-aligned disk $\Delta$ such that $|\bd(\Delta) \cap N(\delta)| \leq 3,$ $\Delta$ or its complement contains a single cell of $\delta,$
\item every ground vertex of $\delta$ is in the same connected component of $G$ as the one containing $W,$ and
\item all other connected components of $G$ are drawn in the interior 
of a cell of $\delta$ without boundary.
\end{itemize}

We say that the triple $(G, \delta, W)$ is
\begin{itemize}
\item a $(t, \Sigma)$-\defi{meadow} if $W$ has no $(t, b, h)$-parterre or $(t, s)$-thicket segments and $\delta$ has a single vortex $c$ that is the unique vortex inside the $\trace(C^{\mathsf{si}})$-avoiding disk of $\trace(C^{\mathsf{ex}}).$

\item a $(t, b, h, \Sigma)$-\defi{garden} if it has no $(t, s)$-thicket segments and $\delta$ has a single vortex $c$ that is the unique vortex that is inside the $\trace(C^{\mathsf{si}})$-avoiding disk of $\trace(C^{\mathsf{ex}}).$ 
Moreover the $(t, b, h)$-parterre segments of a garden have pairwise different societies.

\item a $(t, b, h, s, \Sigma)$-\defi{orchard} if $W$ has as many $(t, s)$-thickets as the vortices of $\delta$ and each vortex of $\delta$ is inside the $\trace(C^{\mathsf{si}})$-avoiding disk of the trace of the inner cycle of a $(t, s)$-thicket.
As before $(t, b, h)$-parterre segments of a garden have pairwise different societies.

A \defi{thicket society} of $(G,\delta,W)$ is a society $\langle H,\Lambda\rangle$ where $\Lambda$ is a linearization of a cyclic ordering of those vertices of $\Lambda$ corresponding to the nodes contained in the trace of the outer cycle $C_{s+1}$ of a $(t,s)$-thicket segment of $W$ and $H$ is the subgraph of $G$ drawn in the $\mathsf{trace}(C^{\mathsf{si}})$-avoiding disk of $\mathsf{trace}(C_{s + 1}).$

If $\delta$ has a single vortex then we say that $(G, \delta, W)$ is a \defi{single}-\defi{thicket} $(t, b, h, s, \Sigma)$-orchard.
\end{itemize}

A $\Sigma$-decomposition $\delta$ of $G$ with vortices $c_{1}, \ldots, c_{\ell}$ is $(x, y)$-\defi{fertile} if there are positive integers $x_{1}, \ldots, x_{\ell}$ where $\sum_{i \in [\ell]} x_i = x$ and such that any vortex society of $\delta,$ 
has a rendition with $x_{i}$ vortices, each of depth at most $y.$
We also say that $\delta$ is $(z, p)$-\defi{ripe} if it has at most $z$ vortices, each of depth at most $p.$ 
Notice that if $\delta$ is $(z,p)$-ripe, then it is also $(z,p)$-fertile.

The $\delta$-\defi{influence} of a fence $C,$ denoted by $\delta\text{-}\mathsf{influence}(C),$ is the union of all simple cells of $\delta$ that contain at least one edge of $C$ and the cells of $\delta$ whose disk is a subset of the $\trace(C^{\mathsf{si}}) \cap N(\delta)$-avoiding disk of the trace of $C.$
Note that a cell can belong to the influence of at most three of the fences of $W.$
Moreover, disjoint fences have disjoint influences.

Let $d \in \Nbbb_{\geq 1}.$
The $d$-folio of a fence $C$ of $W$ is the union of the $d$-folios of all $c$-societies, where $c$ is a cell in the $\delta$-influence of $C,$ i.e.,
 $$d\text{-}\mathsf{folio}(C) \coloneqq \{ d\text{-}\mathsf{folio}(G_{c}, \Lambda_{c, x}) \mid \text{$c \in \delta\text{-}\mathsf{influence}(C)$ and $x \in B_{\Delta_{c}}$} \}.$$
A fence of $W$ is $d$-\defi{homogeneous} if all its bricks have the same (in terms of society isomorphism) $d$-folio.
We moreover say that the triple $(G, \delta, W)$ is
\begin{itemize}
\item $d$-\defi{plowed} if all its $2\mathsf{h} + \mathsf{c} + 1$ enclosures are $d$-homogeneous,
\item $(x,y)$-\defi{fertile} if $\delta$ is $(x,y)$-fertile, 
\item $(z, p)$-\defi{ripe} if $\delta$ is $(z, p)$-ripe and every thicket society of $(G, \delta, W)$ has depth at most $2s + p,$ and
\item $d$-\defi{blooming} if every linear society $\lin{H, \Lambda}$ in $d\text{-}\mathsf{folio}(\delta)$ is the linear society of a $(t, b)$-parterre segment of $W.$ 
\end{itemize}

\section{Preliminary results}

In this section, we present two independent results that will be of use for proofs in the following sections.
In the first subsection, we prove a homogeneity lemma for walls and a mono-dimensional variant for ladders.
In the second subsection, we prove a result that allows us to control to some extent the behavior of a ``minimal'' linkage traversing multiple nests.

\subsection{Finding a homogeneous subwall}

In this subsection we show how given a large enough wall whose bricks come equipped with a subset from a fixed set of colors, we can always find an arbitrarily large subwall of it where all bricks are accompanied by the same color set.
We moreover observe how the same arguments give us a mono-dimensional variant applied to an object we call a ladder.

\medskip
Let $\zeta \in \Nbbb_{\geq 1},$ $z, t \in \Nbbb_{\geq 3},$ and $W$ be a $(z, t)$-wall.
Assume that every brick $B$ of $W$ is assigned a subset of $[\zeta]$ which we call the $\zeta$-\defi{palette} of $B,$ denoted by $\zeta\text{-}\mathsf{palette}(B).$ We prove the following lemma.

\begin{lemma}\llabel{lem_homogeneity_2d} There exists a function $f_{\ref{lem_homogeneity_2d}} : \Nbbb^{2} \to \Nbbb$ and an algorithm that, given $\zeta \in \Nbbb_{\geq 1},$  $z, t \in \mathbb{N}_{\geq 3},$ and an $(f_{\ref{lem_homogeneity_2d}}(z, \zeta), f_{\ref{lem_homogeneity_2d}}(t, \zeta))$-wall $W,$ outputs a $(z, t)$-wall $W'$ that is a subwall of $W$ and such that every brick of $W'$ has the same $\zeta$-palette in time
$${\Ocal(z^{2^{\zeta}} \cdot t^{2^{\zeta}})}.$$
Moreover
$$f_{\ref{lem_homogeneity_2d}}(x, \zeta) = x^{2^{\zeta}}.$$
\end{lemma}
\begin{proof} Let $\alpha : \Nbbb \to \Nbbb$ be such that $\alpha(\zeta) \coloneqq 2^{\zeta}.$
We define $f_{\ref{lem_homogeneity_2d}}(z, \zeta) \coloneqq z^{\alpha(\zeta)}.$

Let $\zeta \in \Nbbb_{\geq 1},$ $z, t \in \mathbb{N}_{\geq 3},$ and $W$ be a $(z^{\alpha(\zeta)}, t^{\alpha(\zeta)})$-wall.
Let $\Pcal_{1}, \ldots, \Pcal_{t^{\alpha(\zeta-1)}}$ be a grouping of the horizontal paths of $W$ in consecutive bundles of $t^{\alpha(\zeta-1)}$ paths ordered from left to right, where the last path in $\Pcal_{i}$ is the first path of $\Pcal_{i+1}.$
Also let $\Qcal_{1}, \ldots, \Qcal_{z^{\alpha(\zeta-1)}}$ be a grouping of the vertical paths of $W$ in consecutive bundles of $z^{\alpha(\zeta-1)}$ paths ordered from top to bottom, where the last path in $\Qcal_{i}$ is the first path of $\Qcal_{i+1}.$

For every $i \in [t^{\alpha(\zeta-1)}]$ and $j \in [z^{\alpha(\zeta-1)}],$ let $W_{i, j}$ be the graph $\cupall \Pcal'_{i} \cup \Qcal'_{j}$ that is a $(z^{\alpha(\zeta-1)}, t^{\alpha(\zeta-1)})$-wall that is a subwall of $W,$ where $\Pcal'_{i}$ and $\Qcal'_{j}$ are defined as follows.
Let $P_{i}$ (resp. $P'_{i}$) be the topmost (resp. bottommost) path of $W$ in $\Pcal_{i}$ and $Q_{j}$ (resp. $Q'_{j}$) be the leftmost (resp. rightmost) path of $W$ in $\Qcal_{j}.$
Let $\Pcal'_{i}$ (resp. $\Qcal'_{j}$) contain the paths in $\Pcal_{i}$ (resp. $\Qcal_{i}$) cropped to contain only vertices of vertical (resp. horizontal) paths between $Q_{i}$ (resp. $P_{i}$) and $Q'_{i}$ (resp. $P'_{i}$).
There are $z^{\alpha(\zeta-1)} \cdot t^{\alpha(\zeta-1)}$ such subwalls depending on the choices of $i$ and $j.$

We also consider $P$ (resp. $P'$) to be the top (resp. bottom) horizontal path in $W,$ and $Q$ (resp. $Q'$) to be the leftmost (resp. rightmost) vertical path in $W.$
Observe that the graph $P \cup (\bigcap_{i \in [t^{\alpha(\zeta-1)}] } \Pcal_{i}) \cup P' \cup Q \cup (\bigcap_{j \in [z^{\alpha(\zeta-1)}]} \Qcal_{j})\cup Q'$ also forms a $(z^{\alpha(\zeta-1)}, t^{\alpha(\zeta-1)})$-wall $W'$ that is a subwall of $W.$

Now, we assume inductively that for every $(z^{\alpha(\zeta-1)}, t^{\alpha(\zeta-1)})$-wall $W''$ that is a subwall of $W,$ there exists a $(z, t)$-wall that is a subwall of $W''$ such that every brick of that subwall has the same $(\zeta-1)$-palette.
Next, we show that there is $W'' \in \{ W_{i, j} \mid i \in [t^{\alpha(\zeta-1)}], j \in [z^{\alpha(\zeta-1)}]\} \cup \{ W' \}$ such that, either the $\zeta$-palette of every brick of $W''$ contains $\zeta$ or no brick contains $\zeta.$
Notice that from the induction hypothesis this suffices to prove our claim.
To conclude, observe that by definition, either the $\zeta$-palette of every brick of $W'$ contains $\zeta$ or there is a $W'' \in \{ W_{i, j} \mid i \in [t^{\alpha(\zeta-1)}], j \in [z^{\alpha(\zeta-1)}]\}$ such that no brick of $W''$ contains $\zeta.$
\end{proof}

Let $L$ be a subdivision of the $(2 \times t)$-grid, where $t \in \Nbbb_{\geq} 2.$
We call $L$ a $t$-\defi{ladder}.
Notice that $L$ can be seen as the union of $t$ horizontal paths $P_{1}, \ldots, P_{t}$ and two vertical paths $Q_{1}, Q_{2}$ which correspond to the subdivided rows and columns respectively of the original $(2 \times t)$-grid.
We say that a ladder $L'$ is a \defi{subladder} of $L$ if $L'$ is obtained from $L$ by removing the edges of some (maybe none) of its horizontal paths.
Consider a plane-embedding of $L$ obtained from the natural plane-embedding of the $(2 \times t)$-grid.
We call all facial cycles of $L,$ except the cycle that bounds the outer face, the \defi{bricks} of $L.$
For $\zeta \in \Nbbb_{\geq 1},$ we define the $\zeta$-palette of a brick of $L$ as before.

\medskip
We get the following mono-dimensional corollary of \autoref{lem_homogeneity_2d}, by applying the same inductive argument only in the dimension of the horizontal paths. 

\begin{corollary}\llabel{cor_homogeneity_1d} There is an algorithm that, given $\zeta \in \Nbbb_{\geq 1},$ $t \in \mathbb{N}_{\geq 2},$ and an $f_{\ref{lem_homogeneity_2d}}(t, \zeta)$-ladder $L,$ outputs a $t$-ladder $L'$ that is a subladder of $L$ and such that every brick of $L'$ has the same $\zeta$-palette in time $\Ocal(t^{2^{\zeta}}).$
\end{corollary}

\subsection{Taming a linkage crossing large nests}
\llabel{tamming_more_sec}

In this subsection we prove that for every linkage in a graph $G$ that ``invades'' a number of large enough nests in a $\Sigma$-decomposition of $G$ we can appropriately define a ``minimal'' linkage equivalent to the first with a number of additional information that will be useful to bound the complexity of the extensions of models we have to consider.
We begin with a few necessary definitions related to the unique linkage function $\ell \colon \nn{1}{1}$ which we already discuss in the introduction.

\medskip
Let $\Lcal$ be a linkage in a graph $G.$
The endpoints of the paths of $\Lcal$ are called \defi{terminals}.
The \defi{pattern} of $\Lcal$ is the set of pairs of endpoints of the paths of $\Lcal.$
$\Lcal$ is \defi{equivalent} to a linkage $\Lcal'$ in $G$ if they have the same pattern.

An \defi{LB-pair} of $G$ is a pair $(\Lcal, B)$ where $B$ is a subgraph of $G$ with maximum degree two and $\Lcal$ is a linkage in $G.$
We define $\mathsf{cae}(\Lcal, B) \coloneqq |E(\Lcal) \setminus E(B)|$ to be the number of edges of the paths of $\Lcal$ that are not edges of $B.$

The following proposition, restated in the terminology of this paper, is proved in \cite{GolovachST22Combing}.
It essentially states that if a linkage in a graph $G$ crosses through a subgraph of $G$ of maximum degree two in a complicated way, i.e., creating large treewidth, then there is a way to simplify the linkage.

\begin{proposition}[Lemma 1,  \cite{GolovachST22Combing}]\llabel{prop_cae} Let $G$ be a graph and 
$(\Lcal, B)$ be an LB-pair of $G.$
Then, if $\tw(\Lcal \cup B) > \ell(|\Lcal|),$ $G$ contains a linkage $\Rcal$ such that
\begin{enumerate}
\item $\mathsf{cae}(\Rcal, B) < \mathsf{cae}(\Lcal, B),$
\item $\Rcal$ is equivalent to $\Lcal,$ and
\item $\Rcal \subseteq \Lcal \cup B.$
\end{enumerate}
\end{proposition}

Given a $\delta$-grounded cycle $C$ of some $\Sigma$-decomposition $\delta$ of some graph, we denote by $G^{-}_{C}$ the graph obtained by the union of the graphs $G_{c},$ where $c \subseteq \Delta_{C}$ and $G_{c} \cap C = \emptyset.$

\medskip
We proceed with the main result of this subsection which gives us our desired minimal linkage with all of the necessary additional properties that we discussed.
The running time follows from standard irrelevant vertex arguments used in order to compute the minimal linkage implied by the lemma.

\begin{lemma}\llabel{lem_taming_linkage_nests} There exist functions $f_{\ref{lem_taming_linkage_nests}}^{1} : \Nbbb \to \Nbbb,$ $f_{\ref{lem_taming_linkage_nests}}^{2} : \Nbbb \to \Nbbb,$ $f_{\ref{lem_taming_linkage_nests}}^{3} : \Nbbb \to \Nbbb,$ and an algorithm that,
\begin{itemize}
\item given $r \in \Nbbb_{\geq 1},$ $s \geq f_{\ref{lem_taming_linkage_nests}}^{1}(r),$
\item a graph $G,$ a linkage $\Lcal$ of order $r$ in $G,$ a surface $\Sigma,$
\item a sequence $(\Ccal^{1}, \ldots, \Ccal^{q})$ where, for every $i \in [q],$ $\Ccal^{i}$ is a set of at least $s$ many cycles of $G,$
\item a sequence $(\Delta_{1}, \ldots, \Delta_{q})$ of pairwise disjoint disks in $\Sigma,$
\item a $\Sigma$-decomposition $\delta$ of $G$ such that
\begin{itemize}
\item for every $i \in [q],$ $(\delta[\Delta_{i}], c_{i}, \Ccal_{i})$ is a $q$-nested cylindrical rendition in $\Delta_{i}$ around a cell $c_{i} \in C(\delta),$
\item $C(\delta) \setminus \{ c_{i} \mid i \in [p] \}$ contains only flap cells, and
\item for every $i \in [p],$ the terminals of $\Lcal$ are either vertices of $G^{-}_{C^{i}_{1}},$ for some $i \in [q],$ or vertices of $G \setminus \bigcup_{i \in [q]} G \cap \Delta_{i},$ and
\end{itemize}
\item a sequence $(\Delta'_{1}, \ldots, \Delta'_{q})$ of arc-wise connected subsets of $(\Delta_{1}, \ldots, \Delta_{q})$ respectively such that, for every $i \in [q],$ the graph $\cupall \Lcal \cap \Delta_{i}$ is drawn in $\Delta'_{i},$
\end{itemize}
outputs a linkage $\Rcal$ such that
\begin{enumerate}
\item\llabel{tame_1} $\Rcal$ is equivalent to $\Lcal,$
\item\llabel{tame_2} $\Rcal$ is a subgraph of the graph obtained from $\cupall \Lcal \cup (\bigcup_{i \in [q]} \cupall \Ccal^{i})$ after removing, for every $i \in [q],$ every vertex and edge of $\cupall \Ccal^{i}$ not drawn in $\Delta'_{i},$ 
\item\llabel{tame_3} for every $i \in [q],$ the number of disjoint subpaths of paths in $\Rcal$ with at least one endpoint in $B_{\Delta_{i}}$ that contain an edge of $G^{-}_{C^{i}_{1}}$ is at most $f_{\ref{lem_taming_linkage_nests}}^{2}(r),$ and
\item\llabel{tame_4} $E(\Rcal) \cap E(G^{-}_{C^{i}_{1}}) \neq \emptyset$ for at most $f_{\ref{lem_taming_linkage_nests}}^{3}(r)$ distinct indices $i \in [p]$
\end{enumerate}
in time
$$2^{\mathsf{poly}(r)} \cdot |G|^{2}.$$
Moreover
$$\textrm{$f_{\ref{lem_taming_linkage_nests}}^{1}(r) = \Ocal(\ell(r)),$ $f_{\ref{lem_taming_linkage_nests}}^{2}(r) = \ell(r),$ and $f_{\ref{lem_taming_linkage_nests}}^{3}(r) = 2^{\Ocal(\ell(r))}.$}$$
\end{lemma}
\begin{proof} We define $f_{\ref{lem_taming_linkage_nests}}^{1}(r) = 3\ell(r) + 2,$ $f_{\ref{lem_taming_linkage_nests}}^{2}(r) = \ell(r),$ and $f_{\ref{lem_taming_linkage_nests}}^{3}(x) = (2r)^{2\ell(r)} \cdot \ell(r) + 2r.$

Let $\mathbf{\Delta} \coloneqq (\Delta'_{1}, \ldots, \Delta'_{q}).$
For brevity, we call $\Lcal$ a $\mathbf{\Delta}$-avoiding linkage.
For every $i \in [q],$ let $\mathbf{C}_{i}$ be the graph obtained from $\cupall \Ccal_{i}$ after removing every vertex and edge not drawn in $\Delta'_{i}.$
First, notice that the pair $(\Lcal, \mathbf{C}_{i})$ is an LB-pair of $G,$ for every $i \in [q].$
Then, consider $\Rcal$ to be a $\mathbf{\Delta}$-avoiding linkage of order $r$ in $G,$ that is equivalent to $\Lcal,$ is a subgraph of $\cupall \Lcal \cup \bigcup_{i \in [q]} \mathbf{C}_{i},$ and is such that for every $i \in [q],$ the quantity $\mathsf{cae}(\Lcal, \mathbf{C}_{i})$ is minimized.
Clearly, by definition $\Rcal$ satisfies properties \ref{tame_1} and \ref{tame_2}.
We argue why $\Rcal$ also satisfies properties \ref{tame_3} and \ref{tame_4}.

\paragraph{Classifying subpaths.}
Consider an index $i \in [q].$
Let $\Rcal_{i}$ be the set containing all $B_{\Delta_{i}}$-$B_{\Delta_{i}}$-paths that are subpaths of the paths in $\Rcal$ and whose vertices and edges belong to $G \cap \Delta_{i}.$
Also, let $\Tcal_{i}$ be the set containing all paths that are subpaths of the paths in $\Rcal$ and whose vertices and edges belong to $G \cap \Delta_{i},$ with one endpoint in $B_{\Delta_{i}}$ and the other being a terminal of $\Rcal$ that is vertex of $G^{-}_{C^{i}_{1}}.$
Notice that the set $\Rcal_{i} \cup \Tcal_{i}$ is a linkage in $G.$
Now, partition $\Rcal_{i}$ into the set $\Rcal'_{i}$ that contains all paths in $\Rcal_{i}$ that contain an edge of $G^{-}_{C^{i}_{1}}$ and the set $\Pcal_{i}^{\mathsf{m}}$ that contains those that do not.
Also, let $\Pcal_{i}^{\mathsf{v}}$ be the linkage containing all $V(C^{i}_{1})$-$V(C^{i}_{1})$-paths that are subpaths of the paths in $\Rcal'_{i} \cup \Tcal_{i}$ that do not intersect $G^{-}_{C^{i}_{1}}.$
Finally, let $\Pcal_{i}^{\mathsf{r}}$ be the linkage containing all minimal $B_{\Delta_{i}}$-$V(C^{i}_{1})$-paths that are subpaths of the paths in $\Rcal'_{i} \cup \Tcal_{i}.$

\paragraph{Bounds from combing.}
Consider an index $i \in [q].$
Let $t_{i} \coloneqq 3 \tw(\Rcal \cup \mathbf{C}_{i}) / 2.$
Let $(C^{i}_{1}, \ldots, C^{i}_{s_{i}}),$ where $s_{i} \geq s,$ be the sequence of cycles of the nest $\Ccal^{i}.$
First by \autoref{prop_cae}, we have that for every $i \in [q],$ $\tw(\Rcal \cup \mathbf{C}_{i}) \leq \ell(r).$
Moreover, in the terminology used in \cite{GolovachST22Combing}, the linkage $\Pcal_{i}^{\mathsf{r}}$ contains all $\mathbf{C}_{i}$-\textsl{rivers}, $\Pcal_{i}^{\mathsf{v}}$ contains $\mathbf{C}_{i}$-\textsl{valleys}, while the linkage $\Pcal_{i}^{\mathsf{m}}$ contains $\mathbf{C}_{i}$-\textsl{mountains} of $\Rcal.$
Moreover, by an appropriate application of Lemmas 4 and 7, from \cite{GolovachST22Combing}, we have that $|\Pcal_{i}^{\mathsf{v}}| \leq \tw(\Rcal \cup \mathbf{C}_{i}) \leq \ell(r),$ that $\Pcal_{i}^{\mathsf{v}}$ may intersect only the first $t_{i} \leq 3 \ell(r) / 2$ cycles of $\Ccal^{i},$ while $\Pcal_{i}^{\mathsf{m}}$ may intersect only the last $t_{i} \leq 3 \ell(r) / 2$ cycles of $\Ccal^{i}.$
Moreover, notice that for each path in $\Rcal'_{i} \cup \Tcal_{i},$ there is at least one path in $\Pcal_{i}^{\mathsf{r}}$ that is a subpath of it.
This implies that $|\Rcal'_{i} \cup \Tcal_{i}| \leq \ell(r).$
Then $\Rcal$ also satisfies property \ref{tame_3}.

\paragraph{Bound on the number of vortices.}
If $q \leq f_{\ref{lem_taming_linkage_nests}}^{3}(r)$ then we can immediately conclude.
Hence assume that $q > f_{\ref{lem_taming_linkage_nests}}^{3}(r).$
First, we can easily observe that the number of distinct indices $i \in [q]$ for which $\Tcal_{i} \neq \emptyset$ is at most $2r$ since $\Rcal$ hast at most $2r$ many terminals.
Let $Q \subseteq [q]$ be the set of indices for which $\Tcal_{i} \neq \emptyset,$ $i \in Q.$
Now, consider an index $i \in [q] \setminus Q.$
Let $\Lambda_{i}$ be the sub-sequence of $\Lambda_{\Delta_{i}, y},$ for a choice of $y \in V_{\Delta_{i}},$ such that the set of maximal subpaths, say $\Pcal_{i},$ of the cycles in $\Ccal_{i}$ that connect vertices of the paths in $\Rcal'_{i}$ and that also contain at least one internal vertex that is a vertex of a path in $\Rcal'_{i},$ are subgraphs of $\mathbf{C}_{i},$ and such that $B_{\Lambda_{i}}$ contains exactly the endpoints of the paths in $\Rcal'_{i}.$
Notice that, by the assumptions for $\Delta'_{i},$ such a $\Lambda_{i}$ must exist.
Hence $\lin{\Qcal_{i} \coloneqq \Rcal'_{i} \cup \Pcal_{i}, \Lambda_{i}}$ is a linear society where $b_{i} \coloneqq |\Lambda_{i}| \leq 2\ell(r),$ $r_{i} \coloneqq |\Rcal'_{i}| = \ell(r),$ and $\Pcal_{i} \leq |\Ccal_{i}|.$

Now, fix a linear ordering $\Lambda_{\Rcal} \coloneqq \lin{(s_{1}, t_{1}), \ldots, (s_{r}, t_{r})}$ of the terminals of the paths in $\Rcal.$
We proceed to define for every $i \in [q] \setminus Q,$ the signature of $\lin{\Qcal_{i}, \Lambda_{i}}$ with respect to $\Lambda_{\Rcal},$ which we denote by $\mathsf{sig}_{\Lambda_{\Rcal}}(\Qcal_{i}, \Lambda_{i}),$ as follows.
We define $$\mathsf{sig}_{\Lambda_{\Rcal}}(\Qcal_{i}, \Lambda_{i}) \coloneqq \lin{x_{1}, \ldots, x_{b_{i}}},$$ where for every $j \in [b_{i}],$ $x_{j}$ equals $s_{n}$ (resp. $t_{n}$) if the $j$-th vertex in $\Lambda_{i},$ say $u,$ is a vertex of the $n$-th path of $\Rcal$ with respect to $\Lambda_{\Rcal},$ say $R,$ and an endpoint of the subpath $R' \in \Rcal'_{i}$ of $R,$ and moreover if the vertex $u$ appears before (resp. after) on $R$ ordered from $s_{n}$ to $t_{n}$ than the other endpoint of $R'.$
Observe that the above definition for $\mathsf{sig}_{\Lambda_{\Lcal}}(\Qcal_{i}, \Lambda_{i}),$ implies that the size of the set $\{ \mathsf{sig}_{\Lambda_{\Lcal}}(\Qcal_{i}, \Lambda_{i}) \mid i \in [q] \}$ is upper bounded by $(2r)^{2\ell(r)}.$

To conclude, assume towards a contradiction, that there are more than $f_{\ref{lem_taming_linkage_nests}}^{3}(r)$ distinct indices $i \in [q]$ such that $E(\Rcal) \cap E(G^{-}_{C^{i}_{1}}) \neq \emptyset.$
This means that there exist $i \neq j \in [q] \setminus Q$ for which $\mathsf{sig}_{\Lambda_{\Rcal}}(\Qcal_{i}, \Lambda_{i}) = \mathsf{sig}_{\Lambda_{\Rcal}}(\Qcal_{j}, \Lambda_{j}).$
W.l.o.g. observe that by a simple parity argument, there exists an index $n \in [b_{i}]$ such that $x_{n} = s_{w}$ and $x_{n+1} = t_{w'},$ for some $w, w' \in [r_{i}].$
Let $R_{w}$ (resp. $R_{w'}$) be the $s_{w}$-$t_{w}$-(resp. $s_{w'}$-$t_{w'}$)-path of $\Rcal.$
Let $I \in \Rcal'_{i}$ (resp. $O \in \Rcal'_{i}$) be the subpath of $R_{w}$ (resp. $R_{w'}$) such that $x_{n}$ (resp. $x_{n+1}$) is an endpoint of $I$ (resp. $O$).
Let $y$ (resp. $z$) be the other endpoint of $I$ (resp. $O$).
Then, let $I_{\mathsf{s}}$ (resp. $I_{\mathsf{t}}$) be the minimal subpath of $I$ whose endpoints are $x_{n}$ (resp. $y$) and a vertex of $V(C^{i}_{1}).$
Symmetrically, let $O_{\mathsf{t}}$ (resp. $O_{\mathsf{s}}$) be the minimal subpath of $O$ whose endpoints are $x_{n+1}$ (resp. $z$) and a vertex of $V_{C^{i}_{1}}.$
Also let $(P_{1}, \ldots, P_{s_{i}})$ be the sequence of the paths in $\Pcal_{i}$ ordered from innermost to outermost, and let $(P'_{1}, \ldots, P'_{s_{i}})$ be the $x_{i}$-$y_{i}$-paths that are minimal subpaths of $P_{i},$ where $u_{i} \in V(I_{\mathsf{s}})$ and $v_{i} \in V(O_{\mathsf{t}}).$
Observe that, by the previous arguments and since the endpoints $x_{n}$ and $x_{n+1}$ are consecutive in $\Lambda_{i},$ $P'_{t_{i} + 1}$ as well as $P'_{s_{i} - t_{i} - 1}$ is not intersected by any other path in $\Rcal_{i}.$
Hence we can obtain a new linkage $\Rcal'$ from $\Rcal$ by swapping the pairs $(s_{w}, t_{w})$ and $(s_{w'}, t_{w})$ of the pattern of $\Rcal$ as follows.
We define $\Rcal' \coloneqq \Rcal \setminus \{ R_{w}, R_{w'} \} \cup \{ S, T \},$ where
$$S \coloneqq s_{w}R_{w}x_{n} \cup x_{n}I_{\mathsf{s}}u_{s_{i} - t_{i} - 1} \cup P'_{s_{i} - t_{i} - 1} \cup v_{s_{i} - t_{i} - 1}I_{\mathsf{t}} \cup yR_{w'}t_{w'}$$
and
$$T \coloneqq s_{w'}R_{w'}z \cup zOv_{t_{i} + 1} \cup P'_{t_{i} + 1} \cup v_{t_{i} + 1}Ox_{n+1} \cup x_{n+1}R_{w'}t_{w'}.$$

Since, $\mathsf{sig}_{\Lambda_{\Lcal}}(\Qcal_{i}, \Lambda_{i}) = \mathsf{sig}_{\Lambda_{\Lcal}}(\Qcal_{j}, \Lambda_{j})$ we can perform the exact same change within $\mathsf{sig}_{\Lambda_{\Lcal}}(\Qcal_{j}, \Lambda_{j})$ and obtain a linkage $\Rcal''$ that is equivalent to $\Rcal$ and we can moreover observe that $\mathsf{cae}(\Rcal''_{i}, \mathbf{C}_{i}) < \mathsf{cae}(\Rcal_{i}, \mathbf{C}_{i})$ and $\mathsf{cae}(\Rcal''_{j}, \mathbf{C}_{j}) < \mathsf{cae}(\Rcal_{j}, \mathbf{C}_{j}),$ which is a contradiction.
Hence $\Rcal$ also satisfies property \ref{tame_4}.
\end{proof}

\section{Plowing a meadow}\llabel{Thidplodowinga}

In this section, we obtain a preliminary version of our $\Sigma$-decomposition in the form of a meadow and our main goal is to prove \autoref{lem_meadow_plowing}.
That is, we show that any meadow with large enough infrastructure can be refined to a meadow that is $d$-plowed.
The first subsection deals with the creation of our preliminary meadow.
The second subsection deals with the proof of \autoref{lem_meadow_plowing}.

\subsection{Creating a meadow}

In this subsection we use the results of \cite{thilikos2023excluding} to create our preliminary meadow.
This preliminary step makes an important case distinction for us as it will have two outcomes, either we find a large clique minor, or we obtain the preliminary meadow which will then be further refined in the following subsections.
Due to this, the only other case where we have to deal with the clique minor is in \autoref{sec_local_structure} where we combine all of the intermediate steps into one central theorem, namely \autoref{thm_local_structure}.

We need a definition to describe what it means for a clique-minor to be ``local'' with respect to some wall.
For this notice that, if a graph $G$ contains $K_t$ as a minor, there exist $t$ pairwise vertex disjoint connected subgraphs $X_1,\dots,X_t$ such that for every pair $i,j\in[t]$ there exists an edge with one end in $X_i$ and the other in $X_j.$
Indeed, every $K_t$-minor in $G$ is witnessed by such a family.
Now let $W$ be a $k$-wall where $k\geq t+1.$
We say that a clique minor is \defi{grasped} by $W$ if there exist sets $\langle X_i | i\in t \rangle$ as above such that for every $i\in[t]$ there exists a pair $(h^i_1,h^i_2)\in[k]\times[k]$ such that the intersection of the $h_1$st row and the $h_2$nd column of $W$ is contained in $X_i.$
Moreover, for $i\neq j$ we have that $h^i_1\neq h^j_1$ and $h^i_2\neq h^j_2.$

\begin{observation}[Meadow creation]\llabel{obs_create_meadow}
    There exist functions $f^1_{\ref{obs_create_meadow}},f^2_{\ref{obs_create_meadow}}\colon\mathbb{N}^3\to\mathbb{N}$ and an algorithm that, given positive integers $g,$ $k,$ $t$ and $h,$ a graph $G,$ and an $f^1_{\ref{obs_create_meadow}}(h,k,t)$-wall $W'$ as input, finds, in time $\mathcal{O}_{h,t}(|G|^3)$ either
    \begin{enumerate}
        \item a $K_{3k^2+k\cdot h}$-minor grasped by $W',$ 
        \item there exist non-negative integers $h'$ and $c'\leq 2$ such that $2h'+c'=g$ and a subdivision $W'$ of a walloid which is a cylindrical concatenation of $h'$ $t$-handle segments and $c'$ $t$-crosscap segments in $G-A$ such that the base cylinder of $W'$ is a subgraph of $W,$ or
        \item a set $A\subseteq V(G)$ of size at most $f_{\ref{obs_create_meadow}}^2(h,k,t)$ and an $(x,y)$-fertile $(t,\Sigma)$-meadow $(G-A,\delta,W)$ where $\Sigma$ is of Euler genus less than $g.$
    \end{enumerate}
Moreover, we have $x\in \mathcal{O}(k^4\cdot h),$ $y\in 2^{\mathsf{poly}_{g}(k)}\cdot t,$ $f^1_{\ref{obs_create_meadow}}(h,k,t)\in \mathsf{poly}_g(k,t),$ and $f^2_{\ref{obs_create_meadow}}(h,k,t)\in 2^{\mathsf{poly}_{g}(k)}\cdot t.$
\end{observation}

For this let us first state the technical key result of \cite{thilikos2023excluding} which acts as our point of departure.
Please notice that the following statement is an abbreviated version of Lemma 5.5 from \cite{thilikos2023excluding} where we only state those parts of the conclusion that are relevant to our case.
The main difference to the way the Lemma is stated in \cite{thilikos2023excluding} is that it produces a walloid representing a surface instead of a clique-minor (this walloid is actually found within the clique-minor in the proof of \cite{thilikos2023excluding}).
For our purposes, however, the clique-minor is much more convenient.
Moreover, notice that the first and the second outcome are merged together in the original statement.
For our purposes, it is slightly more convenient to state them separately.

\begin{lemma}[an abbreviated version of Lemma 5.5 from \cite{thilikos2023excluding}]\llabel{surfex_lemma}
    Let $g,$ $k,$ and $q$ be non-negative integers, where $r\geq  16\cdot 12(162^{2g}+4)k$ and $q\geq 6 k,$ and set 
    \begin{eqnarray*}
  t & \coloneqq &  12(162^{2g}+4)k,\\
    M& \coloneqq &    2(162^{2g}+4)q,\\
   \mu & \coloneqq & \left\lfloor \frac{1}{10^5t ^{26}}t^{10^7t^{26}}\cdot M\right\rfloor, \text{~and~}\\
   R & \coloneqq & 49152\cdot t^{24}(3+4M)+\mu. 
    \end{eqnarray*}
  
  \noindent  Let $W$ be an $R$-wall in a graph $G.$
    Then either, 
    \begin{enumerate}
    \item $G$ contains a $K_t$-minor grasped by $W,$
    \item there exist 
    \begin{itemize}
    \item a surface $\Sigma$ of Euler-genus $g'<g,$
    \item a set $A\subseteq V(G)$ of size at most $\mu,$ and 
    \item a $\Sigma $-decomposition $\delta =(\Gamma,\mathcal{D})$ of $G-A$ with the following properties:
    \begin{enumerate}
      \item $\delta $ is of breadth at most $2t^2$ and depth at most $\mu,$
      \item there exist non-negative integers $h'$ and $c'\leq 2$ such that $2h'+c'=g'$ and a subdivision $W_{\Sigma}$ of a walloid which is a cylindrical concatenation of $h$ $q$-handle segments and $c$ $q$-crosscap segments in $G-A$ such that $W_{\Sigma}$ is grounded in $\delta,$
      \item the base cylinder of $W_{\Sigma}$ is a subgraph of $W,$ and
      \item let $X$ be the set of all vertices of $G$ drawn in the interior of cells of $\delta $ and let $G_{\delta}$ be the torso of the set $V(G)\setminus(X\cup A)$ in $G-A,$ then $G_{\delta}$ is a quasi-$4$-connected component of $G-A.$
    \end{enumerate}
    \end{itemize}
    \end{enumerate}
Moreover, there exists an algorithm that returns one of the outcomes above in time $\mathcal{O}_{g,h,k,q}(|G|^3)$
\end{lemma}
  
\begin{proof}[Proof sketch of \autoref{obs_create_meadow}]
To see how \autoref{obs_create_meadow} follows from \autoref{surfex_lemma} for an appropriate choice of numbers first notice that the first outcome of \autoref{surfex_lemma} gives the first outcome of \autoref{obs_create_meadow}.
So we only have to deal with the second outcome of \autoref{surfex_lemma}.
In this outcome, we are already given the walloid $W_{\Sigma}.$
In case the Euler-genus of $\Sigma$ is at least $g$ we find the desired walloid for the second outcome of \autoref{obs_create_meadow} within $W_{\Sigma}.$
Hence, we may assume that the Euler-genus of $\Sigma$ is less than $g.$
Let now $C$ denote the exception cycle of $W_{\Sigma}$ and $\Delta$ be the disk bounded by its trace that avoids the simple face of $W_{\Sigma}.$
Notice that $\Delta$ induces a society that has a rendition in $\delta$ with breadth at most $x$ and depth at most $y.$
By declaring this disk a new cell of $\delta$ we obtain the desired meadow.
\end{proof}

\subsection{The meadow plowing lemma}

We are now ready to proceed with the proof of the main result of this section.
This lemma marks our first major refinement of our initial $\Sigma$-decomposition.
We prove that given a meadow with large enough infrastructure, we can maintain a still large infrastructure that has the property that we can divide into large territories (enclosures) that all have the property of being $d$-homogeneous, i.e., that small parts of the model (with respect to $d$) that invade such a territory can be represented anywhere in the area defined by the territory.

\begin{lemma}[meadow plowing]\llabel{lem_meadow_plowing}
There exists a function $f_{\ref{lem_meadow_plowing}}: \Nbbb^3 \to \Nbbb$ and an algorithm that, given $t \in \Nbbb_{\geq 3},$ $d \in \Nbbb_{\geq 1},$ $\gamma \in \Nbbb,$ and an $(x, y)$-fertile $(f_{\ref{lem_meadow_plowing}}(t, d, \gamma), \Sigma)$-meadow $(G, \delta, W),$ where $\Sigma = \Sigma^{\mathsf{h} + \mathsf{c}}$ and $\gamma = \mathsf{h} + \mathsf{c},$ outputs an $(x, y)$-fertile $(t, \Sigma)$-meadow that is $d$-plowed in time
$$(t \gamma)^{2^{2^{\Ocal(d^{2})}}} |G|^{2}.$$
Moreover
$$f_{\ref{lem_meadow_plowing}}(t, d, \gamma) \in (t \gamma)^{2^{2^{\Ocal(d^{2})}}}.$$
\end{lemma}
\begin{proof} We define $f_{\ref{lem_meadow_plowing}}(t, d, \gamma) \coloneqq \max \{ f_{\ref{lem_homogeneity_2d}}(4 \cdot f_{\ref{lem_homogeneity_2d}}(2t, f_{\ref{obs_size_dfolio}}(d)) \cdot \gamma + 2t, f_{\ref{obs_size_dfolio}}(d)) / \gamma, 4 \cdot f_{\ref{lem_homogeneity_2d}}(2t, f_{\ref{obs_size_dfolio}}(d)) \cdot \gamma + f_{\ref{lem_homogeneity_2d}}(2t, f_{\ref{obs_size_dfolio}}(d)) \}.$

Let $\Sigma$ be a surface of Euler-genus $2 \mathsf{h} + \mathsf{c},$ $\gamma \coloneqq \mathsf{h} + \mathsf{c},$ $t \in \Nbbb_{\geq 3},$ $d \in \Nbbb_{\geq 1},$ and $(G, \delta, W)$ be an $(x, y)$-fertile $(f_{\ref{lem_meadow_plowing}}(t, d', \gamma), \Sigma)$-meadow, where $d' \coloneqq |d\text{-}\textsf{folio}(\delta)|,$ which by \autoref{obs_size_dfolio} is at most  $f_{\ref{obs_size_dfolio}}(d).$
Our proof works as follows.
We first show how to define a $(t, \Sigma)$-walloid $W'$ that is the cylindrical concatenation of $\mathsf{h}$ many $t$-handle segments and $\mathsf{c}$ many $t$-crosscap segments, as a subgraph of $W,$ with the property that all its $2\mathsf{h} + \mathsf{c} + 1$ enclosures are $d$-homogeneous in $\delta.$
Then, we show how to define a $\Sigma$-decomposition $\delta'$ from $\delta$ such that $(G, \delta', W')$ is the required $(x, y)$-fertile $(t, \Sigma)$-meadow $(G, \delta', W')$ that is $d$-plowed.

\paragraph{Finding a homogeneous subwall.} 
Consider the graph $\widetilde{W}$ which is the base cylinder of $W$ obtained by the cylindrical concatenation of the base walls of its $\gamma$ many $f_{\ref{lem_meadow_plowing}}(t, d, \gamma)$-handle and crosscap segments.
The base wall of each of these is by definition a $(4 \cdot f_{\ref{lem_meadow_plowing}}(t, d, \gamma), f_{\ref{lem_meadow_plowing}}(t, d, \gamma))$-wall.
As a result $\widetilde{W}$ is a $(4 \cdot f_{\ref{lem_meadow_plowing}}(t, d, \gamma) \cdot \gamma, f_{\ref{lem_meadow_plowing}}(t, d, \gamma))$-cylindrical wall.
Then, by definition of $f_{\ref{lem_meadow_plowing}},$ there is a $(f_{\ref{lem_homogeneity_2d}}(4 \cdot f_{\ref{lem_homogeneity_2d}}(2t, d) \cdot \gamma + 2t, d), f_{\ref{lem_homogeneity_2d}}(2t, d))$-wall $W_{2}$ that is a subwall of $\widetilde{W}$ and that is disjoint from the first $4 \cdot f_{\ref{lem_homogeneity_2d}}(2t, d) \cdot \gamma$ cycles of $\widetilde{W}.$
Now, we define the $d'$-palette of every brick $B$ of $W_{2}$ as follows.
Observe that $B$ is a fence of $W$ and as such its $d$-folio is well-defined.
Consider any bijection $g : d\text{-}\mathsf{folio}(\delta) \to [d']$ and let $d'\text{-}\mathsf{palette}(B) = g(d\text{-}\mathsf{folio}(B)).$
Next, we apply \autoref{lem_homogeneity_2d} for $d'$ and $W_{2}$ and we obtain a $(4 \cdot f_{\ref{lem_homogeneity_2d}}(2t, d) \cdot \gamma + 2t, 2t)$-wall $W_{3}$ that is a subwall of $W_{2}$ and such that every brick of $W_{3}$ has the same $d'$-palette, i.e., every brick of $W_{3}$ has the same $d$-folio.

\paragraph{Defining the base cylinder.} Let $b \coloneqq 4 \cdot f_{\ref{lem_homogeneity_2d}}(2t, d) \cdot \gamma + 2t$ and $a \coloneqq 2t.$
Let $P^{2}_{1}, \ldots, P^{2}_{a}$ be the horizontal paths of $W_{2}$ and $Q^{2}_{1}, \ldots, Q^{2}_{b}$ be the vertical paths of $W_{2}.$
We define $t$ pairwise disjoint cycles $\Ccal \coloneqq \{ C_{1}, \ldots, C_{t} \}$ as follows.
For every $i \in [t],$ define $C_{i} \coloneqq P^{2}_{i} \cup P^{2}_{a - i + 1} \cup Q^{2}_{i} \cup Q^{2}_{b - i + 1}.$
Next we define $b' \coloneqq b - 2t$ pairwise disjoint paths $\Pcal \coloneqq \{ P_{1}, \ldots, P_{b'} \}$ that are subpaths of vertical paths in $W_{2}$ as follows.
For every $i \in [t],$ define $P_{i}$ to be the subpath of $Q^{2}_{t + i}$ with one endpoint in $P^{2}_{1}$ and the other in $P^{2}_{t}.$
Observe that the graph $W_{3} \coloneqq \cupall \{ \Pcal, \Qcal \}$ is a $(4 \cdot f_{\ref{lem_homogeneity_2d}}(2t, d) \cdot \gamma, t)$-cylindrical wall that is a subgraph of $W_{2}.$
Moreover, notice that every brick of $W_{3}$ has the same $d$-folio as it is a union of bricks of $W_{2},$ and all such bricks have the same $d$-folio.
\smallskip

\paragraph{Attaching the handle and crosscap linkages.} Let $t' \coloneqq f_{\ref{lem_homogeneity_2d}}(2t, d).$
Let $\Wcal = (W_{1}, \ldots, W_{\gamma})$ be a sequence of $\mathsf{h}$ many handle and $\mathsf{c}$-crosscap segments such that $W$ is the cylindrical concatenation of $\Wcal.$
Additionally, let $\Wcal' = (W'_{1}, \ldots, W'_{\gamma})$ be a sequence of $(4 t', t)$-walls such that $W_{3}$ is the cylindrical concatenation of $\Wcal'.$

For each $i \in [\mathsf{h}],$ let $\Lcal_{i} \coloneqq \{ L_{i, 1}, \ldots, L_{i, t'} \}$ be the first $t'$ paths in left to right order of the leftmost rainbow of $W_{i}.$ Let $u_{i, j}$ and $v_{i,j}$ be the top boundary vertices of the base wall of $W_{i}$ in left to right order that are endpoints of $L_{i, j},$ $j \in [t'].$
Also, let $\Rcal_{i} \coloneqq \{ R_{i, 1}, \ldots R_{i, t'} \}$ be the first $t'$ paths in left to right order of the rightmost rainbow of $W_{i}.$ Let $w_{i, j}$ and $z_{i,j}$ be the top boundary vertices of the base wall of $W_{i}$ in left to right order that are endpoints of $R_{i, j},$ $j \in [t'].$
Additionally, let $u'_{i,1}, \ldots, u'_{i,t'}, w'_{i, 1}, \ldots, w'_{i, t'}, v'_{i, 1}, \ldots v'_{i, t'}, z'_{i, 1}, \ldots, z'_{i, t'}$ be the top boundary vertices of $W'_{i}$ in left to right order.

For each $i \in (\mathsf{h}, \mathsf{c}],$ let $\Tcal_{i} \coloneqq \{ T_{i, 1}, \ldots, T_{i, 2t'} \}$ be the first $2t'$ paths in left to right order of the rainbow of $W_{i}.$ Let $u_{i, j}$ and $v_{i,j}$ be the top boundary vertices of the base wall of $W_{i}$ in left to right order that are endpoints of $T_{i, j},$ $j \in [t'].$
Additionally, let $u'_{i,1}, \ldots, u'_{i,2t'}, v'_{i, 1}, \ldots v'_{i, 2t'}$ be the top boundary vertices of $W'_{i}$ in left to right order.

Our next target is to identify for every $i \in [\mathsf{h}],$ a set $\Ycal_{i} = \Ucal_{i} \cup \Vcal_{i} \cup \Wcal_{i} \cup \Zcal_{i}$ of $4t'$ many paths such that $\Ucal_{i} = \{ U_{i, 1}, \ldots, U_{i, t'} \}$ where, for $j \in [t'],$ $U_{i,j}$ is a $u_{i,j}$-$u'_{i,j}$-path, $\Vcal_{i} = \{ V_{i, 1}, \ldots, V_{i, t'} \}$ where, for $j \in [t'],$ $V_{i,j}$ is a $v_{i,j}$-$v'_{i,j}$-path, $\Wcal_{i} = \{ W_{i, 1}, \ldots, W_{i, t'} \}$ where, for $j \in [t'],$ $W_{i,j}$ is a $w_{i,j}$-$w'_{i,j}$-path, and $\Zcal_{i} = \{ Z_{i, 1}, \ldots, Z_{i, t'} \}$ where, for $j \in [t'],$ $Z_{i,j}$ is a $z_{i,j}$-$z'_{i,j}$-path.
Moreover, to identify for every $i \in (\mathsf{h}, \mathsf{c}],$ a set $\Ycal_{i} = \Ucal_{i} \cup \Vcal_{i}$ of $4t'$ many paths such that $\Ucal_{i} = \{ U_{i, 1}, U_{i, 2t'} \}$ where, for $j \in [2t'],$ $U_{i, j}$ is a $u_{i, j}$-$u'_{i, j}$-path and $\Vcal_{i} = \{ V_{i, 1}, \ldots, V_{i, 2t'} \}$ where, for $j \in [2t'],$ $V_{i, j}$ is a $v_{i, j}$-$v'_{i, j}$ path.
We also want the paths $\{ Y_{1}, \ldots, Y_{4t'\gamma} \}$ in $ \bigcup_{i \in [\gamma]} \{ \Ycal_{i} \}$ to be pairwise disjoint and also internally disjoint from $W_{3}.$

We show how to define the $x$-$y$-path $Y_{i},$ $i \in [4t'\gamma],$ where $x$ and $y$ are the appropriate two vertices corresponding to the path $Y_{i}$ as previously defined.
Consider $C_{j}$ to be the $j$-th cycle of the base cylinder $\widetilde{W}$ of $W$ from top to bottom and $Q_{j}$ to be the $j$-th vertical path of $\widetilde{W}$ from left to right, where $Q_{1}$ is the leftmost vertical path of $W_{1}.$
Also let $\ell$ be minimum such that $C_{\ell+1}$ is not disjoint from $W_{3}.$
Note that $\ell \geq 4t'\gamma.$
Assume w.l.o.g. that $x$ is the topmost vertex of $Q_{i_{x}}$ in $C_{1} \cap Q_{i_{x}}$ while $y$ is the topmost vertex of $Q_{i_{y}}$ in $C_{\ell} \cap Q_{i_{y}}.$
We distinguish three cases.

If $i_{x} = i_{y}$ then $Y_{i}$ is the subpath of $Q_{i_{x}}$ whose endpoints are $x$ and $y.$

If $i_{x} < i_{y}$ then $Y_{i}$ is the union of the subpath of $Q_{i_{x}}$ whose endpoints are $x$ and the topmost vertex of $Q_{i_{x}}$ in $Q_{i_{x}} \cap C_{\ell - i_{x} + 1},$ the subpath of $C_{\ell - i_{x} + 1}$ whose endpoints are the topmost vertex of $Q_{i_{x}}$ in $Q_{i_{x}} \cap C_{\ell - i_{x} + 1}$ and the topmost vertex of $Q_{i_{y}}$ in $Q_{i_{y}} \cap C_{\ell - i_{x} + 1}$ and whose internal vertices that belong to vertical paths are vertices of $Q_{j},$ $j \in (i_{x}, i_{y}),$ and the subpath of $Q_{i_{y}}$ whose endpoints are $y$ and the topmost vertex of $Q_{i_{y}}$ in $Q_{i_{y}} \cap C_{\ell - i_{x} + 1}.$

If $i_{x} > i_{y}$ then $Y_{i}$ is the union of the subpath of $Q_{i_{x}}$ whose endpoints are $x$ and the topmost vertex of $Q_{i_{x}}$ in $Q_{i_{x}} \cap C_{i_{x}},$ the subpath of $C_{i_{x}}$ whose endpoints are the topmost vertex of $Q_{i_{x}}$ in $Q_{i_{x}} \cap C_{i_{x}}$ and the topmost vertex of $Q_{i_{y}}$ in $Q_{i_{y}} \cap C_{i_{x}}$ and whose internal vertices that belong to vertical paths are vertices of $Q_{j},$ $j \in (i_{y}, i_{x}),$ and the subpath of $Q_{i_{y}}$ whose endpoints are $y$ and the topmost vertex of $Q_{i_{y}}$ in $Q_{i_{y}} \cap C_{i_{x}}.$

It is easy to see that the defined set of paths $\Ycal$ are pairwise disjoint and also internally disjoint from $W_{3}.$
Moreover by their definition, the graph $W_{4} \coloneqq \cupall \{ W_{2} \} \cup \Ycal$ is a $(t', \Sigma)$-walloid that is a subgraph of $W$ and that is the cylindrical concatenation of $\Wcal'' \coloneqq \{ W''_{i} \mid i \in [\gamma] \},$ where $W''_{i},$ for $i \in [\mathsf{h}],$ is the $t'$-handle segment obtained by the union of $W'_{i},$ the two rainbows of $W_{i},$ and the linkage $\Ycal_{i},$ and where $W''_{i},$ for $i \in (\mathsf{h}, \mathsf{c}],$ is a $t'$-crosscap segment obtained by the union of $W'_{i},$ the rainbow of $W_{i},$ and the linkage $\Ycal_{i}.$
Moreover, the big enclosure of the base cylinder $\widetilde{W}_{4}$ of $W_{4}$ is $d$-homogeneous, since $\widetilde{W}_{4}$ is the previously defined cylindrical wall $W_{3}$ whose bricks have the same $d$-folio.

\paragraph{Homogenizing them.} For $i \in [\mathsf{h}],$ we define the graph $L_{i}$ as the union of the leftmost enclosure of $W''_{i},$ the leftmost rainbow of $W_{i},$ the linkage $\Ucal_{i},$ and the linkage $\Wcal_{i}.$
Symmetrically, we define the graph $R_{i}$ as the union of the rightmost enclosure of $W''_{i},$ the rightmost rainbow of $W_{i},$ the linkage $\Vcal_{i},$ and the linkage $\Zcal_{i}.$
Similarly, for $i \in (\mathsf{h}, \mathsf{c}],$ we define the graph $S_{i}$ as the union of the enclosure of $W''_{i},$ the rainbow of $W_{i},$ the linkage $\Ucal_{i},$ and the linkage $\Vcal_{i}.$
Observe that the graphs $L_{i}$ and $R_{i}$ are $t'$-ladders where each of their $t'$ many horizontal paths contains a common vertex with two vertical paths of $W_{3}.$
Also the graphs $S_{i}$ are $2t'$-ladders where each of their $2t'$ many horizontal paths contains a common vertex with two vertical paths of $W_{3}.$
Observe that every brick of each of the graphs $L_{i}, R_{i}, S_{i}$ is a fence of $W$ and as such its $d$-folio is defined.
In the same way as in step 1 we can define the $d'$-palette of each of these bricks and then apply \autoref{cor_homogeneity_1d}, to obtain $t$-ladders $L'_{i}, R'_{i}$ and $2t$-ladders $S'_{i}$ whose bricks individually all have the same $d$-folio.
Finally, we can obtain our desired $(t, \Sigma)$-walloid $W_{5}$ by considering a sub-cylindrical wall of $W_{3}$ that contains all cycles of $W_{3}$ but only the vertical paths that contain an endpoint of some horizontal path of some $L'_{i}, R'_{i}, S'_{i},$ union all $L'_{i}, R'_{i}, S'_{i}.$
Notice that all $2\mathsf{h} + \mathsf{c} + 1$ enclosures of $W_{5}$ are now $d$-homogeneous.

\medskip
To conclude it remains to define the desired $\Sigma$-decomposition $\delta' = (\Gamma', \Dcal').$
We obtain $\delta'$ from $\delta$ by considering the same drawing of $G$ in $\Sigma,$ $\Gamma' = \Gamma,$ and defining the cells $\Dcal'$ as follows.
Let $\Delta$ be the $\trace(C^{\mathsf{si}})$-avoiding disk $C^{\mathsf{ex}},$ where $C^{\mathsf{ex}}$ is the exceptional cycle of $W_{5}$ while $C^{\mathsf{si}}$ is the simple cycle of $W_{5}.$
Define $\Dcal'$ to contain all cells of $\delta$ not drawn in $\Delta,$ as well as a new cell that is the closure of all cells of $\delta$ drawn in $\Delta.$
it is easy to observe that $\delta'$ is indeed as required.
\end{proof}

\section{Sowing a garden}\llabel{sec_gardening}

In this section, we obtain the second major refinement of our $\Sigma$-decomposition in the form of a $d$-blooming garden and our main goal is to prove \autoref{lem_garden_creation}.
We show how given a $d$-plowed meadow with enough infrastructure, we can refine it into a garden that is $d$-blooming.

\subsection{Planting the seeds}

In this subsection, we first introduce the notion of a fence at some distance from some other fence in a walloid, which will be useful for several arguments concerning the routing of paths in our walloid in the following proofs.
We then prove a couple of auxiliary results that will serve as intermediate steps towards the proof of \autoref{lem_garden_creation}.

\medskip
Let $(G, \delta, W)$ be a $(t, \Sigma)$-meadow of a graph $G$ in a surface $\Sigma$ for some $t \in \Nbbb_{\geq 3}.$
Let $F$ be a fence of $W.$
We say that $F_{0} \coloneqq F$ is the fence \defi{at distance} $0$ \defi{from} $F$ in $W.$
Given $r \in \Nbbb_{\geq 1},$ we define recursively the fence at distance $r$ from $F$ in $W$ as follows.
If $F_{r-1}$ is disjoint from $C^{\mathsf{si}}$ and $C^{\mathsf{ex}},$ then $F_{r}$ is the unique fence of $W$ whose bricks are exactly the bricks of $F_{r-1}$ union all bricks of $W$ that share a vertex with some brick of $F_{r-1}.$
We call a fence $F$ of $W$ $r$-\defi{internal} if the fence at distance $r$ from $F$ is well-defined.

Given an $1$-internal fence $F$ of $W,$ we define its \defi{internal pegs} (resp. \defi{external pegs}) as its vertices that are incident to edges of $W$ that are not edges of $F$ and that are edges of bricks of $F$ (resp. of bricks that are not bricks of $F$).

\medskip
It follows from the properties of $\delta$ that for every flap cell $c$ of $\delta$ there are $|\widetilde{c}|$ vertex-disjoint paths to any $|\widetilde{c}|$ many distinct vertices not in $\sigma_{\delta}(c).$
This observation allows us to use the following proposition which is a restatement of \cite[Lemma 22]{BasteST19HittingOLD} in the terminology used in this text.

\begin{proposition}[Lemma 22, \cite{BasteST19HittingOLD}]\llabel{prop_cell_paths_to_pegs}
Let $t \in \Nbbb_{\geq 3},$ $\Sigma$ be a surface, and $(G, \delta, W)$ be a $(t, \Sigma)$-meadow of $G.$
For every $2$-internal brick $B$ of $W$ and every flap cell $c$ in the $\delta$-influence of $B,$ $W$ contains $|\widetilde{c}|$ vertex-disjoint paths from $\widetilde{c}$ to the external pegs of the fence $F$ at distance $1$ from $B$ in $W.$
Moreover these paths are drawn inside of the $(\trace(C^{\mathsf{si}}) \cap N(\delta))$-avoiding disk of the trace of the fence at distance $2$ from $B$ in $W.$
\end{proposition}

The following lemma will be useful in order to extract the flower segments needed to obtain the desired garden in \autoref{lem_garden_creation}.

\begin{lemma}\llabel{lem_paths_boundary} Let $t \in \Nbbb_{\geq 3},$ $\Sigma$ be a surface, and $(G, \delta, W)$ be a $(t, \Sigma)$-meadow of $G.$
For every $4$-internal brick $B$ of $W$ and every flap cell $c$ in the $\delta$-influence of $B,$ for every $c$-society $(G_{c}, \Lambda_{c, y}),$ $W$ contains $|\widetilde{c}|$ vertex-disjoint paths from $\widetilde{c}$ to the bottom boundary vertices of a $(10 \times 10)$-wall $W'$ that is a subgraph of $W$ and contains all bricks of the fence at distance $4$ from $B,$ such that the endpoints of these paths not in $\widetilde{c}$ respect the linear ordering $\Lambda_{c, y}.$
\end{lemma}
\begin{proof} Let $B$ be a $4$-internal brick of $W,$ $c$ be a flap cell in the $\delta$-influence of $B,$ and fix a $c$-society $(G_{c}, \Lambda_{c,y}).$ Moreover, let $F_{1}, \ldots, F_{4}$ denote the fences at distances $1$ up to $5$ from $B$ and $W'$ be a $(10 \times 10)$-wall $W'$ that is a subgraph of $W$ that is a wall and contains all bricks of $F_{4}.$

We argue about how to find the desired paths when $|\widetilde{c}| = 3.$ Then, the other cases follow immediately.
W.l.o.g. we assume that the rotation tie-breaker gives $c$ the counter-clockwise rotation.
Let $\Lambda_{c, y} = \lin{x, y, z}.$
First apply \autoref{prop_cell_paths_to_pegs} for $B$ and $c,$ and let $P_{x}, P_{y}, P_{z}$ be the implied paths where $x', y', z'$ are their respective endpoints in $F_{1}.$
Observe that $F_{1}$ has $12$ external pegs which can be partitioned in four sets $T, L, B, R,$ where $T$ contains the $4$ top external pegs, $L$ contains the $2$ leftmost external pegs, $B$ contains the $4$ bottom pegs, and $R$ contains the $2$ rightmost pegs.

We begin visiting the external pegs of $F$ in counter-clockwise ordering starting from $z'.$
Respecting this ordering the next vertex to appear is $x'$ and then $y'.$
We proceed to distinguish cases depending on where the vertices $x',$ $y',$ and $z'$ are on $F_{1},$ and define three paths $P_{x'}, P_{y'}, P_{z'}$ as follows.

\medskip
\noindent
\textbf{Case distinction for $z'.$}

\smallskip
$z' \in T.$ 
We define the path $P_{z'}$ starting from $z',$ moving upwards towards $F_{4}$ following the vertical path of $W'$ that contains $z',$ following $F_{4}$ in clockwise direction until we reach a vertex of $F_{4}$ that belongs to the rightmost possible vertical path of $W',$ and then following this vertical path downwards to its bottom boundary vertex of $W'.$ 

\medskip
\noindent
\textbf{Case distinction for $x'$ and $y',$ when $z' \in T.$}

\smallskip
$x' \in T$ (resp. $y' \in T$). 
We define the path $P_{x'}$ (resp. $P_{y'}$) starting from $x'$ (resp. $y'$), moving upwards towards $F_{3}$ (resp. $F_{2}$) following the vertical path of $W'$ that contains $x'$ (resp. $y'$), following $F_{3}$ (resp. $F_{2}$) in counter-clockwise direction until we reach a vertex of $F_{3}$ (resp. $F_{2}$) that belongs to the leftmost possible vertical path of $W',$ and then following this vertical path downwards to its bottom boundary vertex of $W'.$

\smallskip
$x' \in L$ (resp. $y' \in L$).
We define the path $P_{x'}$ (resp. $P_{y'}$) starting from $x'$ (resp. $y'$), moving to the left towards $F_{3}$ (resp. $F_{2}$) following the horizontal path of $W'$ that contains $x'$ (resp. $y'$), and then following the vertical path that contains the vertex of $F_{3}$ (resp. $F_{2}$) we encountered downwards to its bottom boundary vertex of $W'.$

\smallskip
$x' \in B$ (resp. $y' \in B$).
We define the path $P_{x'}$ (resp. $P_{y'}$) starting from $x'$ (resp. $y'$), moving downwards towards $F_{3}$ (resp. $F_{2}$) following the vertical path of $W'$ that contains $x'$ (resp. $y'$), and then following the vertical path that contains the vertex of $F_{3}$ (resp. $F_{2}$) we encountered downwards to its bottom boundary vertex of $W'.$

\smallskip
$x' \in R$ (resp. $y' \in R$).
We define the path $P_{x'}$ (resp. $P_{y'}$) starting from $x'$ (resp. $y'$), moving to the right towards $F_{2}$ (resp. $F_{3}$) following the horizontal path of $W'$ that contains $x'$ (resp. $y'$), and then following the vertical path that contains the vertex of $F_{2}$ (resp. $F_{3}$) we encountered downwards to its bottom boundary vertex of $W'.$

\medskip
We skip the remaining cases in favor of brevity and note that they are similar to the first.
Now, observe that, by the specifications of \autoref{prop_cell_paths_to_pegs} and by definition, the paths $P_{x} \cup P_{x'},$ $P_{y} \cup P_{y'},$ and $P_{z} \cup P_{z'}$ are indeed vertex-disjoint with one endpoint in $\widetilde{c}$ and the other being a bottom boundary vertex of $W',$ respecting the order of $\Lambda_{c, v},$ ordered left to right.
\end{proof}

\subsection{The garden creation lemma}

We require one more auxiliary lemma that shows how to further refine our garden by extracting a parterre segment that represents a linear society of the $d$-folio of our $\Sigma$-decomposition which is not yet represented by any other parterre segment.

\begin{lemma}\llabel{lem_blooming_one_step}
There exists a function $f_{\ref{lem_blooming_one_step}} \colon \Nbbb^{3} \to \Nbbb$ and an algorithm that, given $t \in \Nbbb_{≥3},$ $b, h, d \in \Nbbb_{≥1},$
\begin{itemize}
\item an $(x,y)$-fertile $d$-plowed $(f_{\ref{lem_blooming_one_step}}(t, b, h), b, h, \Sigma)$-garden $(G,\delta,W),$ and
\item a linear society $\lin{H, \Lambda} \in d\text{-}\mathsf{folio}(\delta),$
\end{itemize}
outputs an $(x,y)$-fertile $d$-plowed $(t, b, h, \Sigma)$-garden $(G,\delta',W')$ such that
\begin{itemize}
\item the set of linear societies of the parterre segments of $W'$ contains $\lin{H, \Lambda},$
\item is a strict superset of the set of linear societies of the parterre segments of $W,$ and
\item the $d$-folio of $\delta'$ is the same as the $d$-folio of $\delta$
\end{itemize}
in time
$$\Ocal(h (t + b) |G|^{2}).$$
Moreover
$$f_{\ref{lem_blooming_one_step}}(t, b, h) \in \Ocal(h (t + b)).$$
\end{lemma}
\begin{proof} We define $f_{\ref{lem_blooming_one_step}}(t, b, h) \coloneqq t + 6 + h \cdot (2t + 10 \cdot 3b).$

Assume that the Euler-genus of $\Sigma$ is $2\mathsf{h} + \mathsf{c}.$
Let $W$ be the cylindrical contatenation of a sequence $\Wcal = (W_{1}, W^{\mathsf{h}}_{1}, \ldots, W^{\mathsf{h}}_{\mathsf{h}}, W^{\mathsf{c}}_{1}, \ldots, W^{\mathsf{c}}_{\mathsf{c}}, W^{\mathsf{p}}_{1}, \ldots, W^{\mathsf{p}}_{\ell}),$ where $W_{1}$ is a $(f_{\ref{lem_blooming_one_step}}(t, b, h), f_{\ref{lem_blooming_one_step}}(t, b, h))$-wall segment, $W^{\mathsf{h}}_{1}, \ldots, W^{\mathsf{h}}_{\mathsf{h}}$ are $f_{\ref{lem_blooming_one_step}}(t, b, h)$-handle segments, $W^{\mathsf{c}}_{1}, \ldots, W^{\mathsf{c}}_{\mathsf{c}}$ are $f_{\ref{lem_blooming_one_step}}(t, b, h)$-crosscap segments, and $W^{\mathsf{p}}_{1}, \ldots, W^{\mathsf{p}}_{\ell}$ are $(f_{\ref{lem_blooming_one_step}}(t, b, h), h)$-parterre segments.

Let $\lin{H, \Lambda}$ be the given linear society in $d\text{-}\mathsf{folio}(\delta).$
By definition $\lin{H, \Lambda} = \dissolve(\mu),$ where $\mu$ is a fiber of a $c$-society $\lin{G_{c}, \Lambda_{c, x}}$ of some cell $c \in C_{\mathsf{f}}(\delta)$ of detail at most $d,$ where $\Lambda$ is a shift of $\Lambda_{c, x}.$
First, notice that $c$ belongs in the $\delta$-influence of at least one of the enclosures of $W.$
Moreover, since $(G, \delta, W)$ is $d$-plowed, i.e., all $2\mathsf{h} + \mathsf{c} + 1$ enclosures of $W$ are $d$-homogeneous, the $\delta$-influence of every brick of an enclosure whose $\delta$-influence contains $c,$ contains a cell $c' \in C_{\mathsf{f}}(\delta)$ such that a $c'$-society of $\delta$ has the same $d$-folio as $\lin{G_{c}, \Lambda_{c, x}}.$

The proof proceeds by demonstrating how to extract a $(t, b, h, H, \Lambda)$-parterre segment depending on which enclosure's $\delta$-influence contains $c$ and then showing how to obtain a new walloid $W'$ and a new $\Sigma$-decomposition $\delta'$ from $W$ and $\delta$ respectively, that also contains this new parterre segment, in a way that maintains all the desired properties.

First, let us observe the following.
Let $t' \coloneqq t + 6 + h \cdot (2t + 10b),$ and $F_{1} \coloneqq C^{\mathsf{ex}}.$
For every fence $F_{r},$ $r  \in [t'],$ at distance $r$ from $C^{\mathsf{ex}},$ we can naturally define a $(t' - r, \Sigma)$-walloid $W_{r}$ with the same number of handle and crosscap segments and with the same number of parterre segments whose linear societies are the same, where only their order changes.
Indeed, let $W_{r}$ be the graph obtained by the union of the cycles of the base cylindrical wall of $W$ that $F_{1}, \ldots, F_{r}$ do not intersect, the paths of the rainbows all handles, crosscap, and parterre segments that $F_{1}, \ldots, F_{r}$ do not intersect, the vertical paths of $W$ that $F_{1}, \ldots, F_{r}$ do not intersect, all linear societies of the parterre segments, and the maximal subpaths of all vertical paths of the walloid's wall segment whose endpoints lie in cycles that $F_{1}, \ldots, F_{r}$ do not intersect.

Observe that $F_{1}, \ldots, F_{r}$ intersect the cycles $C_{1}, \ldots, C_{r}$ of the base cylindrical wall of $W,$ the first $r$ paths of both rainbows of handle segments, the first $r$ and the last $r$ paths of rainbows of crosscap segments, the first $r$ paths of rainbows of parterre segments, as well as all vertical paths that are not disjoint from the former rainbow paths.
Also, note that the linear society of each parterre segment is unaffected as we simply extend the boundary with paths towards the new top boundary vertices of the new base cylindrical wall of $W'.$
Finally, note that the number of vertical paths in the wall segment of $W_{r}$ is the same as in $W.$

\begin{figure}[ht]
  \begin{center}
  \scalebox{0.991}{\includegraphics{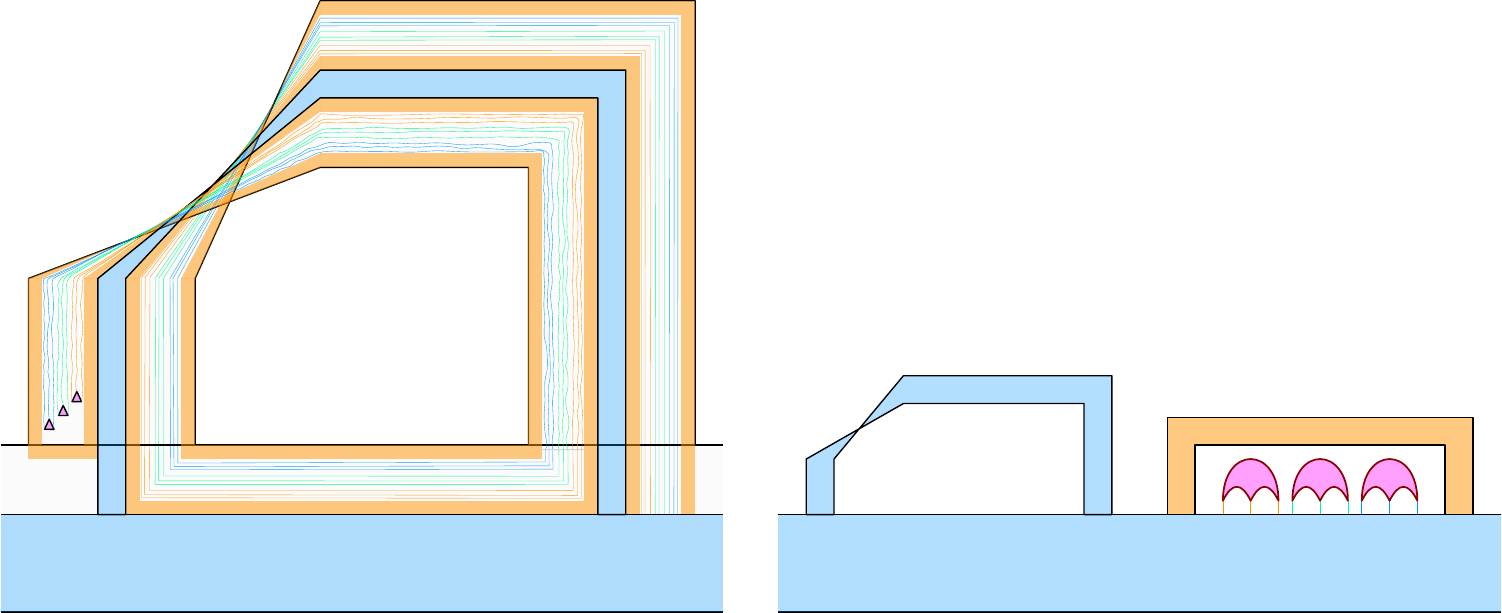}}
  \end{center}
    \caption{Extracting a flower from a crosscap segment.}
  \llabel{fig_garden_crosscap_reverse}
\end{figure}

\paragraph{Extracting from a crosscap enclosure.} 
Let $p \coloneqq 2t + 10b$ and $q \coloneqq hp.$
Consider the fences $\Fcal \coloneqq \{ F_{1} \ldots F_{hp + 6} \}$ at distance $0$ up to $hp + 5$ from $F_{1}.$
Let $W^{\mathsf{c}} \coloneqq W^{\mathsf{c}}_{i},$ for a choice of $i$ such that $c$ is in the $\delta$-influence of the enclosure of $W^{\mathsf{c}}_{i}.$
Let $P_{1}, \ldots, P_{t'}$ be the horizontal and $Q_{1}, \ldots, Q_{4t'}$ be the vertical paths of the base wall of $W^{\mathsf{c}}$ in top to bottom and left to right order respectively.
Moreover, let $\Rcal = \{ R_{1}, \ldots, R_{2t'} \}$ be the rainbow of $W^{\mathsf{c}}$ in left to right order.
Observe that $P_{i},$ $i \in [hp + 5]$ is a subpath of $F_{i},$ and that $R_{i}$ and $R_{2t' - i + 1}$ are also subpaths of $F_{i}.$

Let $\Bcal_{1}, \ldots, \Bcal_{h}$ be $h$ many sets each containing a sequence of $b$ many bricks, say $\Bcal_{i} = (B^{i}_{1}, \ldots, B^{i}_{b})$ of $W$ ordered from the leftmost to the rightmost (respecting the ordering of $\Rcal$) obtained as follows.
For every $\Bcal_{i},$ $i \in [h],$ we reserve a bundle of $p$ many consecutive cycles of $\Fcal$ namely $\Fcal^{i}  \coloneqq \{ F_{6 + (i - 1) \cdot p + 1}, \ldots, F_{6 + i \cdot p} \}.$
We choose $B^{i}_{j},$ $j \in [b]$ as the brick of the enclosure of $W^{\mathsf{c}}$ such that $B^{i}_{j} \cap R_{6 + (i - 1) \cdot p + 10 \cdot (j - 1) + 5} \neq \emptyset$ and $B^{i}_{j} \cap R_{6 + (i - 1) \cdot p + 10 \cdot (j - 1) + 6} \neq \emptyset$ (see \autoref{fig_garden_crosscap_reverse}).

Observe that by definition for every brick $B^{i}_{j},$ $i \in [h],$ $j \in [b],$ there is a $(10 \times 10)$-wall $W^{i}_{j}$ that contains all bricks of the fence at distance $4$ from $B^{i}_{j}$ and such that it is disjoint from any $W^{i'}_{j'},$ $i \neq i' \in [h]$ and $j \neq j' \in [b],$ and it is disjoint from the first $t$ and the last $t$ cycles in $\Fcal^{i}.$
Moreover exactly $2$ subpaths of each of the horizontal paths $P_{1}, \ldots, P_{5}$ correspond to horizontal paths of $W^{i}_{j}.$ We assume that the bottom horizontal path of $W^{i}_{j}$ is the rightmost such subpath of $P_{5}.$

Now for every $i \in [h],$ $j \in [b],$ we are in the position to apply \autoref{lem_paths_boundary} to the brick $B^{i}_{j},$ for a $c^{i}_{j}$-society of a cell $c^{i}_{j} \in C_{\mathsf{f}}(\delta)$ whose $d$-torso is the same as that of $\lin{G_{c}, \Lambda_{c, x}},$ and obtain as a result a set $\Pcal^{i}_{j}$ of $|\widetilde{c}^{i}_{j}|$ many vertex-disjoint paths from $\widetilde{c}^{i}_{j}$ to the bottom boundary vertices of $W^{i}_{j},$ such that the endpoints of these paths not in $\widetilde{c}^{i}_{j}$ respect the linear ordering of the chosen $c^{i}_{j}$-society.
Let $\Lambda^{i}_{j}$ be a left-to-right linear ordering of these endpoints.
Observe that , by definition, the linear society $\lin{G_{c^{i}_{j}} \cup \bigcup \Pcal^{i}_{j}, \Lambda^{i}_{j}}$ has a fiber $\mu_{2}$ of detail $\leq d$ such that $\dissolve(\mu_{2}) = \lin{H, \Lambda}.$

Now for every $i \in [h],$ we define a linkage $\Lcal_{i}$ of order $t$ that will serve as the rainbow of the $i$-th flower of the parterre segment we are trying to define.
Let $\Fcal^{i}_{2}$ be the set of $2t$ subpaths of the cycles in $\Fcal^{i}$ whose one endpoint is the leftmost top boundary vertex of the base wall of $W^{\mathsf{c}}$ that is visited by the respective cycle, and whose other endpoint is the rightmost top boundary vertex of the base wall of the rightmost segment of $W$ that is not its wall segment.
Notice that the paths in $\Fcal^{i}_{2}$ are made up of $2$ bundles of $t$ consecutive paths.
Now we define $\Lcal_{i}$ by connecting the leftmost endpoint of each path of one bundle to the leftmost endpoint of a path in the other bundle, using the $2t$ vertical paths that intersect the paths of the two bundles and the $t$ horizontal paths $P_{6}, \ldots, P_{t + 5}$ of the base wall $W^{\mathsf{c}}.$
Notice that there is a unique pairing of paths between these two bundles that can give us a linkage of order $t.$

Next, notice that for every $i \in [h],$ $j \in [b],$ each endpoint of a path in $\Pcal^{i}_{j}$ not in $\widetilde{c}^{i}_{j}$ is a vertex of a distinct cycle in $\Fcal^{i}.$
Then, we define a new set of paths $\Tcal^{i}_{j}$ as an extension of each path $P \in \Pcal^{i}_{j}$ by considering the union of $P$ with the subpath of that cycle that connects the endpoint of $P$ not in $\widetilde{c}^{i}_{j}$ to the rightmost top boundary vertex visited by this cycle, of the base wall of the rightmost segment of $W$ that is not its wall segment.
Notice that the paths in $\bigcup_{i \in [h], j \in [b]} \{ \Tcal^{i}_{j} \} \cup \bigcup_{i \in [h]} \{ \Lcal^{i} \}$ are pairwise vertex-disjoint by construction.

Now, observe that $t' - |\Fcal| = t.$
Now, to define our desired $(t, \Sigma)$-walloid $W',$ we consider the $(t, \Sigma)$-walloid $W_{|\Fcal|}$ union the linkage $\Lcal_{i},$ the graphs $G_{c^{i}_{j}},$ the linkages $\Tcal^{i}_{j},$ $j \in [b],$ and all vertical paths of $W$ that are intersected by the linkages $\Lcal_{i}$ and $\Tcal^{i}_{j}.$
Moreover, observe that throughout this case the linkages we define do not use any horizontal or vertical path of the wall segment of $W.$
Hence we moreover have to remove vertices not in horizontal paths from $|\Fcal|$ many arbitrarily chosen vertical paths of the wall segment in $W_{|\Fcal|}$ to obtain the correct order in $W'.$

\begin{figure}[ht]
  \begin{center}
  \scalebox{1.04}{\includegraphics{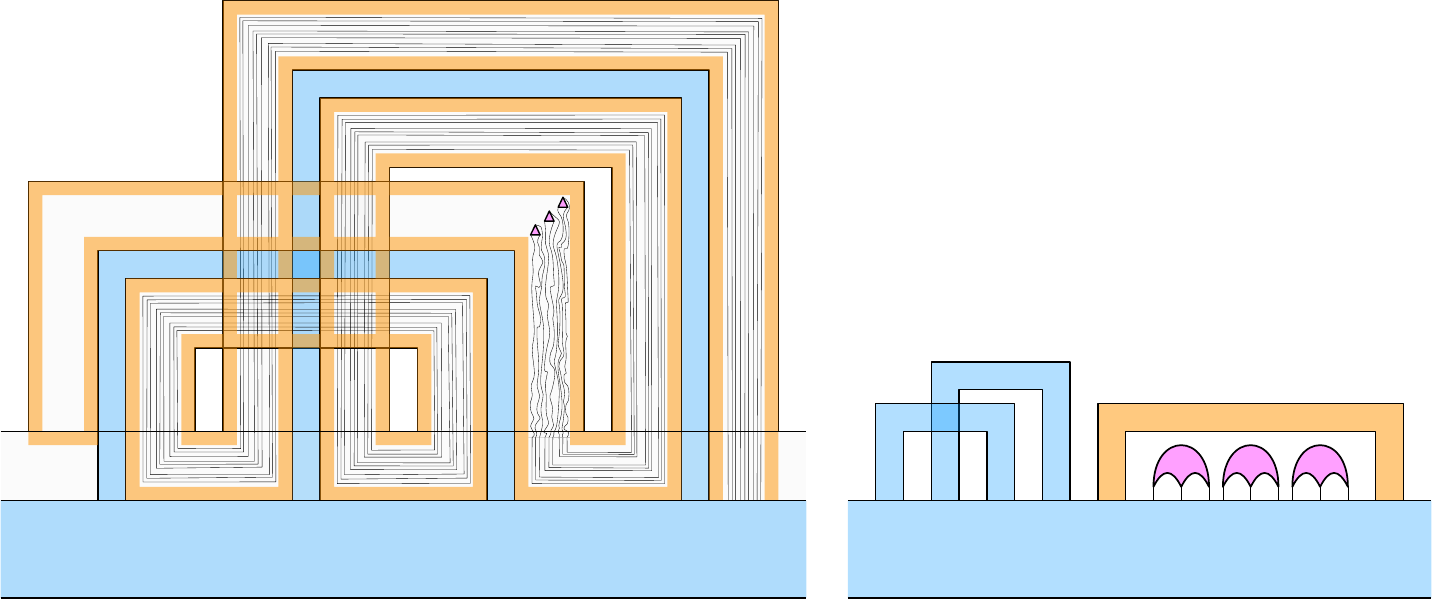}}
  \end{center}
    \caption{Extracting a flower segment from a handle segment.}
  \llabel{fig_garden_handle}
\end{figure}

\paragraph{Extracting from a handle enclosure.} In this case, we have to argue how we can extract our desired $(t, b, h, H, \Lambda)$-parterre segment either from the leftmost enclosure or the rightmost enclosure of a handle segment.
We omit the proof of this case as it is very similar to how we extract from a crosscap enclosure. See \autoref{fig_garden_handle} for an illustration of how we can extract a single flower segment from a handle enclosure.

\paragraph{Extracting from the big enclosure.} 
Observe that the bricks of the big enclosure of $W$ are the bricks of the base cylinder $\widetilde{W}$ of $W.$
Clearly $W_{1}$ is a subgraph of $\widetilde{W}.$
 and let $P_{1}, \ldots, P_{t'}$ be the horizontal and $Q_{1}, \ldots, Q_{t'}$ be the vertical paths of $W_{1}.$

Let $\Bcal_{1}, \ldots, \Bcal_{h}$ be $h$ many sets each containing a sequence of $b$ many bricks, say $\Bcal_{i} = (B_{1}^{i}, \ldots, B_{b}^{h}),$ of $W$ ordered from the leftmost to the rightmost (respecting the ordering of the vertical paths of $W_{1}$) obtained as follows.
For every $\Bcal_{i},$ $i \in [h],$ we reserve a bundle of $p$ consecutive vertical paths of $W,$ from $Q_{(i - 1) \cdot p + 1}$ up to $Q_{i \cdot p}.$
We choose $B^{i}_{j},$ $j \in [b]$ as the brick such that $B^{i}_{j} \cap P_{t + 5} \neq \emptyset$ and $B^{i}_{j} \cap P_{t + 6} \neq \emptyset$ and moreover $B^{i}_{j} \cap Q_{(i - 1) \cdot p + 10 \cdot (j - 1) + 5} \neq \emptyset$ and $B^{i}_{j} \cap Q_{(i - 1) \cdot p + 10 \cdot (j - 1) + 6} \neq \emptyset.$

Observe that by definition for every brick $B^{i}_{j},$ $i \in [h],$ $j \in [b],$ there is a $(10 \times 10)$-wall $W^{i}_{j}$ that contains all bricks of the fence at distance $4$ from $B^{i}_{j}$ and such that, it is disjoint from any $W^{i'}_{j'},$ $i \neq i' \in [h]$ and $j \neq j' \in [b],$ it is disjoint from $P_{1}, \ldots, P_{t},$ and it is disjoint from the first $t$ and the last $t$ vertical paths of the bundle of paths we reserved for $\Bcal_{i}.$
Moreover, the bottom boundary vertices of each of the walls $W^{i}_{j},$ are top boundary vertices of the wall $W'_{1}$ obtained by the union of all horizontal paths of $W_{1}$ except $P_{1}, \ldots, P_{t + 9}$ and the maximal subpaths of all vertical paths of $W_{1}$ whose endpoints lie in $P_{t + 10}$ and $P_{t'}.$

Now for every $i \in [h],$ $j \in [b],$ we are in the position to apply \autoref{lem_paths_boundary} to the brick $B^{i}_{j},$ for a $c^{i}_{j}$-society of a cell $c^{i}_{j} \in C_{\mathsf{f}}(\delta)$ whose $d$-torso is the same as that of $\lin{G_{c}, \Lambda_{c, x}},$ and obtain as a result a set $\Pcal^{i}_{j}$ of $|\widetilde{c}^{i}_{j}|$ many vertex-disjoint paths from $\widetilde{c}^{i}_{j}$ to the bottom boundary vertices of $W^{i}_{j},$ such that the endpoints of these paths not in $\widetilde{c}^{i}_{j}$ respect the linear ordering of the chosen $c^{i}_{j}$-society.
Let $\Lambda^{i}_{j}$ be a left-to-right linear ordering of these endpoints.
Observe that, by definition, the linear society $\lin{G_{c^{i}_{j}} \cup \bigcup \Pcal^{i}_{j}, \Lambda^{i}_{j}}$ has a fiber $\mu_{2}$ of detail $\leq d$ such that $\dissolve(\mu_{2}) = \lin{H, \Lambda}.$
Also for every $i \in [h],$ we define a linkage $\Lcal_{i}$ of order $t$ connecting the $2t$ many top boundary vertices of $W'_{1}$ which are intersected by the first and last $t$ vertical paths reserved for $\Bcal_{i}$ by using the intersection between these vertical paths and the first $t$ horizontal paths of $W_{1}.$
Notice that we can define each $\Lcal_{i}$ to be disjoint from any $\Lcal_{j},$ $j \neq i \in [h].$

Now, to define our new $(t, \Sigma)$-walloid $W',$ we consider the $(t, \Sigma)$-walloid $W_{|\Fcal|}$ union the linkage $\Lcal_{i},$ the graphs $G_{c^{i}_{j}},$ the linkages $\Pcal^{i}_{j},$ $j \in [b],$ and all subpaths of vertical paths $W$ that join the linkages $\Lcal_{i}$ and $\Tcal^{i}_{j}$ to top boundary vertices of the base cylindrical wall of $W_{|\Fcal|}.$
Moreover, observe that in this case, the linkages we define use exactly the first $q$ many vertical paths of the wall segment of $W.$
Hence we moreover have to remove vertices not in horizontal paths from $6$ (distinct from the first $q$) chosen vertical paths of the wall segment in $W_{|\Fcal|}$ to fix its order in $W'$ (see \autoref{fig_garden_base} for an illustration on how to find one flower segment).
Note that the previous arguments give us a new $(t, b, h, H, \Lambda)$-parterre segment in $W'$ with the linear ordering of $\Lambda$ intact.

\medskip
It remains to define the desired $\Sigma$-decomposition $\delta' = (\Gamma', \Dcal').$
We obtain $\delta'$ from $\delta$ by considering the same drawing of $G$ in $\Sigma,$ $\Gamma' = \Gamma,$ and defining the cells $\Dcal'$ as follows.
Let $\Delta$ be the $\trace(C^{\mathsf{si}})$-avoiding disk of $C^{\mathsf{ex}},$ where $C^{\mathsf{ex}}$ is the exceptional cycle of $W'$ while $C^{\mathsf{si}}$ is the simple cycle of $W'.$
Define $\Dcal'$ to contain all cells of $\delta$ not drawn in $\Delta,$ as well as a new cell that is the closure of all cells of $\delta$ drawn in $\Delta.$
Observe that the big enclosure of $W'$ is a fence of $W$ contained in the big enclosure of $W$ and as such remains $d$-homogeneous.
Moreover the remaining $2\mathsf{h} + \mathsf{c}$ enclosures of $W'$ that correspond to the enclosures of its handle and crosscap segments can be seen to also be $d$-homogeneous as they contain all bricks of their corresponding handle and crosscap enclosure in $W$ as well as bricks that were, in $W,$ contained in its big enclosure, and all such bricks had the same $d$-folio.
Additionally, since the new unique vortex cell of $\delta$ only grew by the addition of flap cells of $\delta,$ implies that $(G, \delta', W')$ remains $(x, y)$-fertile, and the previous fact combined with the fact that all enclosures of $W$ are $d$-homogeneous, also implies that the $d$-folio of $\delta'$ is the same as the $d$-folio of $\delta.$
Finally, by construction, it is easy to see that the set of linear societies of the parterre segments of $W'$ contains $\lin{H, \Lambda}$ and is a superset of the set of linear societies of the parterre segments of $W.$

Note that in the way we extract flowers with the previous argument each copy of a linear society within the flowers contains one copy for each linear society in $\mathsf{dec}(H, \Lambda).$
If we replace $b$ with $3b$ we can deal with this issue by keeping from each linear society only the copy of the component of the decomposition we want in the correct order.
\end{proof}

\begin{figure}[ht]
  \begin{center}
  \scalebox{1.04}{\includegraphics{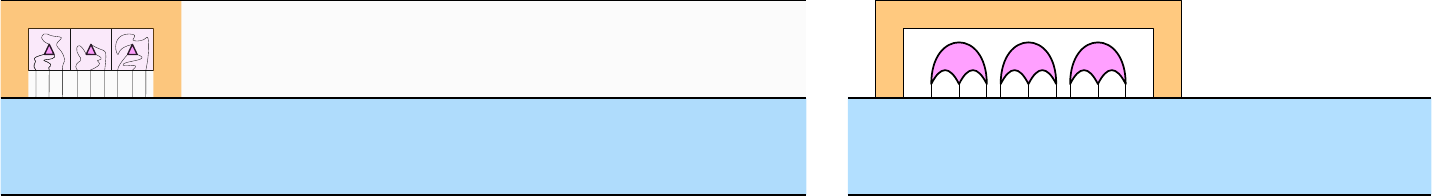}}
  \end{center}
    \caption{Extracting a flower segment from a wall segment.}
  \llabel{fig_garden_base}
\end{figure}

We are now in the position to prove the main result of this section, which is an iterative application of \autoref{lem_blooming_one_step}.
Our goal is to represent each linear society of the $d$-folio of our $\Sigma$-decomposition by a parterre segment containing many enough copies of the linear society it represents.

\begin{lemma}[Garden creation]\llabel{lem_garden_creation}
There exists a function $f_{\ref{lem_garden_creation}} \colon \Nbbb^4 \to \Nbbb$ and an algorithm that, given $t \in \Nbbb_{\geq 3},$ $b, h, d \in \Nbbb_{≥1},$ and an $(x,y)$-fertile $(f_{\ref{lem_garden_creation}}(t, b, h, d), \Sigma)$-meadow $(G,\delta,W)$ that is $d$-plowed, outputs an $(x, y)$-fertile $(t, b, h, \Sigma)$-garden that is $d$-blooming in time
$$h(t + b) 2^{\Ocal(d^{2})} \cdot |G|^{2}.$$
Moreover
$$f_{\ref{lem_garden_creation}}(t, b, h, d) \in h(t + b) 2^{\Ocal(d^{2})}.$$
\end{lemma}
\begin{proof} We define $f_{\ref{lem_garden_creation}}(t, b, h, d) \coloneqq f_{\ref{lem_blooming_one_step}}(t, b, h) \cdot f_{\ref{obs_size_dfolio}}(d).$

The proof of the lemma is an iterative application of \autoref{lem_blooming_one_step} applied each time to a linear society of the $d$-folio of the corresponding $\Sigma$-decomposition that is not represented by the parterre segments of the corresponding walloid.
First observe that by definition $(G, \delta, W)$ can be seen as an $(x, y)$-fertile $d$-plowed $(f_{\ref{lem_garden_creation}}(t, b, h, d), b, h, \Sigma)$-garden with no parterre segments.

\medskip
Let $(G, \delta_{0}, W_{0}) = (G, \delta, W)$ and $i = 0.$
Iteratively apply \autoref{lem_blooming_one_step} on $(G, \delta_{i}, W_{i})$ with a linear society $\lin{H, \Lambda} \in d\text{-}\mathsf{folio}(\delta_{i})$ that is not the linear society of any parterre segment of $W_{i}.$
Let $(G, \delta_{i+1}, W_{i+1})$ be the resulting garden. Set $i \coloneqq i + 1,$ and repeat.

\medskip
Observe that the above procedure will repeat at most $d' \coloneqq |d\text{-}\mathsf{folio}(\delta)|$ times, since the $d$-folio of any of the intermediate $\Sigma$-decompositions is the same.
Then, we conclude with the $(G, \delta_{d'}, W_{d'})$ $(t, b, h, \Sigma)$-garden, which by the specifications of \autoref{lem_blooming_one_step} has all the desired properties.
\end{proof}

\section{Growing an orchard}
\llabel{ripe_obt_orch}

The goal of this subsection is to prove \autoref{ripe_orchard}.
That is, we show that any $d$-blooming garden with big enough infrastructure can be transformed into an orchard that is both $d$-blooming and ripe.
This step marks the last refinement of our $\Sigma$-decomposition before the harvest can begin.

\begin{observation}[Garden seen as a single thicket orchard]\llabel{obs_single_thicket_orchard}
There is an algorithm that, given an $s\in\mathbb{N}$ and a $(x,y)$-fertile and $d$-blooming
$(t+1+s,b,h,d,\Sigma)$-garden $(G,\delta,W),$ outputs an $(x,y)$-fertile and $d$-blooming single-thicket $(t,b,h,s,d,\Sigma)$-orchard $(G,\delta,W)$ in time $\mathcal{O}(t\cdot s\cdot(b+h)|G|).$
\end{observation}

\begin{lemma}[Bounding the depth of a orchard]\llabel{ripe_orchard}
There exist functions $f^1_{\ref{ripe_orchard}},f^2_{\ref{ripe_orchard}},f^3_{\ref{ripe_orchard}},f^4_{\ref{ripe_orchard}}\colon\mathbb{N}^7\to\mathbb{N}$ and an algorithm that, given an $(x,y)$-fertile and $d$-blooming single-thicket $(t',b,h,s',\Sigma)$-orchard $(G,\delta,W)$ where 
\begin{align*}
    t'&\geq f^1_{\ref{ripe_orchard}}(t^*,b,h,s^*,x,y,d)\text{ and}\\
    s'&\geq f^2_{\ref{ripe_orchard}}(t^*,b,h,s^*,x,y,d),
\end{align*}
outputs a $(q,p)$-ripe and $d$-blooming $(t^*,b,h,s^*,\Sigma)$-orchard in time $2^{{t^*}^{\mathcal{O}_h(1)}}|G|^3$ where
\begin{align*}
    q&\leq f^3_{\ref{ripe_orchard}}(t^*,b,h,s^*,x,y,d)\text{ and}\\
    p&\leq f^4_{\ref{ripe_orchard}}(t^*,b,h,s^*,x,y,d).
\end{align*}
\end{lemma}

The proof of \autoref{ripe_orchard} will be an inductive one via a sequence of lemmas which we derive in the following subsections.

\subsection{Developing the tools}\llabel{transaction_prelims}

Before we begin, we need some preliminary results on transactions in societies.

\paragraph{Strip societies.}
Since our goal is to refine an $(x,y)$-fertile blooming garden, or better an $(x,y)$-fertile orchard with a single thicket, it suffices to deal with societies for which we already know that there exists a rendition in a disk with at most $x$ vortices, each of depth at most $y.$
This immediately allows us to deduce that any large enough transaction in such a society contains a still large subtransaction which is planar and satisfies a number of additional properties.
\smallskip

Let $G$ be a graph and $H$ be a subgraph of $G.$
An \defi{$H$-bridge} is either an edge $e$ with both endpoints in $H$ such that $e\notin E(H),$ or a subgraph of $G$ formed by a component $K$ of $G-V(H)$ along with all edges of $G$ with one endpoint in $V(K)$ and the other endpoint in $V(H).$
In the first case, we call the endpoints of $e$ the \defi{attachments} of the bridge, and in the second case the \defi{attachments} of the bridge are those vertices that do not belong to $K.$
\smallskip

Let $\langle G,\Lambda\rangle$ be a society and $\mathcal{P}$ be a planar transaction of order at least two in $\langle G,\Lambda\rangle.$
The \defi{$\mathcal{P}$-strip society of $\langle G,\Lambda\rangle$} is defined as follows.
Let $\mathcal{P}=\{ P_1,P_2,\dots,P_{\ell}\}$ be ordered\footnote{We assume here that neither the segment $a_1\Lambda a_{\ell}$ nor the segment $b_{\ell}\Lambda b_1$ contains both endpoints of $\Lambda.$ If this is not the case we may slightly rotate $\Lambda$ to obtain a new linear ordering $\Lambda'$ neither segment contains both endpoints and use $\Lambda'$ to derive the $\mathcal{P}$-strip society of $\langle G,\Lambda\rangle.$} such that for each $i\in[\ell],$ $P_i$ has the endpoints $a_i$ and $b_i$ and $a_1,a_2,\dots,a_{\ell},b_{\ell},b_{\ell-1},\dots,b_1$ appear in $\Lambda$ in the order listed.
We denote by $H$ the subgraph of $G$ defined by the union of the paths in $\mathcal{P}$ together with all vertices of $\Lambda$ and by $H'$ the subgraph of $H$ consisting of $\mathcal{P}$ together with the vertices in $V(a_1\Lambda a_{\ell})\cup V(b_{\ell}\Lambda b_1).$
We consider the set $\mathcal{B}$ of all $H$-bridges of $G$ with at least one vertex in $V(H')\setminus V(P_1\cup P_{\ell}).$
For each $B\in\mathcal{B}$ let $B'$ be obtained from $B$ by deleting all attachments that do not belong to $V(H').$
Finally, let us denote by $G_1$ the graph defined as the union of $H'$ and all $B'$ where $B\in\mathcal{B}$ and let $\Lambda_1\coloneqq a_1\Lambda a_{\ell}\oplus b_{\ell}\Lambda b_1.$
The resulting society $\langle G_1,\Lambda_1\rangle$ is the desired $\mathcal{P}$-strip society of $\langle G,\Lambda\rangle.$
The paths $P_1$ and $P_{\ell}$ are called the \defi{boundary paths} of $\langle G_1,\Lambda_1\rangle.$

We say that the $\mathcal{P}$-strip society $\langle G_1,\Lambda_1\rangle$ of $\langle G,\Lambda\rangle$ is \defi{isolated} if no edge of $G$ has one endpoint in $V(G_1)\setminus V(P_1\cup P_{\ell})$ and the other endpoint in $V(G)\setminus V(G_1).$

Moreover, an isolated $\mathcal{P}$-strip society is separating if after deleting any non-boundary path of $\mathcal{P}$ there does not exist a path from one of the two resulting segments of $\Lambda$ to the other.

\begin{lemma}\llabel{isolatedstrips}
Let $t,x\geq 1,$ $p, s, y\geq 2,$ and $b,h$ be positive integers.
Moreover, let $(G,\delta,W)$ be an $(x,y)$-fertile $(t,b,h,s,\Sigma)$-orchard and let $\langle H,\Lambda\rangle$ be a thicket society of $(G,\delta,W).$
Then every transaction in $\langle H,\Lambda\rangle$ of order at least $(x+1)(2xy+p)$ contains a planar transaction $\mathcal{P}$ of order $p$ such that the $\mathcal{P}$-strip society of $\langle H,\Lambda\rangle$ is isolated, separating, and rural. 
\end{lemma}

\begin{proof}
Since $(G,\delta,W)$ is $(x,y)$-fertile there exists a rendition $\rho$ of $\langle H,\Lambda\rangle$ in the disk with at most $x$ vortices, each of depth at most $y.$

Let $\mathcal{L}$ be a transaction of order at least $(x+1)(2y+p)$ in $\langle H,\Lambda\rangle.$
Then there exist two disjoint segments, say $I$ and $J$ such that every path in $\mathcal{L}$ has one endpoint in $I$ and the other in $J.$
Let us assume that $I$ and $J$ are chosen minimally with this property and let $\mathcal{L}=\{ L_1,L_2,\dots,L_{\ell}\}$ be ordered with respect to the occurrence of the endpoint $u_i$ of $L_i$ in $I.$

Suppose there exists a transaction $\mathcal{L}'\subseteq\mathcal{L}$ such that there exists a path $P$ in $\mathcal{L}'$ which crosses all other paths in $\mathcal{L}'$
Suppose $|\mathcal{L}'|>xy$ then there must exist a transaction $\mathcal{L}''\subseteq \mathcal{L}'$ of order at least $y+1$ and a vortex $c$ of $\rho$ such that $\mathcal{L}''$ induces a transaction of order at least $y+1$ in the vortex society of $c.$
Since $c$ is of depth at most $y$ this is impossible and thus $|\mathcal{L}'|\leq xy.$

We partition $I$ into $x+1$ consecutive segments $I_1,I_2,\dots,I_{x+1}$ each containing at least $2xy+p$ endpoints of members of $\mathcal{L}.$
For each $i\in[x+1]$ there exists a family $\mathcal{L}_i$ containing $p$ consecutive paths such that $I_i$ has at least $xy$ other endpoints on either side.
Notice that, by the observation above, it is not possible that a path from $\mathcal{L}_i$ crosses a path from $\mathcal{L}$ which does not have an endpoint in $I_i.$
In particular, this means that for any choice of $i\neq j\in[x+1]$ there does not exist a vortex of $\rho$ which contains an edge of a path in $\mathcal{L}_i$ and an edge of a path of $\mathcal{L}_j.$
Therefore, there must exist $i\in[x+1]$ such that no path in $\mathcal{L}_i$ has an edge or a vertex that belongs to a vortex of $\rho.$
It follows that $\mathcal{L}_i$ is isolated, separating, and rural as desired.
\end{proof}

\paragraph{Exposed transactions.}

To ensure we are actually making progress with our constructions we require our transactions to run through the vortex region of our thicket.

\medskip

Let $x,y,t,b,h,$ and $s$ be positive integers and $(G,\delta,W)$ be an $(x,y)$-fertile $(t,b,h,s,\Sigma)$-orchard.
Moreover, let $\lin{H, \Lambda}$ be a thicket society of a $(t, s)$-thicket segment of $(G, \delta, W)$ with nest $\mathcal{C} = \{ C_{0}, \ldots, C_{s+1} \}$ and cylindrical rendition $\rho \coloneqq \delta \cap \Delta,$ where $\Delta$ is the $\trace_{\delta}(C^{\mathsf{si}})$-avoiding disk of $\trace_{\delta}(C_{s + 1}).$
A transaction $\mathcal{P}$ in $\lin{H, \Lambda}$ is said to be \defi{exposed} if the $\trace{\rho}(C_{s+1})$-avoiding disk of $\trace_{\rho}(C_{0})$ intersects every path in $\mathcal{P}.$

Please note that our definition of exposed transactions differs slightly from the definition on \cite{kawarabayashi2020quickly}.
Indeed, we do not require the presence of a vortex at all for a transaction to be exposed, it suffices that the transaction crosses through the disk defined by the innermost cycle of the nest.

\paragraph{Pruning of an orchard.}

In \cite{kawarabayashi2020quickly} it was ensured to always find an exposed transaction because Kawarabayashi et al. only considered crooked transactions which are forced to traverse the vortex in their setting.
For us, however, since our transactions are not necessarily crooked and we do not require our cylindrical renditions to be ``around'' a vortex, this is not the case.
We get around this by using a trick from \cite{thilikos2022killing} that allows us to either find, in every large enough transaction, either a big exposed transaction or a refinement of our orchard that reduced the part of $G$ that belongs to our thicket society $\langle H,\Lambda\rangle.$

Let $x,y,t,b,h,$ and $s$ be positive integers and $(G,\delta,W)$ be an $(x,y)$-fertile $(t,b,h,s,\Sigma)$-orchard.
An $(x,y)$-fertile $(t,b,h,s,\Sigma)$-orchard $(G,\delta,W')$ is said to be a \defi{pruning} of $(G,\delta,W)$ if
\begin{itemize}
    \item all handle, crosscap, and parterre segments of $W$ equal the handle, crosscap, and parterre segments of $W',$
    \item if $W$ has $\ell$ thickets, then there exists $j\in[\ell]$ such that $W'$ also has $\ell$ thickets and all thickets of $W,$ expect for the $j$th thicket, equal the thickets of $W'$ except for the $j$th thicket, and
    \item if $\langle H,\Lambda\rangle$ is a thicket society of the $j$-th thicket of $W$ with nest $\mathcal{C} = \{ C_0, C_1,\dots, C_{s+1}\}$ and cylindrical rendition $\rho \coloneqq \delta \cap \Delta,$ where $\Delta$ is the $\trace_{\delta}(C^{\mathsf{si}})$-avoiding disk of $\trace_{\delta}(C_{s + 1}),$ and $\langle H', \Lambda' \rangle$ is a thicket society of the $j$-th thicket of $W'$ with nest $\mathcal{C} = \{ C'_0, C'_1, \dots, C'_{s+1} \}$ and cylindrical rendition $\rho' \coloneqq \delta \cap \Delta',$ where $\Delta'$ is the $\trace_{\delta}(C^{\mathsf{si}})$-avoiding disk of $\trace_{\delta}(C'_{s + 1}),$ then
    \begin{itemize}
    \item either $\Delta'$ is properly contained in $\Delta$ or
    \item there exists $i \in [s]$ such that
        \begin{itemize}    
        \item if $\Delta_i$ is the $\mathsf{trace}_{\rho}(C_{s+1})$-avoiding disk of $\mathsf{trace}_{\rho}(C_i),$ $\langle H_i,\Lambda_i\rangle$ is a $\Delta_i$-society in $\rho,$ and
        \item $\Delta_i'$ is the $\mathsf{trace}_{\rho'}(C'_{s+1})$-avoiding disk of $\mathsf{trace}_{\rho'}(C'_i),$ $\langle H'_i,\Lambda'_i\rangle$ is a $\Delta'_i$-society in $\rho',$
        \end{itemize}
        then either $H'_i-V(C_i')\subsetneq H_i-V(C_i)$ or $E(H'_i)\subsetneq E(H_i).$
    \end{itemize}
\end{itemize}

Hence, the pruning of an orchard is essentially achieved by ``pushing'' one of the cycles of the nest of one of the thickets either a bit closer to the untamed area in the center of the thicket or by extending the nest into the untamed area, thereby taming a part of the actual vortex inside the thicket.

\begin{lemma}\llabel{exposure}
Let $p,x,y,t,b,h,$ and $s$ be positive integers and $(G,\delta,W)$ be an $(x,y)$-fertile $(t,b,h,s,\Sigma)$-orchard.
Let $\langle H,\Lambda\rangle$ be a thicket society of $(G,\delta,W)$ and $\mathcal{P}$ be a transaction of order at least $2s+p$ in $\langle H,\Lambda\rangle.$
Then there exists
\begin{itemize}
    \item an exposed transaction $\mathcal{Q}\subseteq \mathcal{P}$ of order $p,$ or
    \item a pruning of $(G,\delta,W).$
\end{itemize}
Moreover, there exists an algorithm that finds one of the two outcomes in time $\mathcal{O}(|G|).$
\end{lemma}
\begin{proof}

Let $\{ C_{0}, \ldots, C_{s+1} \}$ be the nest paired with the $(t, s)$-thicket segment that corresponds to $\lin{H, \Lambda},$ $\rho \coloneqq \delta \cap \Delta$ be the cylindrical rendition of $\lin{H, \Lambda},$ where $\Delta$ is the $\trace_{\delta}(C^{\mathsf{si}})$-avoiding disk of $\trace_{\delta}(C_{s + 1}),$ and $\mathcal{C} = \{ C_{1}, \ldots, C_{s} \}.$
If $\mathcal{P}$ contains an exposed transaction of order $p$ we are immediately done.
Moreover, we can check in linear time for the existence of such a transaction by checking for each path in $\mathcal{P}$ individually if it is exposed.

First, observe that we may assume that every path of $\mathcal{P}$ is intersected by the $\trace_{\rho}(C_{s+1})$-avoiding disk of $\trace{\rho}(C_{s}).$
Indeed, if there is a path $P$ which violates this condition then we can define a $\delta$-grounded cycle $C'_{s + 1}$ by replacing a subpath of $C_{s+1}$ whose endpoints are the same as the endpoints of a minimal $V(C_{s+1})$-$V(C_{s+1})$-subpath of $P$ in a way that defines a cycle whose trace in $\delta$ bounds a disk $\Delta'$ which contains $\trace_{\delta}(C_{0}).$
It is then easy to see that $\Delta'$ is properly contained in $\Delta$ and by replacing our working thicket with a thicket whose thicket society corresponds to any $\Delta'$-society in $\delta$ yields a pruning of $(G, \delta, W).$

Hence, we may assume that there exists a linkage $\mathcal{Q}\subseteq \mathcal{P}$ of order $2s+1$ such that \textsl{no} path of $\mathcal{Q}$ is exposed and every path is intersected by the $\trace_{\rho}(C_{s+1})$-avoiding disk of $\trace_{\rho}(C_{s}).$
Let $\Delta^{*}$ denote the $\trace_{\rho}(C_{s+1})$-avoiding disk of $\trace_{\rho}(C_{0}).$
It follows that $\mathcal{Q}$ must be a planar transaction and no path in $\mathcal{Q}$ intersects the $\Delta^{*}.$
Notice that each member $Q$ of $\mathcal{Q}$ naturally separates $\Delta$ into two disks, exactly one of which intersects $\Delta^*$ (in fact this disk must contain $\Delta^*$).
Let us call the other disk the \defi{small side of $Q$}.
Given two members of $\mathcal{Q}$ then either the small side of one is contained in the small side of the other, or their small sides are disjoint.
It is straightforward to see that if we say that two members are \defi{equivalent} if their small sides intersect, this indeed defines an equivalence relation on $\mathcal{Q}.$
Moreover, there are exactly two equivalence classes, one of which, call it $\mathcal{Q}',$ must contain at least $s+1$ members.
Since $|\mathcal{C}|=s$ there must exist some $i\in[s]$ and some subpath $L$ of some path in $\mathcal{Q}'$ such that
\begin{itemize}
    \item both endpoints of $L$ belong to $V(C_i),$
    \item $L$ is internally disjoint from $\bigcup_{i\in[s]}V(C_i),$
    \item $L$ contains at least one edge that does not belong to $C_i,$ and
    \item $L$ is drawn in the $\mathsf{trace}(C_i)$-disk that is disjoint from the nodes corresponding to the vertices of $\Lambda.$
\end{itemize}
Notice that $C_i\cup L$ contains a unique cycle $C'$ different from $C_i$ whose trace separates $\Delta^*$ from the nodes corresponding to the vertices of $\Lambda.$
Moreover, the $\Delta^*$-containing disk $\Delta''$ of $\mathsf{trace}(C')$ is properly contained in the $\Delta^*$-containing disk of $\mathsf{trace}(C_i).$
In particular, there exists an edge of $C_i$ which does not belong to $C_i',$ and there exists an edge that belongs to $C_i'$ but not to $C_i$ and therefore, the graph drawn on $\Delta''$ after deleting the vertices of $C_i'$ is properly contained in the graph drawn on $\Delta'$ after deleting the vertices of $C_i.$

We would now like to define the new nest $\mathcal{C}'\coloneqq \{ C_1,\dots,C_{i-1},C',C_{i+1},\dots,C_s\}.$
However, to fully obtain the desired pruning of $(G,\delta,W)$ we need to use the radial linkage $\mathcal{R}$ of the thicket in $W$ to obtain a radial linkage $\mathcal{R}'$ of the same order which is orthogonal to $\mathcal{C}'.$
First, we are going to change the cycle $C'$ once more.
Consider $C''$ to be the cycle obtained from $C'$ after iteratively replacing a subpath of $C',$ in the same way as above, with any $V(C_{i})$-$V(C_{i})$ subpath $R'$ of path in $\mathcal{R}$ that satisfies the same four properties as $L$ above.
It follows that there can be at most one path in $\mathcal{R},$ say $R,$ that is not orthogonal to $C''$ after performing these replacements.
To conclude, now we are going to change the path $R.$
Let us traverse along $C''$ in the clockwise direction.
We may define $x_R$ to be the first vertex of $L$ on $R$ and $y_R$ to be the last such vertex.
Then there exists a unique $x_R$-$y_R$-subpath $R''$ of $R$ and a unique $x_R$-$y_R$ subpath $L'$ of $C'$ such that replacing $R''$ by $L'$ results in a new path that is now orthogonal to $C''.$
This defines our desired path $R'$ and allows us to define the nest $\mathcal{C}' \coloneqq \{ C_1,\dots,C_{i-1},C'',C_{i+1},\dots,C_s\}$ and the radial linkage $\mathcal{R}' = \mathcal{R} \cup \{ R' \} \setminus \{ R \}$ together with the new walloid $W'$ to obtain our desired pruning of $(G,\delta,W).$
\end{proof}

Notice that whenever we find a pruning, some edge of $G$ crosses over one of the cycles in some nest in some thicket of $(G,\delta,W).$
Since we are in a graph excluding some minor, there are at most $\mathcal{O}(|G|)$ many edges.
Moreover, the total number of thickets is upper bounded by some constant depending only on the parameters which also determine the size of the excluded minor.
Indeed, similarly, the size of each nest of each thicket can also be considered to be a constant with the same arguments.
This means that after finding at most $\mathcal{O}(|G|)$ many big enough transactions on the thicket societies we have either found a big exposed transaction or may conclude that all thicket societies are of bounded depth.

Notice that, in the end, we want to guarantee that not only each thicket society is of bounded depth, but within every thicket is a vortex which is also of bounded depth, and this vortex is surrounded by a large nest.
The tools we develop in the upcoming subsections will allow us to achieve just that.

\paragraph{Orthogonality.}

So far we know that we can always either find a pruning of our blooming orchard, or we find an exposed planar transaction whose strip society is isolated, separating, and rural.
An important ingredient that we will need to develop the tools in the upcoming subsections is that, in the latter case above, we should also be able to make an exposed planar transaction orthogonal, by possibly slightly altering the nest that comes with some thicket society.
For this purpose we develop the following tool that builds on ideas from \cite{PaulPTW24Delineating}.

\begin{lemma}\label{orthogonal_transaction} Let $x,y,t,b,h,$ and $s$ be positive integers and $(G,\delta,W)$ be an $(x,y)$-fertile $(t,b,h,s,\Sigma)$-orchard.
Let $\langle H, \Lambda \rangle$ be a thicket society of $(G, \delta, W)$ with railed nest $(\mathcal{C}, \Rcal)$ and $\rho \coloneqq \delta \cap \Delta$ be the cylindrical rendition of $\lin{H, \Lambda}$ around $c_{\rho}$ where $\Delta$ is the disk bounded by the $\trace(C^{\mathsf{si}})$-avoiding disk of the outer cycle of the  $(t, s)$-thicket segment of $(G, \delta, W)$ that corresponds to $\lin{H, \Lambda}.$
Let $\Pcal$ be a planar exposed transaction of order $s \cdot (p + 2)$ in $\lin{H, \Omega}$ such that the $\Pcal$-strip society of $\lin{H, \Lambda}$ is isolated, separating, and rural.
Then there exists
\begin{itemize}
\item an $s$-nest $\mathcal{C}'$ in $\delta \cap \Delta$ around the same arcwise connected set as $\mathcal{C},$
\item a radial linkage $\mathcal{R}'$ that is orthogonal to $\mathcal{C}'$ and has the same endpoints as $\mathcal{R}$ on $B_{\Lambda},$ and
\item a planar exposed transaction $\mathcal{Q}$ of order $p$ in $\lin{H, \Omega}$ that is orthogonal to $\mathcal{C}'$ such that the $\Qcal$-strip society of $\lin{H, \Lambda}$ is isolated, separating, and rural.
\end{itemize}
Moreover, there exists an algorithm that finds the outcome above in time $(\poly(s \cdot p) + \poly(t)) \cdot |G|.$
\end{lemma}
\begin{proof} It is straightforward to obtain a rendition $\rho^{*}$ of $\lin{H, \Omega}$ by combining a vortex-free rendition of the $\Pcal$-strip society of $\lin{H, \Lambda}$ and $\delta \cap \Delta$ such that $c_{\rho}$ is divided into a set $\Ucal$ of vortex cells in $\rho^{*}.$
Note that, since $\Pcal$ is exposed it follows that $|\Ucal| \geq 2.$
Moreover, we can define $\rho^{*}$ as a \emph{restriction} of $\delta \cap \Delta,$ i.e., so that every cell of $\rho^{*}$ that is not contained in $c_{\rho}$ is a cell of $\delta \cap \Delta.$
For what follows, we work with the rendition $\rho^{*}.$
Notice that in $\rho^{*}$ the drawing of the paths in $\Pcal$ never cross one another and this is crucial for the arguments that follow.

Let $q \coloneqq s \cdot (p + 2).$
Let $\Lambda_{\Pcal}$ be an ordering of the paths in $\Pcal$ such that while traversing the rotation ordering of $\Lambda$ we first encounter one of the two endpoints for all paths in $\Pcal$ and then all others.
Let $P_{1}, \ldots, P_{s}$ be the paths in $\Pcal$ respecting the ordering $\Lambda_{\Pcal}.$
Moreover, for every cycle $C \in \mathcal{C}$ we define the \emph{interior} (resp. \emph{exterior}) of $C$ as the disk bounded by the trace of $C$ that contains (resp. avoids) $B_{\Lambda}.$

Now, for $i \in [s]$ and $j \in [q],$ let $P$ be a $V(P_{j})$-$V(P_{j})$-path $B$ that is a subpath of $C_{i}.$
Let $P_{B}$ be the subpath of $P_{i}$ that shares its endpoints with $B.$
Let $\Delta_{B}$ be the disk bounded by the trace of the cycle $B \cup P_{B}$ that does not fully either the trace of $C_{1}$ or $B_{\Lambda}.$
We call $B$ a \emph{bend} of $C_{i}$ at $P_{j}$ if $\Delta_{B}$ does not intersect the drawing of $C_{i} \setminus B.$
Moreover, we call a bend $B$ \emph{shrinking} if $\Delta_{B}$ is a subset of the interior of $C_{i}$ and \emph{expanding} otherwise.
We call an expanding bend $B$ \emph{loose} if $\Delta_{B}$ does not intersect the drawing of any cycle in $\Ccal$ that is not $C_{i}$ and $\emph{tight}$ if any expanding bend $B'$ of $C_{i}$ where $B'$ is a subpath of $B$ is not loose.

We also define the \emph{height} of a bend $B$ as the non-negative integer $r + 1,$ where $r$ is the maximum value such that $B$ intersects at least one of $P_{j + r}$ or $P_{j - r}$ and $j + r \leq s$ or $j - r \geq 1$ (if $j + r$ exceeds $q$ we consider $P_{j + r} = P_{q}$ and if $j - r$ drops below $1$ we similarly assume that $P_{j - r} = P_{1}$).
Finally, we call any minimal $V(P_{1})$-$V(P_{q})$ path contained in a cycle $C \in \Ccal$ a \emph{pillar} of $C.$

First, notice that an expanding bend $B$ of a cycle $C \in \Ccal$ that is neither loose nor tight must contain a subpath $B'$ that is a loose expanding bend of $C.$
Moreover, if $B$ is any expanding bend of a cycle $C \in \Ccal,$ then the graph obtained by replacing the subpath $B$ of $C$ with the path $P_{B}$ defines a cycle $C'$ and the interior of $C'$ strictly contains the interior of $C,$ hence the name expanding.
In the case of a shrinking bend the opposite happens, particularly we obtain a cycle with strictly smaller interior.

Additionally, notice that for every $C \in \Ccal,$ since $\Pcal$ is an exposed transaction, $C$ contains at least one pillar and since $C$ is a cycle, it must contain an even number of pairwise disjoint pillars.

We proceed to prove the following statement by an induction on the number of cycles in $\Ccal$: If no cycle in $\Ccal$ contains a loose expanding bend then for every $i \in [s],$
\begin{enumerate}
\item every expanding bend of $C_{i}$ has height at most $s - i + 1,$
\item $C_{i}$ contains exactly two disjoint pillars $R_{i},$ $i \in [2],$ and
\item every shrinking bend of $C_{i}$ that is a subpath of $R_{i},$ $i \in [2],$ has height at most $s - i + 1.$
\end{enumerate}

Now, let $k \in [1, s-1]$ and $\Ccal^{k}$ be the subset of $\Ccal$ that consists of its $k$ outermost cycles.
Assume that the statement above holds for $\Ccal^{k}.$
We prove that it also holds for $\Ccal^{k + 1}.$
Assume that no cycle in $\Ccal^{k + 1}$ contains a loose expanding bend.
Let $B$ be any tight expanding bend of $C_{s - k}$ at $P \in \Pcal$ which is the innermost cycle contained in $\Ccal^{k+1}$ and the only cycle of $\Ccal^{k+1}$ not contained in $\Ccal^{k}.$
Notice that any subpath $B'$ of $B$ whose endpoints are the endpoints of a maximal subpath of a path in $\Pcal$ that is drawn in $\Delta_{B}$ is also an expanding bend of $C_{s - k}$ where $\Delta_{B'} \subseteq \Delta_{B}.$
Since $B$ is tight, any such subpath $B'$ must also be tight.
This implies that the disk $\Delta_{B'}$ for any such subpath $B'$ intersects the drawing of some other cycle of $\Ccal.$
Since $B$ is expanding and hence $\Delta_{B}$ is not a subset of the interior of $C_{s - k},$ it follows that $\Delta_{B'}$ intersects the drawing of cycles only in $\Ccal^{k}.$
In fact, it follows that any such intersection corresponds to an expanding bend of some cycle in $\Ccal^{k}$ at $P$ that is contained in $\Delta_{P}.$
Then, by assumption, if all these expanding bends are tight, it must be that the maximum height among all of them is $k.$
Hence, it must be that the height of $B$ is $k + 1.$

Moreover, it is straightforward to observe that $C_{s - k}$ contains exactly two disjoint pillars.
Indeed, if there were at least four disjoint pillars in $C_{s - k}$ then it is easy to see that there would exist an expanding bend of $C_{s - k}$ on either $P_{1}$ and/or $P_{q}$ whose height would be larger than $k + 1,$ which would imply that it is a loose expanding bend, which by assumption cannot exist.

Finally, it is also straightforward to observe that the existence of a shrinking bend of $C_{s - k}$ that is a subpath of either of the two pillars of $C_{s - k}$ of height more than $k + 1,$ implies the existence of an expanding bend of the same height, which as we proved cannot exist.

\medskip
We proceed with the description of an algorithm that runs in time $(\poly(s \cdot p) + \poly(t)) \cdot |G|$ that computes the desired $s$-nest $\mathcal{C}'$ along with the desired railed linkage $\mathcal{C}'$ that is orthogonal to $\mathcal{C}',$ and moreover computes the desired planar exposed transaction $\mathcal{Q}$ of order $p$ that is orthogonal to the new nest $\mathcal{C}'.$

\paragraph{Step 1:} We obtain our new set of cycles $\mathcal{C}' \coloneqq \langle C'_{1}, \ldots, C'_{s} \rangle,$ where $C'_{i},$ $i \in [s],$ is obtained from $C_{i}$ by iteratively applying whenever applicable the following two rules.

\begin{itemize}
\item[\textbf{1:}] As long as there exists a path $R \in \mathcal{R},$ some $i \in [s-1]$ and a subpath $R' \subseteq R$ such that $R'$ has both endpoints in $C_i$ and is otherwise disjoint from all cycles of $\mathcal{C}$ and moreover $R'$ is drawn entirely in the exterior of $C_{i},$ we update $C_{i}$ to the unique cycle $C'$ contained in $C_{i} \cup R'$ different from $C_{i}$ whose trace separates the interior of $C_{1}$ from $B_{\Lambda}.$
We call such a subpath $R'$ of $R$ as above, a $C_{i}$-\emph{expanding} path.
Now, observe that, since $\rho^{*}$ is a restriction of $\delta \cap \Delta,$ $C'$ is also grounded in $\delta \cap \Delta.$
Moreover, the interior of $C'$ properly contains the interior of $C_{i}$ (before the update).
\item[\textbf{2:}] As long as there exists a loose expanding bend $B$ of a cycle $C \in \mathcal{C}$ we replace $B$ by $P_{B}$ as defined above.
Now, observe that, since $\rho^{*}$ is a restriction of $\delta \cap \Delta,$ the updated cycle remains grounded in $\delta \cap \Delta.$
\end{itemize}
Since for both rules above, whenever we update a cycle of our nest, it is guaranteed that the interior of the updated cycle strictly contains the interior of its predecessor, the above procedure will terminate in at most $\poly(s \cdot p) \cdot |G|$ many steps and the resulting nest $\mathcal{C}'$ will satisfy the following two properties.
There is no loose expanding bend of any cycle in $\mathcal{C}',$ which implies that $\mathcal{C}'$ satisfies the three properties above, and there is no $C$-expanding subpath of any path in $\mathcal{R}$ for any cycle $C \in \mathcal{C}'.$

\paragraph{Step 2:}

Now, update $\mathcal{R}$ to contain the minimal $V(\Lambda)$-$V(C'_{1})$ subpath of each path in $\mathcal{R}.$

It follows from the previous step that any subpath $R'$ of any path $R \in \mathcal{R}$ such that $R'$ has both endpoints on $C'_i$ for $i \in [2, s - 1]$ and is otherwise disjoint from $C'_i$ must either intersect $C'_{i+1}$ or it cannot be drawn in the exterior of $C'_i.$
In the case that $R'$ has both endpoints on $C'_{s}$ it can only be that $R'$ is not drawn in the exterior of $C'_{s}.$
From now on we call such a subpath $R'$ of a path $R \in \mathcal{R}$ a \defi{retrogression} of $R.$

Let the paths in $\mathcal{R}=\{ R_1,R_2,\dots,R_t\}$ be ordered according to the occurrence of their endpoints on the rotation ordering of $\Lambda.$
Moreover, assume there exists an $i \in [t]$ and a $j \in [2, s]$ such that $R_i$ has a retrogression $R'$ with both endpoints on $C'_j.$
Notice that there exists a subpath $L_{R'}$ of $C'_j$ that shares its endpoints with $R'$ which is internally disjoint from $R_i.$
This is because otherwise, we could find a subpath of $R_i$ which witnesses that we are in the previous case.
Let $\Delta_{R'}$ be the disk bounded by the trace of $R'\cup L_{R'}$ which avoids the trace of $C_1.$
We assume that $R'$ is chosen maximally with the property of being a retrogression.

Suppose there exists some path $Q,$ possibly $Q=R,$ in $\mathcal{R}$ whose drawing intersects the interior of $\Delta_{R'}$ in a node or arc distinct from those of $R'.$
Then, this intersection of $Q$ with $\Delta_{R'}$ must belong to a maximal retrogression $Q'$ of $Q$ and there must exist $j' \in [2, s],$ $j' > j,$ such that both endpoints belong to $C_{j'}.$
Hence, there must exist some maximal retrogression $P'$ of some path $P' \in \mathcal{R}$ such that no other part of $\mathcal{R}$ intersects the corresponding disk $\Delta_P'.$
Thus, by replacing $P'$ with $L_{P'},$ we obtain a new radial linkage still with the same endpoints as those of $\mathcal{R},$ but with strictly less retrogressions.
This means that, after at most $E(\mathcal{R})$ many such refinement steps, we must have found a radial linkage $\mathcal{R}'$ which is orthogonal to $\mathcal{C}'.$
This step can be performed in time $\poly(t) \cdot |G|.$

\paragraph{Step 3:}

We now define a transaction $\Pcal' \subseteq \Pcal$ of order $p$ by skipping the first and last $s$ paths in $\Pcal$ and then choosing the first path from every bundle of $s$ many consecutive paths from the remaining paths of $\Pcal.$
By the arguments above the height of every expanding bend of every cycle in $\Ccal'$ at $P_{1}$ or $P_{s}$ is at most $s$ and therefore by definition no such bend can intersect any of the paths in $\Pcal'.$
It follows that any other bend of any cycle must be a bend that is a subpath of a pillar of the given cycle.
By definition of $\Pcal',$ this implies that any such bend of any cycle can intersect at most one path in $\Pcal'.$
This implies that the transaction $\Pcal'$ which is a planar and exposed transaction in $(H, \Omega)$ is almost orthogonal to our new set of cycles $\Ccal'.$

We need to make one more adjustment to each of the paths in $\Pcal'$ to finally make them orthogonal to $\Ccal'.$
We define the final transaction $\Qcal$ where each path $Q \in \Qcal$ is a path obtained from a path in $P \in \Pcal'$ that shares the same endpoints with $P$ and is defined by starting from one of the two endpoints of $P$ and while moving towards the other endpoint of $P$ taking any possible ``shortcut'', that is by greedily following along any bend of any pillar of a cycle in $\Ccal'$ at $P$ that appears, until we reach the other endpoint of $P.$
Since as we already observed any bend of any pillar of a cycle can intersect at most one path in $\Pcal',$ it is implied that $\Qcal$ is a linkage.
Moreover, by construction, it now follows that $\Qcal$ is orthogonal to $\Ccal'$ as desired and that the paths in $\Qcal$ are grounded in $\rho^{*}$ which also implies that the $\Qcal$-strip society of $\lin{H, \Lambda}$ is isolated, separating, and rural.
Finally, since the shortcuts we take are subpaths of pillars it follows that $\Qcal$ contains at least one edge drawn in the interior of $C'_{1}$ and therefore is exposed (with respect to $\delta \cap \Delta$).
\end{proof}

\paragraph{Rooting transactions.}

Recall that every thicket comes with a railed nest $(\mathcal{C},\mathcal{R})$ where $\mathcal{R}$ is orthogonal to $\mathcal{C}.$
Moreover, $\mathcal{R}$ acts as our connection of the thicket to the cylindrical wall of the walloid of our orchard.
While refining, and possibly splitting, the thicket we wish to maintain the existence of such a radial linkage and we need to make sure that new flowers that potentially need to be extracted from the thicket can be routed through the paths in $\mathcal{R}.$

\medskip
To this end we prove the following lemma.
For a visualization, see \autoref{omitted_coterminality}.

\begin{figure}[ht]
\begin{center}
\scalebox{1.1}{\includegraphics{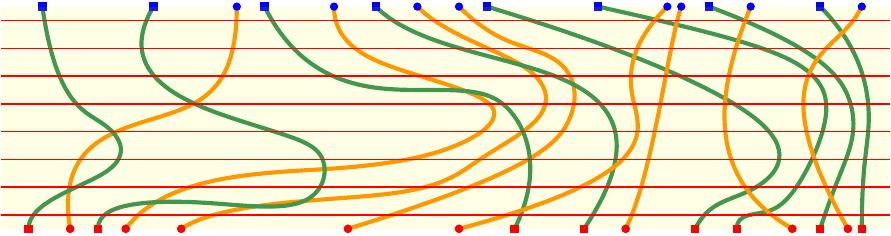}}
\end{center}
    \caption{A visualization of \autoref{easy_consequence_of_coterminality}. The cycles in $\Ccal$ are red, the paths in $\Lcal$ are green, and the paths in $\Lcal'$ are orange. The vertices in  $X\cup Y$ are red and the vertices in $X'\cup Y'$ are blue.}
  \llabel{omitted_coterminality}
\end{figure}

\begin{lemma}\llabel{easy_consequence_of_coterminality}
Assume that $G$ is a graph embedded in an annulus $A$ such that $G$ is the union of 
\begin{itemize}
\item $r$-cycles $\Ccal = \{ C_{1}, \ldots, C_{r} \}$ drawn in its interior such that no disk of $A$ bounds some of the cycles in $\Ccal,$
\item an $X$-$Y$-linkage $\Lcal = \{L_{1},\ldots,L_{r}\},$ and a $X'$-$Y'$-linkage $\Lcal'=\{L_{1}', \ldots, L_{r}'\},$
\end{itemize}
such that the vertices of $X \cup Y$ are drawn in one of the two boundaries of $A,$ the vertices of $X' \cup Y'$ are drawn in the other, and both $\Lcal$ and $\Lcal'$ are orthogonal to $C.$
Then $G$ contains an orthogonal $X$-$Y'$-linkage of order $r.$

\medskip
Moreover, there exists an algorithm that finds the outcome above in time $\poly(r) \cdot |G|.$
\end{lemma}
\begin{proof} Notice that for every set $S \subseteq V(G)$ such that $|S| < r,$ there exists at least one cycle, one path in $\Lcal,$ and one path in $\Lcal'$ fully contained in $G - S.$
By applying Menger's theorem, this implies that there exists an $X$-$Y'$ linkage $\mathcal{R}$ in $G$ of order $r.$
Moreover, in the terminology of the proof of \autoref{orthogonal_transaction}, since $\mathcal{L}$ and $\mathcal{L}'$ are orthogonal to $\mathcal{C},$ we can immediately deduce that there are no $C$-expanding paths of any path in $\mathcal{R}$ for any cycle $C \in \mathcal{C}.$
Then, by applying the arguments of step two in the proof of \autoref{orthogonal_transaction}, we can make $\mathcal{R}$ orthogonal to $\mathcal{C}.$
\end{proof}

\autoref{easy_consequence_of_coterminality} allows us to root any given orthogonal transaction on our orthogonal radial linkage.
To formalize this we borrow the following notion from \cite{kawarabayashi2020quickly}.

\medskip
Let $\rho$ be a rendition of a society $\langle G, \Lambda \rangle$ in a disk $\Delta$ together with a nest $\mathcal{C}=\{ C_1,C_2,\dots,C_s\}$ around an arcwise connected set $\Delta^*\subseteq \Delta.$
We say that a linkage $\mathcal{L}$ is \defi{coterminal} with a linkage $\mathcal{R}$ \defi{up to level $C_i$} if there exists a subset $\mathcal{R}'\subseteq \mathcal{R}$ such that $\mathcal{L}\cap H = \mathcal{R}'\cap H$ where $H$ is the graph drawn in the $\Delta^*$-avoiding disk of $\mathsf{trace}(C_i).$

\subsection{Cultivating a transaction}\llabel{transaction_cultivation}

With the tools developed in the previous section, we know that we can always either find a pruning of our blooming orchard, or we find an exposed planar transaction that is isolated, separating, and rural which because of \autoref{orthogonal_transaction} we can also assume to be orthogonal to the nest of our thicket.
The next step is to apply the ideas of plowing to the transaction to make sure that the area covered by it is homogeneous with respect to its folio of detail $d$ which would then allow us to extract additional flowers and parterre from it.

\medskip
Let $d,x,y,t,b,h,$ and $s$ be positive integers and $(G,\delta,W)$ be an $(x,y)$-fertile $(t,b,h,s,\Sigma)$-orchard.
Moreover, let $\langle H,\Lambda \rangle$ be a thicket society with nest $\mathcal{C} = \{ C_1,\dots,C_s\}$ of $(G,\delta,W).$
We say that a transaction in $\langle H,\Lambda\rangle$ is \defi{barren} if it is planar, exposed, and orthogonal to $\mathcal{C},$ and the $\Pcal$-strip society of $\lin{H, \Lambda}$ is isolated, separating, and rural.
Let $\mathcal{P}$ be a barren transaction in $\langle H, \Lambda \rangle.$
A \defi{parcel} of $\mathcal{P}$ is a cycle $C$ consisting of two vertex-disjoint subpaths of $C_1,$ say $X_1$ and $X_2,$ and a shortest $V(C_{1})$-$V(C_{1})$ subpath from two distinct paths $P_{1} \neq P_{2} \in \mathcal{P},$ such that $C_1$ contains an $X_1$-$X_2$-subpath of $P_i$ for each $i \in [2],$ and $C_1$ intersects no other member of $\mathcal{P}.$
Notice that all but the boundary paths of a barren transaction belong to exactly two parcels and the boundary paths belong to exactly one parcel each.
Let $d$ be a positive integer.
The \defi{$d$-folio of a parcel $O$ of $\mathcal{P}$} is the union of the $d$-folios of all vortex-societies of all cells from the $\delta$-influence of $O.$
A transaction $\mathcal{P}$ is said to be \defi{$d$-cultivated} if it is barren and all of its parcels have the same $d$-folio.

\begin{lemma}\llabel{lem_cultivate_transaction}
Let $d,x,y,t,b,h,p,$ and $s$ be positive integers.
Let $(G,\delta,W)$ be a $d$-blooming and $(x,y)$-fertile $(t,b,h,s,\Sigma)$-orchard, let $\langle H,\Lambda\rangle$ be one of its thicket societies together with its nest $\mathcal{C}=\{ C_1,C_2,\dots,C_s\}.$
There exists a function $f_{\ref{lem_cultivate_transaction}}\colon \mathbb{N}^2\to\mathbb{N}$ such that for  every barren transaction $\mathcal{P}$ of order $f_{\ref{lem_cultivate_transaction}}(p,d)$ in $\langle H,\Lambda\rangle$ there exists a $d$-cultivated transaction $\mathcal{Q}\subseteq\mathcal{P}$ of order $p.$

Moreover, $f_{\ref{lem_cultivate_transaction}}(p,d)\in p^{2^{2^{\mathcal{O}(d^2)}}}.$
\end{lemma}

\begin{proof}
By \autoref{obs_size_dfolio} we know that the folio of any parcel of any barren transaction in $\langle H,\Lambda\rangle$ has at most $f_{\ref{obs_size_dfolio}}(d)\in 2^{\mathcal{O}(d^2)}$ members.
Moreover, recall that $f_{\ref{lem_homogeneity_2d}}(z_1,z_2)= z_1^{2^{z_1}}.$
We set $f_{\ref{lem_cultivate_transaction}}(d,p)\coloneqq f_{\ref{lem_homogeneity_2d}}(p,f_{\ref{obs_size_dfolio}}(d))$ and thus meet the upper bound claimed in the assertion.

Let $P_1$ and $P_2$ be the two outer paths of $\mathcal{P}$ and for each $i\in[2]$ let $P_i'$ be the shortest $V(C_1)$-$V(C_1)$-subpath of $P_i$ with at least one edge not in $C_1.$
Then there exist two vertex-disjoint subpaths $L_1$ and $L_2$ of $C_1$ such that each $L_i$ has one endpoint in common with $P_1'$ and the other endpoint in common with $P_2',$ and every path $P\in\mathcal{P}$ contains a unique $V(L_1)$-$V(L_2)$-subpath $P'.$
It follows that $J\coloneqq L_1\cup L_2\cup \bigcup_{P\in\mathcal{P}P}$ is an $f_{\ref{lem_cultivate_transaction}}(d,p)$-ladder.
Moreover, the bricks of $J$ are in a one-to-one correspondence with the parcels of $\mathcal{P}.$
If we now enumerate all possible $f_{\ref{obs_size_dfolio}}(d)$ members of a $d$-folio of some parcel of $\mathcal{P},$ assign distinct numbers to all of them and then assign to each brick of $J$ the set of all numbers corresponding to the members of the folio of the corresponding parcel of $\mathcal{P}$ we obtain an $f_{\ref{lem_cultivate_transaction}}(d,p)$-ladder where each brick is assigned a $f_{\ref{lem_cultivate_transaction}}(d,p)$-palette.
This allows us to call \autoref{cor_homogeneity_1d} to obtain a $p$-ladder $J'\subseteq J$ where all bricks have the same $f_{\ref{lem_cultivate_transaction}}(d,p)$-palette.
Notice that each of the $p$ horizontal paths of $J'$ is, in fact, a subpath $P'$ of some path in $\mathcal{P}.$
Let $\mathcal{Q}$ be the collection of all paths $Q\in\mathcal{P}$ such that $Q'$ is a horizontal path of $J'.$
It follows that $\mathcal{Q}$ is $d$-cultivated and of order $p$ as desired.
\end{proof}

\subsection{Greenhouses}\llabel{greenhouses}

We are now able to always find large cultivated transactions.
Next, we need an auxiliary tool that will help us refine a thicket and extract additional parterre.
Toward this goal we define an intermediate structure.

\medskip
Let $p\geq t\geq 3$ be integers.
A \defi{$(t,p)$-greenhouse segment} is a $(2t+1,p)$-wall $W$ that is the union of a $(t,p)$-wall $W_1,$ called the \defi{base}, whose $t$ rows are exactly the first $t$ rows of $W,$ a $(t+1,p)$-wall $W_2$ whose $t$ rows are exactly the rows $t+1$ up to $2t+1$ of $W,$ and a $(2,p)$ wall $Q$ whose two rows are exactly the rows $t$ and $t+1$ of $W.$
We call the bricks of $Q$ the \defi{plots} of $W.$

If $G$ is a graph, $\delta$ is a $\Sigma$-decomposition of $G,$ $d$ is a positive integer, and $W$ is a $(t,p)$-greenhouse segment, we say that $W$ is \defi{$d$-homogeneous} if the $d$-folios of all plots of $W$ are the same.

We will use the greenhouse segments as a storage unit for societies we have found deep within a segment and which will be turned into flowers of their respective parterre segments in a second step.
For this reason, greenhouses only appear as an intermediate structure within this subsection.

\begin{figure}[ht]
    \centering
    \scalebox{0.13}{\includegraphics{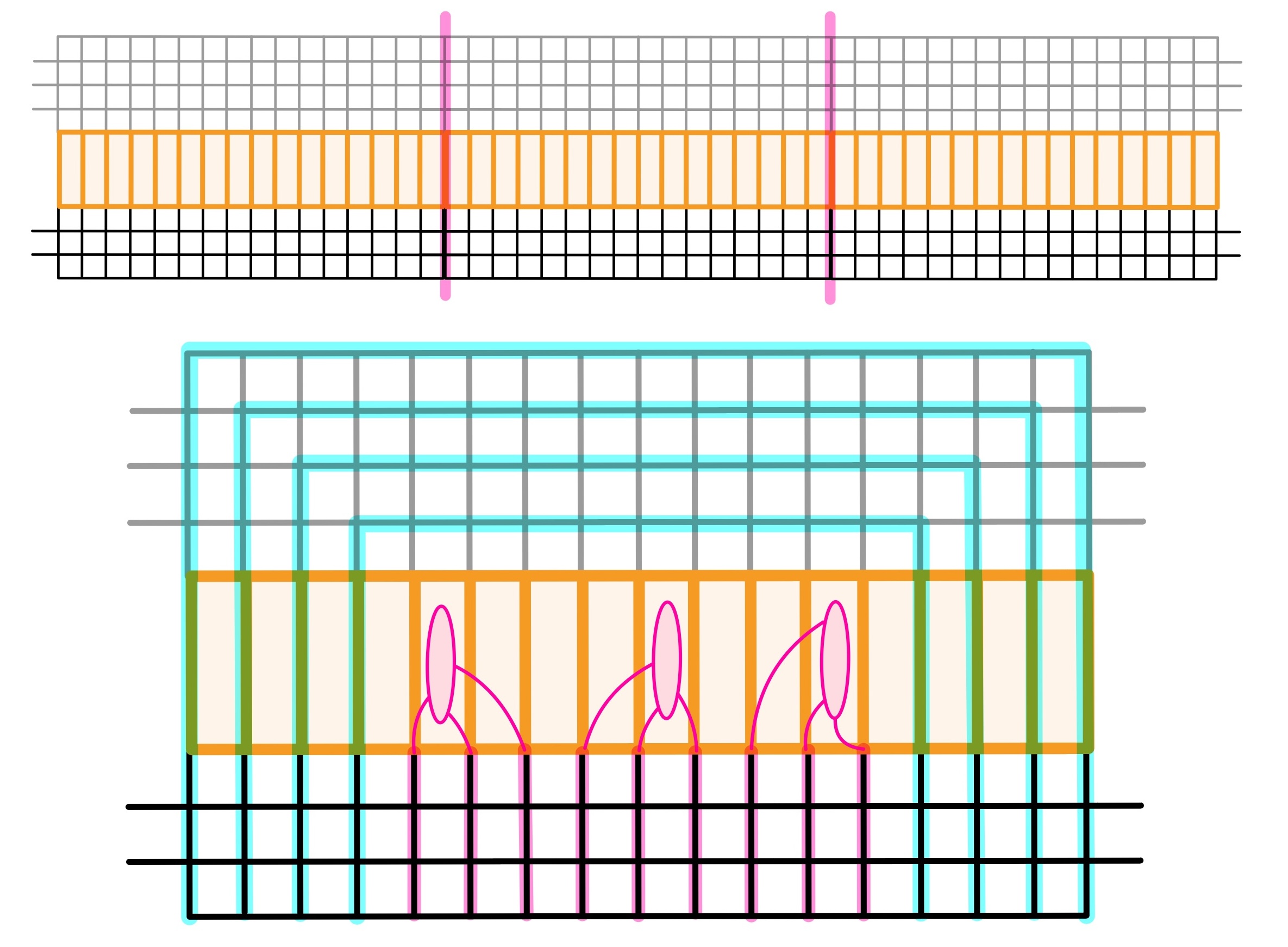}}
    \caption{Two illustrations for the proof of \autoref{lemma_planting_a_greenhause}: a greenhouse subdivided into a small sequence of pieces and one of those pieces transformed into a flower segment.}
    \llabel{fig_greenhouse_2}
\end{figure}

\begin{lemma}\llabel{lemma_planting_a_greenhause}
Let $b,d,h,p,t$ be positive integers where $t\geq 3.$
Let $G$ be a graph $\delta$ be a $\Sigma$-decomposition of $G,$ $W$ be a $d$-homogeneous $(t+10,p)$-greenhouse segment in $G,$ and let $X$ be the perimeter of $W.$
Assume further that $p\geq d^*\cdot h(2t+10bd)$ where $d^*$ is the size of the $d$-folio of the plots of $W.$
Then there exists a sequence $\langle W_1,\dots,W_{d^{*}}\rangle$ of $(t,b,h)$-parterre segments such that
\begin{itemize}
    \item for every $i\in[d^*]$ $W_i$ is fully contained in the graph drawn in the influence of $X,$
    \item for every $i\in[d^*]$ the wall $\widetilde{W_i}$ is a subwall of the base of $W$ and its rows coincide with the rows of $W,$
    \item for every simple society $\langle H,\Lambda\rangle$ in the $d$-folio of the plots there exists some $i\in[d^*]$ such that $W_i$ is a $(t,b,h,H,\Lambda)$-parterre segment.
\end{itemize}
\end{lemma}
\begin{proof}
Let $\mathcal{O}=\{ O_1,O_2,\dots,O_{d^*}\}$ be the collection of all societies that belong to the $d$-folio of the plots of $W.$
Let us begin by partitioning $W$ into a sequence $\langle W'_1,\dots,W'_{d^*}\rangle$ of $(h(2t+10bd),t+10)$-walls.
Fix some $i\in[d^*]$ and now partition $W'_i$ further into a sequence $\langle W^i_1,\dots,W^i_h\rangle$ of $(2t+10bd)$-walls.
It suffices to show that each $W^i_j$ can be turned into a $(t,b,O_i)$-flower segment since this means we can find, for each $i\in[d^*]$ a sequence of $h$ disjoint such flower segments, resulting in a $(t,b,h,O_i)$-parterre segment as desired.

Notice that each $W^i_j$ individually can be seen as a $(t,2t+10bd)$-greenhouse segment.
Since $W$ is $d$-homogeneous, it follows that $W^i_j$ is so as well since every plot of $W^i_j$ is a plot of $W.$
We may now find a linkage $\mathcal{P}$ starting at the first vertex of each of the first $t$ columns of $W^i_j,$ then following these columns into the $t+2$nd row, from there heading to the last $t$ columns of $W^i_j$ and following these down to the bottom row.
See \autoref{fig_greenhouse_2} for an illustration.

This linkage intersects only $2t$ plots of $W^i_j$ and will serve as the rainbow of our flower segment.
To create the petals we group the remaining $bd$ columns into $b$ groups, each containing $10d$ columns.
We may then use the same arguments as in the proof of \autoref{lem_blooming_one_step} and \autoref{lem_paths_boundary} to pick first a plot of each group which is the center of a $(10,10)$-subwall of $W^i_j$ that does not intersect any other group nor $\mathcal{P}.$
Within the influence of this plot, we know we find a cell $c$ whose vortex society contains $O_i$ as a minor.
Now \autoref{lem_paths_boundary} allows us to link the boundary vertices of $c$ to the columns of the group we are currently handling.
Again, see \autoref{fig_greenhouse_2} for an illustration.
Since all groups can be handled individually this indeed results in the desired flower segment. 
\end{proof}

\subsection{Swapping a thicket and a parterre}\llabel{subsec_swapping}

A last issue we have to deal with before we can jump into the proof of our main tool is the following.
Suppose the thicket society of, at least, the first thicket is already of bounded depth.
This means that there is nothing more to be extracted from that thicket and we have to start treating some thicket to the right of it.
However, whenever we extract a parterre segment by first pulling out a part of some $d$-homogeneous transaction in a thicket society of the $i$th thicket and then using \autoref{lemma_planting_a_greenhause}, the parterre will be ``in between'' thicket segments with respect to the linear ordering of the segments.
As we demand all parterre segments of an orchard to occur in order without being interrupted by other segments, we need a way to move this parterre to the ``left'' of all thicket segments.
This will cost us a bit of the infrastructure of our thickets, but since we will only extract, and then move, a bounded number of parterre segments in the end, this becomes just a matter of large enough numbers.

\begin{lemma}\llabel{lemma_moving_a_parterre}
Let $d,x,y,t_0,t,b,h,s_0,s,$ and $\ell$ be positive integers where
\begin{align*}
   s &\geq s_0 + \ell \cdot h(2t_0+bd)\text{ and}\\
   t &\geq t_0 + 2\cdot\ell\cdot h(2t_0+bd)
\end{align*}
Now let, $(G,\delta,W)$ be obtained from a $d$-blooming and $(x,y)$-fertile $(t,b,h,s)$-orchard $(G,\delta^*,W^*)$ by replacing the $i$th $(t,s)$-thicket segment with a sequence $\langle W_1,\dots,W_{\ell},W_{\ell+1},W_{\ell+2}\rangle$ where $W_i,$ for each $i\in [\ell]$ is a $(t_0,b,h)$-parterre segment and both $W_{\ell+1}$ and $W_{\ell+2}$ are $(t_0,s_0)$-thicket segments where we allow either $W_{\ell+2},$ or both $W_{\ell+1}$ and $W_{\ell+2}$ to be empty.
Moreover, assume that every member of $d\text{-}\mathsf{folio}(\delta')$ is either a society of some parterre segment of $\delta,$ or a society of $W_i$ for some $i\in[\ell].$
Then there exists a $d$-blooming $(x,y)$-fertile $(t_0,b,h,s_0)$-orchard $(D,\delta',W')$ such that $d\text{-}\mathsf{folio}(\delta')=d\text{-}\mathsf{folio}(\delta^*)$ and the number of thicket segments of $(D,\delta',W')$ is exactly the number of thicket segments of $(G,\delta,W)$ minus one plus the number of non-empty members of $\{ W_{\ell+1},W_{\ell+2} \}.$
\end{lemma}

The proof of \autoref{lemma_moving_a_parterre} is mostly technical but not very difficult.
We therefore choose to present a rather informal description of the operation instead of a full formalization.
A sketch of the core part of the re-routing necessary is depicted in \autoref{fig_greenhouse_3}. 

\begin{proof}

\begin{figure}[ht]
    \centering
    \scalebox{0.12}{\includegraphics{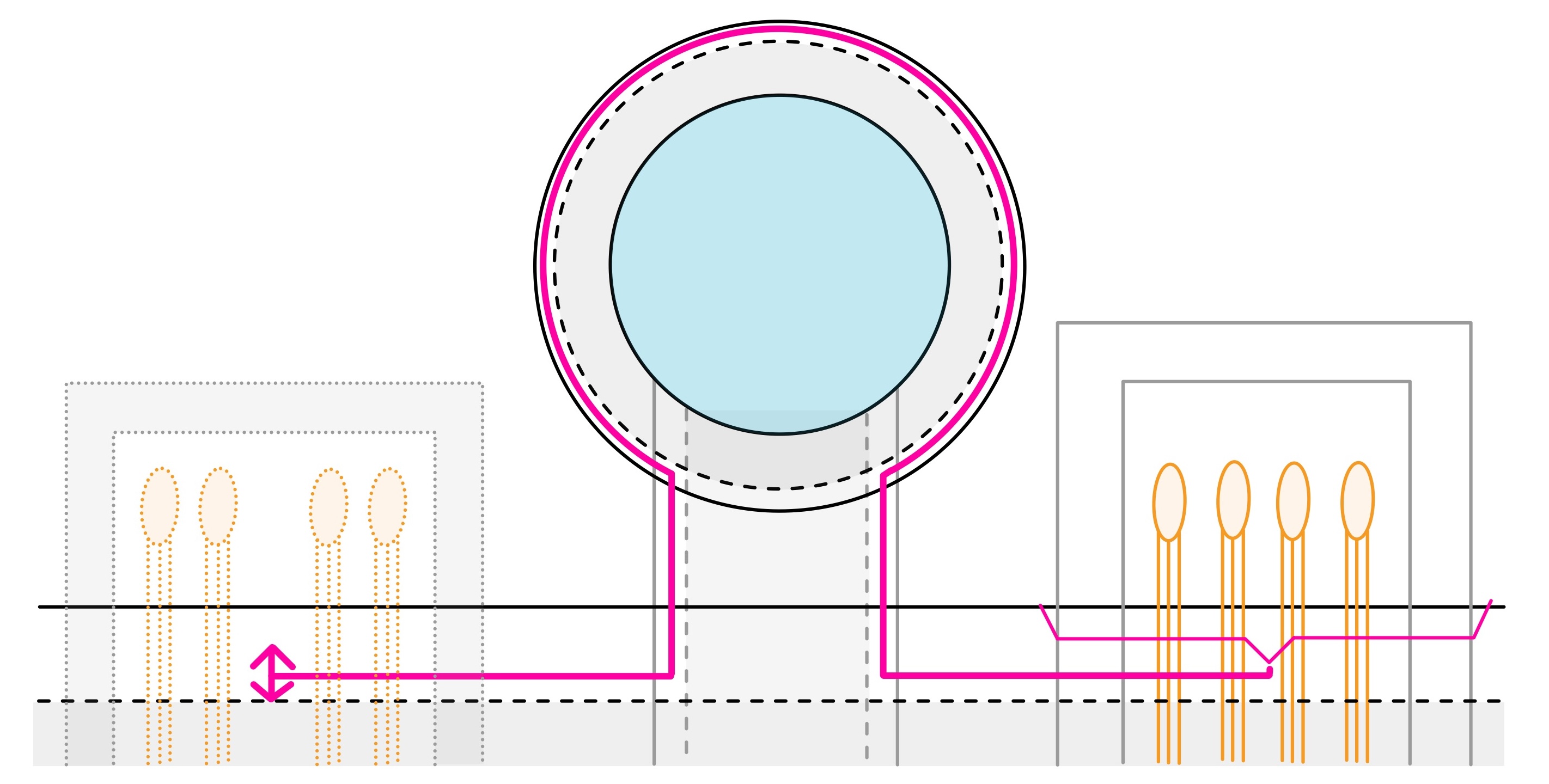}}
    \caption{A sketch on how to move a flower segment from the ``right'' of a thicket segment to its ``left'' in the proof of \autoref{lemma_moving_a_parterre}.}
    \llabel{fig_greenhouse_3}
\end{figure}

Let $C$ denote the exceptional cycle of $W.$
For each $i\in[\ell],$ $W_i$ is a $(t_0,b,h)$-parterre segment.
Hence, $W_i$ consists of $h$ $(t_0,b)$-flower segments in the sequence $\langle W^i_1,\dots, W^i_h\rangle.$

Moreover, for each $i\in[\ell]$ and $j\in[b],$ each flower of $W^i_j$ has at most $d,$ indeed, at most $3,$ vertices in common with $C$ while the rainbow of $W^i_j$ meets $2t_0$ additional vertices of $C.$
Let us call these vertices the \defi{terminals} of $W^i_j.$
To prove out statement it suffices to show that the parterre segments may be ``moved'' sequentially from one side of a single thicket to the other.
If this is true, then, by using the same arguments, we can move the parterre segment to ``the other side'' of any number of consecutive thicket segments since we only have to sacrifice the infrastructure locally.
Moreover, if this works for a single parterre segment, then by induction on $\ell$ it follows that we may repeat the procedure for each of them.

Indeed, to move a single parterre segment, it suffices, by the same arguments, to show that we may ``move'' its flower segments individually, one after the other.
To this end, all we have to prove is that a single flower segment may be ``moved''.

We allocate the following part of the infrastructure to make our routing:
Starting from $C$ we select $2t_0+db$ many consecutive cycles from the concentric cycles of $W.$
From the $(i-1)$st thicket we allocate the first $2t_0+db$ paths from its radial linkage together with the last $2t_0+db$ paths of its radial linkage.
Moreover, we reserve the outermost $2t_0+db$ cycles of its nest.

Let $C'$ denote the new exceptional cycle of $W'$ in case $W'$ is obtained from $W$ by forgetting the first $2t_0+db$ consecutive cycles starting from $C.$

Observe that each part of the infrastructure allocated is large enough to host $2t_0+db$ disjoint paths, which is at least the number of terminals from our flower segment.
By following along the vertical paths of $W$ and the allocated infrastructure, we can now find a linkage from the terminals of our flower segment to the intersection of the first $2t_0+db$ paths of the radial linkage of the $(i-1)$st thicket segment such that
\begin{enumerate}
    \item these paths stay within the allocated infrastructure except for their endpoints, and
    \item the endpoint of each path belongs to $C'.$
\end{enumerate}
See \autoref{fig_greenhouse_2} for an illustration of how those paths are to be chosen.

Notice that this operation reduces the infrastructure by at most $2t_0+db$ many cycles in the thicket and $W$ and at most $2(2t_0+db)$ many paths from its radial linkage.
Moreover, the operation results, after forgetting the used infrastructure and adjusting the numbers accordingly, in the flower segment to now be on ``the other side'' of the $(i-1)$st thicket segment.

Since we perform this operation $h$ times for each parterre segment and we have $\ell$ parterre segments in total, the total amount of lost infrastructure is at most $\ell \cdot h(2t_0+bd)$ cycles from the nest of each involved thicket segment, at most $2\cdot\ell\cdot h(2t_0+bd)$ paths from the radial linkage of each involved thicket segment, and at most $\ell\cdot h(2t_0+bd)$ cycles of $W.$
By our choice of numbers this implies that the nests of each thicket segment, as well as the walloid $W$ itself, each have at least $t_0$ cycles left.
Moreover, each thicket segment has at least $t_0$ paths in its radial linkage that remain untouched and at least $s_0$ cycles in its nest that have not been used.
Thus, our claim follows.
\end{proof}

\subsection{The extraction lemma}\llabel{thicket_extraction}

We are now ready to deal with our most technical tool that will facilitate most of the technicalities of the proof of \autoref{ripe_orchard}.

\medskip

Let $x,y,t,b,h,$ and $s$ be positive integers and $(G,\delta,W)$ be an $(x,y)$-fertile $(t,b,h,s,\Sigma)$-orchard.
Moreover, let $\lin{H, \Lambda}$ be a thicket society of a $(t, s)$-thicket segment of $(G, \delta, W)$ with nest $\mathcal{C} = \{ C_{0}, \ldots, C_{s+1} \}.$
Also, consider $\rho$ to be the cylindrical rendition of $\lin{H, \Lambda}$ obtained from $\delta \cap \Delta$ by declaring the interior of the $\trace_{\delta}(C^{\mathsf{si}})$-avoiding disk of $\trace_{\delta}(C_{0})$ the unique vortex $c_{0}$ of $\rho,$ where $\Delta$ is the $\trace_{\delta}(C^{\mathsf{si}})$-avoiding disk of $\trace_{\delta}(C_{s + 1}).$
Moreover, let $\mathcal{P}$ be a barren transaction in $\langle G,\Lambda\rangle.$
Notice that there exists a rendition $\rho'$ of $\langle H, \Lambda \rangle$ in $\Delta$ which can be obtained by combining $\rho$ with a vortex-free rendition of the strip society of $\mathcal{P}$ such that $c_0$ is divided into two vortices, $c_1$ and $c_2.$
We call these two vortices the \defi{residual vortices} of $\rho$ and $\mathcal{P}.$

\begin{lemma}\llabel{lemma_extraction_from_th}
Let $d,x,y,t_0,b,h,\ell,p,$ and $s_0$ be positive integers where
\begin{align*}
t\coloneqq& \ell\cdot h(2t_0+10bd)+2t_0+2,\\
s\coloneqq& s_0+4t+2t_0+3,\text{ and}\\
p\coloneqq& 3t + 2s + 2.
\end{align*}
Let $(G,\delta,W)$ be a $d$-blooming and $(x,y)$-fertile $(t,b,h,s,\Sigma)$-orchard, $\langle H,\Lambda\rangle$ be one of its thicket societies together with its nest $\mathcal{C}=\{C_{0}, C_1, \dots, C_s, C_{s+1} \},$ and $\mathcal{P}$ be a $d$-cultivated transaction of order $p$ in $\langle H,\Lambda\rangle.$

We assume that either $\langle H,\Lambda\rangle$ has a cross, or its folio contains at least one society that is not a society of $d\text{-}\mathsf{folio}(\delta).$
Let $\mathcal{O}$ be the collection of societies of detail $d$ that belong to the folio of the parcels of $\mathcal{P}$ but not to $d$-$\mathsf{folio}(\delta),$ and assume $|\mathcal{O}|\leq \ell.$

Let $\delta^*=(\Gamma^*,\mathcal{D}^*)$ be a $\Sigma$-decomposition of $G$ obtained by combining $\delta$ with a vortex-free rendition of the strip society of $\mathcal{Q}$ and let $c_1,$ $c_2$ be the residual vortices of $\delta$ and $\mathcal{Q}.$
Then,
\begin{enumerate}
    \item if 
    \begin{itemize}
        \item $\mathcal{O}\neq\emptyset,$ or
        \item $\mathcal{O}=\emptyset$ but for each $i \in [2]$ there either exists a simple society $O_i \notin d\text{-}\mathsf{folio}(\delta)$ such that $O_i$ belongs to the $d$-folio of a vortex society of $c_{i}$ or there exists a cross on the vortex society of $c_i,$
    \end{itemize}
    there exists a $d$-blooming and $(x,y)$-fertile $(t_0,b,h,s_0,\Sigma)$-orchard $(G,\delta',W')$ where $\delta'=(\Gamma',\mathcal{D}'),$ $\Gamma'\subseteq \Gamma^*,$ and every cell of $\delta'$ that is not a cell of $\delta^*$ is a vortex.
    Moreover, every vortex of $\delta'$ is either a vortex of $\delta$ or contained in one of $c_1$ or $c_2.$
    \item Otherwise, there exists a pruning of $(G,\delta,W).$
\end{enumerate}
Moreover, there exists an algorithm that finds one of the two outcomes in time $\mathcal{O}(|G|^2).$
\end{lemma}

\begin{proof}
We break the proof into five steps as follows.
During the proof we keep one bit of information called the \defi{flag} which will determine which of the two outcomes of the lemma we end up in.
In case the bit is $1,$ i.\@ e.\@~ the \defi{flag is set}, we find the first outcome, otherwise, we will end up with a pruning of $(G,\delta,W).$
In the beginning, our bit equals $0$ and the flag is not set.
\begin{description}
    \item[Step 1] This step is to establish the general setup. We divide $\mathcal{P}$ into five transactions (see \autoref{fig_extraction_1}) which will be used for different parts of the construction later on.
    \item[Step 2] We then check if $\mathcal{O}\neq \emptyset.$
    If this is true we set the flag and introduce an intermediate setup that will allow us to extract a large greenhouse.
    Otherwise, the flag remains unset and we move on.
    \item[Step 3] The transaction $\mathcal{P}$ defines two new ``vortex regions''.
    For each of them, we test if it either \textsl{a)} has a cross, or \textsl{b)} its folio contains a society that does neither belong to the folio of $\delta,$ nor to the folio of the parcels of $\mathcal{P}.$
    If for each of the two regions one of \textsl{a)} or \textsl{b)} holds and the flag is unset, we set the flag.
    Otherwise, we do not touch the flag and move on.
    \item[Step 4] Now, if the flag has been set in one of the previous steps we now extract a greenhouse, which will be converted in a sequence of parterre by \autoref{lemma_planting_a_greenhause}, and two new thickets from the original thicket segment.
    This yields the first outcome of the assertion and our case distinction terminates here.
    \item[Step 5] This step is only reached if the flag is not set at any point.
    By our assumptions on $\langle H,\Lambda\rangle$ this implies that exactly one of the two vortex regions satisfies \textsl{a)} or \textsl{b)} from \textbf{Step 3} and every society of the folio of the parcels of $\mathcal{P}$ belongs to $d\text{-}\mathsf{folio}(\delta).$
    In this case, we make use of $\mathcal{P}$ to find a pruning of $(G,\delta,W).$
\end{description}

    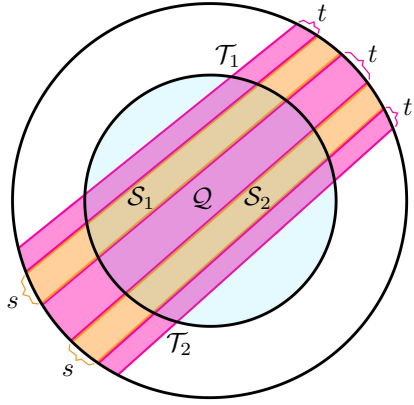
\begin{figure}[h]
        \centering
        \scalebox{1}{
        \begin{tikzpicture}[scale=1]
    
            \pgfdeclarelayer{background}
            \pgfdeclarelayer{foreground}
                
            \pgfsetlayers{background,main,foreground}
                
            \begin{pgfonlayer}{main}
            \node (C) [v:ghost] {};
    
                \pgftext{\includegraphics[width=6cm]{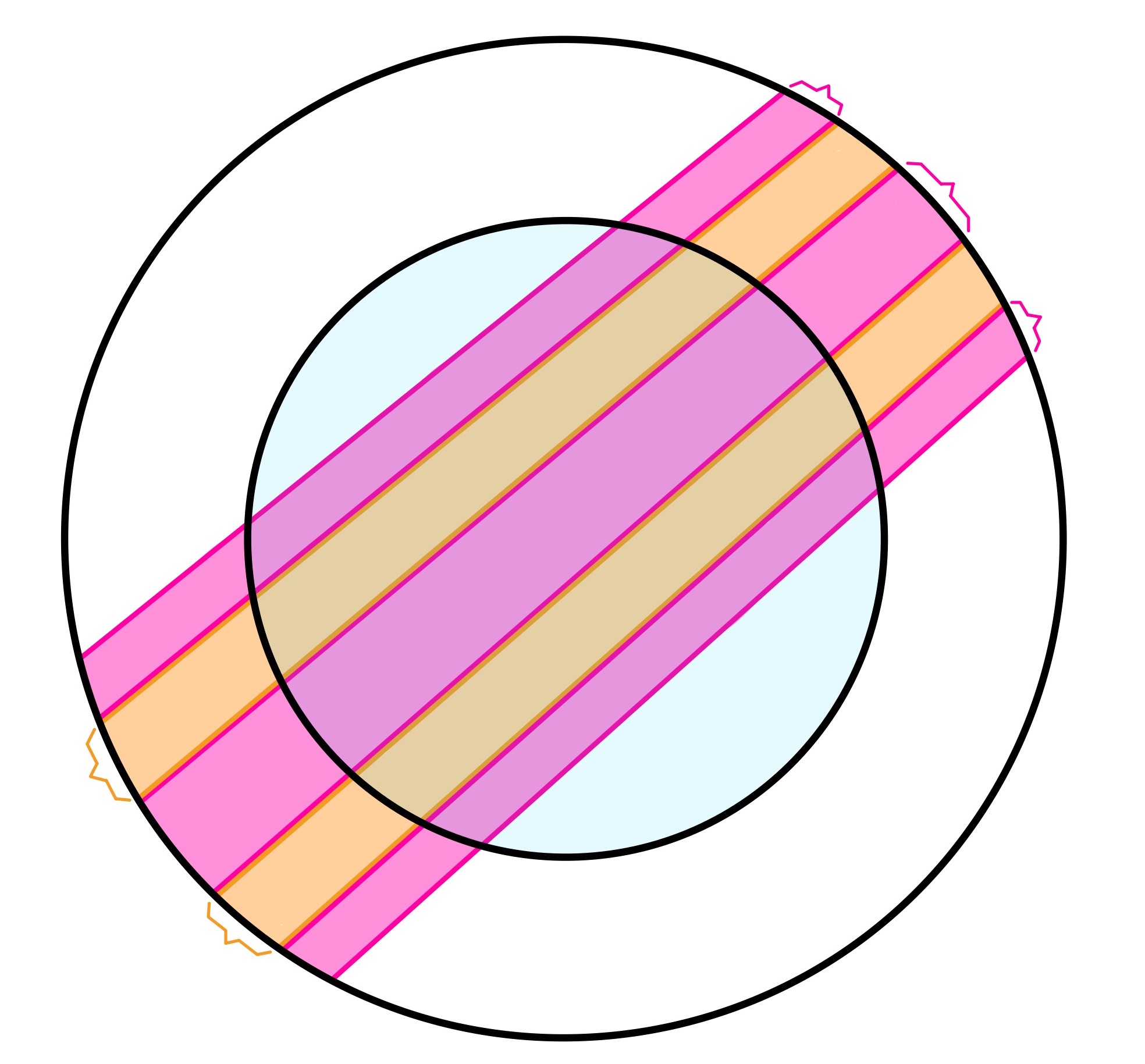}} at (C.center);

                \node (t1) [v:ghost,position=59.5:28.6mm from C] {$t$};
                \node (t2) [v:ghost,position=42:28.6mm from C] {$t$};
                \node (t3) [v:ghost,position=25:28.6mm from C] {$t$};

                \node (s1) [v:ghost,position=207.5:30mm from C] {$s$};
                \node (s2) [v:ghost,position=230:30mm from C] {$s$};

                \node (T1) [v:ghost,position=84:19mm from C] {$\mathcal{T}_1$};
                \node (T2) [v:ghost,position=257:20mm from C] {$\mathcal{T}_2$};

                \node (Q) [v:ghost,position=180:1.5mm from C] {$\mathcal{Q}$};

                \node (S1) [v:ghost,position=180:9.5mm from C] {$\mathcal{S}_1$};
                \node (S2) [v:ghost,position=0:6mm from C] {$\mathcal{S}_2$};

            \end{pgfonlayer}{main}
            
            \begin{pgfonlayer}{foreground}
            \end{pgfonlayer}{foreground}
    
            \begin{pgfonlayer}{background}
            \end{pgfonlayer}{background}
            
        \end{tikzpicture}
        }
        \caption{The partition of the transaction $\mathcal{P}$ into five smaller transaction in the proof of \autoref{lemma_extraction_from_th}.}
        \llabel{fig_extraction_1}
    \end{figure}

We now dive directly into discussing the details of each step.

\paragraph{Step 1: Setup}

Let $\mathcal{R}$ be the radial linkage of $\langle H,\Lambda\rangle$ and $\mathcal{C}.$
Let $I_1$ and $I_2$ be the two minimal segments of $\Lambda$ such that $\mathcal{P}$ is a $V(I_1)$-$V(I_2)$-linkage.
Moreover, by possibly shifting $\Lambda$ slightly, we may assume that every point of $I_1$ appears before any point of $I_2$ in $\Lambda.$
Let $x_1,x_2,\dots,x_{p}$ be the endpoints of the paths in $\mathcal{P}$ in $V(I_1)$ ordered according to their appearance in $\Lambda,$ and $y_1,y_2,\dots,y_p$ be the endpoints of the path in $\mathcal{P}$ in $V(I_2).$
We may number $\mathcal{P}=\{ P_1,P_2,\dots,P_p\}$ such that $P_i$ has endpoints $x_i$ and $y_i$ for each $i\in[p].$
With this, we are ready to define the five transactions that provide the framework for all following construction steps.
\begin{align*}
\mathcal{T}_1 & \coloneqq \{ P_2,\dots, P_{t+1}\}\\
\mathcal{S}_1 & \coloneqq \{ P_{t+2},\dots,P_{t+s+1}\}\\
\mathcal{Q} & \coloneqq \{ P_{t+s+2},\dots,P_{2t+s+1} \}\\
\mathcal{S}_2 & \coloneqq \{ P_{2t+s+2},\dots,P_{2t+2s+1}\}\\
\mathcal{T}_2 & \coloneqq \{ P_{2t+2s+2},\dots,P_{3t+2s+1} \}
\end{align*}
Notice that, with $\mathcal{P}$ being $d$-cultivated, each of the five transactions above is also $d$ cultivated and any member of the $d$-folio of a parcel of one of these transactions belongs to the $d$-folio of every parcel of any of these five transactions.
See \autoref{fig_extraction_1} for an illustration.

We also define two disks $\Delta_1$ and $\Delta_2.$
Let $U_1$ be the cycle consisting of a subpath of $C_{s_0+2t+2t_0+2}$ together with a subpath of $P_{t+s+1}$ whose $C_s$-avoiding disk also avoids the paths in $\mathcal{Q}.$
Similarly, let $U_2$ be the cycle consisting of a subpath of $C_{s_0+2t+2t_0+2}$ together with a subpath of $P_{2t+s+2}$ whose $C_s$-avoiding disk also avoids the paths in $\mathcal{Q}.$
The, for each $i\in[2],$ we define $\Delta_i$ to be the $C_s$-avoiding disk bounded by the trace of $U_i.$
The two disks $\Delta_i$ are our candidates for the new thickets.

Finally, notice that $\langle H,\Lambda \rangle$ has two renditions in the disk:
\begin{itemize}
\item $\delta_H\coloneqq \delta\cap H$ where $\delta_H$ has a single vortex cell, say $c_0$ and the vortex society of $c_H$ has a rendition in the disk with at most $x$ vortices all of which have depth at most $y.$
\item a rendition $\rho$ such that every cell of $\delta_H$ except $c_0$ is also a cell of $\rho$ and, moreover, $\rho$ contains a rendition of the strip society of $\mathcal{P}$ and $\rho$ has exactly two vortex cells, denoted by $c_1$ and $c_2,$ such that 
\begin{itemize}
\item for each $i\in[2],$ $c_i$ is fully contained in $\Delta_i,$ and
\item there exist non-negative integers $x_1$ and $x_2$ such that $x_1+x_2=x$ and the vortex society of $c_{i}$ has a rendition in the disk with at most $x_{i}$ vortices all of which have depth at most $y.$
\end{itemize}
\end{itemize}

At the moment $c_1$ and $c_2$ contain much more than what they need to contain in the end to form the residual vortices.
That is, at this stage they also contain the entire infrastructure that will eventually form their respective nest.
This will be adjusted in \textbf{Step 4} and \textbf{Step 5}.

\paragraph{Step 2: Preparing the greenhouse}

In case $\mathcal{O}=\emptyset$ we completely skip this step and move on to \textbf{Step 3}.

So for the following let us assume that $\mathcal{O}\neq\emptyset.$
In this case, we \textbf{set the flag}.
Moreover, we construct a certain part of the infrastructure which can easily be turned into a $(t,p)$-greenhouse later.
To this end, we start by selecting a linkage $\mathcal{Q}'$ as follows.
For each $Q\in\mathcal{Q}$ let $Q'$ be the shortest subpath of $Q$ which starts in $V(I_1),$ contains an edge which is completely drawn inside $c_0,$ and ends on $C_{s_0+2t+2t_0+1}.$
Then $\mathcal{Q}' \coloneqq\{ Q' \mid Q\in\mathcal{Q} \}.$
Notice that, with respect to $\mathcal{C}'\coloneqq\{ C_{s_0+2t+2t_0+2},C_{s_0+2t+2t_0+3},\dots,C_s\},$ $\mathcal{Q}'$ is an orthogonal radial linkage.
Moreover, the drawing of any path in $\mathcal{Q}'$ is disjoint from $\Delta_1\cup \Delta_2.$

Now, we may select two paths $T_{s_0 + 2t + 2t_0 + 1}$ and $T'_{s_0+ 2t + 2t_0 + 1}$ as follows:
Let $q^+_1$ and $q^+_2$ be the endpoint of $P_{t+s+2}'$ and $P_{2t+s+1}'$ on $C_{s_0 + 2t + 2t_0 + 1}$ respectively.
Then $C_{s_0+2t+2t_0+1}$ contains a unique $q^+_1$-$q^+_2$-subpath which contains all endpoints of $\mathcal{Q}'$ on $C_{s_0+2t+2t_0+1}.$
Let $T_{s_0+2t+2t_0+1}$ be this subpath.
Similarly, let $q^-_1$ and $q^-_2$ be the first vertex of $P_{t+s+2}'$ and $P_{2t+s+1}'$ on $C_{s_0+2t+2t_0+1}$ respectively which is encountered when traversing these paths starting in $V(I_1).$
Then $C_{s_0+2t+2t_0+1}$ contains a unique $q_1^{-}$-$q_2^{-}$-subpath which does not contain any endpoint of $\mathcal{Q}'$ on $C_{s_0+2t+2t_0+1}.$
Let $T'_{s_0+ 2t + 2t_0 + 1}$ be this subpath.
Now notice that there exists a $\delta$-aligned disk $\Delta'$ which intersects the drawing of $G$ exactly in the nodes corresponding to $V(T_{s_0+2t+2t_0+1}) \cup V(P_{t+s+2}')\cup V(P_{2t+s+1}')\cup V(T'_{s_0+2t+2t_0+1}).$
Let $G'$ be the subgraph of $G$ drawn in the disk $\Delta'.$

Then, for every $j \in  [s_0+2t+2t_0+1],$ $C_j\cap G'$ contains two disjoint $V(P_{t+s+2}')$-$V(P_{2t+s+1}')$-subpaths, say $T_{j}$ and $T_{j}'.$
Now, observe that $\cupall \mathcal{Q} \cup\{ T_j \cup T'_{j} \mid j \in [s_0+2t+2t_0+1] \}$ contains a $(2t + 1, t)$-wall $W_Q$ together with a linkage $\mathcal{Q}''$ of order $p$ which links the bottom row of $W_Q$ to $V(I_1)$ such that each path in $\mathcal{Q}''$ contains an edge which is drawn inside $c_0,$ for each column of $W_Q$ there is a unique member of $\mathcal{Q}''$ which has an endpoint in this column, and moreover $W_Q$ is chosen to avoid the cycles $\{C_{t+1}, \dots, C_{s}\}.$
Finally, notice that for every plot $X$ of $\mathcal{Q}$ there exists a $\delta$-aligned disk whose boundary intersects the drawing of $\rho$ only in vertices of $V(I_1),$ the bottom row of $W_Q,$ and two consecutive members of $\mathcal{Q}''$ such that $X$ is contained in this disk.

With this our preparations are complete.
See \autoref{fig_extraction_2} for an illustration of some of the objects defined above.
Notice that, if we made $\mathcal{Q}''$ coterminal with $\mathcal{R}$ within the partial nest $\mathcal{C}',$ we would immediately obtain a $(t, t)$-greenhouse.

\paragraph{Step 3: Preparing the ``split''}

For each $i\in[2],$ let $\langle H_i,\Lambda_i \rangle$ be a vortex society of $c_i$ and let $\mathcal{O}_i\coloneqq d\text{-}\mathsf{folio}(G,\Lambda)\setminus d\text{-}\mathsf{folio}(\delta).$
We say that $c_i$ is \defi{promising} if $\langle H_i,\Lambda_i\rangle$ has a cross or $\mathcal{O}_i\neq\emptyset.$

If there is $i\in[2]$ such that $c_i$ is not promising, we move on to the next step.
Otherwise, $c_i$ is promising for both $i\in[2].$
In this case, we \textbf{set the flag} (if it has not been set yet) and move on to the next step.

\begin{figure}[h]
\centering
\scalebox{1}{
\begin{tikzpicture}[scale=1]
    
            \pgfdeclarelayer{background}
            \pgfdeclarelayer{foreground}
                
            \pgfsetlayers{background,main,foreground}
                
            \begin{pgfonlayer}{main}
            \node (C) [v:ghost] {};
    
                \pgftext{\includegraphics[width=11cm]{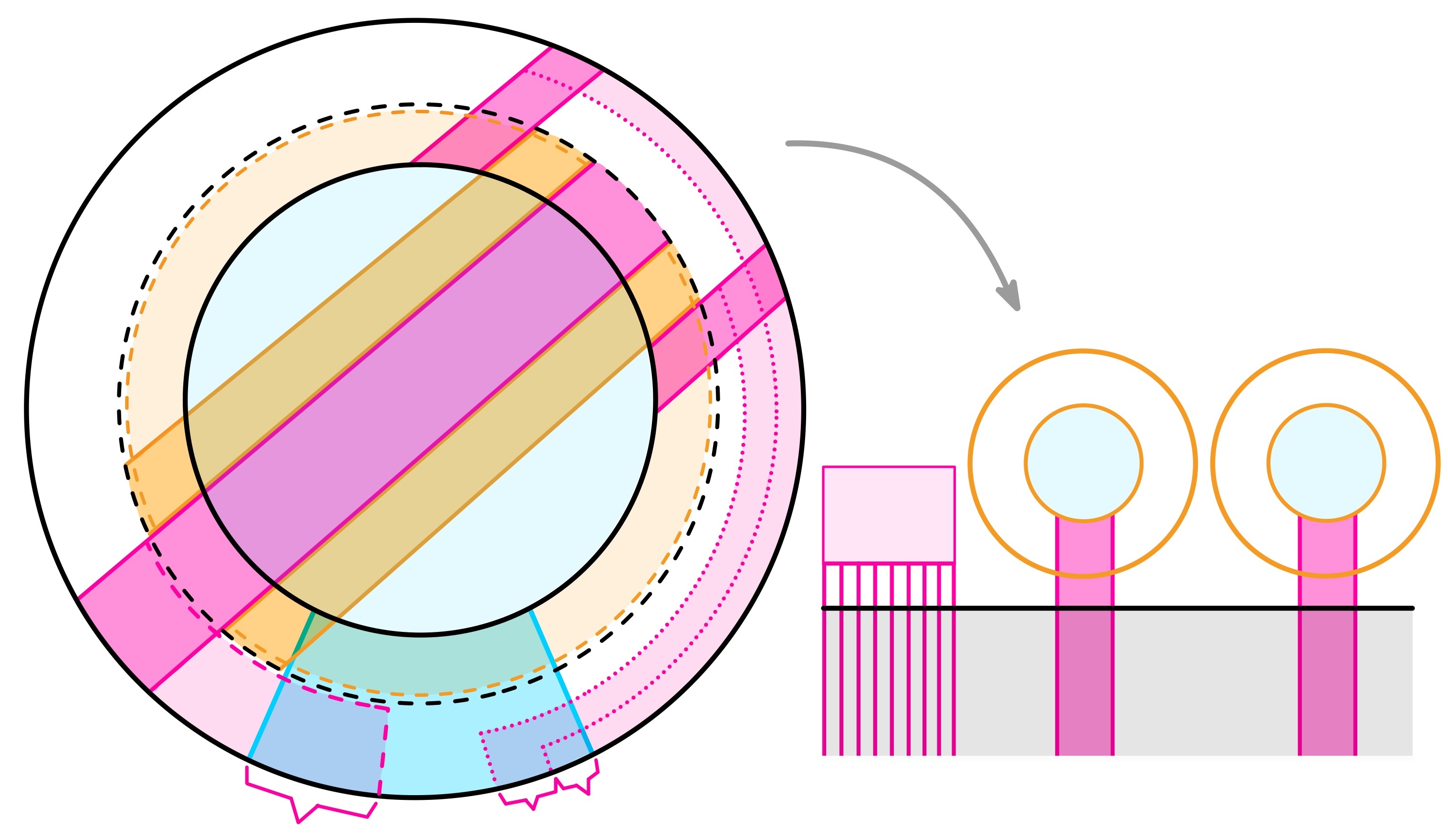}} at (C.center);

                \node (R') [v:ghost,position=200:46.5mm from C] {$\mathcal{R}'$};
                \node (R*_1) [v:ghost,position=208:44mm from C] {$\mathcal{R}^*$};
                \node (R*_2) [v:ghost,position=250:23.5mm from C] {$\mathcal{R}^*$};
                \node (R) [v:ghost,position=229:33.5mm from C] {$\mathcal{R}$};
                \node (T_1') [v:ghost,position=126:33mm from C] {$\mathcal{T}_1'$};
                \node (T_2') [v:ghost,position=62:13.6mm from C] {$\mathcal{T}_2'$};

            \end{pgfonlayer}{main}
            
            \begin{pgfonlayer}{foreground}
            \end{pgfonlayer}{foreground}
    
            \begin{pgfonlayer}{background}
            \end{pgfonlayer}{background}
            
        \end{tikzpicture}
        }
\caption{The ``splitting'' of a thicket into one large greenhouse and two smaller thickets in the proof of \autoref{lemma_extraction_from_th}.}
\llabel{fig_extraction_2}
\end{figure}
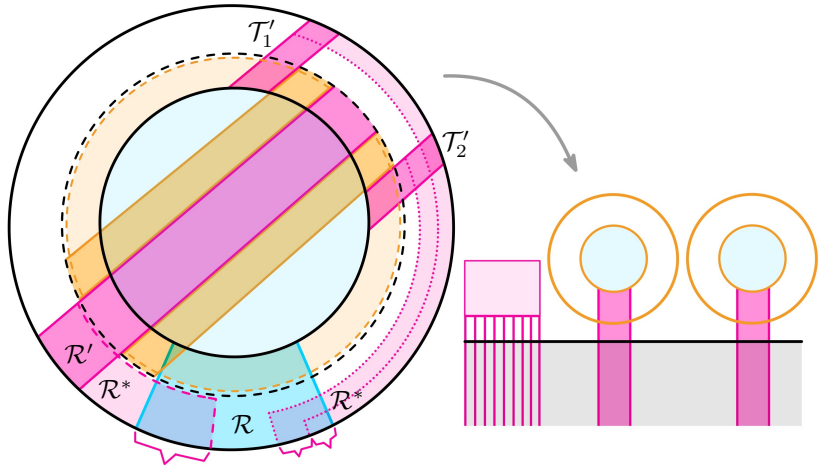

\paragraph{Step 4: The Flag has been set}
If we reach this point and the flag has not been set so far, we directly move on to \textbf{Step 5}.
Otherwise, the flag has been set which means that at least one of the following is true.
\begin{enumerate}
    \item $\mathcal{O}\neq\emptyset.$
    \item $c_i$ is promising for both $i\in[2].$
\end{enumerate}
Let $I\subseteq [2]$ be the maximal set such that $c_i$ is promising for every $i\in I.$

We now fix the following radial linkages which will be used to connect the newly constructed objects to a large portion of the original walloid $W.$

For every $Q\in\mathcal{Q}''$ as constructed in \textbf{Step 2} let $Q'$ be the shortest $V(I_1)$-$V(C_{t+1})$-subpath of $Q.$
We set
\begin{align*}
    \mathcal{Q}'''\coloneqq \{ Q' \mid Q\in\mathcal{Q}'' \}.
\end{align*}
For every $T\in\mathcal{T}_1\cup \mathcal{T}_2$ let $T'$ be the shortest $V(I_2)$-$V(C_{t+1})$-subpath of $T.$
We then set
\begin{align*}
    \mathcal{T}_1'&\coloneqq \{ P_i' \mid i\in[2,t_0+1] \}\text{ and}\\
    \mathcal{T}_2'&\coloneqq \{ P_i' \mid i\in[3t+2s+1-t_0,3t+2s+1]\}.
\end{align*}
In total this means that $\mathcal{R}'\coloneqq \mathcal{Q}'''\cup\mathcal{T}_1'\cup\mathcal{T}_2'$ is an orthogonal radial linkage of order $2t_0+t.$ 
We now call upon \autoref{easy_consequence_of_coterminality} on the (partial) nest $\{C_{t + 1},\dots,C_s \}$ to obtain an orthogonal radial linkage $\mathcal{R}^*$ from $\mathcal{R}'$ such that $\mathcal{R}^*$ is coterminal with $\mathcal{R}$ up to level $C_{2t_0+2t+1}$ and $V(\mathcal{R}^*)\subseteq V(\mathcal{R}) \cup V(\mathcal{R}')\cup \bigcup_{i\in[t + 1, 2t_{0}+2t]}V(C_i).$
Moreover, the endpoints of the paths in $\mathcal{R}^*$ which do not belong to $V(\Lambda)$ coincide with the endpoints of the paths in $\mathcal{R}'$ which do not belong to $V(I_1)\cup V(I_2).$
See \autoref{fig_extraction_2} for an illustration.
For each $i\in I$ let
\begin{align*}
    \mathcal{T}''_i\coloneqq \{ R\in\mathcal{R}^* \mid R-V(\Lambda)\text{ shares an endpoint with some path from }\mathcal{T}_i' \}.
\end{align*}

Next, we need to create new nests for each of the future thicket segments arising from $c_i,$ $i\in I.$
We will define, for each $i \in [I],$ a cycle family $\mathcal{C}_i'$ of $s_{0} + 2$ many concentric cycles ``around'' $c_{i}.$

Fix some point $x_1$ in the interior of $c_1$ which belongs to $c_0$ and is disjoint from the disk defined by the strip society of $\mathcal{P}.$
For each $j \in [s_{0} + 2],$ let $C^1_{j}$ be the cycle contained in $C_{t + 1 + j} \cup P_{t + 1 + j}$ which is fully drawn within $\Delta_1$ and which separates $x_1$ from the trace of $C_s.$
Then
\begin{align*}
    \mathcal{C}_1'\coloneqq \{ C^1_j \mid j \in [s_{0} + 2] \}.
\end{align*}
Similarly, we fix some point $x_2$ in the interior of $c_2$ which belongs to $c_0,$ is fully drawn within $\Delta_2,$ and does not belong to the disk defined by the strip society of $\mathcal{P}.$
Then, for each $j \in [s_{0} + 2],$ let $C^2_{j}$ be the cycle contained in $C_{t + 1 + j} \cup P_{2t+s+1+j}$ whose trace separates $x_2$ from the trace of $C_s.$
Then 
\begin{align*}
    \mathcal{C}_2' \coloneqq \{ C^2_j \mid j \in [s_{0} + 2] \}.
\end{align*}

Next, notice that since $\mathcal{T}''_{i}$ is an railed linkage that is orthogonal to the (partial) nest $\{C_{t + 1}, \dots,C_s \},$ it must also be orthogonal to $\mathcal{C}_i'.$

We now select, for every $i\in I,$ $t_0$ members of $\mathcal{T}_i'',$ excluding its first and last path, and collect them into the radial linkage $\mathcal{T}_i^*.$
Similarly, we set $\mathcal{C}_i^*\coloneqq \{ C^i_2, \dots, C^i_{s_0+1}\}.$
Finally, we select $c_i^*$ to be the cell obtained from the disk bounded by the trace of $C^i_1.$
We know that for each $i\in [2] \setminus I,$ the vortex society of $c_i$ has a vortex-free rendition in the disk and every member of its $d$-folio is also a member of the $d$-folio of $\delta.$
It follows that we may augment $\rho$ to be a rendition $\rho^*$ of $\langle H,\Lambda\rangle$ in the disk where exactly $\{ c_i^* \mid i\in I\}$ are the vortex cells of $\rho^*.$
Moreover, if $c_0$ had a rendition in the disk with at most $x$ many vortices, each of depth at most $y,$ then it follows from our construction that there exist non-negative integers $x_1$ and $x_2$ such that $x_1+x_2=x$ and for each $i\in I,$ the vortex society of $c^*_i$ has a rendition in the disk with at most $x_i$ many vortices, each of depth at most $y.$

What remains is to describe the final step of the construction.
That is, we have to describe how to obtain the desired $(x,y)$-fertile $(t_0,b,h,s_0,\Sigma)$-orchard.
To do this, we extend the radial linkage $\mathcal{R}^*$ through the circles of our walloid $W$ all the way down to its simple face.
The infrastructure described above yields, a large greenhouse and for each $i\in I$ a new thicket segment with a nest and radial linkage of order $t_0.$
Now, \autoref{lemma_planting_a_greenhause} allows us to transform the greenhouse into a sequence of $(t_0,b,h)$-parterre segments such that for each member of $\mathcal{O}$ there exists one such parterre segment.
Together with the (up to two) new thicket segments we have now reached the situation of \autoref{lemma_moving_a_parterre}.
By applying \autoref{lemma_moving_a_parterre} and then forgetting about some of the excess infrastructure of the walloid $W$ and its segments, we finally obtain a $(x,y)$-fertile $(t_0,b,h,s_0,\Sigma)$-orchard $(G,\rho^*,W^*)$ which has the desired properties by construction.

\paragraph{Step 5: Pruning}
All that is left is to discuss the case where the flag was not set.

In the case where neither $c_1$ not $c_2$ is promising it follows that the vortex society of $c_0$ does not contain a cross, nor does its $d$-folio contain a society that does not belong to $d\text{-}\mathsf{folio}(\delta).$
This, however, would mean that the entire thicket could have been forgotten to begin with and thus we may ignore this case.

The only remaining option is that exactly one of $c_1$ or $c_2$ is promising.
Without loss of generality we may assume that $c_2$ is not promising.

For every path $T$ in $\mathcal{T}_1$ let $T'$ be a shortest $V(I_2)$-$V(C_1)$-subpath of $T.$
We set
\begin{align*}
    \mathcal{T}\coloneqq \{T' \mid T\in\mathcal{T}_1 \}.
\end{align*}
Then $\mathcal{T}$ is a radial linkage of order $t$ that is orthogonal to $\mathcal{C}.$
We then apply \autoref{easy_consequence_of_coterminality} to make $\mathcal{T}$ coterminal with $\mathcal{R}$ up to level $t+1.$
Let $\mathcal{L}$ be the resulting orthogonal radial linkage.

Next we define our new nest.
This construction is similar to the construction from \textbf{Step 4}.
We begin by fixing a point $x_1$ in the interior of $c_1\cap c_0$ which avoids the disk induced by the strip society of $\mathcal{P}.$
Now, for each $i \in [0, s+1],$ let $C_i^*$ be the cycle in $C_i \cup P_{t+2+i}$ whose trace in $\delta$ separates $x_1$ from the trace of $C_{s+1}$ in $\delta.$
Moreover, let $\mathcal{C}^*\coloneqq \{ C_i^* \mid i \in [s]\}.$

Now, notice that the radial linkage $\mathcal{L}$ must also be orthogonal to the new nest $\mathcal{C}^{*}.$
Let $c_0^*$ be a cell obtained from the disk bounded by the trace of $C_0^*$ which contains the point $x_1.$
It follows that $\langle H,\Lambda\rangle$ has a rendition $\rho^*$ in the disk where $c_0^*$ is its only vortex cell, and $d\text{-}\mathsf{folio}(\rho^*)\subseteq d\text{-}\mathsf{folio}(\delta).$

With this we may augment $\delta$ to a $\Sigma$-decomposition $\delta'$ of $G$ by using $\rho^*$ and adjust $W$ by replacing the thicket with nest $\mathcal{C}$ and radial linkage $\mathcal{R}$ with a thicket segment with nest $\mathcal{C}^*$ and radial linkage $\mathcal{L},$ thereby obtaining a new walloid $W'.$
It is now straightforward to verify that $(G,\delta',W')$ is indeed a pruning of $(G,\delta,W)$ since no path in $\mathcal{T}_2$ (except its first one), which is an exposed transaction on $\langle H,\Lambda\rangle,$ belongs to any of the thickets of $(G,\delta',W').$
With this, our proof is complete.
\end{proof}

\subsection{The proof of \autoref{ripe_orchard}}

We introduce the following convention for the rest of this section.
Let us define the \defi{$d$-folio number} to be
\begin{align*}
    \mathsf{d}^{\star}\coloneqq f_{\ref{obs_size_dfolio}}(d).
\end{align*}
This is purely for convenience in the proof of \autoref{ripe_orchard} as we need to have a lot of discussion around the maximum number of different members of $d\text{-}\mathsf{folio}(\delta).$
By \autoref{obs_size_dfolio} this number is upper bounded by $\mathsf{d}^{\star}.$

Before we begin with the proof of \autoref{ripe_orchard}, we also fix three recursive functions $\mathsf{s},$ $\mathsf{t},$ and $\mathsf{p}.$
These functions represent the recursive application of \autoref{lemma_extraction_from_th} and ensure that the infrastructure of the final walloid will have order $t^*,$ while its thickets are surrounded by nests of order $s^*$ for any choice of $t^*$ and $s^*.$
Let $s^*,$ $t^*,$ and $z$ be positive integers and $h,b,d\in\mathbb{N}.$
\begin{align*}
   \mathsf{t}(t^*,0) &\coloneqq t^*\\
   \mathsf{t}(t^*,z) &\coloneqq \mathsf{d}^{\star} \cdot h\cdot\big( 2\mathsf{t}(t^*,z-1) + 10bd \big) + 2\mathsf{t}(t^*,z-1)+2\\
   ~&~\\
   \mathsf{s}(s^*,t^*,0) &\coloneqq s^*+6t^*+12\\
   \mathsf{s}(s^*,t^*,z) &\coloneqq s^*+1 + \mathsf{s}(s^*,t^*,z-1) + 4\mathsf{t}(t^*,z) + 2\mathsf{t}(t^*,z-1) + 3\\
\end{align*}
Please notice that the three functions above also depend on $h,b$ and, in particular on $d$ as even $\mathsf{d}^*$ depends on $d.$
However, we consider these numbers to be fixed and wish to highlight mostly the dependency on $s^*$ and $t^*$ as these are the numbers that will mostly be modulated in later parts of the proof of our main theorem.
From the recursive definitions of the functions above, one can derive the following observations on their order:
\begin{align*}
    \mathsf{t}(t^*,z) &\in (\mathsf{d}^{\star}\cdot h)^{\mathcal{O}(z)} + \mathcal{O}(zbd) + t^*\text{ and}\\
    \mathsf{s}(s^*,t^*,z) &\in (\mathsf{d}^{\star}\cdot h)^{\mathcal{O}(z)} + \mathcal{O}\big(z(s^*+bd) \big) + t^*.
\end{align*}
Moreover, whenever we start with an orchard that is $(x,y)$-fertile, the following upper bound $z^*$ on the maximum value $z$ may attain will follow from our proof.
\begin{align*}
    z^* \coloneqq x+\mathsf{d}^{\star}\cdot (\mathsf{d}^{\star}\cdot hb+1)2^{\mathsf{d}^{\star}}
\end{align*}
We also fix\footnote{Recall that $\mathsf{d}^{\star}$ is not a precise value but hides some constant within its $\mathcal{O}$-notation.} a value $p^*$ which will act as an upper bound on the depth of the vortices of $(G,\delta,W)$ as follows.
\begin{align*}
    p^* \coloneqq x^2y + (t^*\cdot s^*\cdot (h+1)^{x+hb})^{2^{\mathsf{d}^{\star}}}
\end{align*}
\smallskip

Based on these functions let us briefly provide explicit bounds for the four functions from the statement of \autoref{ripe_orchard}.
\begin{align*}
f^1_{\ref{ripe_orchard}}(t^*,b,h,s^*,x,y,d) &\coloneqq \mathsf{t}(t^*,z^*) \in t^*+h^{x+hb\cdot 2^{2^{\mathcal{O}(d^2)}}},\\
f^2_{\ref{ripe_orchard}}(t^*,b,h,s^*,x,y,d) &\coloneqq \mathsf{s}(s^*,t^*,z^*) \in t^*+ s^*\cdot(x+hb)2^{2^{\mathcal{O}(d^2)}} + h^{x+hb\cdot 2^{2^{\mathcal{O}(d^2)}}},\\
f^3_{\ref{ripe_orchard}}(t^*,b,h,s^*,x,y,d) &\in x + hb2^{2^{\mathcal{O}(d^2)}} \text{, and}\\
f^4_{\ref{ripe_orchard}}(t^*,b,h,s^*,x,y,d) &\in x^2y + (t^*\cdot s^*\cdot (h+1)^{x+hb})^{2^{2^{\mathcal{O}(d^2)}}}
\end{align*}
Please note that we do not derive these bounds explicitly in the proof below.
The bounds we provide will all be in terms of $z^*,$ $\mathsf{t},$ and $\mathsf{s},$ however, the bounds above then follow from the respective bounds on these functions.

With this, we are ready for the proof of \autoref{ripe_orchard}.

\begin{proof}[Proof of \autoref{ripe_orchard}]
We prove the following slightly stronger claim by induction on $z.$

\paragraph{Inductive claim.}
There exists an algorithm that, given an $(x,y)$-fertile and $d$-blooming single-thicket $(\mathsf{t}(t^*,z),b,h,\mathsf{s}(s^*,t^*,z),\Sigma)$-orchard $(G,\delta,W),$ outputs a $(z^*,p^*)$-ripe and $d$-blooming $(t^*,b,h,s^*,\Sigma)$-orchard in time $\mathcal{O}(|G|^3).$
\medskip

Here $z$ represents an upper bound on the number of times \autoref{lemma_extraction_from_th} can yield the first outcome before, in the resulting orchard $(G,\delta,W),$ $d\text{-}\mathsf{folio}(\delta)$ contains all possible societies of detail $d$ and $\delta$ has $x$ vortices, each with a cross on its society.
The two main points here are that, for an $(x,y)$-fertile orchard, repeated applications of \autoref{lemma_extraction_from_th} can never produce strictly more than $x$ thickets such that the thicket society of each of them has a cross.
Moreover, any thicket whose thicket society does not have a cross must contain some detail-$d$-society in its folio that does not belong to $d\text{-}\mathsf{folio}(\delta)$ since otherwise we could simply discard the thicket as it does not pose an obstacle to embeddability.

Under these principles, there are three situations under which an application of \autoref{lemma_extraction_from_th} does not yield a pruning (which is a situation we will handle online).
\begin{enumerate}
    \item the outcome increases the number of thickets with a cross on their respective thicket societies,
    \item the outcome produces a parterre segment for some detail-$d$-society that did not belong to $d\text{-}\mathsf{folio}(\delta)$ before, or
    \item the outcome increases the number of ``flat'' thickets whose societies contain some detail-$d$-society as a minor which does not belong to $d\text{-}\mathsf{folio}(\delta).$
\end{enumerate}
As discussed above, the first case can only occur $x$ times.
Moreover, the second case can occur at most $\mathsf{d}^{\star}$ times since this is an upper bound on the total number of non-isomorphic societies of detail $d.$
For the third case, we will see that whenever the total number of ``flat'' thickets goes above a certain threshold we are able to apply a pigeonhole argument to turn some of them into parterre segments and thereby increase $|d\text{-}\mathsf{folio}(\delta)|.$
For this reason, also the third case can only occur a bounded number of times.

\paragraph{Base of the induction.}
For the base of the induction, we consider the case $z=0.$
This case reflects the situation where $(G,\delta,W)$ has $x$ thickets which have a cross on their thicket society and where $d\text{-}\mathsf{folio}(\delta)$ contains all societies of detail $d.$
In this situation, by the discussion above, there does not exist a thicket that does not have a cross on its society.
Let $\langle H,\Lambda\rangle$ be the thicket society of some thicket of $(G,\delta,W).$
We know that $(G,\delta,W)$ is an $(x,y)$-fertile and $d$-blooming $(t^*,b,h,s^*+6t^*+12,\Sigma)$-orchard.
What remains to show is that it can be turned into a $(z^*,p^*)$-ripe and $d$-blooming $(t^*,b,h,s^*+6t^*+12,\Sigma)$-orchard.

Let $\mathcal{C}=\{ C_1,\dots,C_{s^*+6t^*+12}\}$ be the nest of $\langle H,\Lambda\rangle$ and let $\Delta$ be the $C_{s^*+6t^*+11}$ avoiding disk defined by the trace of $C_{6t^*+12}.$
Moreover, let $\langle H',\Lambda'\rangle$ be a society defined by $\Delta.$
Notice that $\langle H',\Lambda'\rangle$ has a nested cylindrical rendition $\rho$ in $\Delta$ with a single vortex, say $c_0,$ and nest $\mathcal{C}'\coloneqq \{ C_{1},\dots,C_{6t^*+11}\}.$
Moreover, by cropping the radial linkage of $\langle H,\Lambda\rangle$ we obtain a radial linkage $\mathcal{R}'$ of order $t^*$ for $\langle H',\Lambda'\rangle.$

If $\langle H',\Lambda'\rangle$ is of depth at most
\begin{align*}
    & 12t^*+24 + (x+1)\big( 2xy + (6t^*+12)((15t^*+26)^{2^{\mathsf{d}^{\star}}} + 2) \big)-1\\
    =~ & 2s' + (x+1)\big( 2xy + s'((3t^* + 2s' +2)^{2^{\mathsf{d}^{\star}}} + 2)  \big)-1 
\end{align*}
where $s'\coloneqq 6t^*+12=|{\mathcal{C}}'|,$ we are done.
Hence, we may assume that there exists a transaction $\mathcal{P}_0$ of order $12t^*+24 + (x+1)\big( 2xy + (6t^*+12)((15t^*+26)^{2^{\mathsf{d}^{\star}}} + 2) \big)$ in $\langle H',\Lambda'\rangle.$

By \autoref{exposure}, there is either a \textbf{pruning} of $(G,\delta,W),$ or there exists a transaction $\mathcal{P}_1\subseteq\mathcal{P}_0$ of order $(x+1)\big( 2xy + (6t^*+12)((15t^*+26)^{2^{\mathsf{d}^{\star}}} + 2) \big)$ which is exposed. 

By \autoref{isolatedstrips}, $\mathcal{P}_1$ must contain a planar transaction $\mathcal{P}_2$ of order $(6t^*+12)((15t^*+26)^{2^{\mathsf{d}^{\star}}} + 2)$ whose strip society is isolated, separating, and rural.

Then, by \autoref{orthogonal_transaction}, we can assume that (up to possibly slightly changing the cycles in $\mathcal{C}'$ and the paths in the radial linkage $\mathcal{R}'$) that there exists a barren transaction $\mathcal{P}'_{2}$ of order $(15t^*+26)^{2^{\mathsf{d}^{\star}}}.$

Next we apply \autoref{lem_cultivate_transaction} to obtain a trasaction $\mathcal{P}_3\subseteq \mathcal{P}'_2$ which is $d$-cultivated and of order $15t^*+26 = 3t^* + 2s' +2$ where $s' = |\mathcal{C}'|.$

This finally allows us to use \autoref{lemma_extraction_from_th} on $\langle H',\Lambda'\rangle$ and $\mathcal{P}_3.$
Notice that, since $(G,\delta,W)$ has exactly $x$ thickets which are not flat and $d\text{-}\mathsf{folio}(\delta)$ contains all linear societies of detail at most $d,$ the first outcome of \autoref{lemma_extraction_from_th} is impossible.
Thus, we must find a \textbf{pruning} of $(G,\delta,W)$ in time $\mathcal{O}(|G|^2).$

Notice that every time we find a pruning as the outcome of one of the above steps, at least one of the cycles of $\mathcal{C}'$ is ``pushed'' slightly closer to the vortex.
That means, at least one edge of $G$ is ``pushed'' further to the outside, or fully outside, of a thicket in each of these occurrences.
Since there are at most $|E(G)|$ many edges and only $6t^*+12$ many cycles in $\mathcal{C}',$ we cannot find a pruning more than $(6t^*+12)\cdot |E(G)|$ times.
As a result, and since we may assume that $G$ excludes a graph of bounded size as a minor which means $|E(G)|\in \mathcal{O}(|G|),$ after $\mathcal{O}(|G|^3)$ many iterations of the procedure above the only possible case left is that $\langle H', \Lambda' \rangle$ is of depth at most
\begin{align*}
    12t^*+24 + (x+1)\big( 2xy + (6t^*+12)((15t^*+26)^{2^{\mathsf{d}^{\star}}} + 2) \big)-1. 
\end{align*}
Since there are only $x$ thickets, these arguments establish the base of the induction.

\paragraph{The inductive step.}
The arguments that bounded the number of times we may encounter a pruning in the base of our induction still apply when we find a pruning in the inductive step.
In essence, we cannot encounter a pruning more than $\mathcal{O}(|G|)$ in total and thus, by iterating the arguments enough times, we must always encounter an outcome that is not a pruning.
To avoid repetition, we continue to properly mark all places in the arguments were pruning might occur, but we will always assume that in these cases we encounter one of the non-pruning outcomes.
Formally, we can establish such a situation by assuming that $(G,\delta,W)$ is a minimal counterexample to out inductive claim with respect to improvements via pruning.
Please note that, according to \autoref{exposure}, it is imperative that we perform all possible improvements via pruning, if we are to also guarantee in the end that not only the vortex societies within thickets have bounded depth but also their entire thicket societies.

From here on we may assume that $z\geq 1.$

\paragraph{Step 1: No useless thickets.}
First, we may assume that for every thicket segment $\widehat{W}$ whose society does not have a cross, there exists some society of detail at most $d$ which is a minor of the thicket society of $\widehat{W}$ but does not belong to $d\text{-}\mathsf{folio}(d).$
We may assume this since any thicket that does not meet this criterion can safely be forgotten. 

\paragraph{Step 2: A bounded number of thickets.}
Next, let us assume that there exist more than $(\mathsf{d}^{\star}\cdot hb+1)2^{\mathsf{d}^{\star}}$ many thickets whose thicket society does not contain a cross.
Since there are at most $\mathsf{d}^{\star}$ many societies of detail at most $d$ by \autoref{obs_size_dfolio} we may assign to each thicket the $d$-folio of its thicket society.
Since there are at most $2^{\mathsf{d}^{\star}}$ many possible such $d$-folios, there must be at least $\mathsf{d}^{\star}\cdot hb+1$ many thickets with the same $d$-folio, say $\mathcal{F}'.$
Let $\mathcal{F}\coloneqq \mathcal{F}'\setminus d\text{-}\mathsf{folio}(\delta).$
Notice that $|\mathcal{F}|\geq 1$ by the assumption from \textbf{Step 1}.
Moreover, $z\geq |\mathcal{F}|$ must hold since we at least allow one ``split'' for each possible society of detail at most $d.$

Now, for each member $\langle H,\Lambda\rangle$ of $\mathcal{F}$ we wish to produce $h$ many $(\mathsf{t}(t^*,z-|\mathcal{F}|),b,H,\Lambda)$-flower segments which then can be gathered into a single $(\mathsf{t}(t^*,z-|\mathcal{F}|),b,h,H,\Lambda)$-parterre segment.
To achieve this, the amount of infrastructure of our orchard we must sacrifice is at most
\begin{align*}
    r\coloneqq |\mathcal{F}| \cdot h \cdot (2\mathsf{t}(t^*,z-|\mathcal{F}|) + bd)
\end{align*}
cycles from the nest and at most $2r$ paths from the radial linkage of every thicket.
By doing so, similar arguments as those for \autoref{lemma_moving_a_parterre} allow us to produce a new orchard whose walloid is of order $\mathsf{t}(t^*,z-|\mathcal{F}|)$ and whose $d$-folio equals $d\text{-}\mathsf{folio}(\delta)\cup \mathcal{F}.$
Since these arguments are essentially the same, we omit the discussion on how to move and grow the extracted flower segments through the thickets.
Notice that the following two inequalities follow directly from the recursive definitions of $\mathsf{s}$ and $\mathsf{t}$ for every pair $z\geq z'$ of positive integers.
\begin{align}
  \mathsf{t}(t^*,z-z') + 2(z'h(bd + \mathsf{t}(t^*,z-z'))) &\leq \mathsf{t}(t^*,z)\llabel{equ_t_bound}\\
  \mathsf{s}(s^*,t^*,z-z') + z'h(bd + \mathsf{t}(t^*,z-z')) &\leq \mathsf{s}(s^*,t^*,z)
  \llabel{equ_s_bound}
\end{align}

Therefore, if there are at least $x+(\mathsf{d}^{\star}\cdot hb+1)2^{\mathsf{d}^{\star}}+1$ many thicket segments, under the assumption from \textbf{Step 1}, we may first extract $|\mathcal{F}|\geq 1$ many new parterre segments and then use the induction hypothesis to complete the proof.
Thus, from now on we may assume that there are at most $x+(\mathsf{d}^{\star}\cdot hb+1)2^{\mathsf{d}^{\star}}$ many thicket segments in $(G,\delta,W)$ as desired.

\paragraph{Step 3: Splitting a thicket}
What follows is, essentially, a second iteration of the steps from the base of our induction.
The only two changes that occur are: 
\begin{enumerate}
\item The numbers are now bigger and depend directly on the functions $\mathsf{t}$ and $\mathsf{s}$ and
\item we may now also encounter the first outcome of \autoref{lemma_extraction_from_th} which will lead to a proper ``split'' of the thicket segment.
\end{enumerate}

We begin by selecting $\langle H,\Lambda\rangle$ to be the thicket society of some thicket of $(G,\delta,W)$ as before.

Let $s'\coloneqq \mathsf{s}(s^*,t^*,z-1)+4\mathsf{t}(t^*,z)+2\mathsf{t}(t^*,z-1)+3,$ let $\mathcal{C}=\{ C_1,\dots,C_{s^*+1+s'}\}$ be the nest of $\langle H,\Lambda\rangle$ and let $\Delta$ be the $C_{s^*+1+s'}$ avoiding disk defined by the trace of $C_{s'+1}.$
Moreover, let $\langle H',\Lambda'\rangle$ be a society defined by $\Delta.$
Notice that $\langle H',\Lambda'\rangle$ has a nested cylindrical rendition $\rho$ in $\Delta$ with a single vortex, say $c_0,$ and nest $\mathcal{C}'\coloneqq \{ C_{1},\dots,C_{s'}\}.$
Moreover, by cropping the radial linkage of $\langle H,\Lambda\rangle$ we obtain a radial linkage $\mathcal{R}'$ of order $\mathsf{t}(t^*,z)$ for $\langle H',\Lambda'\rangle.$

If $\langle H',\Lambda'\rangle$ is of depth at most
\begin{align*}
    & 2s' + (x+1)\big( 2xy + s'((3\mathsf{t}(t^*,z)+2s' +2)^{2^{\mathsf{d}^{\star}}}+2) \big)-1 
\end{align*}
We may, by transforming $\Delta$ into a cell of $\delta$ and declare it the vortex of our current thicket segment, obtain a thicket segment with an $s^*$-nested cylindrical rendition and a vortex of depth at most $2s' + (x+1)\big( 2xy + s'((3\mathsf{t}(t^*,z)+2s' +2)^{2^{\mathsf{d}^{\star}}}+2) \big)-1 $ for this segment.
In case this is the outcome for all thicket segments we are done.

Hence, we may assume that there exists a transaction $\mathcal{P}_0$ of order $2s' + (x+1)\big( 2xy + s'((3\mathsf{t}(t^*,z)+2s' +2)^{2^{\mathsf{d}^{\star}}}+2) \big)$ in $\langle H',\Lambda'\rangle.$

By \autoref{exposure}, there is either a \textbf{pruning} of $(G,\delta,W),$ or there exists a transaction $\mathcal{P}_1\subseteq\mathcal{P}_0$ of order $(x+1)\big( 2xy + s'((3\mathsf{t}(t^*,z)+2s' +2)^{2^{\mathsf{d}^{\star}}}+2) \big)$ which is exposed. 

By \autoref{isolatedstrips}, $\mathcal{P}_1$ must contain a planar transaction $\mathcal{P}_2$ of order $s'((3\mathsf{t}(t^*,z)+2s' +2)^{2^{\mathsf{d}^{\star}}}+2)$ whose strip society is isolated, separating, and rural.

Then, by \autoref{orthogonal_transaction}, we can assume that (up to possibly slightly changing the cycles in $\mathcal{C}'$ and the paths in the radial linkage $\mathcal{R}'$) that there exists a barren transaction $\mathcal{P}'_{2}$ of order $(3\mathsf{t}(t^*,z)+2s' +2)^{2^{\mathsf{d}^{\star}}}.$

Next we apply \autoref{lem_cultivate_transaction} to obtain a transaction $\mathcal{P}_3\subseteq \mathcal{P}_2$ which is $d$-cultivated and of order $3\mathsf{t}(t^*,z)+2s' +2.$
Recall that $s' = |\mathcal{C}'|.$
We are now ready to call upon \autoref{lemma_extraction_from_th} for $\langle H',\Lambda'\rangle$ and $\mathcal{P}_3.$
The second outcome of \autoref{lemma_extraction_from_th} is a \textbf{pruning} of $(G,\delta,W).$
Hence, we may assume that \autoref{lemma_extraction_from_th} returns the first outcome.
Let us discuss the numbers involved to see that we may now complete the proof by using the induction hypothesis.

For the purpose of applying \autoref{lemma_extraction_from_th} let us set
\begin{align*}
    t_0 &\coloneqq \mathsf{t}(t^*,z-1)\text{, and}\\
    s_0 &\coloneqq \mathsf{s}(t^*,z-1).
\end{align*}
It then follows from the recursive definition of $\mathsf{s}$ and $\mathsf{t}$ that the size of the nest, the radial linkage of $\langle H,\Lambda\rangle,$ and the transaction $\mathcal{P}_3$ are large enough such that \autoref{lemma_extraction_from_th} outputs an $(x,y)$-fertile $(\mathsf{t}(t^*,z-1),b,h,\mathsf{s}(s^*,t^*,z-1),\Sigma)$-orchard $(G,\delta',W')$ where $d\text{-}\mathsf{folio}(\delta)\subseteq d\text{-}\mathsf{folio}(\delta')$ and one of the following holds:
\begin{itemize} 
    \item $d\text{-}\mathsf{folio}(\delta)\subsetneq d\text{-}\mathsf{folio}(\delta'),$ or
    \item $W'$ has exactly one thicket segment more than $W.$
\end{itemize}
In both cases, the total number of times we are allowed to ``split'' the vortex has been reduced by at least one.
Hence, by applying the induction hypothesis for $z-1$ to $(G,\delta',W')$ the proof is complete.
\end{proof}

\section{Harvesting}\llabel{sec_harvest_crops}

The goal of this section is to prove the two lemmata of \autoref{main_harvesting_lemmata}, namely \autoref{my_crops_findslo} and \autoref{harvesting_lemma}.
The former shows how we can control the behavior of an extension, invading a well-insulated area of a $\Sigma$-decomposition, i.e., an area containing many concentric cycles around a vortex cell of the decomposition.
The latter shows how we can exploit the structure of such a controlled extension in order to show that whenever an extension that invades a well-insulated area cannot be hit with a bounded number of vertices within this area, then we can find a large packing of the invading part of the extension along the boundary of this area. 

\subsection{Lines in cylindrical renditions}
\llabel{lines_cylindrical}

In this subsection, we examine more closely the structure of minimal separators between disjoint segments of the society corresponding to a cylindrical rendition.
In particular, we show that their structure in the flat part of the cylindrical rendition defines a clear frontier in terms of what we call a \textsl{line}.

\medskip
We need the following proposition on the structure of a low-depth society. 

\begin{proposition}[\cite{RobertsonS90GMIX}]\llabel{take_from_GM_IX}
Let $\lin{G, \Lambda}$ be a linear society where $\Lambda = \lin{v_{1},\ldots,v_{l}}.$
If $\lin{G, \Lambda}$ has depth at most $w,$ then there is a sequence $\Bcal = \lin{B_i \mid i \in [l]}$ of subsets of $G$ such that 
\begin{enumerate}
\item $\cupall \Bcal = V(G),$
\item for every edge $xy \in E(G)$ there is some $i \in [l]$ such that $\{ x, y \} \subseteq B_{i},$
\item for every $i \in [l],$ $y_{i} \in B_{i},$
\item for every $i, j, h \in [l],$ if $i ≤ j ≤ h,$ then $X_{i} \cap X_{h} \subseteq X_{j},$ and 
\item\llabel{all_disj_bound} for every $i, j \in [l],$ if $i \neq j,$ then $|B_{i} \cap B_{i+1}| ≤ w.$
\end{enumerate}
 \end{proposition}
 
The proof of \autoref{take_from_GM_IX} builds $\Bcal$ by successively looking for separators along prefixes of the linear society $\Lambda$ (see also \cite[Theorem 12.2]{kawarabayashi2020quickly}). This requires a linear number of calls of a min-flow algorithm for flow, each of size bounded by the depth $w.$ So we can compute $\Bcal$ in time $\Ocal(w \cdot |G|^{2}).$

\paragraph{Fibers of a rendition.}
Let $\rho = (\Gamma, \mathcal{D},c)$ be a cylindrical rendition of a linear society $\lin{G,\Lambda}$ in a disk $\Delta$ around a cell $c.$ 
A \defi{fiber} of $\rho$ is any fiber $\mu = (M, T, \widehat{\Lambda})$ of $\lin{G,\Lambda}.$ 
Whenever we work with a fiber of a rendition $\rho,$ we assume that they are drawn in the closed disk $\Delta$ as indicated by the $\Delta$-decomposition $\rho \cap M$
of $\lin{G,\Lambda},$ where $\mathsf{dissolve}(\mu) = \lin{H, \widehat{\Lambda}}.$

\paragraph{Combined depth of a cylindrical rendition.}
Let $\rho = (\Gamma, \mathcal{D},c)$ be a cylindrical rendition of a linear society $\lin{G,\Lambda}$ in a disk $\Delta$ around a cell $c.$
We define the \defi{combined depth} of $\rho$ as the maximum between the depth of $\lin{G, \Lambda}$ and the depth of a vortex society of $c$ in $\rho$ and we denote it by $\mathsf{cdepth}(\rho).$ 
This allows us to assume that $\rho$ is accompanied with a sequence $\Bcal_{c} = \lin{B_i \mid i \in [l]}$ of subsets of $G_{c}$ for which \autoref{take_from_GM_IX} holds for every vortex society of $c$ in $\rho$ where $w ≤ \mathsf{cdepth}(\rho).$

\paragraph{Lines.} 
A \defi{$\Delta$-line} is a line $L$ of $\Delta$ in $\rho$ connecting (and containing) two distinct points of $B_\Delta,$ with the additional property that $L$ does not meet a flap cell of $\rho$ or a boundary point of a vortex cell that is not the point of a vertex, i.e., 
$$\cupall \{ c \mid c \in C_{\mathsf{f}}(\rho) \} \cap L = \emptyset \ \ \text{and} \ \bigcup_{c\in C_{\mathsf{v}}(\rho)}(\bd(c) - \widetilde{c})\cap L = \emptyset.$$
The vertex set of a $\Delta$-line $L$ is $\pi_{\rho}(L\cap N(\delta))$ and we denote it by $V_{\rho}(L)$ or simply $V(L)$ when $\rho$ is clear from the context.
Given two distinct vertices $y, y' \in V(\Lambda)$ we use $L_{y,y'}$ to denote the $\Delta$-line of $\Delta$ such that $V_{\rho}(L_{y,y'}) \subseteq B_{\Delta},$ starting at the point of $y,$ finishing at the point of $y',$ and following the rotation ordering of $\Delta.$
We also use $\Lambda_{y,y'}$ for the consecutive sub-sequence of $\Lambda$ that starts with $y$ and finishes with $y'.$

\medskip
The following lemma gives us more information about the structure of a minimum size separator between disjoint segments of a linear society corresponding to a cylindrical rendition in a disk $\Delta$ in the form of a $\Delta$-line of minimum length.
The fact that we see this separator as a line is important for the arguments of the rest of the proofs of this section.

\begin{lemma}\llabel{lem_line_sep}
Let $\rho = (\Gamma, \mathcal{D}, c)$ be a cylindrical rendition of a linear society $\lin{G, \Lambda}$ in a disk $\Delta$ around a cell $c$ and let $x$ and $y$ be two vertices in $\Lambda$ where $x$ appears before $y$ in $\Lambda.$
Then there is a $\Delta$-line $L$ of $\Delta$ between $x$ and $y$ such that
\begin{enumerate}
\item\llabel{lin_sep_p1} $|V_{\rho}(L)| \leq \mathsf{cdepth}(\rho) + 2,$
\item\llabel{lin_sep_p2} $L \cap \bd(\Delta_{c})$ contains at most two points, and
\item\llabel{lin_sep_p3} there is no $\Delta$-line $L'$ of $\Delta$ between $x$ and $y$ such that $|V_{\rho}(L')| < |V_{\rho}(L)|.$
\end{enumerate}
\end{lemma}
\begin{proof}
Let $G'$ be the graph obtained from $G$ after removing every vertex and edge drawn in every cell of $\rho$ and then making the vertices in $\widetilde{c}'$ pairwise-adjacent, for every cell $c'$ of $\rho$ that is not $c.$
Also we define $\Gamma'$ to be the drawing of $G'$ in $\Delta$ obtained from $\Gamma$ by drawing the newly added edges in a way so that they do not cross, inside their corresponding cell of $\rho.$
Clearly $\Gamma'$ is a $\Delta$-embedding (and also a plane embedding) of $G'.$
Moreover, $\rho' = (\Gamma', \Dcal)$ is a cylindrical rendition of the linear society $\lin{G', \Lambda}$ in $\Delta$ around $c$ such that $\mathsf{cdepth}(\rho') \leq \mathsf{cdepth}(\rho).$

Let $\Tcal$ be a maximum order transaction between $\Lambda_{xy}$ and $\overline{\Lambda}_{xy}$ in $\rho'.$
By definition $|\Tcal| \leq \mathsf{cdepth}(\rho).$
Let $S \subseteq V(G')$ be a minimum $\Lambda_{xy}$-$\overline{\Lambda}_{xy}$ separator in $G'.$
By Menger's theorem, $|S| \leq \mathsf{cdepth}(\rho).$

Since $\Gamma'$ is a plane embedding of $G'$ and $S$ is a minimum $\Lambda_{xy}$-$\overline{\Lambda}_{xy}$ separator, it is implied that there exists a simple closed curve $N$ in the plane, that intersects $G'$ only in vertices and that intersects exactly the vertices in $S.$
By definition, $N$ must intersect the boundary of $\Delta$ in at least two points.
Also, since $\Delta_{c}$ is a subset of a face of $\Gamma,$ we can assume that $N$ intersects the boundary of $\Delta_{c}$ in at most two points.
Let $L$ be the shortest line between two points of $\bd(\Delta)$ that is a subset of $N$ and that is drawn in $\Delta.$
Since $N$ separates $\Lambda_{xy}$ from its complement it must be that $L$ intersects $\bd(\Delta)$ in points drawn between the point of $x$ and the point of its previous vertex on $\Lambda$ and the point of $y$ and the point of its next vertex on $\Lambda.$
Hence, we can extend $L$ so that its endpoints are the points corresponding to $x$ and $y$ respectively without intersecting any other vertices or edges of $G',$ by possibly increasing the length of $L$ by at most two vertices.

Now, observe that by definition $L$ is a $\Delta$-line in $\rho'$ that satisfies properties \ref{lin_sep_p1}, \ref{lin_sep_p2}, and \ref{lin_sep_p3} in $\rho'.$
We argue that $L$ also satisfies properties \ref{lin_sep_p1}, \ref{lin_sep_p2}, and \ref{lin_sep_p3} in $\rho.$
This is indeed the case, since by definition of $\rho'$ any line of $\Delta$ in $\rho$ must also be a line of $\Delta$ in $\rho'$ with the same vertex set and vice versa. 
\end{proof}

\subsection{Pseudo-disks in cylindrical renditions}
\label{pse_cylindr_rend}

In this subsection, we introduce the concept of a pseudo-disk of a cylindrical rendition to facilitate the procedure of locating a fiber of a linear society, within the area defined by a segment of the linear society corresponding to the rendition and the line that corresponds to the separator of this segment, as demonstrated in the previous subsection.

\paragraph{Pseudo-disks.} Let $\rho$ be a cylindrical rendition of a linear society $\lin{G, \Lambda}$ in a disk $\Delta.$
A \defi{pseudo-disk} of $\rho$ is a set $\Psi \subseteq \Delta$ such that
\begin{enumerate}
\item $\Psi$ is the union of the $\Delta$-line $L_{y, y'},$ a $\Delta$-line $L^{\text{top}}$ between $y$ and $y'$ for some $y, y' \in V_{\rho}(\Lambda)$ where $y$ appears before $y'$ in $\Lambda,$ and the arc-wise connected components of $\Delta - (L_{y, y'} \cup L^{\text{top}})$ whose points in $\Delta$ can reach $L_{yy'}$ without intersecting $L^{\text{top}},$
\item $\Lambda_{\Psi} \coloneqq \Lambda_{y,y'} - V_{\rho}({L}^{\text{top}})$ is non-empty,
\item $L^{\rm top}\cap \bd(\Delta_{c})$ contains at most two points, and
\item there is no $\Delta$-line $L$ between $y$ and $y'$ such that $|V_{\rho}(L)| < |V_{\rho}(L^{\text{top}})|.$
\end{enumerate}

We define the \defi{boundary} of $\Psi$ as $\bd(\Psi) \coloneqq L_{yy'} \cup L^{\text{top}}.$
We call the $\Delta$-line $L^{\text{top}}$ the \defi{top-line} of $\Psi$ and we call $y$ the \defi{starting} and $y'$ the \defi{finishing} vertex of $\Psi.$
We also define \defi{$\mathsf{last}(\Psi) \coloneqq y'$}.
Let $G_{\Psi}$ be the subgraph of $G$ whose vertices and edges are drawn in $\Psi$ minus $L^{\text{top}}.$
We call $\lin{G_{\Psi},\Lambda_{\Psi}}$ the \defi{society} of $\Psi,$ we call the set $V(\Lambda_{\Psi})$ the \defi{base boundary} of $\Psi,$ we call the set $V_{\rho}(L^{\text{top}})$ the \defi{line boundary} of $\Psi.$
Assume that $L^{\text{top}}\cap\bd(c)$ consists of two distinct points $v$ and $v'$ where $y,v,v',y'$ appear in this order in $L^{\text{top}}$ and let $\Lambda^c=\Lambda_{\Delta_{c},v}.$
Then the \defi{vortex boundary} of $\Psi$ is the set
$$(B_{v}\cap B_{\mathsf{next}_{\Lambda^c}(v)})\cup(B_{v'}\cap B_{\mathsf{previous}_{\Lambda^c}(v')})$$
where $B_{\mathsf{next}_{\Lambda^c}(v)}$ and $B_{\mathsf{previous}_{\Lambda^c}(v')}$ are members of $\Bcal^{c}$ and that has size at most $2 \cdot \mathsf{cdepth}(\rho).$
If $|L^{\text{top}} \cap \bd(c)| ≤ 1,$ then the vortex boundary of $\Psi$ is empty.
The \defi{top boundary} of $\Psi$ is the union of its line boundary and its vortex boundary and it is denoted by $V_{\rho}(\Psi).$

\medskip
An \defi{$\lin{H,\widehat{\Lambda}}$-fiber of $\rho$ inside $\Psi$} is an $\lin{H,\widehat{\Lambda}}$-fiber of $\rho$ that is also a fiber of $\lin{G_{\Psi}, \Lambda_{\Psi}}.$
An \defi{$\lin{H,\widehat{\Lambda}}$-fiber of $\rho$ outside $\Psi$} is an $\lin{H,\widehat{\Lambda}}$-fiber $(M, T, \widehat{\Lambda})$ of $\rho$ where none of the points of $V(M)$ are in $\Psi.$

\paragraph{Chains.} 
Let $\rho$ be a cylindrical rendition of a linear society $\lin{G, \Lambda}.$
A sequence $\mathbf{\Psi}=\lin{\Psi_1,\ldots \Psi_{r}}$ of pseudo-disks of $\rho$ is called a \defi{chain} of $\rho$ if 
\begin{itemize}
\item $\Psi_1, \ldots, \Psi_{r}$ are pairwise internally-disjoint,
\item the starting vertex of $\Psi_1$ is the starting vertex of $\Lambda,$ and
\item for every $i \in [r-1],$ if $y'_{i}$ is the finishing vertex of $\Psi_{i},$ $y'_{i}$ is the starting vertex of $\Psi_{i+1}.$
\end{itemize}

The \defi{starting} vertex of $\mathbf{\Psi}$ is the starting vertex of $\Psi_{1},$ i.e., the starting vertex of $\Lambda.$
The \defi{finishing} vertex of $\mathbf{\Psi}$ is the finishing vertex of $\Psi_{r}$ if $\mathbf{\Psi} \neq \lin{},$ and the last vertex of $\Lambda$ otherwise.
An \defi{$\lin{H,\widehat{\Lambda}}$-fiber of $\rho$ outside $\mathbf{\Psi}$} is an $\lin{H,\widehat{\Lambda}}$-fiber of $\rho$ that is outside $\Psi_{i},$ for every $i \in [r],$ if $\mathbf{\Psi} \neq \lin{},$ and an $\lin{H,\widehat{\Lambda}}$-fiber of $\rho$ where $V(M)$ does not contain the starting and finishing vertex of $\mathbf{\Psi}$ otherwise.

\medskip
The following lemma permits us to associate pairs of boundary vertices of $\lin{G,\Lambda}$ to pseudo-disks of some suitably chosen cylindrical rendition of it.

\begin{lemma}\llabel{lem_pseudodisk}
Let $\rho = (\Gamma, \mathcal{D},c)$ be a cylindrical rendition of a linear society $\lin{G,\Lambda}$ in a disk $\Delta$ around a cell $c,$ $\mathbf{\Psi}$ be a chain of $\rho$ with starting vertex $x$ and finishing vertex $y,$ and $z$ be a vertex of $\Lambda$ that is not in the consecutive sub-sequence of $\Lambda$ with first vertex $\mathsf{next}_{\Lambda}(x)$ and last vertex $y.$
Then $\rho$ has a pseudo-disk ${\Psi},$ starting from the finishing vertex of $\mathbf{\Psi}$ and finishing at $z,$ that is internally-disjoint from every $\Psi' \in \mathbf{\Psi},$ and where $|V_{\rho}(\Psi)| ≤ 3 \cdot \mathsf{cdepth}(\rho) + 2.$
\end{lemma}
\begin{proof}
First, we apply \autoref{lem_line_sep} to obtain a $\Delta$-line $L$ between $y$ and $z$ such that $|V_{\rho}(L)| \leq \mathsf{cdepth}(\rho) + 2,$ $L \cap \bd(\Delta_{c})$ contains at most two points, and such that there is no $\Delta$-line $L'$ between $y$ and $z$ such that $|V_{\rho}(L')| < |V_{\rho}(L)|.$

Next, we argue that $L$ does not contain any internal point of any of the pseudo-disks of $\mathbf{\Psi}.$
Assume towards contradiction that it is not the case.
Let $\Psi' \in \mathbf{\Psi}$ such that $L$ contains an internal point of $\Psi.$
Let $L^{\textrm{top}}$ be the top-line of $\Psi'.$
This particularly implies that $L$ and $L^{\text{top}}$ intersect in at least in two points.
Let $L'$ be a {subline} of $L$ whose endpoints are points of $L^{\text{top}}.$
Let $L^{\text{top}}_{2}$ be the corresponding subline of $L^{\text{top}}.$
Consider the $\Delta$-line $L''$ obtained by replacing the subline $L^{\text{top}}_{2}$ of $L^{\text{top}}$ by $L'.$
Now, observe that by the minimality of $L$ guaranteed by the specifications of \autoref{lem_line_sep}, it must be that $|L''|$ is strictly smaller the line boundary of $\Psi'$ which contradicts the definition of $\Psi'.$

Then, we define the desired pseudo-disk $\Psi$ as the union of $L \cup L_{yz}$ and the arc-wise connected components of $\Delta - (L \cup L_{yz})$ whose points in $\Delta$ can reach $L_{yz}$ without intersecting $L.$
\end{proof}

\subsection{Critical pseudo-disks and fibers}
\llabel{Critical_pseudo_disks}

In this subsection we define the notion of a critical pseudo-disk which will be used to find fibers of a given linear society within pseudo-disks, in a way that will allow us to either find many disjoint, in a sequential fashion along the boundary of the linear society corresponding to our cylindrical rendition or hit all of them by removing a small number of vertices which will be the frontiers of the corresponding critical pseudo-disks in which we locate our fibers.

\paragraph{Critical pseudo-disks.}
Let $\rho = (\Gamma, \mathcal{D},c)$ be a cylindrical rendition of a linear society $\lin{G,\Lambda}$ in a disk $\Delta$ around a cell $c$ and let $\lin{H, \widehat{\Lambda}}$ be some linear society.
A pseudo-disk $\Psi$ of $\rho$ is an \defi{$\lin{H, \widehat{\Lambda}}$-critical pseudo-disk} if there is a pseudo-disk $\widetilde{\Psi}$ of $\rho$ where
\begin{enumerate}
\item the starting vertex of $\widetilde{\Psi}$ is the same as $\Psi$ and its finishing vertex is the previous of the finishing vertex of $\Psi$ in $\Lambda,$
\item $\widetilde{\Psi}\subseteq\Psi,$
\item there is an $\lin{H,\widehat{\Lambda}}$-fiber of $\rho$ inside $\Psi,$
\item  there is no $\lin{H,\widehat{\Lambda}}$-fiber of $\rho$ inside $\widetilde{\Psi},$ and
\item $|V_{\rho}(\Psi)|,|V_{\rho}(\widetilde{\Psi})| \leq 3 \cdot \mathsf{cdepth}(\rho) + 2.$ 
\end{enumerate}

\begin{figure}[ht]
\begin{center}
\scalebox{1.1}{\includegraphics{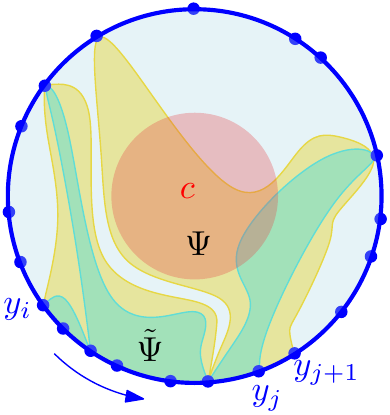}}
\end{center}
 \caption{The two pseudo-disks of the definition of an $\lin{H,\widehat{\Lambda}}$-critical pseudodisk.}
 \llabel{two_pseudo_disks_critical}
\end{figure}

Whenever we consider some $\lin{H,\widehat{\Lambda}}$-critical pseudo-disk $\Psi$ of $\rho,$ we accompany it with a pseudo-disk $\widetilde{\Psi},$ as above. See \autoref{two_pseudo_disks_critical} for a visualization of the notion of a $\lin{H, \widehat{\Lambda}}$-critical pseudo-disk.
The \defi{frontier} of $\Psi$ is the set $\mathsf{frontier}_{\rho}(\Psi) \coloneqq V_{\rho}(\Psi)\cup V_{\rho}(\widetilde{\Psi}).$ 
The \defi{debris} of $\Psi,$ denoted by $\mathsf{debris}_{\rho}(\Psi)$ is the set of vertices of $G-\mathsf{frontier}_{\rho}(\Psi)$ that are not $\rho$-bound.
{Notice} that $\partial_{G}(\mathsf{debris}_{\rho}(\Psi)) \subseteq \mathsf{frontier}_{\rho}(\Psi).$
Therefore, $|\partial_{G}(\mathsf{debris}_{\rho}(\Psi))| ≤ 6 \cdot \mathsf{depth}(\rho) + {3}.$

We define $$\mathsf{frontier}_{\rho}(\mathbf{\Psi}) \coloneqq \bigcup_{i\in [|\mathbf{\Psi}|]} \mathsf{frontier}_{\rho}(\mathbf{\Psi}_{i})$$
and $$\mathsf{debris}_{\rho}(\mathbf{\Psi}) \coloneqq \bigcup_{i \in [|\mathbf{\Psi}|]} \mathsf{debris}_{\rho}(\mathbf{\Psi}_{i}).$$
It is easy to observe that, as all pseudo-disks in $\mathbf{\Psi}$ are pairwise internally disjoint, it holds that
\begin{eqnarray}
|\partial_{G}(\mathsf{debris}_{\rho}(\mathbf{\Psi}))| & ≤ & (6 \cdot \mathsf{depth}(\rho) + {3}) \cdot |\mathbf{\Psi}|\llabel{debris_bound_all}
\end{eqnarray}

\begin{lemma}\llabel{pseudodissks}
There is an algorithm that, given
\begin{itemize}
\item a cylindrical rendition $\rho = (\Gamma, \mathcal{D}, c)$ of a linear society $\lin{G,\Lambda}$ in a disk $\Delta$ around a cell $c,$ where $w \coloneqq \mathsf{cdepth}(\rho),$
\item a linear society $\lin{H,\widehat{\Lambda}},$
\item a chain $\mathbf{\Psi}$ of $\rho,$ and
\item a pseudo-disk $\Psi$ of $\rho$ starting from the finishing vertex of $\mathbf{\Psi}$ and finishing at the starting vertex of $\mathbf{\Psi}$ that is internally-disjoint from every $\Psi'' \in \mathbf{\Psi}$ where $|V_{\rho}(\Psi)| ≤ 3w,$
\end{itemize}
if there is an $\lin{H,\widehat{\Lambda}}$-fiber of $\rho$ outside $\mathbf{\Psi}$ that is also a fiber of $\lin{G_{\Psi}, \Lambda_{\Psi}},$ outputs an $\lin{H,\widehat{\Lambda}}$-critical pseudo-disk $\Psi'$ of $\rho$ such that $\mathbf{\Psi} \oplus \lin{\Psi'}$ is a chain of $\rho,$
in time
$$\mathcal{O}_{|H|}(w \cdot |G|^3).$$
\end{lemma}
\begin{proof} By the assumptions of the lemma there exists an $\lin{H,\widehat{\Lambda}}$-fiber of $\rho$ outside $\mathbf{\Psi}$ that is a fiber of $\lin{G_{\Psi}, \Lambda_{\Psi}}.$
Let $x$ be the finishing and $y$ be the starting vertex of $\mathbf{\Psi}$ respectively.
Also, let $\Lambda_{xy}$ be the maximal sub-sequence of $\Lambda$ with first vertex $x$ and last vertex $y$ such that the point of any other vertex of $\Lambda_{xy}$ is not in $\cupall \mathbf{\Psi}.$
Notice that since $\Psi$ is internally disjoint from every $\Psi'' \in \mathbf{\Psi},$ it must be that every vertex of the top-line of any $\Psi'' \in \mathbf{\Psi},$ that is also a vertex of $\Lambda_{\Psi},$ must also be a vertex of the top-line of $\Psi.$
This implies that $\Lambda_{xy} = \Lambda_{\Psi}.$

Now, for every vertex $z$ of $\Lambda_{xy}$ distinct from $y,$ let $\Psi_{xz}$ be the pseudo-disk of $\rho$ given by \autoref{lem_pseudodisk}.
Moreover, let $\Psi_{xy} \coloneqq \Psi.$
Now, consider a vertex $z$ of $\Lambda_{xy}$ such that the consecutive sub-sequence $\Lambda_{xz}$ of $\Lambda_{xy}$ with first vertex $x$ and last vertex $z$ is minimal in length so that there is an $\lin{H,\widehat{\Lambda}}$-fiber of $\rho$ outside $\mathbf{\Psi}$ that is a fiber of $\lin{G_{\Psi_{xz}}, \Lambda_{\Psi_{xz}}}.$
Clearly, the existence of $\Psi$ certifies that such a vertex $z$ exists.
Now let $z'$ be the previous vertex from $z$ on $\Lambda_{xy}.$
Now observe that since $z$ is minimal with this property it means that there is no $\lin{H,\widehat{\Lambda}}$-fiber of $\rho$ outside $\mathbf{\Psi}$ that is a fiber of $\lin{G_{\Psi_{xz'}}, \Lambda_{\Psi_{xz'}}}.$
Then the pseudo-disk $\Psi_{xz}$ is an $\lin{H,\widehat{\Lambda}}$-critical pseudo-disk certified by the pseudo-disk $\Psi_{xz'}.$
Also by the specifications of \autoref{lem_pseudodisk} (or by the assumptions of this lemma in case $z = y$) we have that $\mathbf{\Psi} \oplus \lin{\Psi'}$ is a chain of $\rho.$

In order to compute $\Psi,$ we start by computing a decomposition as the one in  \autoref{take_from_GM_IX}, which as we already briefly discussed, can be done in time $\Ocal(w \cdot |G|^2).$
This gives us all the top boundaries that correspond to the pseudo-disks we define by calling upon \autoref{lem_pseudodisk}.
To finally compute the critical pseudo-disk $\Psi$ we iteratively apply at most $\Ocal(|G|)$ times the minor-checking algorithm of \cite{KawarabayashiKR12Thedisjoint} that runs in time $\Ocal_{|H|}(|G|^2).$
\end{proof}

Critical pseudo-disks will allow us to find ``minimal packings'' of fibers along the boundary of linear societies.
The ``criticality'' in their definition is important in order to show an Erd\H{o}s-P{\'o}sa duality for fibers within bounded-depth societies.
This is the subject of the following subsection.

\subsection{Harvesting fibers}
\llabel{harvesting_fibers}

In this subsection, we show how to combine all the partial results in the previous subsection in order to either find many sequential fibers of a given linear society or hit all of them with a small number of vertices found within the disk of the given cylindrical rendition.

\paragraph{$k$-fold multiplications and shifts.}
Let $\lin{H,\widehat{\Lambda}}$ be a linear society and consider that $\mathsf{dec}({H,\widehat{\Lambda}}) =  \lin{ \lin{H_{1}, \widehat{\Lambda}_{1}}, \ldots, \lin{H_{q}, \widehat{\Lambda}_{q}} }.$
For $k \in \Nbbb_{≥1},$ we define the $k$-\defi{fold multiplication} of $\lin{H,\widehat{\Lambda}},$ denoted by $\lin{H,\widehat{\Lambda}}^{(k)}$ to be the linear society where 
$$\mathsf{dec}(\lin{H,\widehat{\Lambda}}^{(k)})=\lin{\lin{H_1^1,\widehat{\Lambda}_1^1},\ldots,\lin{H_1^k,\widehat{\Lambda}_1^k},\ \ldots\ ,\lin{H^1_q,\widehat{\Lambda}^1_q},\ldots,\lin{H^k_q,\widehat{\Lambda}^k_q}}$$
and where, for $(i,j) \in [k]\times[q],$ $\lin{H_j^{i}, \widehat{\Lambda}_j^{i}}$ is isomorphic to $\lin{H_j,\widehat{\Lambda}_j}.$

\paragraph{Corridors and $\Delta$-compact sets.}
Let $\rho = (\Gamma, \mathcal{D}, c)$ be a cylindrical rendition of a linear society $\lin{G,\Lambda}$ in a disk $\Delta$ around a cell $c.$
Let $\Psi$ be a pseudo-disk of $\rho$ with starting vertex $y$ and finishing vertex $y'$ whose top line is $L^{\text{top}}$ and assume that $\Psi$ is a closed disk of $\Delta.$
Then we define $\widehat{\Psi}$ as the set containing the interior of the set $\Psi \setminus c$ and all internal points of $L_{y,y'}.$
We call any subset of $\Delta$ that is a set $\widehat{\Psi}$ as above a \defi{corridor} of $\rho$ and its \defi{frontier} is the set of all vertices of $G$ whose points belong in $\bd(\Psi)$ but not in the interior of $L_{y,y'} \cup \Delta_{c}.$
The union $K$ of a finite set of pairwise internally disjoint corridors of $\Delta$ is called a \defi{corridor system} of $\rho$ and the union of the frontiers of the corridors of this set is the \defi{frontier} of $K.$

\medskip
We say that a vertex set $S\subseteq V(G)$ is $\Delta$-\defi{compact} if there is a corridor system $K$ of $\rho$ such that 
\begin{itemize}
\item none of the points of $S$ belongs in $K$ and 
\item the frontier of $K$ is a subset of $S.$
\end{itemize}
Moreover, we assume any $\Delta$-compact set $S$ to be \defi{accompanied} by such a corridor system of $\rho.$

\begin{observation}\llabel{obs:compact_union}
Let $\rho = (\Gamma, \mathcal{D},c)$ be a cylindrical rendition of a linear society $\lin{G,\Lambda}$ in a disk $\Delta$ around a cell $c.$
If $S$ and $S'$ are two $\Delta$-compact sets of $G$ then $S \cup S'$ is also a $\Delta$-compact set of $G.$
\end{observation}

The notion of a corridor and its frontier is important in our proofs as when the outcome of our Erd\H{o}s-P{\'o}sa duality is a ``cover'' $S,$ then we require that this cover does not contain vertices drawn on the interior of the corridor as well as that all vertices of the boundary of the corridor belong in $S.$
That way, the covers that we consider give us ``valuable space'' for rerouting the fibers that are not intersected by $S$ in a way that we can bound the number of vertices on the boundary.

\medskip
The algorithm that extracts the Erd\H{o}s-P{\'o}sa duality is heavily based on the notion of a critical pseudo-disk in bounded depth societies and is the following.

\begin{tabbing}
~~~~~\=~~~~~\=~~~~~\=~~~~~\=~~~~~\=~~~~~\=\\

{\bf Algorithm} $\mathbf{find\_all}(\Delta,G,\Lambda,\rho,k, \lin{H,\widehat{\Lambda}})$\\

{\sf Input}: \= A cylindrical rendition $\rho = (\Gamma, \mathcal{D},c_0)$ of a linear society $\lin{G,\Lambda}$ in a disk $\Delta,$ \\
\> an integer $k\in\Nbbb,$ \\
\> a linear society $\lin{H,\widehat{\Lambda}}$ where $\mathsf{dec}({H,\widehat{\Lambda}}) \coloneqq \lin{\lin{H_{1},\widehat{\Lambda}_{1}},\ldots,\lin{H_{q},\widehat{\Lambda}_{q}}}.$ \\
{\sf Output}: \= either a $\lin{H,\widehat{\Lambda}}^{(k)}$-fiber of $\rho$ or \\
\> a set $S \subseteq V(G)$ such that $\lin{G, \Lambda}-S$ has no $\lin{H,\widehat{\Lambda}}$-fiber
\\ 
$\mathbf{\Psi}=\lin{}$\\
use \autoref{lem_pseudodisk} to find a pseudo-disk $\Psi$ of $\rho$ starting from the finishing vertex of $\mathbf{\Psi}$\\
and finishing at the starting vertex of $\mathbf{\Psi}.$\\
for \= $i\in[q]$\\
\> $\ell=0,$ \\
\> while \= $\ell≤k$ and $\lin{G,\Lambda}\setminus\mathbf{\Psi}$ has an $\lin{H_i,\widehat{\Lambda}_i}$-fiber that is also a fiber of $\lin{G_{\Psi}, \Lambda_{\Psi}}$\\ 
\> \>use \= \autoref{pseudodissks} to find an $\lin{H_i,\widehat{\Lambda}_i}$-critical pseudo-disk $\Psi'$ of $\rho$ such that\\
\> \>$\mathbf{\Psi}\oplus\lin{\Psi'}$ is a chain of $\rho'.$\\
\> \>$\mathbf{\Psi}\coloneqq\mathbf{\Psi}\oplus\lin{\Psi'}$\\
\> \>$\ell\coloneqq\ell+1$\\
\> \>Set $\Psi$ to the pseudo-disk given \autoref{lem_pseudodisk} of $\rho$ starting from the finishing vertex of\\
\> \>$\mathbf{\Psi}$ and finishing at the starting vertex of $\mathbf{\Psi}.$\\
\> if \= $\ell<k$\> then output $\mathsf{frontier}_{\rho}(\mathbf{\Psi}) \cup V_{\rho}(\Psi)$ and stop\\
\> \>\> otherwise use the last $k$ $\lin{H,\widehat{\Lambda}}$-critical pseudo-disks of $\mathbf{\Psi}$ in order to find a\\
\> \>\> $\lin{H_{i},\widehat{\Lambda}_{i}}^{(k)}$-fiber of $\rho.$\\
concatenate every $\lin{H_{i}^1,\widehat{\Lambda}_{i}^1}^{(k)}$-fiber of $\rho,$ for $i\in[q],$ in order to build a $\lin{H,\widehat{\Lambda}}^{(k)}$-fiber of $\rho.$
\end{tabbing}

Notice that whenever the output of the previous algorithm is a ``cover'' set $S,$ it is important that $S$ is $\Delta$-compact.
This is indeed the case since the separators extracted are the boundaries of the pseudo-disks found in the algorithm.

\begin{lemma}[Harvesting lemma]\llabel{func_sharvsp}
There exists a function $f_{\ref{func_sharvsp}}:\nn{3}{1}$ and an algorithm that, given
\begin{itemize}
\item a cylindrical rendition $\rho = (\Gamma, \mathcal{D},c)$ of a linear society $\lin{G,\Lambda}$ in a disk $\Delta$ around a cell $c,$ where $w \coloneqq \mathsf{cdepth}(\rho),$ or
\item a linear society $\lin{H,\widehat{\Lambda}},$ and
\item a $k\in\Nbbb_{≥1},$
\end{itemize}
either outputs
\begin{itemize}
\item an $\lin{H,{\Lambda}'}^{(k)}$-fiber of $\lin{G,\Lambda},$ for a shift $\Lambda'$ of $\widehat{\Lambda},$ or
\item a $\Delta$-compact set $S\subseteq V(G),$ where $|S| ≤ f_{\ref{func_sharvsp}}(k, w, |H|)$ and $\lin{G,\Lambda} - S$ has no $\lin{H, \Lambda'}$-fiber, for any shift $\Lambda'$ of $\widehat{\Lambda},$
\end{itemize}
in time
$$\Ocal_{|H|}(kw \cdot |G|^3 ).$$
Moreover
$$f_{\ref{func_sharvsp}}(k,w, |H|)=\Ocal_{|H|}(kw).$$
\end{lemma}
\begin{proof} We define $f_{\ref{func_sharvsp}}(k, w, |H|) \coloneqq k \cdot (3w + 2) \cdot |H|^{2}.$  

Let $\Lambda_{1}, \ldots, \Lambda_{q}$ be the shifts of $\widehat{\Lambda},$ where $q = |\widehat{\Lambda}| \leq |H|.$
For every $i \in [q],$ we execute the algorithm $\mathbf{find\_all}(\Delta, G, \Lambda, \rho, k, \lin{H, \Lambda_{i}}).$
If for some $i \in [q]$ the output is an $\lin{H, \Lambda_{i}}^{(k)}$-fiber of $\rho$ we can conclude.
Hence, we can assume that for every $i \in [q],$ the output of the corresponding execution is a set $S_{i} = \mathsf{frontier}_{\rho}(\mathbf{\Psi}_{i}) \cup V_{\rho}(\Psi_{i}),$ where $\mathbf{\Psi}_{i}$ is a chain of $\rho$ and $\Psi_{i}$ is a pseudo-disk of $\rho$ starting from the finishing vertex of $\mathbf{\Psi}_{i},$ finishing at the starting vertex of $\mathbf{\Psi}_{i},$ and that is internally-disjoint from every pseudo-disk of $\mathbf{\Psi}_{i}.$
First, observe that by the definition of $\Delta$-compact sets, for every $i \in [q],$ the set $S_{i}$ is $\Delta$-compact.
We define $S \coloneqq \bigcup_{i \in [q]} S_{i}.$
Then, \autoref{obs:compact_union} implies that $S$ is $\Delta$-compact, as required.

Also, note that for fixed $i \in [q],$ $\mathbf{\Psi_{i}}$ can contain at most $k \cdot |H|$ critical pseudo-disks of $\rho,$ since $|\mathsf{dec}(H, \Lambda_{i})| \leq |H|,$ and for each such critical pseudo-disk by definition we have that its frontier has size at most $3w + 2.$ Hence in total $|S_{i}| \leq k \cdot (3w + 2) \cdot |H|$ and thus $|S| \leq k \cdot (3w + 2) \cdot |H|^{2}$ as required.

It remains to show that $\lin{G, \Lambda} - S$ has no $\lin{H, \Lambda_{i}}$-fiber, for any $i \in [q].$
Assume towards contradiction that $\lin{G, \Lambda} - S$ has an $\lin{H, \Lambda_{i}}$-fiber, for some $i \in [q],$ and that $\mathsf{dec}(H, \Lambda_{i}) \coloneqq \lin{ \lin{H_{1}, \Lambda^{i}_{1}}, \ldots, \lin{H_{q}, \Lambda^{i}_{d_{i}}} }.$
Then, it must be that $\lin{G, \Lambda} - S$ has a $\lin{H,  \Lambda^{i}_{j}}$-fiber, for every $j \in [d_{i}],$ such that all vertices of $\Lambda^{i}_{j}$ appear before the vertices of $\Lambda^{i}_{j'}$ in $\Lambda,$ where $j < j' \in [d_{i}]$ and moreover such that these fibers are pairwise-disjoint.

First, observe that since the algorithm outputted the set $\mathsf{frontier}_{\rho}(\mathbf{\Psi}_{i}) \cup V_{\rho}(\Psi_{i})$ it is implied that there is an $l \in [d_{i}]$ such that if there is a $\lin{H, \Lambda^{i}_{l}}$-fiber outside $\mathbf{\Psi}_{i},$ this fiber is not a fiber of $\lin{G_{\Psi_{i}}, \Lambda_{\Psi_{i}}},$ implying that any such fiber is intersected by $V_{\rho}(\Psi_{i}).$
So, our only hope to find a fiber for each $\lin{H,  \Lambda^{i}_{j}},$ $j \in [l],$ in the correct order, is within $\mathbf{\Psi}_{i}.$
Now, observe that by the way the chain $\mathbf{\Psi}_{i}$ of $\rho$ is defined, for any sequence $\Psi_{1}, \ldots, \Psi_{l}$ of pseudo-disks of $\mathbf{\Psi}_{i},$ respecting the order in $\mathbf{\Psi}_{i},$ such that for every $j \in [l],$ there is an $\lin{H,  \Lambda^{i}_{j}}$-fiber of $\rho$ inside $\Psi_{j},$ there must be an index $j' \in [l]$ such that $\Psi_{j'}$ is a $\lin{H,  \Lambda^{i}_{j'}}$-critical pseudo-disk of $\rho,$ which implies that the set $\mathsf{frontier}_{\rho}(\mathbf{\Psi}_{i})$ intersects this fiber.
This contradicts the existence of the sequence of fibers we require in their correct order within $\mathbf{\Psi}_{i},$ and this concludes the proof. 

The algorithm above assumes that we have at hand all separators from a decomposition as the one in \autoref{take_from_GM_IX}.
As we already briefly explained, this can be done in time $\Ocal(w \cdot |G|^2)$ and it gives us frontiers of the corresponding pseudo-disks we compute by calling upon \autoref{pseudodissks}.
The total running time is dominated by the time required to find a $\lin{H_i,\widehat{\Lambda}_i}$-critical pseudo-disk $\Psi'$ of $\rho$ in the while-loop of algorithm \textbf{find\_all}.
This is done, using \autoref{pseudodissks}, in time $\Ocal_{|H|}(w \cdot |G|^3)$ time and is applied $\Ocal(k)$ many times.
\end{proof}

\subsection{Main harvesting lemmata}\llabel{main_harvesting_lemmata}

In this subsection, we conclude with the proofs of the main lemmata of the entire section.
An important notion for our proofs is the one of an extension of some graph $Z$ in a graph $G.$

\paragraph{Extensions.} \llabel{def_extensions}
Given a graph $M$ and a set $T \subseteq V(M)$ we call $(M, T)$ a \defi{subdivision pair} if every vertex not in $T$ has degree two in $M.$
Given a subdivision pair $(M, T)$ we define $\mathsf{dissolve}(M, T)$ to be the graph obtained from $M$ after dissolving all vertices in $V(M) \setminus T.$
We define the \defi{detail} of a subdivision pair $(M, T)$ to be the number of edges of $\mathsf{dissolve}(M, T).$
Consider two graphs $Z$ and $G.$
We say that a subdivision pair $(M, T)$ is an \defi{extension of $Z$ in $G$} if
\begin{enumerate}
\item $M$ is a connected subgraph of $G,$
\item every vertex of $V(M) \setminus T$ has degree two, and
\item $\mathsf{dissolve}(M, T)$ contains $Z$ as a minor.
\end{enumerate}

It is important to take into account that the graph $M$ of an extension is always connected.
Also the vertices of $V(M)\setminus T$ can be seen as subdivision vertices of paths joining the ``terminals'' in $T.$
By considering a minimal subgraph $M$ satisfying the above properties, we can obtain the following observation.

\begin{observation}\llabel{observation_how_to_find}
Let $Z$ and $G$ be two graphs.
If $Z$ is a minor of $G,$ then there exists an extension of $Z$ in $G$ that has detail $\Ocal(|Z|^2).$
\end{observation}

We define the function $\mathsf{h} \colon \Nbbb\to\Nbbb$ so that for every graph $Z$ that is a minor of a graph $G,$ $\mathsf{h}(|Z|)$ is the bound on the detail of the minimal extension given by \autoref{observation_how_to_find}.

An extension $(M,T)$ of a graph $Z$ in a graph $G$ is \defi{$\delta$-bound} if $M$ is $\delta$-bound. For the purposes of our local structure theorem we are interested in ``covering'' the  $\delta$-bound extensions of some obstruction $Z\in\obs(\Hcal).$

\paragraph{Disk collection of a $\Sigma$-decomposition.}
Let $\delta$ be a $\Sigma$-decomposition of a graph $G.$
A \defi{disk collection of $\delta$} is a set $\mathbf{\Delta}$ of $|\Ccal_{\mathsf{v}}(\delta)|$ pairwise-disjoint $\delta$-aligned disks of $\Sigma$ where every disk $\Delta$ of $\mathbf{\Delta}$ contains a vortex $c^{\Delta}$ of $\delta.$

We assume that a disk collection $\mathbf{\Delta}$ is accompanied with a \defi{starting vertex assignment function} $y_{\mathbf{\Delta}} \colon \mathbf{\Delta} \to V(G)$ mapping $\Delta \in \mathbf{\Delta}$ to a vertex $y_{\mathbf{\Delta}}(\Delta).$
For brevity, we use the term \defi{svaf} instead of the longer one ``starting vertex assignment function''.
We also define $\Lambda^{(\Delta)} \coloneqq \Lambda_{\Delta, y_{\mathbf{\Delta}}(\Delta)},$ i.e., $y_{\mathbf{\Delta}}(\Delta)$ ``fixes'' the starting vertex of a linear ordering of the vertices on the boundary of $\Delta.$

Hence, for every $\Delta \in \mathbf{\Delta},$ $\rho^{\Delta} \coloneqq \delta \cap \Delta$ is a cylindrical rendition of the society $\lin{G\cap \Delta,\Lambda^{(\Delta)}}.$ around a vortex cell $c^{\Delta}.$
Moreover, we say that a disk $\Delta \in \mathbf{\Delta}$ is \defi{$q$-insulated} if $\rho^{\Delta}$ is a $q$-\defi{nested} cylindrical rendition  around the vortex $c^{\Delta}.$

We call the triple $(\delta,\mathbf{\Delta},y_{\mathbf{\Delta}})$ a \defi{$\Sigma$-disk-collection triple} of $G$ and moreover \defi{$q$-insulated} if every disk $\Delta \in \mathbf{\Delta}$ is {$q$-insulated}.

We say that $(M,T)$ is \defi{$\mathbf{\Delta}$-respecting} if at least one of the points of $V(M)$ is drawn in the closure of $\Sigma \setminus \cupall \mathbf{\Delta}.$

\paragraph{Covers from harvesting.}
Let $k, d \in \Nbbb,$ and  let $(\delta,\mathbf{\Delta},y_{\mathbf{\Delta}})$ be a {$\Sigma$-disk-collection triple} of a graph $G.$
Consider a disk $\Delta \in \mathbf{\Delta}$ and for every linear society $\lin{H, \widehat{\Lambda}}$ of detail at most $d,$ apply \autoref{func_sharvsp} to $\rho^{\Delta},$ $\lin{H, \widehat{\Lambda}},$ and $k.$
Let $S^{\Delta}$ be the union of all sets that correspond to the second outcome of all previous applications of \autoref{func_sharvsp}.
We define a \defi{$(d, k)$-cover of $(\delta, \mathbf{\Delta}, y_{\mathbf{\Delta}})$} as the set $\bigcup_{\Delta \in \mathbf{\Delta}} S^{\Delta}.$

Notice that, using \autoref{func_sharvsp}, we get the following observation on the time it takes to compute a $(d, k)$-cover of $(\delta, \mathbf{\Delta}, y_{\mathbf{\Delta}}).$
\begin{observation}\llabel{how_to_find_fast_cover}
There is an algorithm that, given $k, d \in \Nbbb_{≥1}$ and a $\Sigma$-disk-collection triple $(\delta, \mathbf{\Delta},y_{\mathbf{\Delta}})$  of a graph $G,$ outputs a $(d, k)$-cover $S$ of $(\delta, \mathbf{\Delta}, y_{\mathbf{\Delta}})$ in time $\Ocal_{d}(w\cdot k\cdot |\mathbf{\Delta}|\cdot |G|^3 ),$ where $w \coloneqq \max\{\mathsf{cdepth}(\rho^{\Delta}) \mid \Delta \in \mathbf{\Delta}\}.$
Moreover $|S| ≤ \Ocal_{d}(w \cdot k\cdot |\mathbf{\Delta}|).$
\end{observation}

\paragraph{Crops.}
Let $(\delta,\mathbf{\Delta},y_{\mathbf{\Delta}})$ be a {disk-collection triple} of a graph $G,$ $(M, T)$ be a subdivision pair where $M$ is a subgraph of $G,$ and consider a disk $\Delta \in \mathbf{\Delta}.$
Let $M'$ be the union of all paths in $M \cap \Delta$ connecting a vertex of $B_{\Delta}$ with a vertex whose point is in $\Delta_{c^{\Delta}}$ and whose internal vertices are not in $B_{\Delta}.$
We say that $(M, T)$ \defi{has a crop in $\Delta$} if $M'$ is a non-empty graph and, if this is the case, we define $\mathsf{crop}(M, T, \Delta)$ as the fiber $(M', T \cap V(M'), \widehat{\Lambda})$ of $\rho^{\Delta}$ where $\widehat{\Lambda} \subseteq \Lambda^{(\Delta)}.$ 
For an example of the crop operation, see~\autoref{example_of_crops}.

\begin{figure}[ht]
\begin{center}
\scalebox{1.1}{\includegraphics{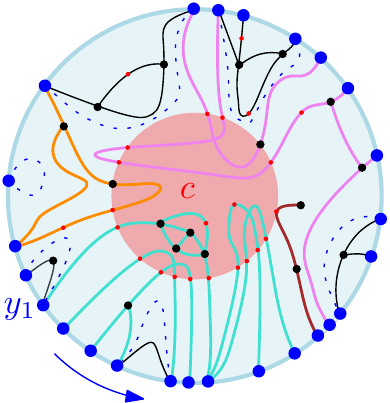}}
\end{center}
\caption{A disk $\Delta$ containing some vortex $c$ in is interior and a fiber $(M,T,\Lambda)$ of $\lin{G\cap \Delta,\Lambda}.$ We depict by red the vertices not in $T.$ 
If $M'$ is the subgraph of $M$ obtained after removing the black edges and the resulting set $S$ of isolated vertices, then $(M',T-S,\Lambda-S)={\sf crop}(M,T,\widehat{\Lambda}).$  The subgraphs formed by the black edges are excluded because of the blue-dotted $\Delta$-lines.
}
\llabel{example_of_crops}
\end{figure}

Intuitively the operation of cropping filters out parts of the subdivision pair that invade the disks of our disk-collection, that do not reach all the way to the vortex cells of the renditions associated to the disk-collection.
The ``forgotten'' parts have the nice property that they use only flat parts of $\delta.$
Exploiting this fact, we shall later see that all forgotten parts can be reconstructed as all such flat parts are sufficiently represented in the parterre segments of our walloid.
Hence, in the final proofs of this section, we only deal with the essential part of the extension of $Z$ which is what remains after cropping.
Our goal is to sufficiently represent all essential parts in the form of flowers extracted from the thickets of our walloid.

Notice that the parts of an extension that survive the crop operation and are not covered by a $(d, k)$-cover $S$ must invade the vortices only through the corridors corresponding to $S.$
The space provided by the corridors allows for the application of the rerouting of \autoref{lem_taming_linkage_nests} (based on the {vital linkage theorem}), which implies that if the graph $M$ of an extension invades the vortex of a disk $\Delta,$ then an equivalent extension will invade it in a way that only a bounded number of vertices of $M$ intersect the boundary of $\Delta.$
Therefore we may assume that the restriction of $M$ in $\Delta$ is ``represented'' by an alternative fiber of bounded detail.
This is proved in \autoref{my_crops_findslo}.
With such a ``representation'' lemma at hand we can later argue, in \autoref{harvesting_lemma}, that if $S$ does not hit such a fiber, then a big enough packing of this fiber exists in $G-S.$

\begin{lemma}[Representation lemma]\llabel{my_crops_findslo}
There exist functions $f_{\ref{my_crops_findslo}}^{1} \colon \nn{1}{1},$ $f_{\ref{my_crops_findslo}}^{2} \colon \nn{1}{1},$ and $f_{\ref{my_crops_findslo}}^{3} \colon \nn{1}{1}$ such that,
\begin{itemize}
\item for every graph $Z,$
\item every $q$-insulated $\Sigma$-disk-collection triple $(\delta,\mathbf{\Delta},y_{\mathbf{\Delta}})$ of a graph $G,$ for some $q \geq f_{\ref{my_crops_findslo}}^{1}(|Z|),$
\end{itemize}
if,
\begin{itemize}
\item for every $\Delta \in \mathbf{\Delta},$ $\Psi^{\Delta}$ is a corridor system of $\rho^{\Delta},$
\item $\Psi^+ \coloneqq \cupall \{\Psi^{\Delta} \cup \Delta_{c^{\Delta}} \mid \Delta \in \mathbf{\Delta} \},$ and
\item $(M,T)$ is a $\delta$-bound $\mathbf{\Delta}$-respecting extension of $Z$ in $G' \coloneqq G - \pi_{\delta}((\cupall \mathbf{\Delta}) \setminus \Psi^+ ),$
\end{itemize}
then there exists a $\delta$-bound $\mathbf{\Delta}$-respecting extension $(M', T')$ of $Z$ in $G'$ where 
\begin{itemize}
\item $(M', T')$ has a crop in $\Delta$ for at most $f_{\ref{my_crops_findslo}}^{2}(|Z|)$ many distinct disks $\Delta \in \mathbf{\Delta},$
\item for every $\Delta \in \mathbf{\Delta},$ if $(M', T')$ has a crop in $\Delta,$ then $\mathsf{crop}(M', T', \Delta)$ has detail $f_{\ref{my_crops_findslo}}^{3}(|Z|),$ and
\item all edges in $E(M') \setminus E(M)$ are drawn inside $\Psi = \cupall \{ \Psi^{\Delta} \mid \Delta \in \mathbf{\Delta} \}.$
\end{itemize}
Moreover
$$\textrm{$f_{\ref{my_crops_findslo}}^{1}(|Z|) = \Ocal(\ell(\mathsf{h}(|Z|))),$ $f_{\ref{my_crops_findslo}}^{2}(|Z|) = 2^{\Ocal(\ell(\mathsf{h}(|Z|)))},$ and $f_{\ref{my_crops_findslo}}^{3}(|Z|) = \ell(\mathsf{h}(Z)).$}$$
\end{lemma}
\begin{proof} We define $f_{\ref{my_crops_findslo}}^{1}(|Z|) \coloneqq (3\mathsf{h}(|Z|) + 1) (3\ell(\mathsf{h}(|Z|)) + 2),$ $f_{\ref{my_crops_findslo}}^{2}(|Z|) \coloneqq f_{\ref{lem_taming_linkage_nests}}^{3}(\mathsf{h}(|Z|)),$ and $f_{\ref{my_crops_findslo}}^{3}(|Z|) \coloneqq f_{\ref{lem_taming_linkage_nests}}^{2}(\mathsf{h}(|Z|))).$

To begin, we assume that the detail of $(M, T)$ is at most $d \coloneqq \mathsf{h}(|Z|)$ and that no two vertices in $T$ are adjacent in $M.$
Indeed, the former property is implied by \autoref{observation_how_to_find} and can be achieved by removing vertices of $M$ as long as the conditions that $M$ is connected, $\mathsf{dissolve}(M, T)$ contains $Z$ as a minor, $(M, T)$ is $\delta$-bound, and $\Delta$-respecting are satisfied.
For the latter, if two vertices in $T$ are adjacent in $M,$ we simply subdivide this edge once, without affecting any of the desired properties.

The fact that $(M, T)$ is $\mathbf{\Delta}$-respecting implies that $M$ is not entirely drawn in the interior of some $\Delta \in \mathbf{\Delta}$ without intersecting its boundary $B_{\Delta}.$

Also, notice that by definition of an extension, the graph $M - T$ is a graph whose connected components are paths.
Let $\Lcal$ be a linkage in $G'$ that is the set of connected components of $M - T.$
Observe that since the detail of $(M, T)$ is at most $d,$ we also have that the order of $\Lcal,$ $r \coloneqq |\Lcal|,$ is at most $d.$

Now, the proof is essentially a careful application of \autoref{lem_taming_linkage_nests}.
Assume that $\mathbf{\Delta} = \{ \Delta_{1}, \ldots, \Delta_{b} \},$ for some $b \in \Nbbb.$
Now, for every $i \in [b],$ we define the following.
Let $\Ccal^{i}$ be the nest around $c^{\Delta_{i}}$ in $\rho^{\Delta_{i}}.$
Since by assumption, $|\Ccal^{i}| = q \geq (3d + 1)(3\ell(r) + 2),$ by the pigeonhole principle, there is a subset $\Ccal^{i}_{2} = \{ C^{i}_{1}, \ldots, C^{i}_{p}\} \subseteq \Ccal^{i}$ of $p \coloneqq 3\ell(\Ocal(r)) + 2$ many consecutive cycles of $\Ccal^{i}$ such that, no vertex in $T$ and no terminal vertex of $\Lcal$ is drawn in the disk $\Delta_{C_{p}}$ minus every point of this disk drawn in the interior of the disk $\Delta_{C^{i}_{1}}.$
Let $\rho_{2}^{\Delta_{i}}$ be the $p$-nested cylindrical rendition obtained from $\rho^{\Delta_{i}}$ by considering only the subset $\Ccal^{i}_{2}$ of the original nest $\Ccal^{i}.$
For every $i \in [b],$ we define $\Delta'_{i} \coloneqq \Psi^{\Delta_{i}} \cup \Delta_{c^{\Delta_{i}}}.$
Note that $\Delta'_{i}$ is an arcwise connected set and that $M \cap \Delta_{i}$ is drawn in $\Delta'_{i}.$

We are now in the position to apply \autoref{lem_taming_linkage_nests} for the graph $G,$ the linkage $\Lcal$ in $G,$ the $\Sigma$-decomposition $\delta$ of $G,$ the sequence $(\rho^{\Delta_{1}}_{2}, \ldots, \rho^{\Delta_{b}}_{2})$ of $p$-insulated disks of $\delta,$ and the arc-wise connected sets $\Delta'_{1}, \ldots, \Delta'_{b}$ that are subsets of $\Delta_{1}, \ldots, \Delta_{b}$ respectively, with the property necessary to apply that lemma, that is every vertex and edge of $\cupall \Lcal \cap \Delta_{i}$ is drawn in $\Delta'_{i},$ in order to obtain a linkage $\Rcal$ with the following properties.
\begin{enumerate}
\item $\Rcal$ is equivalent to $\Lcal,$
\item $\Rcal$ is a subgraph of the graph $\cupall \Lcal \cup (\bigcup_{i \in [b]} (\cupall \Ccal^{i}_{2} \cap G')),$
\item\llabel{linkprop3} for every $i \in [b],$ the number of disjoint subpaths of paths in $\Rcal$ with at least one endpoint in $B_{\Delta_{i}}$ that contains an edge of drawn in the interior of the disk $\Delta_{C^{i}_{1}}$ is at most $f^{2}_{\ref{lem_taming_linkage_nests}}(r),$ and
\item\llabel{linkprop4} a path in $\Rcal$ contains an edge drawn in the interior of the disk $\Delta_{C^{i}_{1}}$ for at most $f^{3}_{\ref{lem_taming_linkage_nests}}(r)$ distinct indices $i \in [b].$
\end{enumerate}

To conclude, we define the desired extension of $Z$ in $G'$ as the pair $(M', T')$ where $M'$ is obtained from $M$ by replacing every path in $M'$ that corresponds to a path in $\Lcal$ by the corresponding path having the same terminals in $\Rcal$ and $T' = T.$
Clearly, since $\Rcal$ is equivalent to $\Lcal$ and is a subgraph of ~$\cupall \Lcal \cup (\bigcup_{i \in [r]} (\cupall \Ccal^{i}_{2} \cap G')),$ and as such intersects no vertex in $T,$ it is implied that $(M', T')$ is indeed a $\delta$-bound extension of $Z$ in $G'.$
Moreover, all edges in $E(M') \setminus E(M)$ are drawn inside $\Psi$ as required, since $\Rcal$ is a subgraph of  ~$\cupall \Lcal \cup (\bigcup_{i \in [r]} (\cupall \Ccal^{i}_{2} \cap G'))$ and by definition the graph $\bigcup_{i \in [r]} (\cupall \Ccal^{i}_{2} \cap G')$ is drawn inside $\Psi.$

Additionally, the fact that $\Rcal$ satisfies property \ref{linkprop4} above, combined with the definition of crop, implies that $(M', T')$ has a crop in $\Delta$ for at most $f_{\ref{my_crops_findslo}}^{2}(|Z|)$ many distinct disks $\Delta \in \mathbf{\Delta}.$
Finally, the fact that $\Rcal$ satisfies property \ref{linkprop3} above, again combined with the definition of crop, implies that for every $\Delta \in \mathbf{\Delta},$ if $(M', T')$ has a crop in $\Delta,$ 
then $\mathsf{crop}(M', T', \Delta)$ has detail $f_{\ref{my_crops_findslo}}^{3}(|Z|).$ 
This follows from the fact that, for every $i \in [r],$ every path of $M'$ that connects a vertex of $B_{\Delta_{i}}$ with a vertex whose point is in $\Delta_{c^{\Delta_{i}}}$ must contain a subpath of a path in $\Rcal$ that satisfies \ref{linkprop3}.
Since the number of such paths is bounded, this bounds the detail of any crop of $(M', T')$ in $\Delta_{i}.$
\end{proof}

We are now in the position to prove the main lemma of this section which can be seen as a form of a local Erd\H{o}s-P{\'o}sa duality inside a collection of bounded depth disks.
It essentially implies that for every graph on $h$ vertices, an $(\alpha,k)$-cover of a $\Sigma$-disk-collection triple provides such a duality for all their extensions.
The key condition is that $\alpha$ is bounded by a function in $h,$ depending on the unique linkage function $\ell.$

\begin{lemma}[Harvesting lemma]\llabel{harvesting_lemma}
Let $k \in \Nbbb,$ $Z$ be a graph on $h$ vertices, and $(\delta,\mathbf{\Delta},y_{\mathbf{\Delta}})$ be a $q$-insulated $\Sigma$-disk-collection triple of a graph $G,$ where $q \geq f_{\ref{my_crops_findslo}}^{1}(h).$
If $S$ is the $(f_{\ref{my_crops_findslo}}^{3}(h), k)$-cover of $(\delta, \mathbf{\Delta}, y_{\mathbf{\Delta}})$ and $(M,T)$ is a $\delta$-bound extension of $Z$ in $G' \coloneqq G - S,$ then there exists a $\delta$-bound and $\mathbf{\Delta}$-respecting extension $(M',T')$ of $Z$ in $G'$ such that
\begin{itemize}
\item $|T'|≤ \mathsf{h}(h)$ and
\item for every $\Delta \in \mathbf{\Delta},$ if $(M',T')$ has a crop in $\Delta$
and $\mu = \mathsf{crop}(M', T', \Delta),$ then
\begin{itemize}
\item $\lin{G\cap \Delta, \Lambda^{(\Delta)}}$ has an $\lin{H, {\Lambda}'}^{(k)}$-fiber, where $\lin{H, {\Lambda}'}$ is a shift of $\lin{H,\widehat{\Lambda}} \coloneqq \mathsf{dissolve}(\mu),$ and
\item $|V({\Lambda}')|≤f_{\ref{my_crops_findslo}}^{3}(h).$
\end{itemize}
\end{itemize}
\end{lemma}
\begin{proof} Let $d \coloneqq f_{\ref{my_crops_findslo}}^{3}(h).$
Firstly, by definition of the $(d, k)$-cover of $(\delta, \Delta, y_{\mathbf{\Delta}}),$ $S = \bigcup_{\Delta \in \mathbf{\Delta}} S_{\Delta},$ where for every $\Delta \in \mathbf{\Delta},$ $S^{\Delta}$ is the union of $\Delta$-compact sets obtained by the application of \autoref{func_sharvsp} to every linear society of detail at most $d.$
As a result, by \autoref{obs:compact_union}, each $S^{\Delta},$ $\Delta \in \mathbf{\Delta},$ is itself a $\Delta$-compact set and as such it is accompanied by a corridor system $K^{\Delta}$ of $\rho^{\Delta}.$
Let $$K^{+} = \cupall \{ K^{\Delta} \cup \Delta_{c^{\Delta}} \mid i \in [r] \}.$$
Now, notice that since, for every $\Delta \in \mathbf{\Delta},$ every $\Delta$-compact set $S^{\Delta}$ is accompanied by the corridor system $K^{\Delta},$ by definition it must be that $(M, T)$ is also an extension of $Z$ in the graph $G'' \coloneqq G - \pi_{\delta}((\cupall \mathbf{\Delta}) \setminus K^{+}).$

Assume that it is the case that $(M, T)$ is also $\mathbf{\Delta}$-respecting.
Then, we are in the position to apply \autoref{my_crops_findslo} on $Z,$ $\delta,$ $\mathbf{\Delta},$ $K^{\Delta},$ for every $\Delta \in \mathbf{\Delta},$ and $(M, T),$ in order to obtain a $\delta$-bound and $\mathbf{\Delta}$-respecting extension $(M', T')$ of $Z$ in $G''$ with the following additional properties.
\begin{enumerate}
\item\llabel{harvprop1} $(M', T')$ has a crop in $\Delta$ for at most $f_{\ref{my_crops_findslo}}^{2}(h)$ many distinct disks $\Delta \in \mathbf{\Delta},$
\item\llabel{harvprop2} for every $\Delta \in \mathbf{\Delta},$ if $(M', T')$ has a crop in $\Delta,$ then $\mathsf{crop}(M', T', \Delta, \Delta)$ has detail $f_{\ref{my_crops_findslo}}^{3}(h),$ and
\item\llabel{harvprop3} all edges in $E(M') \setminus E(M)$ are drawn inside $K = \cupall \{ K^{\Delta} \mid \Delta \in \mathbf{\Delta} \}.$
\end{enumerate}

The fact that $(M', T')$ satisfies property \ref{harvprop3} says that $M'$ differs from $M$ only on vertices and edges drawn within the corridor systems accompanying each $\Delta$-compact set $S^{\Delta},$ for every $\Delta \in \mathbf{\Delta}.$
In particular, by definition of compact sets, this implies that $(M', T')$ is also an extension of $Z$ in $G'.$
This combined with the fact that $(M', T')$ satisfies property \ref{harvprop2} implies the following.
Assume that for a given $\Delta \in \mathbf{\Delta},$ $(M', T')$ has a crop in $\Delta.$
This implies that $\mu^{\Delta} \coloneqq \mathsf{crop}(M', T', \Delta)$ is a $\mathsf{dissolve}(\mu^{\Delta})$-fiber of $\rho^{\Delta}$ of detail at most $f_{\ref{my_crops_findslo}}^{3}(|Z|).$
Assume that $\lin{H, \widehat{\Lambda}} = \mathsf{dissolve}(\mu^{\Delta}).$
As we previously observed, $\mu^{\Delta}$ is a $\lin{H, \widehat{\Lambda}}$-fiber of the graph $G'$ and more specifically it is also 
a $\lin{H, \widehat{\Lambda}}$-fiber of $\lin{G \cap \Delta, \Lambda^{(\Delta)}} - S^{\Delta}.$
Hence, an application of \autoref{func_sharvsp} on $\rho^{\Delta},$ $\lin{H, \widehat{\Lambda}},$ and $k,$ must output an $\lin{H, \Lambda'}^{(k)}$-fiber of $\lin{G \cap \Delta, \Lambda^{(\Delta)}}$ as required.

It remains to handle the case where $(M, T)$ is not $\mathbf{\Delta}$-respecting.
By definition, this implies that $M$ is drawn in the interior of exactly one of the disks say $\Delta \in \mathbf{\Delta}.$
However, the fact that $(M, T)$ is $\delta$-bound, implies that there is a path $P$ in $G$ that is internally disjoint from $M$ whose endpoints are a vertex of $M,$ say $x$ and a vertex of $B_{\Delta}.$
We then define $M'' \coloneqq M \cup P$ and $T'' \coloneqq T \cup \{ x \}.$
Clearly, $(M'', T'')$ is a $\delta$-bound and $\mathbf{\Delta}$-respecting extension of $Z$ in $G'$ and we can now reduce to the previous case.
\end{proof}

\section{Plucking flowers from thickets}\llabel{extr_flowers_thickets}

The main outcome of \autoref{sec_harvest_crops}, namely \autoref{harvesting_lemma}, states a type of local Erd\H{o}s-P{\'o}sa duality confined inside a collection of bounded depth disks of our $\Sigma$-decomposition.
In this section, we advance towards our goal by first applying that lemma to the $\Sigma$-disk-collection $(\delta,\mathbf{\Delta}, y_{\mathbf{\Delta}})$ that naturally corresponds to the thickets of our refined orchard, for a given graph $Z.$
If none of the $(f_{\ref{my_crops_findslo}}^{3}(|Z|), k)$-covers, that correspond to the disks of $(\delta,\mathbf{\Delta}, y_{\mathbf{\Delta}}),$ obtained from the application of \autoref{harvesting_lemma}, intersect the graph $M$ of an extension of $Z,$ then we can find a sequence of many copies of the crop of $\Delta \cap M$ inside $\Delta,$ for every $\Delta \in \mathbf{\Delta}.$
We, then proceed with a series of successive transformations of these sequences that, in the end, will allow us to see every thicket as a flower that will be used in order to certify the existence of a guest walloid, as required for the proof of our local structure theorem in \autoref{sec_local_structure}.

\subsection{Certifying flatness}\label{certif_flat}

As we already mentioned, we use \autoref{harvesting_lemma} in order to extract 
a ``packing'' of fibers rooted on a collection of disks around the vortices of our $\Sigma$-decomposition $\delta.$ 
This collection will later permit us to see the thickets of our orchard as the flowers of a guest walloid.
Recall that in an orchard all flaps of $\delta$ are already represented by its flower segments, so the goal of applying \autoref{harvesting_lemma} for the purposes of this section, is in order to guarantee that we can also sufficiently represent the parts of an extension that invade the vortices of our $\Sigma$-decomposition.
To make the interaction 
between the vortex-invading part of an extension of some graph $Z$ with its flat part more precise,
we first give a way to formally certify flatness.

\paragraph{Flatness certifying complements.}
Let $Z$ be a graph and $\text{\bf \mu} = \lin{\mu^{1}, \ldots, \mu^{z}}$ be a sequence of fibers where $\mu^{i} = (M^i, T^i, \Lambda^i), i\in[z].$

A $\Sigma$-\defi{flat} $Z$-\defi{certifying complement} of $\text{\bf \mu}$ is a tuple $\bfl{C} =({\Sigma}^{\mathsf{c}}, {M}^{\mathsf{c}},{T}^{\mathsf{c}}, {\delta}^{\mathsf{c}},Y)$ where
\begin{enumerate}
\item ${\Sigma}^{\mathsf{c}}$ is the closure of $\Sigma \setminus \cupall \mathbf{\Delta}$ where $\mathbf{\Delta} = \{\Delta^{1},\ldots,\Delta^{z}\}$ is a collection of pairwise-disjoint disks of $\Sigma,$ all oriented counter-clockwise,

\item $({M}^{\mathsf{c}}, {T}^{\mathsf{c}})$ is a pair, where ${M}^{\mathsf{c}}$ is a graph, ${T}^{\mathsf{c}} \subseteq V({M}^{\mathsf{c}})$ and a vertex $x \in V(M^{\mathsf{c}}) \setminus {T}^{\mathsf{c}}$ has degree two if $x \not \in \bigcup_{i\in[z]} V(\Lambda^i),$ otherwise it has  degree zero or one in $M^{\mathsf{c}}.$

\item ${\delta}^{\mathsf{c}}$ is a vortex-free ${\Sigma}^{\mathsf{c}}$-decomposition of ${M}^{\mathsf{c}},$

\item for every $i\in[z],$ $V(M^{i}) \cap V({M}^{\mathsf{c}}) = V(\Lambda^i)$ and the points, relatively to $\delta^{\mathsf{c}},$ of all vertices of $V(\Lambda^i)$ are on the same boundary of ${\Sigma}^{\mathsf{c}},$

\item for every $i\in[z]$ and every $y \in V(\Lambda^{i}) \setminus T^{i}$ it holds that $\deg_{M^i}(y) + \deg_{M^{c}}(y) = 2,$ 

\item $Y = \{y^1,\ldots,y^{z}\},$ where for $i\in[z],$ the ordering of the vertices of $V(M^{i})$ whose points are on the boundary of ${\Sigma}^{\mathsf{c}},$ when starting from $y^i$ and following the clockwise ordering of the boundary of ${\Sigma}^{\mathsf{c}}$ where they belong, is the same as the ordering of $\Lambda^i$ (following the counterclockwise ordering of $\bd(\Delta_i)$), and

\item\llabel{condition_extension_all_union} $({M}^{\mathsf{c}}\cup\bigcup_{i\in[z]}M^{i},T^{\mathsf{c}}\cup\bigcup_{i\in[z]}T^{i})$ is an extension of $Z.$
We call this extension the $(\bfl{C},\text{\bf \mu})$-\defi{certifying} extension of $Z.$
\end{enumerate}

The \defi{detail} of the pair $(\mathbf{C}, \text{\bf \mu})$ is the number of terminals of the extension of $Z$ in condition \ref{condition_extension_all_union} above.
Notice that a bound on the detail of $\mathbf{C}$ implies also a bound on the total number of terminals of the fibers of $\text{\bf \mu}.$
We also define
$$d\folio(\bfl{C}) = d\folio(\delta^{\mathsf{c}}).$$

We say that a society $\lin{H, \Lambda}$ is \defi{prime} if $\mathsf{dec}(H,\Lambda) = \lin{\lin{H, \Lambda}}.$
Also, we say that a fiber $\mu = (M, T, \Lambda)$ is \defi{prime} if $\mathsf{dissolve}(\mu)$ is prime.
Given two disjoint fibers $\mu_{1} = (M_{1}, T_{1}, \Lambda_{1})$ and $\mu_{2} = (M_{2}, T_{2}, \Lambda_{2}),$ we say that they are \defi{disjoint} if $V(M_{1}) \cap V(M_{2}) = \emptyset.$
Given that $\mu_{1}$ and $\mu_{2}$ are disjoint, we define $\mu_{1} \oplus \mu_{2} = (M_{1} \cup M_{2}, T_{1} \cup T_{2}, \Lambda_{1} \oplus \Lambda_{2}).$

\begin{figure}[ht]
\begin{center}
\scalebox{1}{\includegraphics{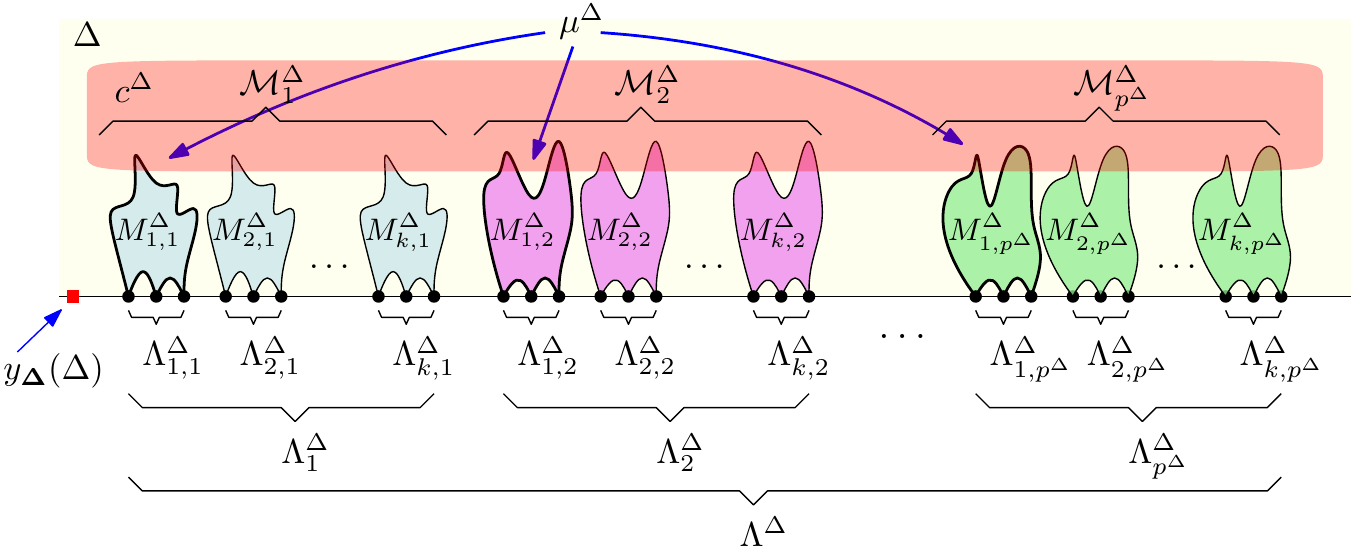}}
\end{center}
    \caption{A visualization of \autoref{internal_certification_without_thickets} for one, say $\Delta,$  of the disks of $\mathbf{\Delta}'.$ The fiber $\mu^\Delta$ consists of the first prime fiber of each of the
    $p^{\Delta}$ groups of pairwise equivalent fibers.
    Each fiber intersects the disk $\Delta_{c^{\Delta}}$ of the vortex $c^{\Delta}$ because it is a result of the cropping operation. 
    }
  \llabel{first_pic_of_many_fibers}
\end{figure}

The following lemma is a consequence of \autoref{func_sharvsp} and the definition of a $\Sigma$-flat $Z$-certifying complement (for a visualization, see \autoref{first_pic_of_many_fibers}).

\begin{lemma}\llabel{internal_certification_without_thickets}
There exists a function $f_{\ref{internal_certification_without_thickets}} \colon \nn{4}{1}$ and an algorithm that, given
\begin{itemize}
\item $k, h, q \in \Nbbb_{≥1},$ $z, w\in \Nbbb,$
\item a graph $G,$
\item a $q$-insulated $\Sigma$-disk-collection triple $( \delta, \mathbf{\Delta}, y_{\mathbf{\Delta}} )$ of $G,$ where
\begin{itemize}
\item $q \geq f_{\ref{my_crops_findslo}}^{1}(h),$  \item $|\mathbf{\Delta}|≤z,$  and 
\item for every disk $\Delta \in \mathbf{\Delta},$ $\rho^{\Delta}$ has combined depth $≤w,$
\end{itemize}
\end{itemize}
outputs a set $S,$ where $|S|≤f_{\ref{internal_certification_without_thickets}}(z,w,k,h),$
such that for every graph $Z$ on $≤h$ vertices, either 
$G-S$ contains no $\delta$-bound extension of $Z$
or there exists a $\mathbf{\Delta}' \subseteq \mathbf{\Delta}$ and a set $\text{\bf \mu}\coloneqq\{\mu^{\Delta}\mid \Delta\in\mathbf{\Delta'}\}$ of fibers such that
\begin{itemize}
\item for every $\Delta\in\mathbf{\Delta}'$ there exists a $p^{\Delta}\in\Nbbb,$ where 
\begin{itemize}
\item for every $j\in [p^{\Delta}],$ there exists a set $\Mcal_j^{\Delta}\coloneqq\{\mu_{1,j}^{\Delta},\ldots,\mu_{k,j}^{\Delta}\}$ of $k$ pairwise disjoint and pairwise equivalent prime fibers, where 
\begin{itemize}
\item for every $i\in[k],$ $\mu_{i,j}^{\Delta} \coloneqq (M_{i,j}^{\Delta}, T_{i,j}^{\Delta}, \Lambda_{i,j}^{\Delta})$ and $M_{i,j}^{\Delta}$ has a vertex drawn in $\Delta_{c^{\Delta}},$ and
\item $\Lambda_{j}^{\Delta}\coloneqq\Lambda_{1,j}^{\Delta}\oplus\cdots \oplus \Lambda_{k,j}^{\Delta},$
\end{itemize}
\item $\Lambda^{\Delta} \coloneqq \Lambda_{1}^{\Delta} \oplus \cdots \oplus \Lambda_{p^{\Delta}}^{\Delta}$ and $\Lambda^{\Delta} \subseteq \Lambda^{(\Delta)},$ and  
\item $\mu^{\Delta} = \mu_{1,1}^{\Delta} \oplus \ldots, \oplus \mu_{1,p^{\Delta}}^{\Delta}$ is a fiber of $(G\cap \Delta,\Lambda^{(\Delta)}),$
\end{itemize}
\item $\text{\bf \mu}$ has a $\Sigma$-flat $Z$-certifying complement $\bfl{C}$ where $d\folio(\bfl{C}) \subseteq d\folio(\delta),$ for every $d \in \Nbbb,$ and 
\item the boundary of each of the fibers in $\text{\bf \mu}$ has at most $f_{\ref{my_crops_findslo}}^{3}(h)$ vertices, $|\mathbf{\Delta}'| ≤ f_{\ref{my_crops_findslo}}^{2}(h),$ and the detail of $(\bfl{C}, \text{\bf \mu})$ is at most $\mathsf{h}(h),$
\end{itemize}
in time
$$\Ocal_{h}(zwk \cdot |G|^3).$$
Moreover
$$f_{\ref{internal_certification_without_thickets}}(z, w, k, h) = \Ocal_{h}(zwk).$$
\end{lemma}
\begin{proof}
Let $S$ be the $(f_{\ref{my_crops_findslo}}^{3}({\sf h}(h)), k)$-cover of $(\delta,\mathbf{\Delta}, y_{\mathbf{\Delta}}),$ where $\delta=(\Gamma,\Dcal).$ As each disk in $\mathbf{\Delta}$ has  {combined depth}  $≤w,$ from \autoref{func_sharvsp}
we have that $|S|≤f_{\ref{func_sharvsp}}(k, w,f_{\ref{my_crops_findslo}}^{3}({\sf h}(h))))\cdot z.$ 
Also, from \autoref{how_to_find_fast_cover}
$S$ can be found in time  $\Ocal_{h}(k\cdot w\cdot  z\cdot |G|^3 )$-time.

We claim that the lemma holds if we set $f_{\ref{internal_certification_without_thickets}}(z,w,k,h)=f_{\ref{func_sharvsp}}(k, w,f_{\ref{my_crops_findslo}}^{3}({\sf h}(h)), k))\cdot z.$
Suppose that $Z$ is a graph on at most $h$ vertices 
such that $G'\coloneqq G-S$ contains some $\delta$-bound extension $(M,T)$ of $Z,$ as the one in \autoref{observation_how_to_find}, that has detail bounded by ${\sf h}(h).$ Then, 
consider a  $\delta$-bound  and  $\mathbf{\Delta}$-respecting extension $(M',T')$ of $Z$ in $G'$ where $|T'|≤{\sf h}(h)$ as the one of  \autoref{harvesting_lemma}. Let also $\mathbf{\Delta}'\subseteq \mathbf{\Delta}$
be the set of the disks of $\mathbf{\Delta}$
where $(M',T')$ has a crop and, for every $\Delta\in  \mathbf{\Delta}',$ let 
$\overline{\mu}^{\Delta}=\mathsf{crop}(M', T', \Delta).$ From \autoref{my_crops_findslo}, $|\mathbf{\Delta}'|≤f_{\ref{my_crops_findslo}}^{2}({\sf h}(h)).$

We next claim that $\overline{\text{\bf \mu}}\coloneqq\{\overline{\mu}^{\Delta}\mid \Delta\in\mathbf{\Delta}'\}$ has a $\Sigma$-flat  $Z$-certifying complement
$\bfl{C}.$
Indeed, this complement is the tuple $\bfl{C}=({\Sigma}^{\mathsf{c}},{M}^{\mathsf{c}},{T}^{\mathsf{c}},{\delta}^{\mathsf{c}},Y)$ where 
\begin{enumerate}
\item  ${\Sigma}^{\mathsf{c}}$ is the closure of $\Sigma\setminus \cupall \mathbf{\Delta}',$ 
\item ${M}^{\mathsf{c}}\coloneqq {M'\cap {\Sigma}^{\mathsf{c}}}$ and ${T}^{\mathsf{c}}\coloneqq T'\cap {\Sigma}^{\mathsf{c}}$ (as $(M',T')$ is $\mathbf{\Delta}$-respecting, $M^{\mathsf{c}}$ is non-empty.),
\item $\delta^{\mathsf{c}}= (\Gamma\cap {\Sigma}^{\mathsf{c}}, \{\Delta_{c} \in \mathcal{D} \mid \Delta_{c} \subseteq {\Sigma}^{\mathsf{c}} \})\cap M^{\mathsf{c}},$ i.e., we first take the restriction of $\delta$ inside ${\Sigma}^{\mathsf{c}}$ and then we consider the rendition induced by  $M^{\mathsf{c}},$ and 
\item $Y=\{y_{\mathbf{\Delta}}(\Delta)\mid \Delta\in\mathbf{\Delta}'\}.$
\end{enumerate} 
Notice that  $d\folio(\bfl{C})=d\folio(\delta^{\mathsf{c}})\subseteq d\folio(\delta),$ for every $d\in\Nbbb.$ Also $|T'|≤{\sf h}(h)$ implies 
directly that $(\mathbf{C},\overline{\text{\bf \mu}})$ has detail $≤{\sf h}(h).$

From \autoref{harvesting_lemma}, $\mathsf{dissolve}(\overline{\mu}^{\Delta})$ has a shift $\lin{H^\Delta, {\Lambda}^{\prime\Delta}}$ such that  $\lin{G \cap \Delta, \Lambda^{(\Delta)}}$ contains  an $\lin{H^{\Delta}, {\Lambda}^{\prime\Delta}}^{(k)}$-fiber $\widetilde{\mu},$
where $|V({\Lambda}^{\prime\Delta})|≤f_{\ref{my_crops_findslo}}^{3}(h).$
Let $${\sf dec}(\lin{H^{\Delta}, {\Lambda}^{\prime\Delta}})=\lin{\lin{H_{1}^{\Delta}, {\Lambda}^{\prime\Delta}_1},\ldots,\lin{H_{p^{\Delta}}^{\Delta}, {\Lambda}_{p^{\Delta}}^{\prime\Delta}}}.$$
For every $j\in[p^{\Delta}],$
the fibers  $\mu_{1,j}^{\Delta},\ldots,\mu_{k,j}^{\Delta}$ are defined 
as the pairwise equivalent prime fibers such that  ${\sf dissolve}(\mu_{i,j}^{\Delta})=\lin{H_{j}^{\Delta}, {\Lambda}^{\prime\Delta}_j}$ and where 
\begin{eqnarray}
\widetilde{\mu} &  = & \mu_{1,1}^{\Delta}\oplus\cdots\oplus\mu_{k,1}^{\Delta}\ \oplus\ \ \cdots\ \ \oplus\ \ \mu_{1,j}^{\Delta}\oplus\cdots\oplus\mu_{k,j}^{\Delta}\ \ \oplus\ \ \cdots\ \ \oplus\ \mu_{1,p^{\Delta}}^{\Delta}\oplus\cdots\oplus\mu_{k,p^{\Delta}}^{\Delta}.\llabel{many_fibers_all_together}
\end{eqnarray}
Let $\mu_{i,j}^{\Delta}\coloneqq(M_{i,j}^{\Delta},T_{i,j}^{\Delta},\Lambda_{i,j}^{\Delta}).$ If some $j\in[p^{\Delta}],$ $M_{i,j}^{\Delta}$ has no vertex drawn in 
$\Delta_{c^{\Delta}},$ this means that $M_{i,j}^{\Delta}$ is drawn 
in a $\delta$-aligned closed disk $\Delta_{j}$ of  $\Delta'$ that does not intersect the vortex disk $\Delta_{c^{\Delta}}$ and where $\bd(\Delta_{j})\cap\bd(\Delta)$ is a line   where all the boundary vertices of $M_{i,j}^{\Delta}$ are drawn. We update $\overline{\mu}$ by eliminating $\mu_{1,j}^{\Delta}\oplus\cdots\oplus\mu_{k,j}^{\Delta}$ in \eqref{many_fibers_all_together} and update $\bfl{C}$ by enhancing it with  the ``flat'' drawing of $M_{i,j}^{\Delta}$  by setting 
$\Sigma^{\mathsf{c}}\coloneqq \Sigma^{\mathsf{c}}\cup\Delta_{j},$
$M^{\mathsf{c}}\coloneqq M^{\mathsf{c}}\cup M_{i,j}^{\Delta},$ $T^{\mathsf{c}}\coloneqq T^{\mathsf{c}}\cup T_{i,j}^{\Delta}$
and update $\delta^{\mathsf{c}}$ as in step 3 above.
By applying these enhancements, we {guarantee} that 
every graph $M_{i,j}^{\Delta}$ must have  some vertex drawn in $\Delta_{c^{\Delta}}.$ 

We next observe that  $\mu^{\Delta}=\mu_{1,1}^{\Delta}\oplus\ldots,\oplus\mu_{1,p^{\Delta}}^{\Delta}$ is created by concatenating the first, i.e., 
$\mu_{1,j}^{\Delta},$ fiber from each of the above $p^{\Delta}$ groups  of $k$ pairwise equivalent fibers 
each. Notice also that for every $\Delta\in\mathbf{\Delta},$ ${\sf dissolve}(\mu^{\Delta})$ is a shift of 
${\sf dissolve}(\overline{\mu}^{\Delta})$ and that $\mu^{\Delta}$ as well as $\overline{\mu}^{\Delta}$
have a boundary of  $|V({\Lambda}^{\prime\Delta})|≤f_{\ref{my_crops_findslo}}^{3}(h)$ vertices. This implies that  $\bfl{C}'=({\Sigma}^{\mathsf{c}},{M}^{\mathsf{c}},{T}^{\mathsf{c}},{\delta}^{\mathsf{c}},Y')$
is  a $\Sigma$-flat $Z$-certifying complement of ${\text{\bf \mu}}\coloneqq \{{\mu}^{\Delta}\mid \Delta\in\mathbf{\Delta}'\}$
where $Y'$ is a suitable update of $Y$ according to to the shifts required for transforming  $\overline{\mu}^{\Delta}$ to ${\mu}^{\Delta},$ for each $\Delta\in\mathbf{\Delta}'.$
Clearly, it holds that $d\folio(\bfl{C}')=d\folio(\bfl{C})\subseteq d\folio(\delta),$ for every $d\in\Nbbb.$
\end{proof}

Notice that in the proof above, the $Z$-certifying complement $\bfl{C}$ is constructed based on the way that $(M', T')$ is drawn outside of the disks of $\mathbf{\Delta},$ according to $\delta.$
Moving on, we will successively derive several collections $\text{\bf \mu}$ of fibers whose $\Sigma$-flat $Z$-certifying complements will be constructed based on the previous $\Sigma$-flat $Z$-certifying complements and not explicitly from the $\Sigma$-decomposition $\delta.$

We refer to each set $\Mcal_{j}^{\Delta},$ $j\in[p^{\Delta}]$ as in the conclusion of \autoref{internal_certification_without_thickets}, as the $j$-\defi{th package} (of fibers) in the disk $\Delta\in \mathbf{\Delta}'.$
While $\mu^{\Delta}$ is built by concatenating the first fiber of each $j$-th package, one might 
instead use any other in $\Mcal_{j}^{\Delta},$ as they are all equivalent.
We refer to the sets $\Mcal_{1}^{\Delta}, \ldots, \Mcal_{p^{\Delta}}^{\Delta}$ as \defi{packages} of $\text{\bf \mu}.$

\subsection{Flatness certifying complements of thickets}

Our next step is to apply \autoref{internal_certification_without_thickets} in the thickets of a $(t, b, h, s, \Sigma)$-orchard.
In this subsection we consider such a thicket and we explain how the packages of the fibers in the output of \autoref{internal_certification_without_thickets} can be seen as fibers rooted on the boundary of a thicket society of our given orchard.

\paragraph{Thickets and their renditions.}

Let $(G,\delta,W)$ be a $(t,b,h,s,\Sigma)$-orchard and let $c \in\Ccal_{\mathsf{f}}(\delta)$ be the vortex of one of the thickets of $W$ whose railed nest is $(\Ccal, \Rcal).$

Let also $\Delta$ be the disk bounded by the trace of the outer cycle $C_{s+1}$ of the thicket containing the vortex $c.$
We define $\mathbf{\Delta}$ as the set of all such disks corresponding to the thickets of $(G,\delta,W)$ and we refer to it as the \defi{disk  collection} of $(G,\delta,W).$
For each $\Delta \in \mathbf{\Delta},$ we denote its corresponding railed nest by $(\Ccal^{\Delta}, \Rcal^{\Delta}),$ where $\Ccal^{\Delta} = \{C_{0}^{\Delta},\ldots,C_{s+1}^{\Delta}\}$ and $\Rcal^{\Delta} = \{R_{1}^{\Delta}, \ldots, R_{t}^{\Delta}\}.$
Also,
\begin{itemize}
\item for $i \in [t],$ we update $R_{i}^{\Delta}$ to be the unique connected component of $R_{i}^{\Delta} \cap \Delta$ that has one endpoint in $V(C_{0}^{\Delta})$ and we call the other endpoint $\widehat{y}_{i}^{\Delta}$ (that is a vertex of $B_{\Delta}$),
\item we define $\widehat{\Lambda}^{\Delta} \coloneqq \lin{\widehat{y}_{1}^{\Delta},\ldots,\widehat{y}_{t}^{\Delta}},$ and 
\item we update $\Ccal^{\Delta}\coloneqq \{C_{1}^\Delta,\ldots,C_{s}^{\Delta}\},$
\end{itemize}

We refer to $\widehat{\mathbf{\Lambda}} = \{\widehat{\Lambda}^{\Delta} \mid \Delta \in \mathbf{\Delta}\}$ as the \defi{boundary} collection of the $(t,b,h,s,\Sigma)$-orchard $(G,\delta,W).$
We also use $\widehat{y}^{\Delta}$ as a simpler alternative for the first vertex of $\widehat{\Lambda}^{\Delta},$ i.e., the vertex $\widehat{y}_{1}^{\Delta}.$

\medskip
We now explain how to see the fibers in $\{\mu^{\Delta}\mid \Delta\in{\mathbf{\Delta}}'\}$ in the output of \autoref{internal_certification_without_thickets} as fibers whose boundaries are in $B_{\Delta}.$
This is done in \autoref{extending_internal_Certification_in_thickets} below. 
We stress that the constants $h',$ $k',$ and $t'$ in the statement of \autoref{extending_internal_Certification_in_thickets} are irrelevant as they are not used in neither its conclusion nor its proof.

\begin{lemma}\llabel{extending_internal_Certification_in_thickets}
There exists a function $f_{\ref{rooted_certification_into_flower}} \colon \nn{5}{1}$ and an algorithm that, given
\begin{itemize}
\item $k, h \in \Nbbb_{≥ 1}$ and 
\item a $(z,p)$-ripe and $d'$-blooming $(t',k',h',s,\Sigma)$-orchard $(G,\delta,W),$ where $s ≥ f_{\ref{my_crops_findslo}}^{1}(h)+1,$ 
\end{itemize}
outputs a set $S,$ where $|S| ≤ f_{\ref{extending_internal_Certification_in_thickets}}(z,p,k,h,s)$ such that for every graph $Z$ on at most $h$ vertices, either $G - S$ contains no $\delta$-bound extension of $Z$ or there is a subset ${\mathbf{\Delta}}'$ of the disk collection $\mathbf{\Delta}$ of $(G,\delta,W),$ a svaf $y_{\mathbf{\Delta}'},$ and a set $\text{\bf \mu} \coloneqq \{\mu^{\Delta}\mid \Delta\in{\mathbf{\Delta}}'\}$ of fibers such that  
\begin{itemize}
\item for every $\Delta\in{\mathbf{\Delta}'}$ there exists some $p^{\Delta}\in\Nbbb$ where
\begin{itemize}
\item for every $j\in [p^{\Delta}],$ there exists a $j$ package $\Mcal_j^{\Delta}\coloneqq \{\mu_{1,j}^{\Delta},\ldots,\mu_{k,j}^{\Delta}\}$ of $k$ pairwise disjoint and pairwise equivalent prime fibers, where 
\begin{itemize}
\item for every $i \in [k],$ $\mu_{i,j}^{\Delta} \coloneqq (M_{i,j}^{\Delta}, T_{i,j}^{\Delta}, \Lambda_{i,j}^{\Delta})$ and $M_{i,j}^{\Delta}$ has some vertex drawn in $\Delta_{c^{\Delta}}$ and
\item $\Lambda_{j}^{\Delta}\coloneqq \Lambda_{1,j}^{\Delta}\oplus\cdots \oplus \Lambda_{k,j}^{\Delta},$
\end{itemize}
\item $\Lambda^{\Delta}\coloneqq \Lambda_{1}^{\Delta}\oplus\cdots\oplus\Lambda_{p^{\Delta}}^{\Delta}$ and $\Lambda^{\Delta}\subseteq \Lambda^{(\Delta)},$ and  
\item  $\mu^{\Delta}=\mu_{1,1}^{\Delta}\oplus\ldots,\oplus\mu_{1,p^{\Delta}}^{\Delta}$ is a fiber of $(G\cap \Delta,\Lambda^{(\Delta)}),$
\end{itemize}
\item $\text{\bf \mu}$ has a $\Sigma$-flat $Z$-certifying complement $\bfl{C}$ where $d\folio(\bfl{C})\subseteq d\folio(\delta),$ for every $d\in\Nbbb,$ and 
\item the boundary of each of the fibers in $\text{\bf \mu}$ has at most $f_{\ref{my_crops_findslo}}^{3}(h)$ vertices, $|\mathbf{\Delta}'| ≤ f_{\ref{my_crops_findslo}}^{2}(|Z|),$ and the detail of $(\bfl{C},\text{\bf \mu})$ is at most ${\sf h}(h).$
\end{itemize}
Moreover, the above algorithm runs in time $\Ocal_{h}({k} \cdot (s+p) \cdot  z \cdot |G|^3 )$ and $f_{\ref{extending_internal_Certification_in_thickets}}(z, p, k, h, s) = {\Ocal_{h}(z \cdot k \cdot (s +p))}.$
\end{lemma}
\begin{proof}
We prove the lemma for the following quantities.
\begin{itemize}
\item $w \coloneqq 2s + p$ and
\item $f_{\ref{extending_internal_Certification_in_thickets}}(z,p,k,h,s) \coloneqq f_{\ref{internal_certification_without_thickets}}(z, w, k, h).$
\end{itemize}

We pick some svaf $y_{\mathbf{\Delta}}$ and we consider the $\Sigma$-disk-collection triple $(\delta, \mathbf{\Delta}, y_{\mathbf{\Delta}})$ of $G.$
This gives rise to some linear ordering $\Lambda^{(\Delta)}$ for every $\Delta \in \mathbf{\Delta}.$
Let $\Delta \in \mathbf{\Delta}$ and let $(\Ccal^{\Delta}, \Rcal^{\Delta})$ be the corresponding railed nest.
We also consider the $\delta$-aligned disks $\Delta_{1},\ldots,\Delta_{s}$  where $\bd(\Delta_{i})$ is the $\trace_{\delta}(C^{\mathsf{si}})$-avoiding disk of the $\delta$-trace of the cycle $C_{i}^{\Delta},$ $i\in[s],$ $\Delta_{c^{\Delta}}$ is a subset of the interior of $\Delta_{1},$ $\Delta_{1} \subseteq \cdots \subseteq \Delta_{s},$ and $\Delta_{s}$ is a subset of the interior of $\Delta.$
Also consider the $s$-nested cylindrical rendition $\rho^{\Delta} = \delta \cap \Delta$ of the linear society $\lin{G \cap \Delta,\Lambda^{(\Delta)}}$ around the cell $c_{\Delta}.$
Observe that $\lin{G \cap \Delta, \Lambda^{(\Delta)}}$ corresponds to a thicket society of the $(t, s)$-thicket that corresponds to $\Delta,$ and that since $(G, \delta, W)$ is $(z, p)$-ripe, $\rho^{\Delta}$ has combined depth at most $w.$

We are now in a position to apply \autoref{internal_certification_without_thickets} for the $q$-insulated $\Sigma$-disk-collection triple $(\delta, \mathbf{\Delta}, y_{\mathbf{\Delta}})$ of $G,$ for $q \coloneqq s,$ and $k.$
This takes time $\Ocal_{h}(k \cdot w\cdot  z\cdot |G|^3) = \Ocal_{h}(k \cdot  z \cdot (s+p) \cdot |G|^3).$
Given that $S$ is the output of the application of \autoref{internal_certification_without_thickets} above, we have that $|S| \leq f_{\ref{internal_certification_without_thickets}}(z, w, k, h) = f_{\ref{extending_internal_Certification_in_thickets}}(z, p, k, h, s).$
Moreover, it holds that for every graph $Z$ on at most $h$ vertices, either $G - S$ contains no $\delta$-bound extension of $Z$ or there is a set $\mathbf{\Delta}' \subseteq \mathbf{\Delta}$ and a set $\text{\bf \mu} \coloneqq \{ \mu^{\Delta} \mid \Delta \in \mathbf{\Delta}' \}$ of fibers such that
\begin{enumerate}   
\item for every $\Delta \in \mathbf{\Delta}',$ there exists some $p^{\Delta} \in \Nbbb$ where
\begin{itemize}
\item for every $j \in [p^{\Delta}]$ there exists a $j$ package $\Mcal_j^{\Delta} \coloneqq \{\mu_{1,j}^{\Delta}, \ldots, \mu_{k, j}^{\Delta} \}$ of $k$ pairwise disjoint and pairwise equivalent prime fibers, where 
\begin{itemize}
\item for every $i \in [k],$ $\mu_{i,j}^{\Delta} \coloneqq (M_{i,j}^{\Delta}, T_{i,j}^{\Delta}, \Lambda_{i,j}^{\Delta})$ and $M_{i,j}^{\Delta}$ has some vertex drawn in $\Delta_{c^{\Delta}}$ and 
\item $\Lambda_{j}^{\Delta} \coloneqq \Lambda_{1,j}^{\Delta} \oplus \cdots \oplus \Lambda_{k,j}^{\Delta},$
\end{itemize}
\item $\Lambda^{\Delta} \coloneqq \Lambda_{1}^{\Delta} \oplus \cdots \oplus \Lambda_{p^{\Delta}}^{\Delta}$ and $\Lambda^{\Delta} \subseteq \Lambda^{(\Delta)},$ and 
\item $\mu^{\Delta} = \mu_{1,1}^{\Delta} \oplus \ldots, \oplus \mu_{1,p^{\Delta}}^{\Delta}$ is a fiber of $(G \cap \Delta, \Lambda^{({\Delta})}),$
\end{itemize}
\item $\text{\bf \mu}$ has a $Z$-certifying complement $\bfl{C}$ where $d\folio(\bfl{C})\subseteq d\folio(\delta),$ for every $d\in\Nbbb,$ and 
\item the boundary of each of the fibers in $\text{\bf \mu}$ has at most $f_{\ref{my_crops_findslo}}^{3}(h)$ vertices, $|\mathbf{\Delta}'| ≤ f_{\ref{my_crops_findslo}}^{2}(|Z|),$ and the detail of $(\bfl{C}, \text{\bf \mu})$ is at most ${\sf h}(h).$
\end{enumerate}
exactly as desired and we can conclude.
\end{proof}

\subsection{A variant of the combing lemma}\label{Variant_combing}

We need the following variant of the Combing lemma in \cite{GolovachST22Combing}.
The statement is slightly different (and simpler) than the one of  \cite[Theorem 5]{GolovachST22Combing} and is suitable for the needs 
of the final result of this subsection.
The proof uses the same ideas and the bulk of them have already been used for the results of \autoref{tamming_more_sec}.
This lemma will be useful for the proof of the main result of this section, \autoref{rooted_certification_into_flower}, that is proven in  \autoref{the_main_extr_flowers_final}.

\begin{lemma}
\llabel{combing_variant}
There exists two functions $f^1_{\ref{combing_variant}}:\nn{1}{1}$ and  $f^2_{\ref{combing_variant}}:\nn{1}{1}$ such that the following holds: Let  $r\coloneqq 2\overline{r}+1,$ let   
 $(M,T)$ be an extension {of detail $≤h$}, 
and   let $G$
be the graph obtained by the union of $M,$
with the cycles of some collection $\Ccal=\{C_{1},\ldots,C_{d}\}$ of pairwise disjoint nested cycles and the paths of some linkage $\Lcal=\{P_{1},\ldots,P_{a}\}$
such that $a\coloneqq f^1_{\ref{combing_variant}}(h),$ $d\coloneqq f^2_{\ref{combing_variant}}(h)+r$ and is odd, and $G$ has a $\Sigma^{(0,0)}$-decomposition  $\delta$ and there is an annulus $A\subseteq \Sigma^{(0,0)}$
with the following properties
\begin{itemize}
\item  If $\Delta_{\rm in}$ and $\Delta_{\rm out}$ are the closures of  the two connected components of $ \Sigma^{0,0}\setminus A,$
then $\Delta_{\rm in}$ and $\Delta_{\rm out}$ are both $\delta$-alinged. 
\item  For every $i\in[d],$ the boundary $\bd(\Delta_{\rm out}),$ as well as the traces of the cycles $C_{1},\ldots,C_{i-1},$
are subsets of the one, say $\Delta_{i},$ of the two open disk bounded by  $\trace_{\delta}(C_{i})$
while the boundary $\bd(\Delta_{\rm in}),$ as well as the traces of the cycles $C_{i+1},\ldots,C_{d},$ are subsets of the other. We also set $\Delta_{0}\coloneqq \Delta_{\rm out}$ and 
$\Delta_{d+1}\coloneqq A\cup\delta_{\rm out}.$ 
\item For every $i\in[a],$ one, say $y_{i}^{\rm out},$ of the   endpoints of $P_{i}$ is drawn in $\bd(\Delta_{\rm out})$
and the other, say $y_{i}^{\rm in},$ is drawn   in $\bd(\Delta_{\rm in})$ and all other vertices of 
$P$ are drawn in the interior of $A.$
\item the vertices $y_{1}^{\rm out},\ldots,y_{a}^{\rm out}$ are drawn in this order on the boundary of $\Delta_{\rm out}.$ 
\item $M$ and $\cupall\Lcal$ are vertex disjoint.
\item None of the vertices in $T$ is drawn in $A.$
 \item the linkage $\Lcal$ is orthogonal to $\Ccal.$
 \item $\delta$ has only two vortices $c^{\rm out}$ and $c^{\rm in}$
 where $\Delta_{c^{\rm out}}=\Delta^{\rm out}$ and  $\Delta_{c^{\rm in}}=\Delta^{\rm in}.$ 
\end{itemize}
For every $l,l'\in[0,d+1]$ where $l<l',$  we use notation  $A_{l,l'}$ for the closure of $\Delta_{l}\setminus\Delta_{l'}$ and we denote by   ${\Lcal}_{l,l'}$ be the linkage consisting of 
every connected component of $\cupall \Lcal\cap A_{l,l'}$ that has  endpoints drawn in different boundaries of $A_{l,l'}.$

Then there exists an extension $(M',T)$ where, given that $q\coloneqq f^2_{\ref{combing_variant}}(h)/2$ 
\begin{itemize}
\item ${\sf dissolve}(M',T)={\sf dissolve}(M,T),$  
\item  $M'\cap(\Delta_{\rm in}\cup\Delta_{\rm out})\subseteq M\cap(\Delta_{\rm in}\cup\Delta_{\rm out}),$ and 
\item $M'\cap A_{z+1,z+r}\subseteq \Lcal_{{q}+1,{q}+r}.$
\end{itemize} \end{lemma}

\begin{figure}[ht]
\begin{center}
\scalebox{.93}{\includegraphics{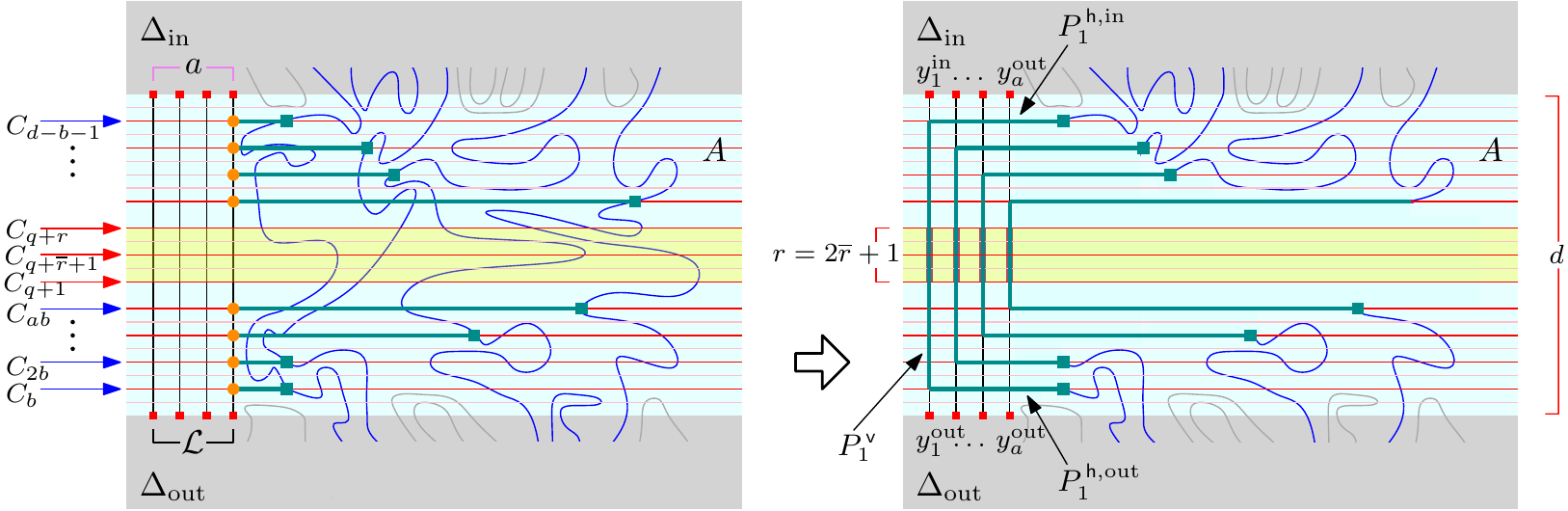}}
\end{center}
    \caption{A visualization of the proof of \autoref{combing_variant}. The rivers of $\Lcal_{G}$
    are blue while the $0$-mountains and the $(d+1)$-valleys are grey. The cycles in $\Ccal$ are red and the substitution paths $P_{i}^{\,{\sf h},{\rm out}},$ $P_{i}^{\,{\sf h},{\rm in}},$ and  $P_{i}^{\,{\sf v}},$ for $ i\in[a]$ are green. The annulus $A_{z+1,z+r}$ is light green.}
  \llabel{new_comb_special}
\end{figure}
 
\begin{proof} 
Let $a$ be the smallest even number that is no smaller than $\ell(h),$ where $\ell \colon \mathbb{N} \to \mathbb{N}$ is the linkage function (see \autoref{tamming_more_sec}).
Let also $b\coloneqq \frac{3}{2}a+1,$ $d\coloneqq 2ab+r.$ We prove the lemma for $f^1_{\ref{combing_variant}}(h)\coloneqq a$ and  $f^{2}_{\ref{combing_variant}}(h)\coloneqq 2ab.$ Observe that $q=ab.$
Also keep in mind that $f^1_{\ref{combing_variant}}(h)=\Ocal(\ell(h))$ and $f^2_{\ref{combing_variant}}(h)=\Ocal((\ell(h))^{2}).$

Notice first that, as none  of the vertices in $T$ is drawn in $A,$
we may subdivide once all neighbors in $G$ of the terminals in $T,$ remove the terminals and by considering their neighbors as new terminals create a linkage $\Lcal_{G}$ of size {$≤h$}.
Notice that the result follows if instead we prove it for the graph $\cupall\Lcal_{G}$: in the result of the theorem we may apply the reverse transformation, that is add back the terminal vertices and the removed edges and then dissolve the subdivision vertices.
Therefore we assume that $G=\cupall\Lcal_{G}$ for some linkage $\Lcal_{G}.$ 

Observe that all components of  
$G\cap A$  are paths that, in turn, are all disjoint from $\cupall\Lcal.$ 
Notice also that for  every cell $c$ of $\delta$ where  $c\subseteq A,$
it holds that $\sigma_{\delta}(c)$ is either a path between two of the vertices in $\pi_{\delta}(\widetilde{c})$ or three paths from a vertex 
of $\sigma_{\delta}(c)-\pi_{\delta}(\widetilde{c})$  to the three vertices of $\pi_{\delta}(\widetilde{c}).$ This permits us to see the part of $G$ that is drawn inside $A$ as a planar embedding.
Next observe that $(\Lcal_{G},B)$ is an LB-pair of $G,$ where $B=\cupall\Ccal.$ 
Among all linkages that are equivalent to $\Lcal_{G}$ we assume that  $\Lcal_{G}$  minimizes the quantity
$\mathsf{cae}(\Lcal_{G}, B) \coloneqq |E(\Lcal) \setminus E(B)|.$
By \autoref{prop_cae} we know that $\tw(\Lcal_{G}\cup B)≤\ell(h)=a.$

Similarly to what we did in \autoref{tamming_more_sec}, we may classify the connected components of $G\cap A,$ to rivers, mountains, and 
valleys: A \defi{river} is a path whose endpoints are drawn in different 
boundaries of $A.$
For every $i\in [0,d]$ an {$i$-mountain} is a connected component of $G\cap A_{i,d+1}$ and for every $i\in[1,d+1]$ an {$i$-valley} is a connected component of $G\cap A_{0,i}.$
Also the \defi{height} (resp. \defi{depth}) of an $i$-mountain ($i$-valey) $P$ is 
 the maximum $j$ such that $C_{i+j-1}$ (resp. $C_{i-j+1})$ intersects $P.$ 
 As we already argued 
in \autoref{tamming_more_sec}, because of \cite[Lemma 4]{GolovachST22Combing}, there are no more than $\tw(\Lcal_{G}\cup B)≤a$ rivers. we set $b\coloneqq \frac{3}{2}a+1.$
Moreover, because of  \cite[Lemma 7]{GolovachST22Combing},
the height (resp. depth) of  an $l$-valey (resp.   $l$-mountain) 
is never more than $\frac{3}{2}\tw(\Lcal_{G}\cup B)<b.$ 
With these bounds at hand, we may reroute the 
 $≤a$ rivers by reproducing  a simpler version of the proof of  \cite[Theorem 5]{GolovachST22Combing}. Let $R_{1},\ldots,R_{q}$ be the rivers, for $q\leq a.$
For $i=1,\ldots,q,$ let $x_{\rm out}^i$ be the last vertex of $C_{i b}$ that  $R_{i}$ meets when traversed from its endpoint in $B_{\Delta_{\rm out}}$ to its endpoint  in $B_{\Delta_{\rm in}}.$ Symmetrically, 
for $i=1,\ldots,q,$ let $x_{\rm in}^i$ is the last vertex of $C_{d-ib-1}$ that  $R_{i}$ meets when traversed from its endpoint in $B_{\Delta_{\rm in}}$ to its endpoint  in $B_{\Delta_{\rm out}}.$  We then discard from each $P_{i}$
the subpath with endpoints $x_{\rm out}^{i}$ and $x_{\rm in}^{i}$ and replace each of the removed paths by the union of three paths: two ``horizontal'' paths $P_{i}^{\,{\sf h},{\rm out}},$ $P_{i}^{\,{\sf h},{\rm in}}$ and one ``vertical'' $P_{i}^{\,{\sf v}},$ as indicated in \autoref{new_comb_special}.
The fact that, for $i\in[q]$ the horizontal path $P_{i}^{\,{\sf h},{\rm out}},$ does not intersect the surviving portions of the paths $P_{1},\ldots,P_{i-1}$ follows by the fact that 
all $(i-1)b$-mountains have height $<b$ and the symmetric argument holds for the horizontal path $P_{i}^{\,{\sf h},{\rm in}}.$ It is now easy to see that the above substitution creates a linkage that is equivalent to $\Lcal_{G}$ does not
change outside $A,$ and moreover, it traverses the annulus $A_{z+1,z+r}$ by exclusively using the paths in $\Lcal_{z+1,z+r},$ as required.
\end{proof}

\subsection{Plucking a flower}\llabel{the_main_extr_flowers_final}

We have now gathered all the tools that we need for the proof of the final result of this subsection.
The statement is similar to the one of \autoref{extending_internal_Certification_in_thickets} however, we now need to ``root''  the fibers in $\{\mu^{\Delta}\mid \Delta\in{\mathbf{\Delta}}'\}$ not only to vertices of the boundaries are in $B_{\Delta}$ (this is already done in \autoref{extending_internal_Certification_in_thickets}) but to the particular vertices of this boundary that are crossed by the linkage $\Rcal^{\Delta}$ of the thicket. i..e, the linear ordering $V(\widehat{\Lambda}^{\Delta}).$
This rerouting is important as it permits us to eventually see each thicket as a flower.

Similarly to the case of  \autoref{extending_internal_Certification_in_thickets}, constants $h'$ and  $k'$ in the statement of \autoref{rooted_certification_into_flower} are irrelevant, as they are not used neither in its conclusion nor in its proof.
However, now we need to give a lower bound, not only for $s,$ but also to the order of the thicket, that is $t = |\Rcal|^{\Delta},$ in order to obtain our rerouting.

\begin{lemma}\llabel{rooted_certification_into_flower}
There exist functions $f^1_{\ref{rooted_certification_into_flower}}, f^2_{\ref{rooted_certification_into_flower}} \colon \nn{2}{1},$ and $f^3_{\ref{rooted_certification_into_flower}} \colon \nn{4}{1}$ and an algorithm that, given
\begin{itemize}
\item $k,h\in \Nbbb_{≥1}$ and
\item a $(z,p)$-ripe  and $d'$-blooming  $(t,k',h',s,\Sigma)$-orchard $(G,\delta,W),$ where $s ≥ f^1_{\ref{rooted_certification_into_flower}}(k,h),$ $t ≥ f^2_{\ref{rooted_certification_into_flower}}(k,h),$
\end{itemize}
outputs a set $S,$ where $|S|≤f^3_{\ref{rooted_certification_into_flower}}(z,p,k,h)$ such that for every graph $Z$ in $h$ vertices, either $G-S$ contains no $\delta$-bound extension of $Z$ or there is some subset $\mathbf{\Delta}'$ of the disk collection $\mathbf{\Delta}$ of $(G,\delta,W)$ and a set $\text{\bf \mu}\coloneqq \{\mu^{\Delta} \mid \Delta \in \mathbf{\Delta'}\}$ of fibers such that
\begin{enumerate}   
\item for every $\Delta\in\mathbf{\Delta}'$ there exists some $p^{\Delta}\in\Nbbb$ where 
\begin{itemize}
\item[i.] for every $j\in [p^{\Delta}],$ there exists a  $j$-package $\Mcal_j^{{\Delta}}\coloneqq \{\mu_{1,j}^{\Delta},\ldots,\mu_{k,j}^{\Delta}\}$ of $k$ pairwise disjoint and pairwise equivalent prime fibers, where 
\begin{itemize}
\item for every $i\in[k],$ $\mu_{i,j}^{\Delta}\coloneqq (M_{i,j}^{\Delta},T_{i,j}^{\Delta},\Lambda_{i,j}^{\Delta})$ and 
\item $\Lambda_{j}^{\Delta}\coloneqq \Lambda_{1,j}^{\Delta}\oplus\cdots \oplus \Lambda_{k,j}^{\Delta},$
\end{itemize}
\item[ii.] $\Lambda^{\Delta}\coloneqq \Lambda_{1}^{\Delta}\oplus\cdots\oplus\Lambda_{p^{\Delta}}^{\Delta},$  $\Lambda^{\Delta}\subseteq \widehat{\Lambda}^{\Delta},$ and $\Lambda^{\Delta}$ and $\widehat{\Lambda}^{\Delta}$ have the same starting vertex $y^{\Delta}.$
\item[iii.]  $\mu^{\Delta}=\mu_{1,1}^{\Delta}\oplus\ldots,\oplus\mu_{1,p^{\Delta}}^{\Delta}$ is a fiber of $(G \cap \Delta,\widehat{\Lambda}^{\Delta}),$
\end{itemize}
\item $\text{\bf \mu}$  has a $\Sigma$-flat $Z$-certifying complement $\bfl{C}$ where $d\folio(\bfl{C})\subseteq d\folio(\delta),$ for every $d\in\Nbbb,$ and 
\item the boundary of each of the fibers in ${\text{\bf \mu}}$ has at most $f_{\ref{my_crops_findslo}}^{3}(h)$ vertices, $|\mathbf{\Delta}'|≤f_{\ref{my_crops_findslo}}^{2}(|Z|),$ and the detail of $(\bfl{C},{\text{\bf \mu}})$ is at most ${\sf h}(h).$
\end{enumerate}
Moreover, the above algorithm runs in $\Ocal_{h}(k\cdot (s+p)\cdot  z\cdot |G|^3 )$-time.
Moreover, $f^1_{\ref{rooted_certification_into_flower}}(k,h)=\Ocal_{h}(k^2),$ $f^2_{\ref{rooted_certification_into_flower}}(k,h)=\Ocal_{h}(k),$ and $f^3_{\ref{rooted_certification_into_flower}}(k,h)= {\Ocal_{h}(zk(k^2+p))}.$
\end{lemma}
\begin{proof}
Let also $\gamma:\Nbbb\to\Nbbb$ be a function
such that $\gamma(x)$ is an upper bound to the number of different societies on $≤x$ vertices. Clearly, we may assume that $\gamma(x)=2^{\Ocal(x^2)}.$
We define the following quantities:
\begin{itemize}

\item  $a\coloneqq f^1_{\ref{combing_variant}}({\sf h}(h)),$
\item $\breve{k}=a\cdot k,$
\item $\widetilde{k}\coloneqq \gamma(a)\cdot \breve{k}=\Ocal_{h}(k),$  

\item $f^2_{\ref{rooted_certification_into_flower}}(k,h)=a\cdot \widetilde{k}$
\item  $\widehat{k}\coloneqq (a+1)\cdot \widetilde{k},$
\item  $\overline{r}\coloneqq 1+a\cdot \widetilde{k},$
 
\item  $r=2\overline{r}+1,$
 
\item  $d=f_{\ref{combing_variant}}^{2}({\sf h}(h))+r,$ and 
 
\item  $f^1_{\ref{rooted_certification_into_flower}}({k},h)\coloneqq (d+2)\cdot (\widehat{k}\cdot {\sf h}(h)+1)=\Ocal_{h}(k^2).$
\item $f^3_{\ref{rooted_certification_into_flower}}(z,p,k,h)\coloneqq f_{\ref{extending_internal_Certification_in_thickets}}(z,p,\widehat{k},h,f^1_{\ref{rooted_certification_into_flower}}(\widehat{k},h))=\Ocal_{h}(zk(k^2+p))$
\end{itemize}

We fist apply \autoref{extending_internal_Certification_in_thickets} for {$s\coloneqq f^1_{\ref{rooted_certification_into_flower}}(\widehat{k},h)$},  $h,$ and $k\coloneqq \widehat{k}.$
As a result of this, we obtain 
a set $S,$ where  $|S|≤f_{\ref{extending_internal_Certification_in_thickets}}(z,p,\widehat{k},h,f^1_{\ref{rooted_certification_into_flower}}(\widehat{k},h))$ and 
such that for every graph $Z$ on $≤h$ vertices, if $G-S$ contains a $\delta$-bound extension of $Z$ then there is some subset ${\mathbf{\Delta}}'$ of the disk collection $\mathbf{\Delta}$ of $(G,\delta,W),$ a svaf $y_{\widetilde{\mathbf{\Delta}}},$  
and a set $\text{\bf \mu}\coloneqq \{\mu^{\Delta}\mid \Delta\in{\mathbf{\Delta}}'\}$ of fibers such that  
all the conditions of the lemma are satisfied, 
except from sub-condition \defi{1.ii}, where we only have that
$\Lambda^{\Delta}\subseteq \Lambda^{(\Delta)}.$ 
Let $\text{\bf \mu}\coloneqq \{\mu^{\Delta}\mid \Delta\in\mathbf{\Delta'}\}$ and 
$\bfl{C}=({\Sigma}^{\mathsf{c}},{M}^{\mathsf{c}},{T}^{\mathsf{c}},{\delta}^{\mathsf{c}},Y)$ be such that conditions \defi{2} and \defi{3} are satisfied for $\widehat{k}.$ The algorithm of \autoref{extending_internal_Certification_in_thickets} runs in $\Ocal_{h}(\widehat{k}\cdot (s+p)\cdot  z\cdot |G|^3 )$-time.

Let $\Delta\in\mathbf{\Delta}.$ As we did in the proof of \autoref{extending_internal_Certification_in_thickets} we consider the railed nest  $(\Ccal^{\Delta}, \Rcal^{\Delta})$ corresponding to $\Delta$
and the $\delta$-aligned disks $\Delta_{1},\ldots,\Delta_{s}$  where $\bd(\Delta_{i})$ is the $\delta$-trace of the cycle $C_{i}^{\Delta},$ $i\in[s],$ $\Delta_{c^{\Delta}}$ is a subset of the interior of $\Delta_{1},$ $\Delta_{1}\subseteq\cdots\subseteq\Delta_{s},$
and $\Delta_{s}$ is a subset of the interior of $\Delta.$
Recall that every graph $M_{i,j}^{\Delta}$ contains at least one vertex from $B_{\Delta}$ and at least one vertex drawn in $\Delta_{c^{\Delta}}.$ This implies the existence 
of a $B_{\Delta}$-$B_{c^{\Delta}}$-linkage  $\Lcal$ of $k\cdot p^{\Delta}$ paths in $G \cap \Delta,$ where each path is a subgraph of some $M_{i,j}^{\Delta}.$
Our next observation is that, because the detail of $(\bfl{C},{\text{\bf \mu}})$ is  at most ${\sf h}(h),$  
each $|T_{i,j}^{\Delta}|$ is bounded by ${\sf h}(h).$
As a result of this,  the set $\overline{T}=\bigcup_{(i,j)\in[\widehat{k}]\times[p^{\Delta}]}T_{i,j}^{\Delta}$ has at most $\widehat{k}\cdot {\sf h}(h)$ vertices.
Given  that {$s≥f^1_{\ref{rooted_certification_into_flower}}(k,h)=(d+2)\cdot (\widehat{k}\cdot {\sf h}(h)+1)$}, there is some 
$i'\in [s]$ such that none of the terminals in $\overline{T}$
is drawn in the closure of $\Delta_{i}\setminus\Delta_{i'+d+1}.$
This closure is an annulus containing the drawings of all ${d}$
cycles in $\{C_{i'+1}^{\Delta},\cdots,C_{i'+d-1}^{\Delta}\}.$
For the sake of simplicity, we rename these cycles and we index them in reverse ordering, that is from outside towards inside, that is 
$C_{1}=C_{i'+d-1}^{\Delta},$ \ldots, $C_{d}=C_{i'+1}^{\Delta}.$
Similarly, we set $D_{0}=\Delta_{i'+d}, \ldots, D_{d+1}=\Delta_{i'+2}.$
Resuming we have $d$ cycles drawn in the interior of the annulus 
formed by the closure, denoted by $A^{\Delta},$ of $D_{0}\setminus D_{d+1}.$ Moreover none of the terminals of the fibers in the $j$-packages of $\Delta$ is drawn in $A^{\Delta}.$ 

As  {$\widehat{k}=(a+1)\cdot \widetilde{k}$},
in each $j$-package $\Mcal_j^{{\Delta}}=\{\mu_{1,j}^{\Delta},\ldots,\mu_{\widehat{k},j}^{\Delta}\}$ we discard every $\mu_{i,j}^{\Delta}$ where $i \neq 0\pmod{a+1}$ and we shift the 
indices of the surviving elements so that $\Mcal_j^{{\Delta}}=\{\mu_{1,j}^{\Delta},\ldots,\mu_{\widetilde{k},j}^{\Delta}\}.$ 
Clearly, conditions \defi{1}, \defi{2}, and \defi{3} above hold also if we replace $\widehat{k}$ by $\widetilde{k}.$ We next extract from every discarded 
fiber a path of the  $B_{\Delta}$-$B_{c^{\Delta}}$-linkage  $\Lcal.$
This gives rise to $\widetilde{k}\cdot p^{\Delta}$ sub-linkages $\Lcal_{i,j},[i,j]\in[\widetilde{k}]\times [p^{\Delta}]$ of $\Lcal,$ each containing $a$ paths, such that for every $[i,j]\in[\widetilde{k}]\times [p^{\Delta}],$ the graph $M_{i,j}^{\Delta}$
is drawn ``between'' the paths in $\Lcal_{i,j}$
 and the paths in $\Lcal_{i+1,j},$ where we agree that
 $\Lcal_{\widetilde{k}+1,j}=\Lcal_{1,j+1},$ for $ j\in[p^{\Delta}],$ and $\Lcal_{1,p^{\Delta}+1}=\Lcal_{1,1}$ (see \autoref{many_before_application}).
 Furthermore, we may assume that the the linkages $\Lcal_{i,j},[i,j] \in [\widetilde{k}]\times [p^{\Delta}]$ are orthogonal to the cycles of $\Ccal^{\Delta}$ (up to possibly obtaining a pruning of $(G, \delta, W)$).
 Observe also that, because of Condition 2,
 for each $[i,j]\in[\widetilde{k}]\times [p^{\Delta}]$
 there is an expansion  $(M_{i,j},T_{i,j})$
of $Z$ such that $M_{i,j}^{\Delta}$ is a subgraph of $M_{i,j},$ $T_{i,j}^{\Delta}\subseteq T_{i,j},$  $|T_{i,j}^{\Delta}|≤|T_{i,j}|≤{a},$ and  $M_{i,j}\cap A^{\Delta}=M_{i,j}^{\Delta}\cap A^{\Delta}.$ 

\begin{figure}[ht]
\begin{center}
\scalebox{.74}{\includegraphics{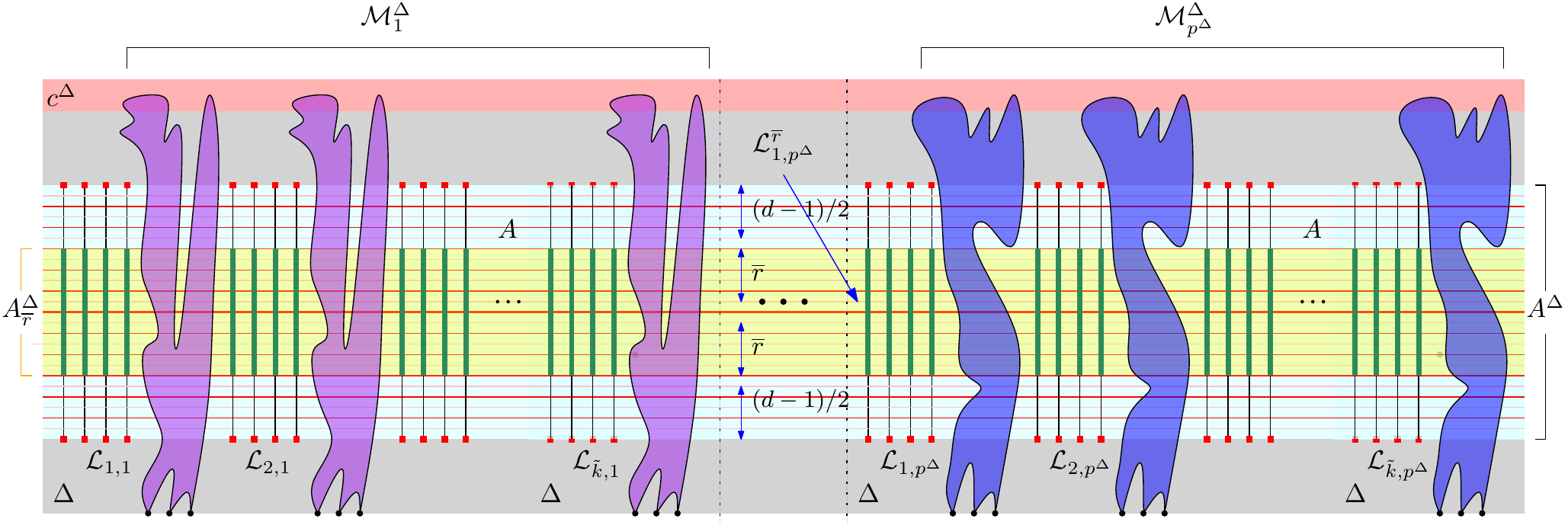}}
\end{center}
    \caption{A visualization of the proof of \autoref{easy_consequence_of_coterminality}
    before the application of \autoref{combing_variant}. The  sub-linkages $\Lcal_{i,j},[i,j]\in[\widetilde{k}]\times [p^{\Delta}]$ are depicted in black and  green.}
  \llabel{many_before_application}
\end{figure}

Recall that  {$d=f^2_{\ref{combing_variant}}({\sf h}(h))+r$}, where {$r=2\overline{r}+1$}.
For every $l\in[0,\ell],$
let $A^{\Delta}_{l}$ be the closure of $\Delta_{l}\setminus\Delta_{d-l+1}$ and let  ${\Lcal}_{i,j}^l$ be the linkage consisting
every connected component of $\cupall \Lcal_{i,j}\cap A^{\Delta}_{l}$ that has  endpoints drawn in different boundaries of $A^{\Delta}_{l}.$
Clearly $A^{\Delta},=A^{\Delta}_{0}.$ We also set $A^{\Delta}_{*}\coloneqq A_{(d-1)/2}^{\Delta}.$

Consider the $\Sigma^{(0,0)}$-decomposition $\delta_{i,j}$
of $M_{i,j}$ obtained by the drawing of the graph 
$G_M^{i,j}\coloneqq M_{i,j}\cup \cupall\{C_{1},\ldots,C_{d}\}\cup\cupall{\Lcal}_{i,j}$ inside $A^{\Delta}$ with two additional vortices, attached to  the boundaries of $A^{\Delta},$  that contain the parts of $M_{i,j}$ that are drawn outside $A^{\Delta}.$
As there are no terminals in $T_{i,j}^{\Delta}$ that are drawn in $A^{\Delta}$
we may apply \autoref{combing_variant} for the terminals in $T_{i,j},$ the graph $G_M,$  the collection of cycles $\Ccal=\{C_{1},\ldots,C_{d}\},$
 the linkage $\Lcal_{i,j},$ the $\Sigma^{(0,0)}$-decomposition $\delta_{i,j},$ and the annulus $A^{\Delta}.$
 As a result of this, we obtain a rerouting 
 of  $M_{i,j}$ where $M_{i,j}\cap A_{\overline{r}}^{\Delta}\subseteq 
 \Lcal_{i,j}^{\overline{r}}$ (see \autoref{bef_sec_a_fig_all_previous} for a visualization of the result of the application of 
   \autoref{combing_variant}).
   By setting $M_{i,j}^{\Delta}\coloneqq M_{i,j}\cap 
   \Delta$ we have an update of all $M_{i,j}^{\Delta}$'s as well.

Recall now that $A_{\overline{r}}^{\Delta}$ is also traversed by the linkage $\Rcal^{\Delta}=\{R_{1}^{\Delta},\ldots,R_{t}^{\Delta}\}$
which is orthogonal to $\Ccal^{\Delta}.$
Let also $\Lcal^{\Delta}=\bigcup_{(i,j)\in[\widetilde{k}]\times[p^{\Delta}]} \Lcal_{i,j}^{\overline{r}}.$
In \autoref{bef_sec_a_fig_all_previous}, the linkage
$\Rcal^{\Delta}$ is drawn in orange, the linkage $\Lcal^{\Delta}$ is depicted in green and the annulus $A_{1}^{\Delta}$ is depicted in light green. Notice also that $|\Lcal^{\Delta}|=q^{\Delta}\cdot \widetilde{k}≤a\cdot \widetilde{k}$ and that $|\Rcal^{\Delta}|=t≥f^2_{\ref{rooted_certification_into_flower}}(k,h)=a\cdot \widetilde{k}.$

Let $A^{\Delta}_{\rm in}$
and $A^{\Delta}_{\rm out}$ be  the closures of the connected components
of $A^{\Delta}_{\overline{r}}\setminus A_{1}^{\Delta}$ as indicated in  \autoref{bef_sec_a_fig_all_previous}, i.e., $A^{\Delta}_{\rm out}$ is ``below'' $A_{1}^{\Delta}$ and $A^{\Delta}_{\rm in}$ is ``above'' $A_{1}^{\Delta}.$
Also we define $\Lcal^{\Delta}_{{\rm out}}$ (resp. $\Lcal^{\Delta}_{{\rm in}}$) as the restriction of $\Lcal^{\Delta}$  in $A^{\Delta}_{\rm out}$ (resp. $A^{\Delta}_{\rm in}$).  Similarly, we define $\Rcal^{\Delta}_{\rm out}$
(resp. $\Rcal^{\Delta}_{\rm in}$) as the restriction of $\Rcal^{\Delta}$ in $A_{\rm out}^{\Delta}$ (resp.  $A_{\rm in}^{\Delta}$).
In \autoref{bef_sec_a_fig_all_previous}, the linkages $\Lcal^{\Delta}_{{\rm out}}$ and 
$\Lcal^{\Delta}_{{\rm in}}$ are depicted in bold green
while the linkages $\Rcal^{\Delta}_{\rm out}$ and $\Rcal^{\Delta}_{\rm in}$ are depicted in bold orange.

\begin{figure}[ht]
\begin{center}
\scalebox{.75}{\includegraphics{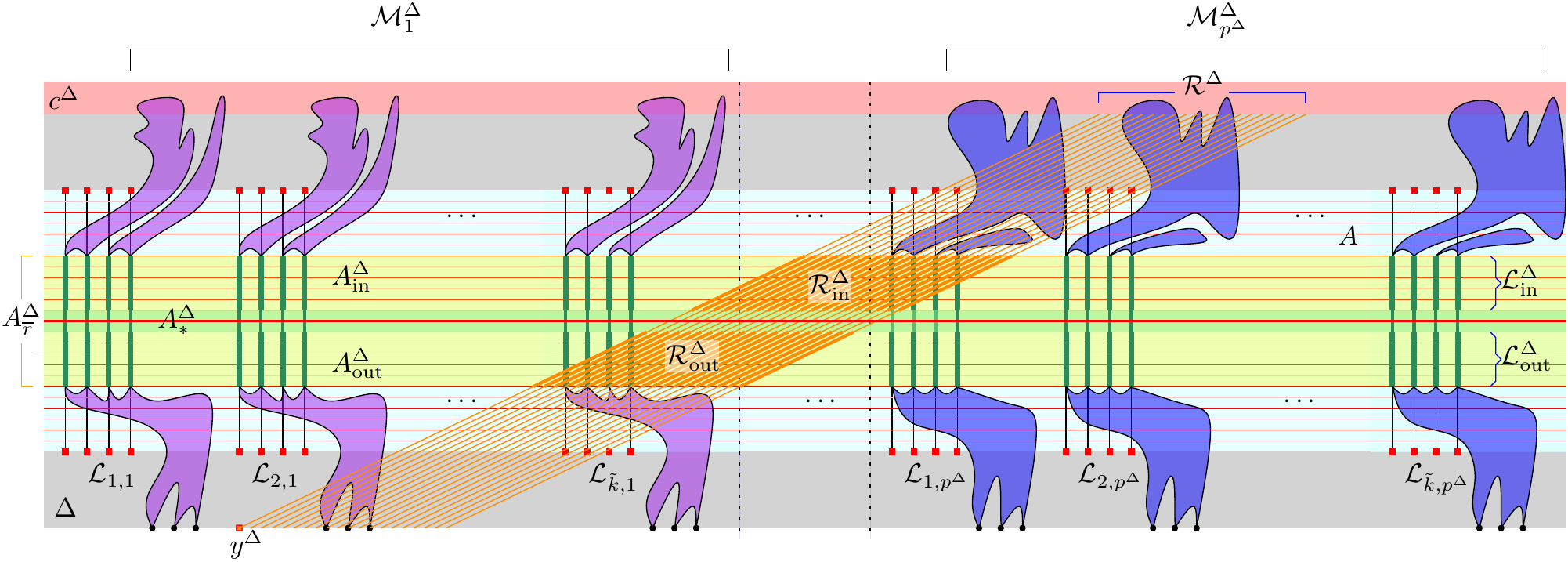
}}
\end{center}
    \caption{A visualization of the rerouting 
    after the application of \autoref{combing_variant}. 
    The linkage $\Rcal^{\Delta}$ is depicted in orange.}
  \llabel{bef_sec_a_fig_all_previous}
\end{figure}

Our next step is to partially reroute the 
central part of $\Lcal^{\Delta}$
so that it becomes a subset of $\Rcal^{\Delta}$ inside the annulus $A_{1}^{\Delta}.$
This is done 
by applying \autoref{easy_consequence_of_coterminality}
twice: the first time we apply it in the annulus $A^{\Delta}_{\rm out}$
and for the linkages  $\Lcal^{\Delta}_{{\rm out}}$   and  $\Rcal^{\Delta}_{\rm out}$  
and reroute $\Lcal^{\Delta}_{{\rm out}}$ so that its revised version is a linkage joining the terminals of $\Lcal^{\Delta}_{{\rm out}}$ on the lower boundary of $A_{\overline{r}}^{\Delta}$ to the terminals of $\Rcal_{\rm out}^{\Delta}$ on the  lower boundary of 
$A^{\Delta}_{*}$ and the second time is applied symmetrically  in the annulus $A^{\Delta}_{\rm in}$ and 
reroute $\Lcal^{\Delta}_{{\rm out}}$ so that its revised version is a linkage joining the terminals of $\Lcal^{\Delta}_{{\rm in}}$ on the upper boundary of $A_{\overline{r}}^{\Delta}$ to the terminals of $\Rcal_{\rm in}^{\Delta}$ on the  upper boundary of 
$A_{*}^{\Delta}.$  This, in turn 
implies an update of all $M_{i,j}^{\Delta}$'s as well, so that, for each $(i,j)\in[\widetilde{k}]\times[p^{\Delta}],$  $M_{i,j}^{\Delta}\cap A_{1}^{\Delta}\subseteq A_{1}\cap \cupall \Lcal.$
Given the statement of \autoref{rooted_certification_into_flower}, the above rerouting 
is possible, as $\overline{r}=1+\widetilde{k}\cdot a.$

\begin{figure}[ht]
\begin{center}
\scalebox{.75}{\includegraphics{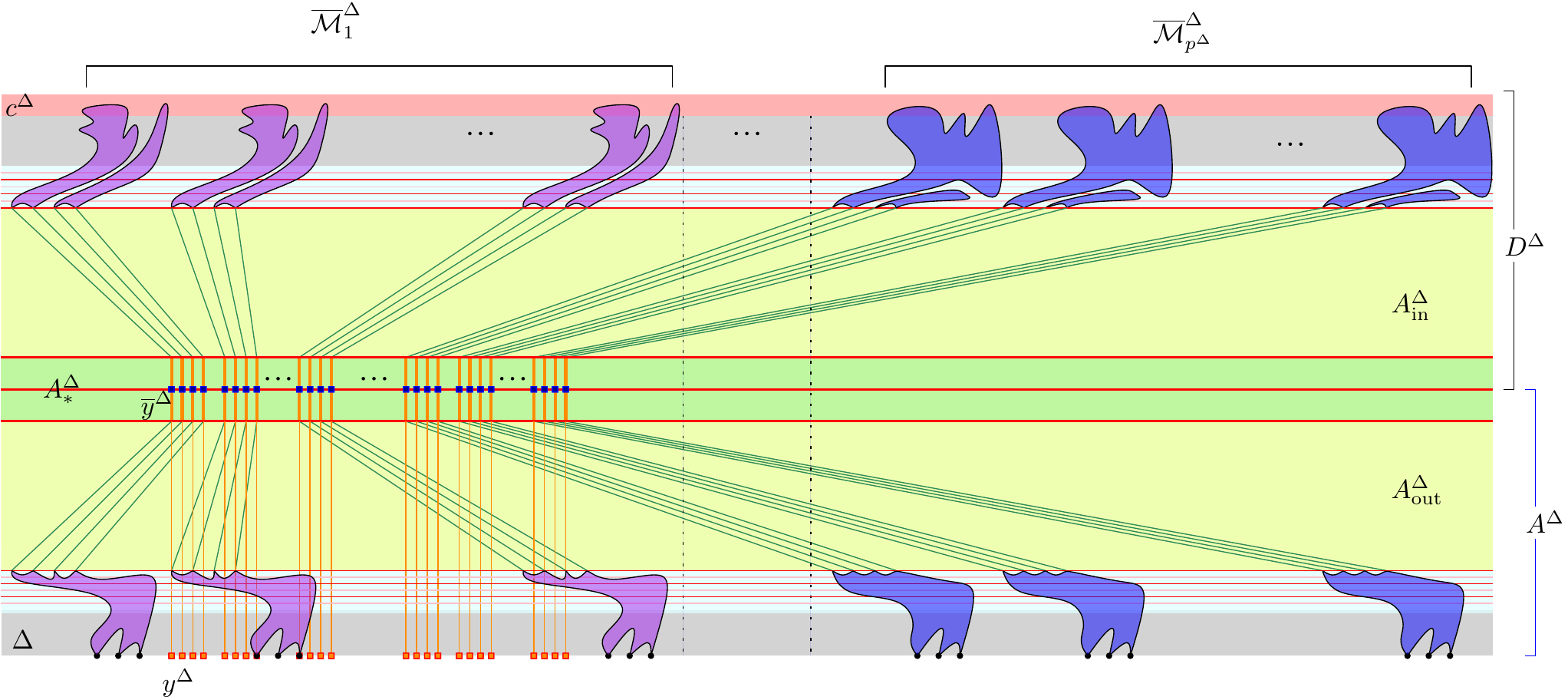
}}
\end{center}
    \caption{The result of the rerouting after the duplicate application of \autoref{easy_consequence_of_coterminality}. The blue square vertices represent the boundary $\overline{\Lambda}$ and the red square vertices represent the boundary of $\widehat{\Lambda}.$}
  \llabel{sec_a_fig_all}
\end{figure}

Let $D^{\Delta}$ be the disk whose boundary is the trace of cycle $C_{\overline{r}+1},$ i.e., the disk $D_{(d-1)/2+\overline{r}+1}.$ 
We also use the notation $A^{\Delta}$ for the annulus  that is the closure 
of $\Delta\setminus D^{\Delta}.$
For $(i,j)\in[\widetilde{k}]\times[p^{\Delta}],$ we define  
$\overline{\mu}_{i,j}=(\overline{M}_{i,j}^{\Delta},\overline{T}_{i,j}^{\Delta},\overline{\Lambda}_{i,j}^{\Delta}),$ where 
$\overline{M}_{i,j}^{\Delta}\coloneqq M_{i,j}^{\Delta}\cap D^{\Delta}$ 
$\overline{T}_{i,j}^{\Delta}\coloneqq T_{i,j}^{\Delta}\cap D^{\Delta}$
and  $\overline{\Lambda}_{i,j}$ is 
{naturally} defined by the  intersection of $M_{i,j}^{\Delta}$  with the boundary of  $D^{\Delta}.$ In \autoref{sec_a_fig_all} the vertices of $\overline{\Lambda}^{\Delta}\coloneqq \overline{\Lambda}_{1}^{\Delta}\oplus\cdots\oplus\overline{\Lambda}_{p^{\Delta}}^{\Delta}$
are depicted in blue ).

We have now arrived to the construction of the $j$-packages 
 $\overline{\Mcal}^{\Delta}_{j}\coloneqq \{\overline{\mu}_{1,j}^{\Delta},\ldots,\overline{\mu}_{\widetilde{k},j}^{\Delta}\}$ that are close to what we need.
 A first difference is that the concatenation of the boundaries of the fibers in these $j$-packages are $\overline{\Lambda}^{\Delta}$ and not $\widehat{\Lambda}^{\Delta}$ (denoted be the square orange vertices in \autoref{sec_a_fig_all}). 
Two more differences are that   the 
 $\overline{\mu}_{i,j}^{\Delta}$'s in each $j$-package 
 are not necessarily prime and they are not necessarily pairwise equivalent. We first discard fibers so to enforce the pairwise equivalence.
  For this observe first that as each $\overline{\mu}_{i,j}$ has at most 
 ${\sf h}(h)$ terminals and at most $a$ boundary vertices,
 there are $\gamma(a)$  pairwise non-equivalent graphs in $\overline{\Mcal}^{\Delta}_{j}.$ As $\widetilde{k}= \gamma(a)\cdot \breve{k},$ we may find $\breve{k}$-paiwise equivalent 
 fibers in  $\overline{\Mcal}^{\Delta}_{j}$ and update 
 $\overline{\Mcal}^{\Delta}_{j}$ by discarding all the rest.

For every $\Delta\in\mathbf{\Delta}'$
and every $j\in[p^{\Delta}],$ we  define $\overline{\mu}^{\Delta}\coloneqq \overline{\mu}_{1,1}^{\Delta}\oplus\cdots\oplus\overline{\mu}^{\Delta}_{1,j}$ and 
$\overline{\text{\bf \mu}}=\{\overline{\mu}^{\Delta}\mid \Delta\in\mathbf{\Delta}'\}$ and we claim that
  $\overline{\text{\bf \mu}}$  has a $Z$-certifying complement
$\overline{\bfl{C}}$ where $d\folio(\bfl{C})\subseteq d\folio(\delta),$ for every $d\in\Nbbb.$ Indeed this assignment is now build if in $\bfl{C}=({\Sigma}^{\mathsf{c}},{M}^{\mathsf{c}},{T}^{\mathsf{c}},{\delta}^{\mathsf{c}},Y)$ 
we update 
\begin{itemize}
\item ${\Sigma}^{\mathsf{c}}\coloneqq {\Sigma}^{\mathsf{c}}\cup\bigcup_{\Delta\in\mathbf{\Delta}'}A^{\Delta},$
\item $M^{\mathsf{c}}\coloneqq M^{\mathsf{c}}\cup\bigcup_{j\in[p^{\Delta}]}(M_{1,j}^{\Delta}\cap A^{\Delta}),$  
\item $\delta^{\mathsf{c}}= (\Gamma\cap {\Sigma}^{\mathsf{c}}, \{ \Delta_{c} \in \mathcal{D} \mid \Delta_{c} \subseteq {\Sigma}^{\mathsf{c}} \})\cap M^{\mathsf{c}}$  and 
\item $Y=\{\overline{y}^{\Delta}\mid \Delta\in\mathbf{\Delta}'\}$
where, for $\Delta\in\mathbf{\Delta}',$ $\overline{y}^{\Delta}$ is the first vertex of $\overline{\Lambda}^{\Delta}.$
\end{itemize}
Clearly,  the boundary of each of the fibers in $\overline{\text{\bf \mu}}$ has at most $f_{\ref{my_crops_findslo}}^{3}(h)$ vertices, $|\mathbf{\Delta}'|≤f_{\ref{my_crops_findslo}}^{2}(|Z|),$ and the detail of $(\overline{\bfl{C}},\overline{\text{\bf \mu}})$ is  not bigger than the one of $\text{\bf \mu},$ that is  at most ${\sf h}(h).$ Finally, as no vortex of $\delta$ has been moved in the new $\delta^{\mathsf{c}},$ we also have that $d\folio(\overline{\bfl{C}})\subseteq d\folio(\delta),$ for every $d\in\Nbbb.$

 We next force the remaining missing property of the $j$-packages, that is primality, again to the cost of sacrificing   
 some of the fibers in each $j$-package $\overline{\Mcal}^{\Delta}_{j}.$ For this, we will substitute each $\overline{\Mcal}^{\Delta}_{j}$
 by many others each one containing $k$ prime fibers. 
 For every $(i,j)\in[\breve{k}]\times [p^{\Delta}]$ we define $\lin{H_{i,j},\widehat{\Lambda}_{i,j}}={\sf dissolve}({\overline{\mu}_{i,j}^{\Delta}}),$ and let $$\mathsf{dec}({H_{i,j},\widehat{\Lambda}_{i,j}}) \coloneqq  \lin{\lin{H^{1}_{i,j},\widehat{\Lambda}^{1}_{i,j}},\ldots,\lin{H^{\ell_j}_{i,j},\widehat{\Lambda}^{\ell_j}_{i,j}}}.$$
 We next  partition, for each $(i,j)\in[\breve{k}]\times[p^{\Delta}],$  the fiber  $\overline{\mu}_{i,j}^{\Delta}$ into $\ell_j$ fibers  $\mu_{i,j}^{\Delta,1},\ldots,\mu_{i,j}^{\Delta,\ell_{j}}$ 
so that, for $l\in[\ell_{j}],$ ${\sf dissolve}(\mu_{i,j}^{\Delta,l})=\lin{H^{l}_{i,j},\widehat{\Lambda}^{l}_{i,j}}.$
 
 As, for every $j\in[p^{j}],$ all fibers in $\overline{\Mcal}^{\Delta}_{j}$ are pairwise isomorphic, it follows that 
 for every $j\in[p^{\Delta}]$ and every $l\in[\ell_{j}]$ all fibers 
in $\{\mu_{1,j}^{\Delta,l},\ldots,\mu_{\breve{k},j}^{\Delta,l}\}$ are also pairwise isomorphic. Also the fibers corresponding to the $j$-package  $\overline{M}_{j}^{\Delta}$ are in order:
\begin{eqnarray}
{\mu}_{1,j}^{{\Delta},1},\ldots,{\mu}_{1,j}^{{\Delta},\ell_{j}},\ \ \ldots\ \  ,{\mu}_{{i},j}^{{\Delta},1}\ldots,{\mu}_{{i},j}^{{\Delta},\ell_{j}},\ \ \ldots\ \ , 
{\mu}_{\breve{k},j}^{{\Delta},1},\ldots,{\mu}_{\breve{k},j}^{{\Delta},\ell_{j}}\llabel{what_a_sec}\end{eqnarray}
 We apply the following procedure on the sequence in  \eqref{what_a_sec}:
\begin{quote}
For $l\coloneqq 1,\ldots,\ell_{j},$ we consider the first $k$ 
fibers of the set $\{\mu_{1,j}^{\Delta,l},\ldots,\mu_{\widehat{k},j}^{\Delta,l}\}$ that appear in \eqref{what_a_sec} and up to the position where 
the $k$-th such fiber appears, discard all elements appearing before it that do not belong in $\{\mu_{1,j}^{\Delta,l},\ldots,\mu_{\breve{k},j}^{\Delta,l}\}.$ In the end of this repetition, discard all remaining elements that have not be considered.
\end{quote}
We index the  sequence obtained by the above procedure   as follows:
\begin{eqnarray}
{\mu}_{1,j}^{{\Delta},1},\ldots,{\mu}_{k,j}^{{\Delta},1},\ \ \ldots\ \  ,{\mu}_{{1},j}^{{\Delta},l}\ldots,{\mu}_{{k},j}^{{\Delta},l},\ \ \ldots\ \ , 
{\mu}_{{1},j}^{{\Delta},\ell_{j}},\ldots,{\mu}_{{k},j}^{{\Delta},\ell_{j}}\llabel{before_one_end}\end{eqnarray}
Recall that $\ell_{j}≤a$ and $\breve{k}=a\cdot k,$ therefore 
$\ell_{j}\cdot k≤\breve{k}.$ This implies that for each $l\in[\ell_{j}],$ the set $\{\mu_{1,j}^{\Delta,l},\ldots,\mu_{\widehat{k},j}^{\Delta,l}\}$ has enough elements so to produce the sequence in \eqref{before_one_end}. Moreover, for every $l\in[\ell],$ all fibers 
in $\Mcal_{j}^{\Delta,l}=\{{\mu}_{{1},j}^{{\Delta},l},\ldots,{\mu}_{{k},j}^{{\Delta},l}\}$ are pairwise isomorphic. By applying the above construction for every 
$j=1,\ldots,p^{\Delta},$ we produce a sequence  
This gives rise to the following sequence of $\ell^{\Delta}\coloneqq \sum_{j\in[p^{\Delta}]}\ell_{j}$ packages.
\begin{eqnarray}
\Mcal_{1}^{\Delta,1},\ldots,\Mcal_{1}^{\Delta,\ell_{1}},\ \ \ldots\ \  ,\Mcal_{j}^{\Delta,1},\ldots,\Mcal_{j}^{\Delta,\ell_{l}},\ \ \ldots\ \ , 
\Mcal_{p^{\Delta}}^{\Delta,1},\ldots,\Mcal_{p^{\Delta}}^{\Delta,\ell_{l}}\llabel{many_fibers_in_a_line}\end{eqnarray}
and each of them contains $k$ pairwise isomorphic prime fibers.

Notice that if, after the above rearrangement, we consider the concatenation of the first fibers of each of the packages in  \eqref{many_fibers_in_a_line} we again obtain the same fiber $\overline{\mu}^{\Delta}.$ This implies that the newly defined sequence of packages in \eqref{many_fibers_in_a_line} along with $\overline{\bfl{C}}$ and $\overline{\text{\bf \mu}})$ satisfy all conditions of the conclusion of the lemma, with the difference that, now 
the boundary of each $\overline{\mu}^{\Delta}$ is  $\overline{\Lambda}^{\Delta}$ and not $\widehat{\Lambda}^{\Delta}.$
In order to deal with this, we recall first that  for every $\overline{y}\in V(\overline{\Lambda}^{\Delta})$ 
there is some  $y\in V(\widehat{\Lambda}^{\Delta})$
such that $y$ and $\overline{y}$ are the endpoints of 
some subpath $P_{\overline{y}}$ of  ${\Rcal}^{\Delta}.$
Next, we enhance each fiber $(M^{\Delta},T^{\Delta},\Lambda^{\Delta})$ in the packages of 
\eqref{many_fibers_in_a_line} as follows
\begin{itemize}
\item we set 
$M^{\Delta}=M^{\Delta}\cup\bigcup_{\overline{y}\in V(\overline{\Lambda}^{\Delta})}P_{\overline{y}},$ 
\item  $T^{\Delta}$ remains the same. Notice that for the current  $\overline{\bfl{C}}$ and $\overline{\text{\bf \mu}},$ every vertex 
of $\overline{\mu}^{\Delta}$ is a non-terminal vertex (i.e., a subdivision vertex) of the graph of the  $(\overline{\bfl{C}},\overline{\text{\bf \mu}})$-certifying extension of $Z,$
\item given that $\Lambda^{\Delta}=\lin{\overline{y}^{1},\ldots,\overline{y}^{r}},$ we set $\Lambda^{\Delta}=\lin{{y}^{1},\ldots,{y}^{r}}.$
\end{itemize}

\begin{figure}[ht]
\begin{center}
\scalebox{.96}{\includegraphics{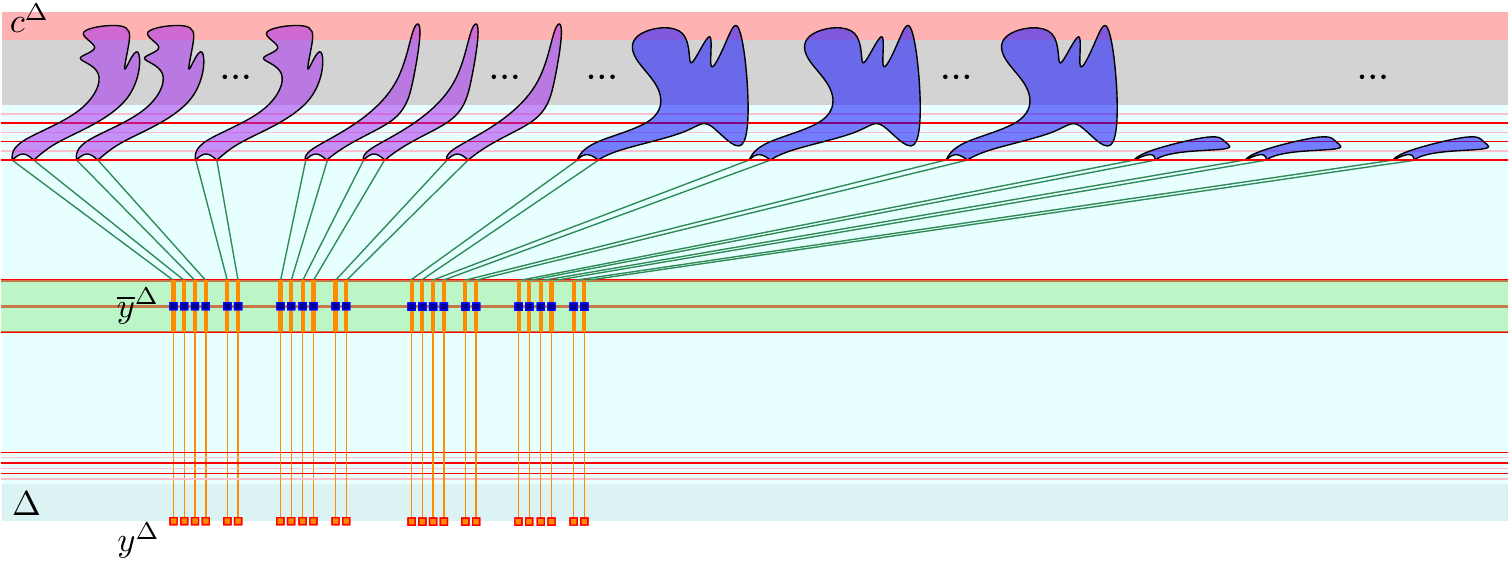
}}
\end{center}
    \caption{A visualization of the ``upper part'' of the fibers in the packages in \eqref{many_fibers_in_a_line}. By adding the red lines in them we have the enhanced fibers at the end of the proof of \autoref{easy_consequence_of_coterminality}. The red lines represent the subpaths of the linkage $\Rcal^{\Delta}$ that were used for this enhancement.}
  \llabel{really_one_end}
\end{figure}

In other words, we use the paths in the linkage $\Rcal^{\Delta}$
In order to root the boundary of every fiber in the packages  of \eqref{many_fibers_in_a_line} to vertices of $V(\widehat{\Lambda}^{\Delta}),$
starting with the vertex $y^{\Delta}$ (corresponding to $\overline{y}^{\Delta}$ -- see \autoref{really_one_end}). This yields all sub-conditions of condition   \defi{1}. Also, after  the above enhancement, $\bfl{C}$ remains 
a  $Z$-certifying complement of $\overline{\text{\bf \mu}}=\{\overline{\mu}^{\Delta},\Delta\in\mathbf{\Delta}'\}$ and conditions \defi{2} and \defi{3} hold as well.

All the transformations of this proof can be implemented in linear time.  Therefore the running time of the algorithm is dominated by one of the applications of the algorithm of \autoref{extending_internal_Certification_in_thickets} plus $\Ocal(t)$ and that runs in $\Ocal_{h}(\widehat{k}\cdot (s{+p})\cdot  z\cdot |G|^3)$-time. As $\widehat{k}=\Ocal_{h}(k),$ we obtain the claimed running time.
\end{proof}

Notice that in the statement of \autoref{rooted_certification_into_flower} we do not demand that each $M_{i,j}^{\Delta}$ has some vertex drawn in $\Delta_{c^{\Delta}},$ as we did in \autoref{internal_certification_without_thickets}  or in \autoref{extending_internal_Certification_in_thickets}. This extra feature in  \autoref{extending_internal_Certification_in_thickets} was necessary for the proof of \autoref{rooted_certification_into_flower} as it permitted
us to detect the  sub-linkages $\Lcal_{i,j},[i,j]\in[\widetilde{k}]\times [p^{\Delta}]$ of $\Lcal$ and then apply the combing  \autoref{combing_variant} in other to reroute the fibers produced by \autoref{extending_internal_Certification_in_thickets}.
For the use we make of \autoref{rooted_certification_into_flower} in the next section,  this additional feature is not necessary anymore.
However, we could also further discard all fibers extracted by \autoref{rooted_certification_into_flower} that do not invade their corresponding vortices and update $\overline{\bfl{C}}$ by enhancing it with the ``flat'' drawings of the discarded fibers as we did in  \autoref{internal_certification_without_thickets}. For instance, this extra modification would discard the rightmost package in \autoref{really_one_end} that contains fibers that have flat embeddings away from the vortex (depicted by the red territory). In the visualization of the flower of \autoref{mo_really_more_one_flower_in_thicket} we assume that this rightmost package has been cropped.

\section{The local structure theorem}\llabel{sec_local_structure}

This section is the first key step in unifying most of the previously introduced machinery into a single statement: the local structure theorem.
This theorem takes as input simply a large wall $W$ and integers $k$ and $h$ and outputs either a large clique-minor which is grasped by $W,$ or produces the entire orchard infrastructure necessary for the application of \autoref{rooted_certification_into_flower} together with a small set $S$ and a list of walloids representing all graphs of size at most $h$ which are not fully controlled by $S.$

\subsection{Finding the universal obstruction}
To make the outcome of \autoref{rooted_certification_into_flower} more accessible, we start with an intermediate statement that directly extracts our obstructions from the orchard and its thickets whenever the set $S$ misses a graph $Z$ on at most $h$ vertices.

\begin{lemma}\llabel{lemma_grounded_extensions_to_obstructions}
There exist functions $f^1_{\ref{lemma_grounded_extensions_to_obstructions}},f^2_{\ref{lemma_grounded_extensions_to_obstructions}}\colon \mathbb{N}^2 \to \mathbb{N},$ $f^3_{\ref{lemma_grounded_extensions_to_obstructions}} \colon \mathbb{N}^4 \to \mathbb{N},$ and an algorithm that, given two positive integers $k$ and $h$ and a $(z,p)$-ripe and $h$-blooming $(t,k,{\sf h}(h),s,\Sigma)$-orchard $(G,\delta,W)$ as input, where
\begin{align*}
  s\coloneqq f^1_{\ref{lemma_grounded_extensions_to_obstructions}}(k,h)\text{ and }t\coloneqq f^2_{\ref{lemma_grounded_extensions_to_obstructions}}(k,h),
\end{align*}
outputs  a set $S\subseteq V(G)$ of size at most $f^3_{\ref{lemma_grounded_extensions_to_obstructions}}(z,p,k,h)$ such that for every graph $Z$ on at most $h$ vertices, either
\begin{enumerate}
  \item $G-S$ contains no $\delta$-bound extension of $Z,$ or
  \item there exists $\mathscr{W}^Z=\langle \mathscr{W}^Z_k \mid k\in\mathbb{N}\rangle\in\mathfrak{W}_{\mathsf{excl}(Z)}$ such that $G$ contains a subdivision $W^Z_k$ of $\mathscr{W}^Z_k$ and the base cylinder of $W^Z_k$ is a subgraph of the base cylinder of $W,$
\end{enumerate}
in time
$$\Ocal_{h}(k\cdot (k^2+p)\cdot  z\cdot |G|^3).$$
Moreover
$$\textrm{$f^1_{\ref{lemma_grounded_extensions_to_obstructions}}(k,h)=\Ocal_{h}(k^2),$
{$f^2_{\ref{lemma_grounded_extensions_to_obstructions}}(k,h)=\Ocal_{h}(k^3),$}
and 
$f^3_{\ref{lemma_grounded_extensions_to_obstructions}}(z,p,k,h)=\Ocal_{h}(zk(k^2+p)).$}$$ 
\end{lemma}

\begin{figure}[ht]
\begin{center}
\scalebox{.927}{\includegraphics{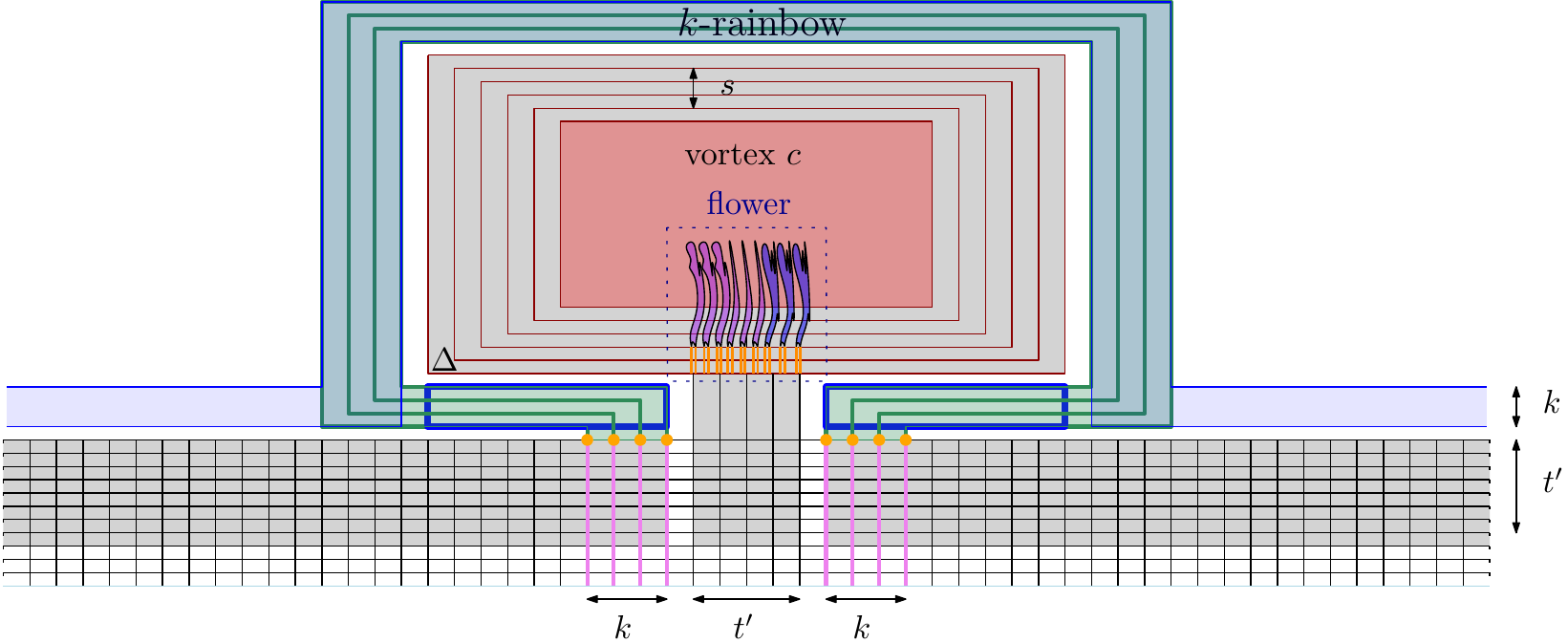}}
\end{center}
    \caption{A visualization of the thicket in the proof of \autoref{lemma_grounded_extensions_to_obstructions} and the position of the flower and the rainbow in it.}
  \llabel{mo_really_more_one_flower_in_thicket}
\end{figure}

\begin{proof}
Let $t'=f^2_{\ref{rooted_certification_into_flower}}(k,{\sf h}(h))\cdot k,$ 
$\overline{t}\coloneqq  t'+2k,$ $s'=f^1_{\ref{rooted_certification_into_flower}}(k,{\sf h}(h))$ and $\overline{s}\coloneqq s'+k.$
We prove the lemma for 
\begin{itemize}
\item  $f^1_{\ref{lemma_grounded_extensions_to_obstructions}}(k,h)\coloneqq f^1_{\ref{rooted_certification_into_flower}}(k,{\sf h}(h))+k,$ \item  $f^2_{\ref{lemma_grounded_extensions_to_obstructions}}(k,h)\coloneqq f^2_{\ref{rooted_certification_into_flower}}(k,{\sf h}(h))\cdot k,$
and 
\item  $f^3_{\ref{lemma_grounded_extensions_to_obstructions}}(z,p,k,h)\coloneqq f^3_{\ref{rooted_certification_into_flower}}(z,p,k,{\sf h}(h)).$
\end{itemize}

Consider a  $(q,p)$-ripe and $h$-blooming $(\overline{t},k,{\sf h}(h),\overline{s},\Sigma)$-orchard $(G,\delta,W).$ Let also $W^{-}$ be the walloid obtained if we first remove 
the $k$ cycles around the exceptional face plus, for each thicket, the parts that are indicated by the two bold blue cycles and the purple edges in  \autoref{mo_really_more_one_flower_in_thicket}. We also use $\Lcal^{*}$ for the linkage corresponding to the purple edges.
This gives rise to a  $(z,p)$-ripe and
 $h$-blooming $({t}',k,{\sf h}(h),s',\Sigma)$-orchard $(D,\delta,W^-).$ 
 It also gives us space to define a linkage of $k$-paths ``above'' each 
of the thickets of $(G,\delta,W).$ In \autoref{mo_really_more_one_flower_in_thicket}, this linkage is visualized as the set of bold green paths joining the orange vertices and we refer to it as the \defi{$k$-rainbow} of this thicket.

 We next apply  \autoref{rooted_certification_into_flower}
on  $(G,\delta,W^-)$ for $k\coloneqq k,$ $s\coloneqq s',$ $t\coloneqq t',$ $h={\sf h}(h),$ and find a 
 set $S,$ where $|S|≤f^3_{\ref{rooted_certification_into_flower}}(z,p,k,{\sf h}(h)) =  {\Ocal_{h}(zk(k^2+p))}$ and such that, for every $\delta$-bound extension of $Z$ in $G-S$  there is some subset $\mathbf{\Delta}'$ of the disk collection $\mathbf{\Delta}$ of $(G,\delta,W)$
and   a set $\text{\bf \mu}\coloneqq \{\mu^{\Delta}\mid \Delta\in\mathbf{\Delta'}\}$ of fibers so that
conditions \defi{1}, \defi{2}, and \defi{3} of \autoref{rooted_certification_into_flower} are satisfied. Let also $\mu^{\Delta}=\mu_{1,1}^{\Delta}\oplus\ldots,\oplus\mu_{1,p^{\Delta}}^{\Delta},$ where, for $j\in[p^{\Delta}],$  $\mu_{1,j}^{\Delta}\coloneqq (M_{1,j}^{\Delta},T_{1,j}^{\Delta},\Lambda_{1,j}^{\Delta}).$ 

Also, again from \autoref{rooted_certification_into_flower}, $S$ can be computed in  time $\Ocal_{h}(k\cdot (s'+p)\cdot  z\cdot |G|^3 )$-time.
As $s'=f^1_{\ref{rooted_certification_into_flower}}(z,p,k,h)$ and $f^1_{\ref{rooted_certification_into_flower}}(k,h)= \Ocal_{h}(k^2),$ the algorithm runs in time $\Ocal_{h}(k\cdot (k^2+p)\cdot  z\cdot |G|^3),$ as claimed.

Let now  $\bfl{C}$ be a  $\Sigma$-flat $Z$-certifying complement
$\bfl{C}=({\Sigma}^{\mathsf{c}},{M}^{\mathsf{c}},{T}^{\mathsf{c}},{\delta}^{\mathsf{c}},Y)$ where $d\folio(\bfl{C})\subseteq d\folio(\delta),$ for every $d\in\Nbbb,$ because of  by Condition \defi{2}. Let also $(Q,T)$ be 
the $(\bfl{C},\text{\bf \mu})$-certifying extension of $Z.$ We define $$M={\sf dissolve}(Q,T\cup\bigcup_{\Delta\in\mathbf{\Delta'}}B_{\Delta}\cap V(Q))$$ 
and keep in mind that $Z$ is a minor of {\sf dissolve}$(Q,T)$ that, in turn, is a minor of $M.$ Moreover, because of Condition \defi{3},  the boundary of each of the fibers in ${\text{\bf \mu}}$ has at most $f_{\ref{my_crops_findslo}}^{3}(h)$ vertices,  $|\mathbf{\Delta}'|≤f_{\ref{my_crops_findslo}}^{2}(|Z|),$ and $|T|={\sf h}(h).$
Therefore, we are now in the position to define the function $\mathsf{ext}(h) \coloneqq f_{\ref{my_crops_findslo}}^{2}(h)\cdot f_{\ref{my_crops_findslo}}^{3}(h)+ \mathsf{h}(h),$ introduced in \autoref{subsec_universal_obstructions}, so that $|M| ≤ \mathsf{ext}(h).$
We call a cell $c\in\Ccal(\delta^{\mathsf{c}})$ \defi{trivial}
if $\sigma_{\delta^{\mathsf{c}}}(c)$ is the graph with one edge joining the two vertices in 
$\pi_{\delta^{\mathsf{c}}}(\widetilde{c}).$ We denote by $\Ccal^*(\delta^{\mathsf{c}})$
the set of non-trivial cells in $\delta^{\mathsf{c}}(c).$
We now build a subgraph partition of $M$ as follows. $$\Pcal=\{\sigma_{\delta^{\mathsf{c}}}(c)\mid c\in \Ccal^*(\delta^{\mathsf{c}})\}\cup\{M_{1,1}^{\Delta}\oplus\cdots\oplus M_{1,p^{\Delta}}^{\Delta}\mid \Delta\in\mathbf{\Delta'} \}.$$
 Consider now the hypergraph $K_{\Pcal}$
 and observe that $\delta^{\mathsf{c}}$ and $\mathbf{\Delta}'$ define a $\Sigma$-embedding $\Gamma$ of $\Kcal_{\Pcal}$ where 
 every edge $e\in \Ecal_{\Pcal}$ is embedded as a disk $\Delta_{e}.$
 Given the triple $(M,\Pcal,\Gamma),$ we also define the function 
 $\eta:\Ecal_{\Pcal}\to V(M)$ as follows: let $e\in \Ecal_{\Pcal}.$
 If $e=\pi_{\delta^{\mathsf{c}}}(\widetilde{c})$ for some $c\in\Ccal^{*}(\delta^{\mathsf{c}}),$ then $|e|≤3$ and we define $\eta(e)$ to be an, {arbitrarily chosen},  vertex of $e.$
 If $e=V(\Lambda_{1}^{\Delta}),$ for some $\Delta\in\mathbf{\Delta}',$ 
 then we define $\eta(e)$ to be the first vertex of $\Lambda_{e,\eta(e)}=\Lambda_{1,1}^{\Delta}\oplus\cdots\oplus\Lambda_{1,p^{\Delta}}^{\Delta}.$  This generates the set of linear societies $$\mathbf{B}=\{\lin{G_{e},\Lambda_{e,\eta(e)}}\mid e\in \Ecal_{\Pcal} \text{~and $(G_{e},\Lambda_{e,\eta(e)})$ is not disk-embeddable}\},$$
 and therefore ${\bf p}\coloneqq (\Sigma,\mathbf{B})$ is an embedding pair.
 Let $\mathscr{W}^{\mathbf{p}}=\lin{\mathscr{W}^{\mathbf{p}}_{k}}_{k\in\Nbbb}$ be the guest walloid generated by ${\bf p}.$
It remains to prove that a subdivision of $\lin{\mathscr{W}^{\mathbf{p}}}_{k\in\Nbbb}$ can be found as a subgraph of $G.$
We examin two cases:
\medskip

\noindent{\em Case 1.}
Suppose that $e=\pi_{\delta^{\mathsf{c}}}(c),$ for some $c\in \Ccal^*(\delta^{\mathsf{c}}),$  where $\sigma_{\delta^{\mathsf{c}}}(c)$ is $\Delta_{c}$-embeddable so that the vertices in $\pi_{\delta^{\mathsf{c}}}(\widetilde{c})$ are drawn in the boundary of $\Delta_{c}.$ This means that the society $\lin{\sigma_{\delta^{\mathsf{c}}}(c),\Lambda_{e,\eta(e)}}$
is not disk embeddable and belongs in $\mathbf{B}$ and also belongs in ${\sf h}(h)\folio(\bfl{C}).$ From Condition \defi{2}
of \autoref{rooted_certification_into_flower} we have that 
$\lin{\sigma_{\delta^{\mathsf{c}}}(c),\Lambda_{e,\eta(e)}}\in  {\sf h}(h)\folio(\delta),$ therefore $W^-,$ anf therefore also $W,$ contains a  $(t',k,{\sf h}(h),\sigma_{\delta^{\mathsf{c}}}(c),\Lambda_{e,\eta(e)})$-{parterre segment}. Notice that societies isomorphic to $\lin{\sigma_{\delta^{\mathsf{c}}}(c),\Lambda_{e,\eta(e)}}$
can appear at most $|\Ccal^*(\delta^{\mathsf{c}})|$ times
and, because  $|\Ccal^*(\delta^{\mathsf{c}})|≤|T|≤{\sf h}(h),$   
the aforementioned $(t',k,{\sf h}(h),\sigma_{\delta^{\mathsf{c}}}(c),\Lambda_{e,\eta(e)})$-{parterre segment} contains enough many 
 $(t',k,\sigma_{\delta^{\mathsf{c}}}(c),\Lambda_{e,\eta(e)})$-{flower segments} to represent all of them.
 We just proved that all flowers of $\mathscr{W}^{\mathbf{p}}_{k}$ generated by cells of $\delta^{\mathsf{c}}$ are present in $W.$ 
\medskip

\noindent{\em Case 2.}
 Let now $e=V(\Lambda^{\Delta})$ for some $\Delta\in\mathbf{\Delta}'.$ Condition \defi{1} of \autoref{rooted_certification_into_flower}  implies that the thicket of 
 $(G,\delta,W)$ corresponding  to $\Delta$ also contains a 
 $(t',k,M_{1,1}^{\Delta}\oplus\cdots\oplus M_{1,p^{\Delta}}^{\Delta},\Lambda_{1,1}^{\Delta}\oplus\cdots\oplus\Lambda_{1,p^{\Delta}}^{\Delta})$-{flower segment}. For this, we may just see the graphs in $ M_{1,1}^{\Delta},\ldots, M_{1,p^{\Delta}}^{\Delta}$ below the $k$-rainbow of the thicket and bring back the edges of the paths in the linkage that we initially discarded from $W,$ for every thicket $\Lcal^*$ (denoted in purple in \autoref{mo_really_more_one_flower_in_thicket}). \medskip
 
 As all $k$-flowers 
 of  $W^{\mathbf{p}_{k}}$ can be found either in the $k$-parterres of 
 $(G,\delta,W)$ or in the thickets of $(G,\delta,W),$ we conclude 
 that $G$ indeed contains a subdivision of $W^{\mathbf{p}}_{k},$ as required.
\end{proof}

\subsection{Bringing everything together}

With \autoref{lemma_grounded_extensions_to_obstructions} at hand, we are finally ready to prove the main result of this section and the most central statement of this paper.

\begin{theorem}[Local structure theorem]
\llabel{thm_local_structure}
There exist functions $f^1_{\ref{thm_local_structure}},f^2_{\ref{thm_local_structure}}\colon \mathbb{N}^2\to\mathbb{N}$ and an algorithm that, given two integers $k$ and $h,$ a graph $G$ and an $f^1_{\ref{thm_local_structure}}(h,k)$-wall $W$ in $G$ finds, in time $\mathcal{O}_{h,k}(|G|^3)$ either
\begin{enumerate}

  \item a $K_{3k^2+k\cdot h}$ minor grasped by $W,$ 
  
  \item there are non-negative integers $h'$ and $c'\leq 2$ such that every graph on at most $h$ vertices can be embedded in every surface of Euler-genus at least $2h'+c'$ and $G$ contains a subdivision of a walloid $W'$ which is the cylindrical concatenation of $h'$ many $k$-handle segments and $c'$ many $k$-crosscap segments such that the base cylinder of $W'$ is a subgraph of $W,$ or

  \item a set $S\subseteq V(G)$ of size at most $f^2_{\ref{thm_local_structure}}(h,k)$ and a list of graphs $\mathcal{Z},$ all of which are of size at most $h,$ such that, if $J$ is the unique component of $G-S$ which contains a full row and a full column of $W,$
  \begin{enumerate}
    \item for every graph $Z$ on at most $h$ vertices, if $Z\notin\mathcal{Z}$ then $J$ is $Z$-minor-free, and
    \item for every graph $Z\in\mathcal{Z}$ there exists an embedding pair $\mathbf{p}\in \mathfrak{P}^{(Z)}$ such that the algorithm outputs a subdivision $W_Z$ of $\mathscr{W}^{\mathbf{p}}_k$ of $G$ where the base cylinder of $W_Z$ is a subgraph of $W.$
  \end{enumerate}
\end{enumerate}
Moreover, we have that
\begin{align*}
    f^1_{\ref{thm_local_structure}}(k,h)&\in~ 2^{k^{\mathcal{O}_h(1)}}\text{ and } f^2_{\ref{thm_local_structure}}(k,h)\in~ 2^{k^{\mathcal{O}_h(1)}}.
\end{align*}
\end{theorem}

\begin{proof}
Let us start by discussing the numbers.
We set $g\coloneqq g(h)$ to be the smallest integer $g$ such that every graph on at most $h$ vertices embeds in \textsl{every} surface of Euler-genus at least $g.$
Observe that, since there are, up to isomorphism, only finitely many graphs on at most $h$ vertices, and thus $g$ is well-defined. 
Indeed, it holds that $g \in \mathsf{poly}(h).$
Similarly, we know that for every extension of a graph on at most $h$ vertices it holds that the graph of the extension has at most $\mathsf{h} \coloneqq \mathsf{h}(h)\in\mathcal{O}(h^2)$ terminals.
Hence, it will suffice to only consider folios of detail $\mathsf{h}.$

We then set
\begin{align*}
  f^1_{\ref{thm_local_structure}}(h,k) \coloneqq~ &  f^1_{\ref{obs_create_meadow}}\bigg( h,k,f_{\ref{lem_meadow_plowing}} \bigg( f_{\ref{lem_garden_creation}}\big(f^1_{\ref{ripe_orchard}}(f^2_{\ref{lemma_grounded_extensions_to_obstructions}}(k,h),k,h,f^1_{\ref{lemma_grounded_extensions_to_obstructions}}(k,h),x,y,\mathsf{h})\\
  &+f^2_{\ref{ripe_orchard}}(f^2_{\ref{lemma_grounded_extensions_to_obstructions}}(k,h),k,h,f^1_{\ref{lemma_grounded_extensions_to_obstructions}}(k,h),x,y,\mathsf{h})+1,k,h,\mathsf{h}\big),h,g\bigg)\bigg)
\end{align*}
and
\begin{align*}
  f^2_{\ref{thm_local_structure}}(h,k) & \coloneqq  f^2_{\ref{obs_create_meadow}}(h,k,k) + f^3_{\ref{lemma_grounded_extensions_to_obstructions}}(q,p,k,h).
\end{align*}
The values for $x,y,q,$ and $p$ are bounded in functions purely depending on $k$ and $h$ as explained below.
\smallskip

As a first step, we want to get the first two outcomes of the assertion out of the way and create a meadow for further refinement.
Calling upon \autoref{obs_create_meadow} yields either
\begin{enumerate}
  \item a $(3k^2+k\cdot h)$-minor grasped by $W,$
  \item a ``$k$-Dyck-walloid\footnote{That is a walloid which is the cylindrical concatenation of only $k$-handle and $k$-crosscap segments.}'' representing a surface of Euler-genus at least $g$ whose base cylinder is a subgraph of $W,$ or
  \item a set $S_1\subseteq V(G)$ of size at most $f^2_{\ref{obs_create_meadow}}(h,k,k)$ and an $(x,y)$-fertile $\big(z_1,\Sigma\big)$-meadow $(G_1,\delta_1,W_1)$ where $\Sigma$ is of Euler-genus less than $g$ where $G_1\coloneqq G-S_1.$
  Here we have
  \begin{align*}
    z_1 =~ f_{\ref{lem_meadow_plowing}} \big(&f_{\ref{lem_garden_creation}}(f^1_{\ref{ripe_orchard}}(f^2_{\ref{lemma_grounded_extensions_to_obstructions}}(k,h),k,h,f^1_{\ref{lemma_grounded_extensions_to_obstructions}}(k,h),x,y,\mathsf{h})\\
    &+f^2_{\ref{ripe_orchard}}(f^2_{\ref{lemma_grounded_extensions_to_obstructions}}(k,h),k,h,f^1_{\ref{lemma_grounded_extensions_to_obstructions}}(k,h),x,y,\mathsf{h})+1,k,h,\mathsf{h}),\mathsf{h},g\big).
  \end{align*}
\end{enumerate}
In the first and second cases, we are immediately done.
In the third case, we obtain the desired meadow and the following bounds on $x$ and $y$:
\begin{align*}
  x & \in \mathcal{O}(k^4\cdot h)\text{ and}\\
  y & \in 2^{k^{\mathcal{O}_h(1)}}.
\end{align*}

Next, we want to plow our meadow.
By using \autoref{lem_meadow_plowing} we obtain an $(x,y)$-fertile $\big(z_2,\Sigma\big)$-meadow $(G_1,\delta_2,W_2)$ which is, in addition, $\mathsf{h}$-plowed.
Moreover,
\begin{align*}
  z_2 =~ f_{\ref{lem_garden_creation}}\big(&f^1_{\ref{ripe_orchard}}(f^2_{\ref{lemma_grounded_extensions_to_obstructions}}(k,h),k,h,f^1_{\ref{lemma_grounded_extensions_to_obstructions}}(k,h),x,y,\mathsf{h})\\
  &+f^2_{\ref{ripe_orchard}}(f^2_{\ref{lemma_grounded_extensions_to_obstructions}}(k,h),k,h,f^1_{\ref{lemma_grounded_extensions_to_obstructions}}(k,h),x,y,\mathsf{h})+1,k,h,\mathsf{h}\big).
\end{align*}

Once we have our plowed meadow, it is time for proper domestication.
We will turn the meadow into a garden using \autoref{lem_garden_creation}.
This returns an $(x,y)$-fertile $\big(z_3,k,h,\Sigma\big)$-garden $(G_1,\delta_3,W_3)$ which is $\mathsf{h}$-blooming.
Here it holds that
\begin{align*}
  z_3 = f^1_{\ref{ripe_orchard}}(f^2_{\ref{lemma_grounded_extensions_to_obstructions}}(k,h),k,h,f^1_{\ref{lemma_grounded_extensions_to_obstructions}}(k,h),x,y,\mathsf{h})+f^2_{\ref{ripe_orchard}}(f^2_{\ref{lemma_grounded_extensions_to_obstructions}}(k,h),k,h,f^1_{\ref{lemma_grounded_extensions_to_obstructions}}(k,h),x,y,\mathsf{h})+1
\end{align*}

From here we want to proceed and create a ripe orchard.
By using \autoref{obs_single_thicket_orchard} we may assume that $(G_1,\delta_3,W_3)$ is in fact an $(x,y)$-fertile and $\mathsf{h}$-blooming single-thicket $\big(z_{4,1},k,h,z_{4,2},\mathsf{h},\Sigma\big)$-orchard.
We then apply \autoref{ripe_orchard} to obtain a $(q,p)$-ripe and $\mathsf{h}$-blooming $\big(f^2_{\ref{lemma_grounded_extensions_to_obstructions}}(k,h),k,h,f^1_{\ref{lemma_grounded_extensions_to_obstructions}}(k,h),\Sigma\big)$-orchard $(G_1,\delta_4,W_4).$
Here we have
\begin{align*}
  z_{4,1} &= f^1_{\ref{ripe_orchard}}(f^2_{\ref{lemma_grounded_extensions_to_obstructions}}(k,h),k,h,f^1_{\ref{lemma_grounded_extensions_to_obstructions}}(k,h),x,y,\mathsf{h}),\\
  z_{4,2} &= f^2_{\ref{ripe_orchard}}(f^2_{\ref{lemma_grounded_extensions_to_obstructions}}(k,h),k,h,f^1_{\ref{lemma_grounded_extensions_to_obstructions}}(k,h),x,y,\mathsf{h}),\\
  q & \leq f^3_{\ref{ripe_orchard}}(f^2_{\ref{lemma_grounded_extensions_to_obstructions}}(k,h),k,h,f^2_{\ref{lemma_grounded_extensions_to_obstructions}}(k,h),x,y,\mathsf{h})\text{, and}\\
  p & \leq f^4_{\ref{ripe_orchard}}(f^2_{\ref{lemma_grounded_extensions_to_obstructions}}(k,h),k,h,f^1_{\ref{lemma_grounded_extensions_to_obstructions}}(k,h),x,y,\mathsf{h}). 
\end{align*}

We now make use of \autoref{lemma_grounded_extensions_to_obstructions}.
This produces a set $S_2\subseteq V(G_1)$ of size at most $f^3_{\ref{lemma_grounded_extensions_to_obstructions}}(q,p,k,h)$ such that $W$ together with $S\coloneqq S_1\cup S_2$ satisfy the third outcome of our assertion as desired.
Notice that the algorithm from \autoref{lemma_grounded_extensions_to_obstructions} partitions the set of all graphs on at most $h$ vertices exactly into the set $\mathcal{Z}$ and those graphs which are fully removed by $S.$
With this, our proof is complete.
\end{proof}

\section{Extensions of graphs in their guest walloids}\llabel{sec_Htw}

This section is dedicated to the proof of the following theorem which is of crucial importance in our quest of proving Erd\H{o}s-P{\'o}sa dualities.
The results of these sections will be used in the proof of \autoref{lemma_half_integral_Brambles} in order to show that the walloids we consider contain large half-integral brambles.

\begin{theorem}\llabel{lemma_walloid_contains_obstruction}
For every graph $Z,$ there exists a function $f^{(Z)}_{\ref{lemma_walloid_contains_obstruction}} \colon\nn{2}{1}$ such that, for every embedding pair $\mathbf{p}$ of $\mathfrak{P}^{(Z)},$ $Z$ is a minor of $\mathscr{W}^{\mathbf{p}}_{f^{(Z)}_{\ref{lemma_walloid_contains_obstruction}}(|Z|)}.$
Moreover
$$f^{(Z)}_{\ref{lemma_walloid_contains_obstruction}}(|Z|) = 2^{\Ocal(\ell(\mathsf{h}(|Z|)))}.$$
\end{theorem}

This theorem can be seen as an extension of the fact (see \cite{RobertsonST94Quickly}) that every planar graph $Z$ is a minor of some large enough grid, whose size only depends on the size of $Z.$
This result has been further extended by Gavoille and Hilaire \cite{gavoille2023minoruniversal}
for the Dyck-walloids\footnote{See \autoref{thm_local_structure} for an explanation of this term.} that are the $\mathscr{W}^{\bf p}$'s corresponding to embedding pairs of the form $\mathbf{p} = (\Sigma, \emptyset).$
\autoref{lemma_walloid_contains_obstruction} is an extension of these results to every embedding pair.

\subsection{Designs of embedding pairs}

In this subsection, we introduce the notion of a \textsl{design} of embedding pairs.
This notion, while similar to that of a walloid (they are in fact equivalent as parametric graphs), its definition will prove more beneficial towards showing \autoref{lemma_walloid_contains_obstruction}.

\paragraph{Filaments.} Let $z, t \in \Nbbb_{\geq 3}$ and consider a $(z, t)$-wall $W.$
Let $P_{1}, \ldots, P_{t}$ and $Q_{1}, \ldots, Q_{z}$ be the horizontal and vertical paths of $W$ respectively.
For every $i \in [t-1]$ and $j \in [z-1],$ let $B^{j}_{i}$ denote the brick of $W$ such that $B^{j}_{i} \cap P_{i} \neq \emptyset,$ $B^{j}_{i} \cap P_{i+1} \neq \emptyset,$ $B^{j}_{i} \cap Q_{j} \neq \emptyset,$ and $B^{j}_{i} \cap Q_{j+1} \neq \emptyset.$
Also, let $Z^{j}_{i}$ be the path $P^{i+1} \cap Q_{i+1}.$

Now, given an $h \in \Nbbb$ and a choice of $i \in [t-1]$ and $j \in [z-1],$ we say that the graph $B^{j}_{i} \cup \Pcal$ is an \defi{$h$-filament} of $W \cup \cupall \Pcal$ if $\Pcal$ is a set of $h$ many pairwise disjoint $Z^{j}_{i}$-$P_{t}$-paths that are orthogonal to the horizontal paths of $W$ and are only allowed to intersect the paths $P_{i'} \cap Q_{j+1},$ for every $i' \in [i+1, t].$
We also refer to the endpoints of the paths in $\Pcal$ in $Z^{j}_{i}$ as the \defi{boundary} vertices of the $h$-filament.
We refer to the set $\Pcal$ as the \defi{paths} of the filament.

\paragraph{Frames.} Let $t \in \Nbbb_{\geq 2},$ $f \in \Nbbb$ and consider an $((f + 1)2t, 2t)$-wall $W.$
Let $\mathbf{P} \coloneqq \{ \Pcal_{1}, \ldots, \Pcal_{f} \}$ be a set of sets of paths and $I \coloneqq \{ h_{1}, \ldots, h_{f} \} \subseteq \Nbbb.$
Moreover, let $t' \coloneqq t/2.$
We say that the graph $W \cup \cupall \mathbf{P}$ is a \defi{$(t, I)$-frame} if for every $i \in [f],$ the graph $F_{i} \coloneqq B^{it'}_{t'} \cup \cupall \Pcal_{i}$ is an $h_{i}$-filament of $W \cup \cupall \mathbf{P}.$
We refer to the set $\{ F_{1}, \ldots, F_{f} \}$ as the \defi{set of filaments} of $W \cup \cupall \mathbf{P}.$

\paragraph{Designs of embedding pairs.} Let $Z$ be a graph and $\mathbf{p} \coloneqq (\Sigma, \mathbf{B})$ be an embedding pair of $\mathfrak{P}^{(Z)}.$
We define the \defi{$(t, \mathbf{p})$-design $\mathscr{D}^{\mathbf{p}}$} to be a parametric graph $\lin{\mathscr{D}^{\mathbf{p}}_{t}}_{t \in \Nbbb_{\geq 2}}$ as follows.
Let $b \coloneqq |\mathbf{B}|,$ $\mathbf{B} = \{ \lin{Z_{1}, \Lambda_{1}}, \ldots, \lin{Z_{b}, \Lambda_{b}} \}$ and $f \coloneqq 2\mathsf{h} + \mathsf{c} + b.$
For every $t \in \Nbbb_{\geq 2},$ $\mathscr{D}^{\mathbf{p}}_{t}$ is obtained from a $(t, I)$-frame $F,$ where $I = \lin{h_{1}, \ldots, h_{f}},$ as follows.
\begin{itemize}
\item For $i \in [2\mathsf{h}],$ $h_{i} \coloneqq t,$
\item for $i \in [2\mathsf{h} + 1, 2\mathsf{h} + \mathsf{c}],$ $h_{i} \coloneqq 2t,$ and
\item for $i \in [2\mathsf{h} + \mathsf{c} + 1, f],$ $h_{i} \coloneqq |\Lambda_{i}|t.$
\end{itemize}
Moreover, for the $i$-th $h_{i}$-filament of $F,$ let $u^{i}_{1}, \ldots, u^{i}_{h_{i}}$ be its boundary vertices in left to right order. Then
\begin{itemize}
\item for odd $i \in [2\mathsf{h}],$ connect $u^{i}_{j}$ with $u^{i+1}_{j},$ for every $j \in [h_{i}],$
\item for $i \in [2\mathsf{h} + 1, 2\mathsf{h} + \mathsf{c}],$ connect $u^{i}_{j}$ with $u^{i}_{t' + j},$ for every $j \in [t'],$ and
\item for $i \in [2\mathsf{h} + \mathsf{c} + 1, f],$ let $\mathsf{dec}(Z_{i}, \Lambda_{i}) = \lin{ \lin{Z^{i}_{1}, \Lambda^{i}_{1}}, \ldots, \lin{Z^{i}_{q_{i}}, \Lambda^{i}_{q_{i}}} }.$
Then, for every $p \in [q_{i}],$ we take $t$ many disjoint copies of $\lin{Z^{i}_{p}, \Lambda^{i}_{p}}$ and connect, for every $p \in [q_{i}],$ every $j \in [t],$ every $l \in |\Lambda^{i}_{p}|,$ the $l$-th vertex of the $j$-th copy of $\Lambda^{i}_{p}$ with $u^{i}_{j'},$ where $j' \coloneqq (\sum_{p' \in [p-1]} |\Lambda_{p'}|)t + |\Lambda_{p}|(j - 1) + l.$
\end{itemize}

The following observation follows easily from the definition of designs.

\begin{observation}\llabel{designs_contain_walloids} Let $Z$ be a graph and $\mathbf{p} \in \mathfrak{P}^{(Z)}$ be an embedding pair of $Z.$
Then, for every $t \in \Nbbb_{\geq 2},$ $\mathscr{W}^{\mathbf{p}}_{t}$ is a minor of $\mathscr{D}^{\mathbf{p}}_{2t}.$
\end{observation}

\subsection{Designs as minors of walloids}

In this subsection, we demonstrate how for a given embedding pair of some graph, we can find a large design of it as a minor of a large enough instance of its corresponding walloid.

\begin{lemma}\llabel{walloids_contain_designs} Let $Z$ be a graph and $\mathbf{p} \in \mathfrak{P}^{(Z)}$ be an embedding pair of $Z.$
Then, for every $t \in \Nbbb_{\geq 2},$ $\mathscr{D}^{\mathbf{p}}_{t}$ is a minor of $\mathscr{W}^{\mathbf{p}}_{5t}.$
\end{lemma}
\begin{proof} Let $\mathbf{p} = (\Sigma, \mathbf{B}),$ let $\mathsf{h}$ and $\mathsf{c}$ be such that $\Sigma = \Sigma^{(\mathsf{h}, \mathsf{c})},$ and assume that $\lin{ \lin{Z_{1}, \Lambda_{1}}, \ldots, \lin{Z_{q}, \Lambda_{q}} }$ are the linear societies contained in $\mathbf{B}$ arranged in the order they appear in $\mathscr{W}^{\mathbf{p}},$ $\lin{Z_{1}, \Lambda_{1}}$ being the linear society of the leftmost flower segment of $\mathscr{W}^{\mathbf{p}},$ while $\lin{Z_{q}, \Lambda_{q}}$ being the linear society of the rightmost flower segment of $\mathscr{W}^{\mathbf{p}}.$
Fix $t \in \Nbbb_{\geq 2}.$
We describe how to find $\mathscr{D}^{\mathbf{p}}_{t}$ as a minor of $\mathscr{W}^{\mathbf{p}}_{5t}.$

\begin{figure}[h]
    \centering
    \scalebox{.95}{
    \begin{tikzpicture}[scale=1]

        \pgfdeclarelayer{background}
        \pgfdeclarelayer{foreground}
            
        \pgfsetlayers{background,main,foreground}
            
        \begin{pgfonlayer}{main}
        \node (C) [v:ghost] {};
	
	\scalebox{0.9}{\pgftext{\includegraphics{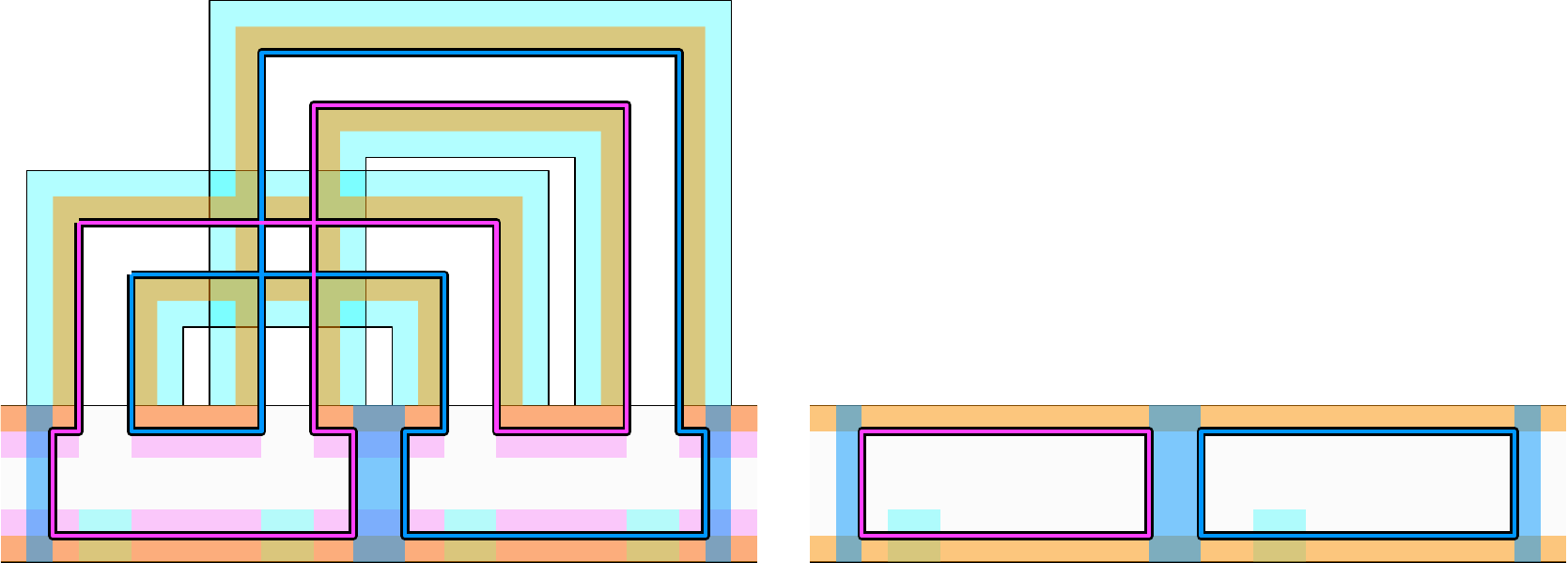}} at (C.center);}

            \node (center) [v:ghost,position=270:19.7mm from C] {};            
            \node (center_2) [v:ghost,position=180:21.6mm from center] {};
            \node (center_1) [v:ghost,position=180:35.4mm from center_2] {};
            
            \node (label_a) [v:ghost,position=167:11.8mm from center_1] {$a$};
            \node (label_b) [v:ghost,position=191:11.8mm from center_1] {$b$};
            \node (label_c) [v:ghost,position=348:13.2mm from center_1] {$c$};
            \node (label_d) [v:ghost,position=11:13.2mm from center_1] {$d$};
            
            \node (label_a_2) [v:ghost,position=167:12.6mm from center_2] {$a'$};
            \node (label_b_2) [v:ghost,position=191:12.6mm from center_2] {$b'$};
            \node (label_c_2) [v:ghost,position=348:11.8mm from center_2] {$c'$};
            \node (label_d_2) [v:ghost,position=11:11.8mm from center_2] {$d'$};
       
        \end{pgfonlayer}{main}
        
        \begin{pgfonlayer}{foreground}
        \end{pgfonlayer}{foreground}

        \begin{pgfonlayer}{background}
        \end{pgfonlayer}{background}
        
    \end{tikzpicture}
    }
    \caption{Finding two $t$-filaments that represent a handle in a $5t$-handle segment in the proof of \autoref{walloids_contain_designs}.}
    \llabel{designs_in_walloids_handle}
\end{figure}
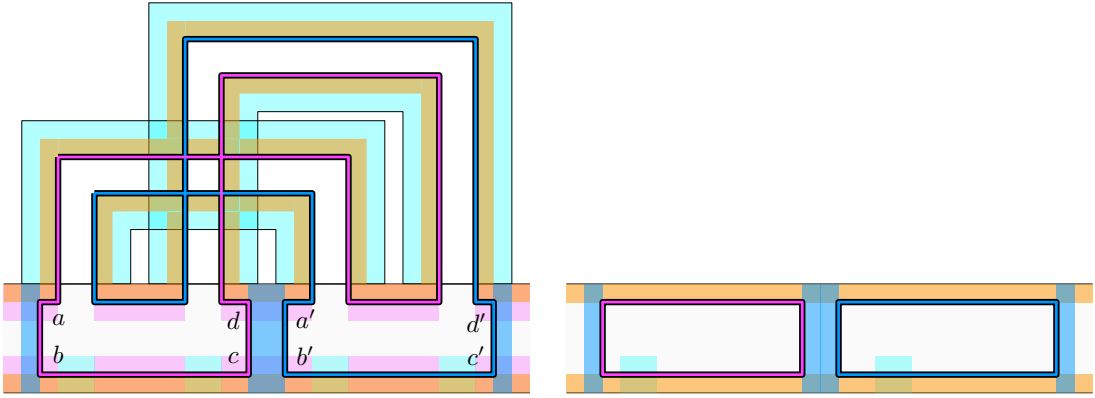

\paragraph{Defining the underlying wall.} We define two sets of $t$ pairwise disjoint paths, say $\Pcal$ and $\Pcal',$ which are subgraphs of $\mathscr{W}^{\mathbf{p}}_{5t}$ as follows.
We define $\Pcal$ as the paths starting from the leftmost boundary vertices of the first handle segment of $\mathscr{W}^{\mathbf{p}}_{4t},$ ending at the rightmost boundary vertices of the last flower segment of $\mathscr{W}^{\mathbf{p}}_{5t},$ where all such endpoints are vertices of the first $t$ cycles of the base cylinder of $\mathscr{W}^{\mathbf{p}}_{5t},$ with the additional property that whenever these paths intersect a handle, crosscap, or flower segment, they use the edges of their corresponding rainbows as depicted in \autoref{designs_in_walloids_handle}, \autoref{designs_in_walloids_crosscap}, and \autoref{designs_in_walloids_flower} respectively.
Moreover, we define $\Pcal'$ as the subpaths of the last $t$ cycles of the base cylinder of $\mathscr{W}^{\mathbf{p}}_{5t}$ starting from the leftmost boundary vertices of the first handle segment of $\mathscr{W}^{\mathbf{p}}_{5t},$ ending at the rightmost boundary vertices of the last flower segment of $\mathscr{W}^{\mathbf{p}}_{5t},$ intersecting all segments of the walloid.
The paths in $\Pcal \cup \Pcal'$ will serve as the horizontal paths of the underlying wall of $\mathscr{D}^{\mathbf{p}}_{t}.$

We proceed with the definition of multiple sets of pairwise disjoint paths, depending on the three different types of segments present in $\mathscr{W}^{\mathbf{p}}_{5t}.$

Let $i \in \mathsf{h}$ and let $W$ be the $i$-th handle segment of $\mathscr{W}^{\mathbf{p}}_{5t}.$
We define three sets, say $\Hcal_{i}^{j},$ $j \in [3],$ of paths that are vertical paths of $\mathscr{W}^{\mathbf{p}}_{5t}$ as follows.
$\Hcal_{i}^{1}$ (resp. $\Hcal_{i}^{3}$) contains the first (resp. last) $t$ many vertical paths of $W$ and $\Hcal_{i}^{2}$ is the set that contains the $t$ many vertical paths of $W$ whose topmost endpoints are the last $t$ among the leftmost endpoints of the rightmost rainbow of $W,$ union the set that contains the $t$ many vertical paths of $W$ whose topmost endpoints are the first $t$ among the rightmost endpoints of the leftmost rainbow of $W.$

Now, let $i \in \mathsf{c}.$
We define two sets, say $\Ccal_{i}^{j},$ $j \in [2],$ of paths that are vertical paths of $\mathscr{W}^{\mathbf{p}}_{5t}$ as follows.
$\Ccal_{i}^{1}$ (resp. $\Ccal_{i}^{2}$) contains the first (resp. last) $t$ many vertical paths of the $i$-th crosscap segment of our walloid.
Finally, let $i \in [b],$ where $b$ is the number of flower segments of our walloid.
As before, we define two sets, say $\Fcal_{i}^{j},$ $j \in [2],$ of paths that are vertical paths of $\mathscr{W}^{\mathbf{p}}_{5t}$ as follows.
$\Fcal_{i}^{1}$ (resp. $\Fcal_{i}^{2}$) contains the first (resp. last) $t$ many vertical paths of the $i$-th flower segment of our walloid.

Clearly, the graph
$$U \coloneqq \cupall (\Pcal \cup \Pcal' \cup \bigcup_{i \in [\mathsf{h}]} (\Hcal_{i}^{1} \cup \Hcal_{i}^{2} \cup \Hcal_{i}^{3}) \cup \bigcup_{i \in [\mathsf{c}]} (\Ccal_{i}^{1} \cup \Ccal_{i}^{2}) \cup \bigcup_{i \in [f]} (\Fcal_{i}^{1} \cup \Fcal_{i}^{2}))$$
is an $((f+1)2t, 2t)$-wall, where we define $f \coloneqq 2\mathsf{h} + \mathsf{c} + b.$
$U$ will serve as the underlying wall of the frame we will define for the design we're looking for.

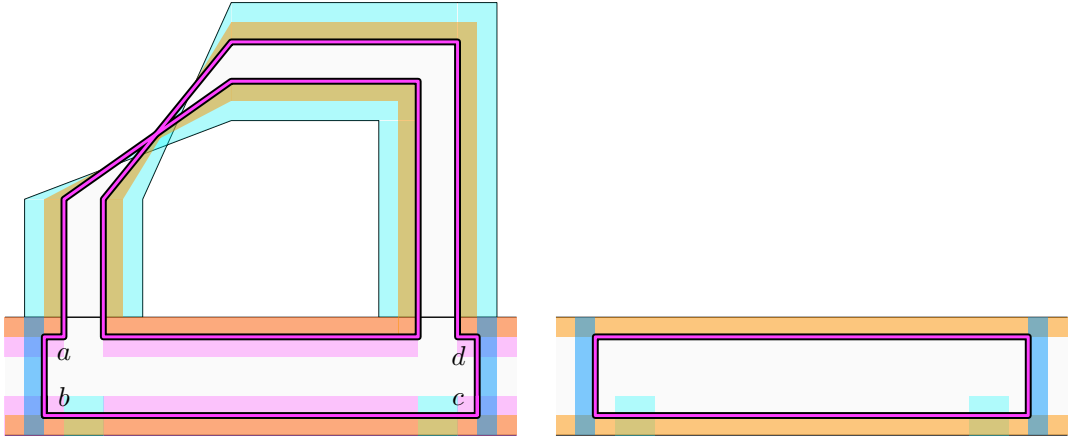
\begin{figure}[h]
    \centering
    \scalebox{1.03}{
    \begin{tikzpicture}[scale=1]

        \pgfdeclarelayer{background}
        \pgfdeclarelayer{foreground}
            
        \pgfsetlayers{background,main,foreground}
            
        \begin{pgfonlayer}{main}
        \node (C) [v:ghost] {};

	\scalebox{0.9}{\pgftext{\includegraphics{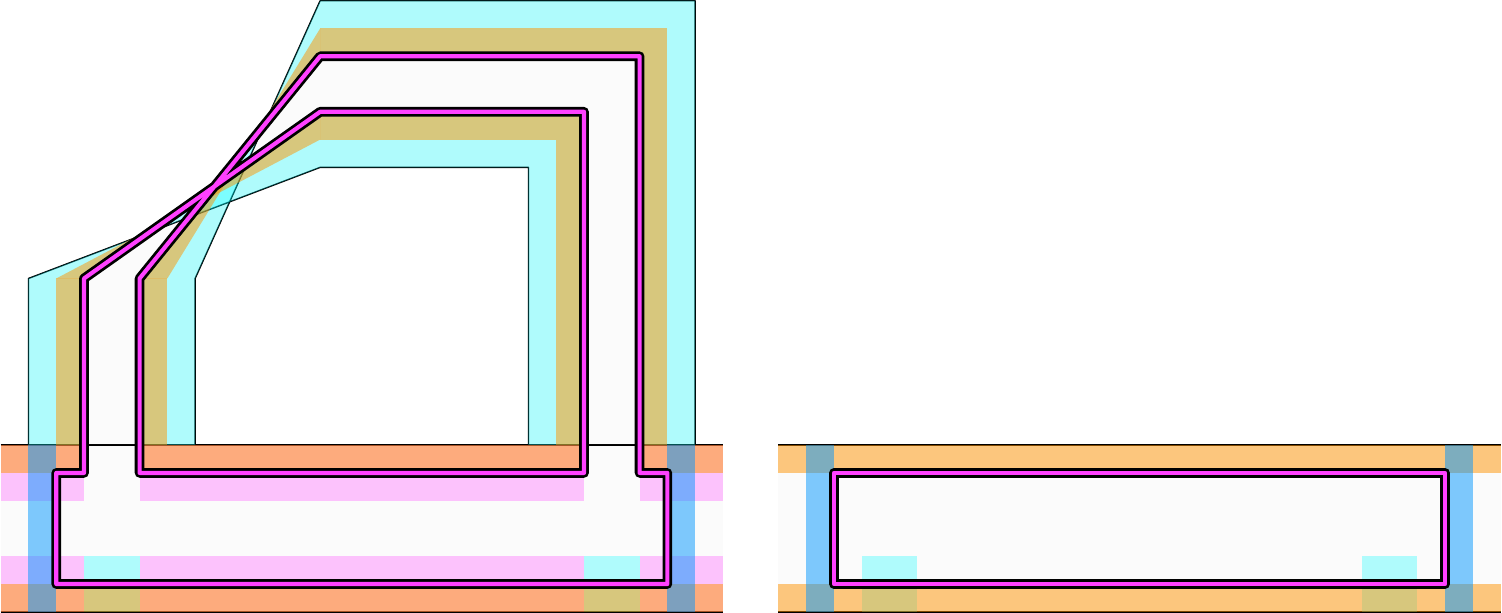}} at (C.center);}
            
        \node (center) [v:ghost,position=270:20.2mm from C] {};
        \node (center_1) [v:ghost,position=180:35.1mm from center] {};
        
            \node (label_a) [v:ghost,position=174:26mm from center_1] {$a$};
            \node (label_b) [v:ghost,position=186:26mm from center_1] {$b$};
            \node (label_c) [v:ghost,position=353:25.2mm from center_1] {$c$};
            \node (label_d) [v:ghost,position=5.8:25.2mm from center_1] {$d$};
       
        \end{pgfonlayer}{main}
        
        \begin{pgfonlayer}{foreground}
        \end{pgfonlayer}{foreground}

        \begin{pgfonlayer}{background}
        \end{pgfonlayer}{background}
        
    \end{tikzpicture}
    }
    \caption{Finding a $2t$-filament that represents a crosscap in a $5t$-crosscap segment in the proof of \autoref{walloids_contain_designs}.}
    \llabel{designs_in_walloids_crosscap}
\end{figure}

\paragraph{Defining the appropriate frame.}
Let $t' \coloneqq t/2.$
We proceed to define a $(t, I)$-frame $F,$ where $I \coloneqq \lin{h_{1}, \ldots, h_{f}},$ from $\mathscr{W}^{\mathbf{p}}_{5t},$ whose underlying wall is $U.$

Let $i \in [2\mathsf{h}]$ be odd and $i' \coloneqq i \ \mathsf{mod} \ 2.$
Notice that by construction, the brick $B^{it'}_{t'}$ (resp. $B^{(i+1)t'}_{t'}$) of $U$ corresponds to a unique cycle of the $i'$-th handle segment of $\mathscr{W}^{\mathbf{p}}_{5t},$ that is the cycle bounded by $abcd$ (resp. $a'b'c'd'$) as depicted in \autoref{designs_in_walloids_handle}.
Moreover, again by construction, there remain at least $t$ edges from each of the rainbow linkages of the $i'$-th handle segment we have not used.
W.l.o.g. we can use $t$ of them from the leftmost rainbow and extend both of their endpoints through their incident vertical paths of $\mathscr{W}^{\mathbf{p}}_{5t}$ until both endpoints are bottom boundary vertices of $U.$
These paths give us the necessary infrastructure to argue that $B^{it'}_{t'}$ and $B^{(i+1)t'}_{t'}$ union these paths produce an $h_{i}$-filament and $h_{i+1}$-filament of $F$ respectively, where $h_{i} \coloneqq h_{i+1} \coloneqq t,$ with the necessary edges between the boundary vertices of the two filaments required by the definition of a $(t, \mathbf{p})$-design.

Let $i \in [2\mathsf{h} + 1, 2\mathsf{h} + \mathsf{c}]$ and $i' \coloneqq i - 2\mathsf{h}.$
In a similar way as before, observe that by construction, the brick $B^{it'}_{t'}$ of $U$ corresponds to a unique cycle of the $i'$-th crosscap segment of $\mathscr{W}^{\mathbf{p}}_{5t},$ that is the cycle bounded by $abcd$ as depicted in \autoref{designs_in_walloids_crosscap}.
As before, there remain at least $2t$ edges from the rainbow linkage of the $i'$-th crosscap segment we have not previously used.
Now, we can use $2t$ of these edges and extend both of their endpoints through their incident vertical paths of $\mathscr{W}^{\mathbf{p}}_{5t}$ until both endpoints are bottom boundary vertices of $U.$
These paths give us the necessary infrastructure to argue that $B^{it'}_{t'}$ union these paths produces an $h_{i}$-filament of $F,$ where $h_{i} \coloneqq 2t,$ with the necessary edges between the boundary vertices of the filament required by the definition of a $(t, \mathbf{p})$-design.

Finally, let $i \in [2\mathsf{h} + \mathsf{c} + 1, f]$ and $i' \coloneqq i - 2\mathsf{h} + \mathsf{c}.$
In a similar way as before, observe that by construction, the brick $B^{it'}_{t'}$ of $U$ corresponds to a unique cycle of the $i'$-th flower segment of $\mathscr{W}^{\mathbf{p}}_{5t},$ that is the cycle bounded by $abcd$ as depicted in \autoref{designs_in_walloids_flower}.
Now, we can simply connect all linear societies $\lin{Z_{i'}, \Lambda_{i'}}$ of the $i'$-th flower to bottom boundary vertices of $U$ by using the vertical paths of $\mathscr{W}^{\mathbf{p}}_{5t}$ whose topmost endpoint is a top boundary vertex of its base cylinder that is adjacent to some vertex of a copy of $\Lambda_{i'}.$
This gives the necessary infrastructure to argue that $B^{it'}_{t'}$ union these paths produces an $h_{i}$-filament of $F,$ where $h_{i} \coloneqq |\Lambda_{i'}|t,$ as required by the definition of a $(t, \mathbf{p})$-design.
\end{proof}

\begin{figure}[h]
    \centering
    \scalebox{1}{
    \begin{tikzpicture}[scale=1]

        \pgfdeclarelayer{background}
        \pgfdeclarelayer{foreground}
            
        \pgfsetlayers{background,main,foreground}
            
        \begin{pgfonlayer}{main}
        \node (C) [v:ghost] {};

	\scalebox{0.9}{\pgftext{\includegraphics{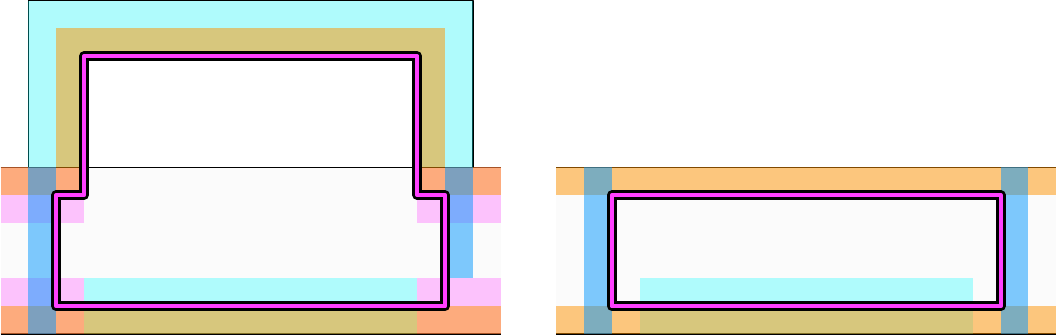}} at (C.center);}
            
        \node (center) [v:ghost,position=270:7.5mm from C] {};
        \node (center_1) [v:ghost,position=180:25mm from center] {};
        
            \node (label_a) [v:ghost,position=170:15.5mm from center_1] {$a$};
            \node (label_b) [v:ghost,position=188:15.5mm from center_1] {$b$};
            \node (label_c) [v:ghost,position=350:15mm from center_1] {$c$};
            \node (label_d) [v:ghost,position=10:15mm from center_1] {$d$};
       
        \end{pgfonlayer}{main}s
        
        \begin{pgfonlayer}{foreground}
        \end{pgfonlayer}{foreground}

        \begin{pgfonlayer}{background}
        \end{pgfonlayer}{background}
        
    \end{tikzpicture}
    }
    \caption{Finding a $|\Lambda|t$-filament representing a $(t,t,Z,\Lambda)$-flower in a $(t,t,Z,\Lambda)$-flower segment in the proof of \autoref{walloids_contain_designs}.}
    \llabel{designs_in_walloids_flower}
\end{figure}
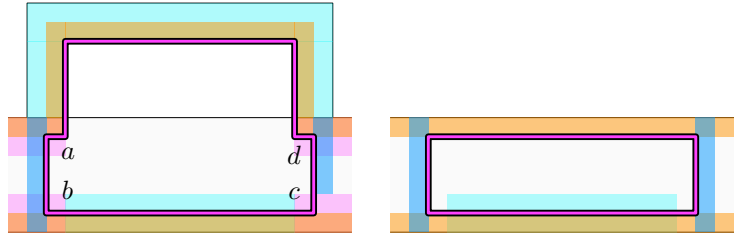

The previous lemma combined with \autoref{designs_contain_walloids} shows that for any embedding pair $\mathbf{p} \in \mathfrak{P}^{(Z)}$ of any graph $Z,$ $\mathscr{D}^{\mathbf{p}} \approx \mathscr{W}^{\mathbf{p}}.$

\subsection{Graphs as minors of their designs}

In this subsection, we demonstrate how for a given embedding pair of some graph $Z,$ we can find $Z$ as a minor of a large enough instance of the design corresponding to its embedding pair, whose size only depends on the size of $Z.$

\begin{lemma}\llabel{designs_contain_graphs} For every graph $Z,$ there exists a function $f^{(Z)}_{\ref{designs_contain_graphs}} \colon \nn{1}{1}$ such that, for every embedding pair $\mathbf{p} \in \mathfrak{P}^{(Z)},$ $Z$ is a minor of $\mathscr{D}^{\mathbf{p}}_{f^{(Z)}_{\ref{designs_contain_graphs}}(|Z|)}.$
Moreover
$$f^{(Z)}_{\ref{designs_contain_graphs}}(|Z|) = 2^{\Ocal(\ell(\mathsf{h}(|Z|)))}.$$
\end{lemma}
\begin{proof} We define $f^{(Z)}_{\ref{designs_contain_graphs}}(|Z|) \coloneqq \Ocal(\mathsf{ext}(|Z|)^{4}).$

Let $\mathbf{p} = (\Sigma, \mathbf{B})$ and assume that $\mathbf{p}$ is generated by the triple $(M, \Pcal, \Gamma)$ where, $M$ is a connected graph that contains $Z$ as a minor, where $|M| \leq \mathsf{ext}(|Z|),$ $\Pcal \in \Pcal(M),$ and $\Gamma$ is a $\Sigma$-embedding of $\Kcal_{\Pcal}.$
Let $\mathsf{h}, \mathsf{c} \in \Nbbb$ be such that $\Sigma = \Sigma^{(\mathsf{h}, \mathsf{c})}.$
We treat hyperedges of order two in $\Kcal_{\Pcal}$ as regular edges of a graph and whenever we refer to a hyperedge as an edge we mean a hyperedge of order two.
We first obtain a new hypergraph $H$ from $\Kcal_{\Pcal}$ as follows.
\begin{itemize}
\item For every vertex $u \in V(\Kcal_{\Pcal}),$ let $E_{u}$ denote the set of hyperedges that contain $u.$
We define a hypergraph $H''$ from $\Kcal_{\Pcal}$ by replacing $u$ by a path $P_{u}$ on $|E_{u}|$ many fresh vertices, where every vertex of $P_{u}$ replaces $u$ in exactly one distinct hyperedge in $E_{u}.$
Observe that $H''$ is $\Sigma$-embeddable.
Let $\Gamma'$ be a $\Sigma$-embedding of $H''$ obtained from $\Gamma$ in the obvious way.
\item Next, from $H''$ we define a hypergraph $H'$ by considering, for every $i \in [\mathsf{h}],$ a non-contractible cycle $C_{i}^{h}$ cutting through the $i$-th handle of $\Sigma$ that intersects $\Gamma'$ only in edges of $H''$ and exactly once per edge.
We then subdivide each edge that $C_{i}^{h}$ intersects twice and we consider $\Tcal_{i}^{h}$ to be the unique set of edges incident to the newly introduced vertices. 
Additionally, for every $i \in [\mathsf{c}],$ we consider a non-contractible cycle $C_{i}^{c}$ cutting through the $i$-th crosscap of $\Sigma,$ i.e., a cycle such that one of the two sides is a Moebius strip, that intersects $\Gamma'$ only in edges of $H''$ and exactly twice per edge.
We then subdivide each edge that $C_{i}^{c}$ intersects twice, drawing a vertex at each point of intersection, and we consider $\Tcal_{i}^{c}$ to be the unique set of edges incident to the newly introduced vertices. 
\item Finally, for every $i \in [\mathsf{h}]$ (resp. every $i \in [\mathsf{c}]$), we define the sets $A_{i}$ and $B_{i}$ (resp. the set $S_{i}$) such that $\Tcal_{i}^{h}$ (resp. $\Tcal_{i}^{c}$) is an $A_{i}$-$B_{i}$ (resp. $S_{i}$-$S_{i}$) linkage in $H'.$
We define $H$ from $H'$ by adding $A_{i}$ and $B_{i}$ (resp. $S_{i}$) as hyperedges.
\end{itemize}

Let $\Tcal \coloneqq \bigcup_{i \in [\mathsf{h}]} \Tcal_{i}^{h} \cup \bigcup_{i \in [\mathsf{c}]} \Tcal_{i}^{c}.$
Observe that by construction, hyperedges of order at least three in $H$ are pairwise disjoint, $|H| \leq \Ocal(|\Kcal_{\Pcal}|^{2}) = \Ocal(\mathsf{ext}(|Z|)^{2}),$ and that the hypergraph $H - \Tcal$ is plane-embeddable.
Let
$$\mathbf{B}^{+} \coloneqq \{ \lin{A_{1}, \Lambda^{A}_{1}}, \lin{B_{1}, \Lambda^{B}_{1}} \mid i \in [\mathsf{h}] \} \cup \{ \lin{S_{1}, \Lambda^{S}_{1}}\mid i \in [\mathsf{c}] \} \cup \mathbf{B}$$
which is a set of linear societies each corresponding to a hyperedge of $H$ such that, for every $i \in [\mathsf{h}],$ $V(\Lambda^{A}_{i}) = A_{i},$ $V(\Lambda^{B}_{i}) = B_{i},$ and the $j$-th vertex of $\Lambda^{A}_{i}$ is adjacent (via $\Tcal_{i}^{h}$) the $|\Lambda^{B}_{i}| - j + 1$-th vertex $\Lambda^{B}_{i},$ while for every $i \in \mathsf{c},$ $V(\Lambda^{C}_{i}) = C_{i}$ and the $j$-th vertex of $\Lambda^{C}_{i}$ is adjacent (via $\Tcal_{i}^{c}$) to the $j + |\Lambda^{C}_{i}|/2$-th vertex of $\Lambda^{C}_{i}.$
Moreover, every hyperedge of $H$ that does not correspond to a linear society in $\mathbf{B}^{+}$ is of order two.
Moreover, let $b \coloneqq |\mathbf{B}|$ and let
$$\mathbf{B} = \{ \lin{Z_{i}, \Lambda_{i}} \mid i \in [b] \}.$$
Let $t' \coloneqq f^{(Z)}_{\ref{designs_contain_graphs}}(|Z|).$
We are now in the position to describe how to find $Z$ as a minor of the $(t', \mathbf{p})$-design $\mathscr{D}^{\mathbf{p}}_{t'}.$
 and $f \coloneqq |\mathbf{B}^{+}| = 2\mathsf{h} + \mathsf{c} + b.$
Let $F$ be the underlying $(t', I)$-frame of $\mathscr{D}^{\mathbf{p}}_{t'},$ where $I = \lin{h_{1}, \ldots, h_{f}}.$

We can first observe that by construction, for every odd $i \in [2\mathsf{h}],$ we can map $\Lambda^{A}_{i}$ (resp. $\Lambda^{B}_{i}$) to the boundary vertices of the $i$-th (resp. $i+1$-th) filament of $F$ respecting the linear ordering as well as embed the edges in $\Tcal^{h}_{i}$ by using the edges between the boundary vertices of the two filaments.
Moreover, for every $i \in [2\mathsf{h} + 1, 2\mathsf{h}+\mathsf{c}+1]$ we can map $\Lambda^{C}_{i'},$ where $i' \coloneqq 2\mathsf{h} - i,$ to the boundary vertices of the $i$-th filament of $F$ respecting the linear ordering as well as embed the edges in $\Tcal^{c}_{i}$ using the edges between the boundary vertices of the filament.
Lastly, for every $i \in [2\mathsf{h} + \mathsf{c} + 1, f],$ we can map $\Lambda_{i'},$ where $i' \coloneqq 2\mathsf{h} + \mathsf{c},$ to  boundary vertices of the $i$-th filament of $F$ respecting the linear ordering such that we find $\lin{Z_{i'}, \Lambda_{i'}}$ as a minor within the $i$-th filament of our design.

It remains to show how to appropriately model the edges of $H$ that do not correspond to hyperedges of $\mathbf{B}^{+}.$
This part is implied from standard techniques in the following way.
First, we subdivide all such edges twice.
For every $i \in [f],$ let $\Lambda$ denote the linear society we have associated to the $i$-th filament in the previous step.
We begin by considering a wall $W_{i}$ of order $\Ocal(\mathsf{ext}(|Z|)^{4})$ in our design whose central brick is the brick of the $i$-th filament.
Then, we connect each vertex $v_{i}$ of $\Lambda$ to a set, say $B_{i},$ of boundary vertices of $W_{i}$ that correspond to each neighbour of $v_{i}$ that was created by the previously subdivided edges.
We do this while making sure that the subgraph of $W_{i}$ we have used to make this connection is disjoint from all others and that the vertices in $B_{i}$ are pairwise at distance at least $\max\{ |B_{i}| \mid v_{i} \in \Lambda \} \in \Ocal(\mathsf{ext}(|Z|)^{2})$ (i.e., if they are on the same vertical path of $W_{i}$ there should be at least $\Ocal(\mathsf{ext}(|Z|)^{2})$ horizontal paths between the ones they belong to and symmetrically if they belong on the same horizontal path).
By performing this step we have successfully managed to find a subcubic model for each of the vertices of $\Lambda.$
This can be done via standard routing arguments (similarly to \autoref{lem_paths_boundary} for example).
Now, every edge of $H$ we want to model, can be modelled as a $B_{i}$-$B_{j}$-path for distinct $i, j \in [f].$
Moreover, the collection of all these paths we are looking form a linkage.
However, this is an instance of the disjoint paths problem in a planar graph as certified by $H - \Tcal$ being plane-embeddable.
Hence, again by standard routing arguments, we can show that if we carefully choose the constants hidden in the definition of $f^{(Z)}_{\ref{designs_contain_graphs}}(|Z|)$ we can find the desired paths using the infrastructure of $\mathscr{D}^{\mathbf{p}}_{t'}$ away from the walls $W_{i},$ $i \in [f]$ we have chosen at the previous step.
\end{proof}

We are now ready to proceed with the proof of the main theorem.

\begin{proof}[Proof of \autoref{lemma_walloid_contains_obstruction}]
We define $f_{\ref{lemma_walloid_contains_obstruction}}(|Z|) \coloneqq 5f^{(Z)}_{\ref{designs_contain_graphs}}(|Z|).$

Let $z \coloneqq f^{(Z)}_{\ref{designs_contain_graphs}}(|Z|).$
Then, the proof follows by an application of \autoref{walloids_contain_designs}, in order to find the $(t, \mathbf{p})$-design $\mathscr{D}^{\mathbf{p}}_{z}$ as a minor of $\mathscr{W}_{5z}.$
We conclude by calling upon \autoref{designs_contain_graphs} to find $Z$ as a minor of $\mathscr{D}^{\mathbf{p}}_{z}.$
\end{proof}

\section{Universal obstructions for $\Hcal$-treewidth}
\llabel{sec_obstructions_for_Htw}

In this section, we make a first step in tying together all of the previously established structural theorems to show that \autoref{thm_local_structure} either provides an obstruction to $\mathcal{H}$-treewidth or can be used to produce an object we call a ``$\mathcal{H}$-tree decomposition'' of small width.

We first, in \autoref{subsec_H_tw}, provide a formal definition of $\mathcal{H}$-tree decompositions and their width and then show that the resulting parameter equals $\mathcal{H}\text{-}\mathsf{tw}.$

Then, in \autoref{subsec_H_brambles} we provide an extension of the notion of brambles to the setting of $\mathcal{H}$-treewidth.
These notions build the foundation for both the lower and the upper bound through our universal obstructions.
The proof of \autoref{thm_H_tw_lowerbound} is then presented in \autoref{subsec_lowerbound_H_tw}.

Finally, towards the proof of \autoref{main_grid_general} in \autoref{subsec_upperbound_H_tw} we provide an argument that shows how \autoref{thm_local_structure} can be used to obtain an $\mathcal{H}$-tree decomposition of small width in the absence of any of our universal obstructions.

\subsection{A torso-free definition of $\mathcal{H}$-treewidth}\llabel{subsec_H_tw}

Let us begin by finally introducing a definition of tree-decompositions.

Let $G$ be a graph.
A \defi{tree decomposition} of $G$ is a pair $(T,\beta)$ where $T$ is a tree and $\beta\colon V(T)\to 2^{V(G)}$ such that
\begin{itemize}
    \item $\bigcup_{t\in V(T)}\beta(t)=V(G),$
    \item for every $e\in E(G)$ there exists $t\in V(T)$ such that $e\subseteq \beta(t),$ and
    \item for every $v\in V(G)$ the set $\{ t\in V(T) \mid v\in\beta(t) \}$ induces a connected subgraph of $T.$
\end{itemize}
We refer to the vertices of $T$ as the \defi{nodes} of the tree decomposition and to the sets $\beta(t)$ as its \defi{bags}.
For every $t\in V(T)$ the sets in $\{ \beta(d)\cap \beta(t) \mid dt\in E(T) \}$ are called the \defi{adhesions} of the node $t.$
The \defi{width} of $(T,\beta)$ is defined as $\max_{t\in V(T)}|\beta(t)|-1$ and its \defi{adhesion} is defined as $\max_{t\in V(T)}\max_{A\in \{ \beta(d)\cap \beta(t) ~\mid~ dt\in E(T) \}} |A|.$
The \defi{treewidth} of $G,$ denoted by $\mathsf{tw}(G),$ is defined to be the minimum width over all tree decompositions of $G.$

Now let $\mathcal{H}$ be any proper minor-closed graph class\footnote{Please note that the definition can also be extended to any graph class $\mathcal{H},$ but in this paper we are only interested in minor-closed graph classes.}.
A tree decomposition $(T,\beta)$ for a graph $G$ is called an \defi{$\mathcal{H}$-tree decomposition} if for every leaf $d\in V(T)$ it holds that $G[\beta(d)\setminus A_d]\in \mathcal{H}$ where $A_d$ is the unique adhesion set of $d.$
The \defi{width} of an $\mathcal{H}$-tree decomposition $(T,\beta)$ is defined as
\begin{align*}
     \max ~\{ 0\} &\cup \{ |\beta(t)|-1 ~\!\mid\!~ t\in V(T)\text{ is not a leaf of }T \}\\
     &\cup \{ |A_d|-1 ~\!\mid\!~ d\in V(T)\text{ is a leaf with adhesion set }A_d\}.
\end{align*}
Where we consider $\emptyset$ to be the adhesion set of the node $d$ if $d$ is the unique vertex of $T.$

We begin with the simple observation that $\mathcal{H}$-tree decompositions allow for a purely decomposition-based definition of $\mathcal{H}$-treewidth.

\begin{lemma}\llabel{lemma_Htw_via_decompositions}
Let $\mathcal{H}$ be a proper minor-closed graph class.
Then for every graph $G$ it holds that $\mathcal{H}\text{-}\mathsf{tw}(G)$ is equal to the minimum width over all $\mathcal{H}$-tree decompositions of $G.$
\end{lemma}

\begin{proof}
We first show that there always exists an $\mathcal{H}$-tree decomposition of width $\mathcal{H}\text{-}\mathsf{tw}(G).$
To see this let $X\subseteq V(G)$ be a set of vertices such that every component of $G-X$ belongs to $\mathcal{H}$ and $\mathsf{tw}(\mathsf{torso}(G,X))$ is minimum over all choices for such a set $X.$
Let $G^+\coloneqq \mathsf{torso}(G,X)$ and let $(T',\beta')$ be a tree decomposition of minimum width for $G^+.$
It follows that the width of $(T',\beta')$ equals $\mathcal{H}\text{-}\mathsf{tw}(G).$

Let $J$ be any component of $G-X,$ then $N(J)$ is a clique of $G^+.$
Hence, there exists a node $t_J\in V(T')$ such that $N(J)\subseteq \beta(t_J).$
Therefore, for every such component $J$ we may introduce a new vertex $d_J$ adjacent to $t_J$ and declare the set $V(J)\cup N(J)$ to be the bag of $d_J.$
Let $(T'',\beta'')$ be the resulting tree decomposition of $G.$
Suppose there exists some leaf $d\in V(T'')$ such that $G[\beta(d)\setminus A_d]\notin\mathcal{H}.$
This means, in particular, that $d\in V(T').$
For each of these leaves $d$ we introduce a new vertex $d'$ adjacent to $d$ and declare its bag to also be $\beta(d)$ which implies that $|\beta(d')|-1=|\beta(d)|-1\leq \mathcal{H}\text{-}\mathsf{tw}(G).$

It follows that the resulting decomposition, let us call it $(T,\beta),$ is now an $\mathcal{H}$-tree decomposition for $G.$
Moreover, its width is exactly the width of $(T',\beta').$
\medskip

For the reverse inequality let $(T,\beta)$ be an $\mathcal{H}$-tree decomposition of $G$ of minimum width, say $k.$
In case $T$ consists of a single vertex we must have $G\in\mathcal{H}$ and $\mathcal{H}\text{-}\mathsf{tw}(G)=0=k.$
Hence, we may assume that $T$ has more than one vertex.
Let $L$ be the set of all leaves of $T.$ 
Then, for every leaf $d\in L$ let $J_d\coloneqq G[\beta(d)\setminus A_d]$ and let $F\coloneqq \bigcup_{d\in L}V(J_d).$
Finally, let $(T',\beta')$ be obtained from $(T,\beta)$ by deleting, for each $d\in L,$ the set $V(J_d)$ from $\beta(d)$ and setting $T'\coloneqq T.$
Notice that for each $d\in L$ and every component $K$ of $J_d$ it holds that $N(K)\subseteq A_d.$
Hence, $(T,\beta')$ is indeed a tree decomposition of $\mathsf{torso}(G,V(G)\setminus F).$
Thus, the $\mathcal{H}$-treewidth of $G$ is at most $k$ as desired.
\end{proof}

Let $\Hcal_{\varnothing}$ be the graph class containing only the empty graph $K_{0}$ and observe that $\obs(\Hcal_{\varnothing})=\{K_{1}\}.$
It is also easy to observe that $\mathcal{H}_{\varnothing}\text{-}\mathsf{tw}$ is the same as $\tw.$ 
According to the next observation, excluding $K_{1}$ defines an equivalent parameter as when excluding any $\leq$-antichain containing some planar graph.

Recall that due to the Grid Theorem of Robertson and Seymour \cite{RobertsonS86GMV}, a proper minor-closed graph class $\mathcal{H}$ has bounded treewidth if and only if $\mathsf{obs}(\mathcal{H})$ contains a planar graph.

\begin{observation}\llabel{obs_Htw_boundedtw}
Let $\mathcal{H}$ be a proper minor-closed graph class.
Then $\mathcal{H}\text{-}\mathsf{tw} \sim \mathsf{tw}$ if and only if $\mathsf{obs}(\mathcal{H})$ contains a planar graph.
\end{observation}

\begin{proof}
To see this, notice that by \autoref{lemma_Htw_via_decompositions} it suffices to consider $\mathcal{H}$-tree decomposition of small width since any tree decomposition of bounded width can easily be turned into a $\mathcal{H}$-tree decomposition by duplicating all leaves.
Let $G$ be some graph with $\mathcal{H}\text{-}\mathsf{tw}(G)\leq k$ and let $(T,\beta)$ be an $\mathcal{H}$-tree decomposition of minimum width of $G.$
Now let $d\in V(T)$ be some leaf with adhesion set $A_d.$
Then $G_d\coloneqq G[\beta(d)\setminus A_d]\in\mathcal{H}$ which means that $\mathsf{tw}(G_d)\leq h_{\mathcal{H}}$ by the Grid Theorem \cite{RobertsonS86GMV}.
We may now replace $d$ by a whole tree decomposition of width at most $h_{\mathcal{H}}$ for $G_d,$ then add $A_d$ to every bag of this decomposition and finally join an arbitrary node of its tree to the ancestor of $d$ in $T.$
The result of performing this operation for every leaf of $T$ is a tree decomposition of width at most $k+h_{\mathcal{H}}$ for $G.$
\end{proof}

\subsection{Strict $\mathcal{H}$-brambles}\llabel{subsec_H_brambles}

To provide good lower bounds for $\mathcal{H}$-treewidth we adapt the notion of brambles to our setting.
The basic ideas can be traced back to the papers by Reed but we adapt the proof ideas of Adler, Gottlob, and Grohe \cite{Adler2007Hypertree}.

Let $\mathcal{H}$ be a proper minor-closed graph class\footnote{Notice that, as in the previous section, we only define these notions for minor-closed graph classes due to our setting, but they work the same way for general graph classes. For \autoref{lemma_bramble_lowerbound} to hold, one should assume just hereditarity though.}.
Let $G$ be a graph.
A \defi{strict $\mathcal{H}$-bramble} in $G$ is a collection $\mathcal{B}$ of connected subgraphs of $G$ such that
\begin{enumerate}
    \item $V(B')\cap V(B)\neq \emptyset,$ and
    \item $\mathcal{B} \cap \mathcal{H} = \emptyset.$
\end{enumerate}
A \defi{hitting set} for $\mathcal{B}$ is a set $S\subseteq V(G)$ such that $V(B)\cap S\neq\emptyset$ for all $B\in\mathcal{B}.$
The \defi{order} of $\mathcal{B}$ is then defined as the minimum size of a hitting set for $\mathcal{B}.$
Finally, we define the \defi{strict $\mathcal{H}$-bramble number} of $G,$ denoted by $\mathcal{H}\text{-}\mathsf{sbn}(G),$ to be the maximum order of a strict $\mathcal{H}$-bramble in $G.$

\medskip
Note that the notion of strict brambles is a special case of the classic and more general notion of brambles that provides an exact min-max duality in relation to the treewidth of a graph (see \cite{SeymourT93Graph}).
The relation between strict brambles and classic brambles, as well as precise dualities for strict brambles, have also been studied in the literature (see \cite{KozawaOY14lower, Reed97Tree, AidunDMYY20Treewidth, LardasPTZ23onstr}).

\begin{lemma}\llabel{lemma_bramble_lowerbound}
Let $\mathcal{H}$ be a proper minor-closed graph class.
Then for every graph $G$ it holds that $\mathcal{H}\text{-}\mathsf{sbn}(G)\leq\mathcal{H}\text{-}\mathsf{tw}(G)+1.$
\end{lemma}

\begin{proof}
Let $k\coloneqq \mathcal{H}\text{-}\mathsf{tw}(G)+1$ and suppose towards a contradiction that $G$ contains a strict $\mathcal{H}$-bramble $\mathcal{B}$ of order at least $k+1.$
Notice that for every set $S\subseteq V(G)$ of size at most $k$ there most exist at least one member $B$ of $\mathcal{B}$ such that $S\cap V(B)=\emptyset.$
Moreover, there exists a unique component $K_S$ of $G-S$ such that for every $B\in\mathcal{B}$ that avoids $S$ we have that $B\subseteq K$ since the union of any two members of $\mathcal{B}$ is connected.

Now, notice that for every edge $t_1t_2\in E(T)$ there exists a unique $i\in [2]$ such that the union of the bags in the unique subtree $T_{t_i}$ of $T$ which is rooted at $t_i$ and avoids $t_{3-i}$ contains the vertex set of the component $K_{\beta(t_1)\cap\beta(t_2)}$ as defined above.
This is simply because $|\beta(t_1)\cap\beta(t_2)|\leq k$ while $\mathcal{B}$ is of order at least $k+1.$
Let us call $t_i$ the \defi{dominant} endpoint of $t_1t_2.$
This observation induces an orientation $\vec{T}$ of $T$ where each edge is oriented towards its dominant endpoint.

Suppose there exists a node $t\in V(T)$ such that $t$ is not the dominant endpoint of at least two of its incident edges, say $td_1$ and $td_2.$
Since the unions of the bags of the two trees $T_{d_i}$ must intersect in a subset of $\beta(t)$ it follows that for each $i\in[2],$ the bags in the subtree $T_{d_i}$ must contain the vertex set of some element $B_i$ that avoids $\beta(t).$
Since $V(B_1)\cap V(B_2)\neq\emptyset$ this is impossible and thus for each node of $T$ there exists at most one incident edge for which the node is not dominant.

Thus, there must exist some node $t$ of $T$ which is not dominant for any of its incident edges, i.e., $t$ is a sink of $\vec{T}.$
In this case, none of the subtrees rooted at a neighbor of $t$ can contain the vertex set of a member of $\mathcal{B}$ which avoids $\beta(t).$
Hence, $\beta(t)$ is a hitting set for $\mathcal{B}$ which contradicts our assumptions.
\end{proof}

Similar to strict $\mathcal{H}$-brambles we may define a more local variant which also poses as an obstruction to $\mathcal{H}\text{-}\mathsf{tw}.$

\paragraph{Highly linked sets.}
Let $\Hcal$ be a proper minor-closed class. 
Let also $G$ be a graph, $X\subseteq V(G)$ be a set of vertices, and $k\in \Nbbb.$
We  say that the set $X$ is \defi{$k$-$\Hcal$-linked in $G$} if for every $S\subseteq V(G),$ where $|S|\leq k,$ there exists a component $J_S$ of $G-S$ such that $J_S\not\in\Hcal$ and $|V(J)\cap X|>\frac{2}{3}|X|.$

As before, we associate a graph parameter with the notion of $\mathcal{H}$-linked sets.
For every graph $G$ let $\mathcal{H}$-$\mathsf{link}(G)$ denote the maximum $k$ such that there exists a $k$-$\mathcal{H}$-linked set in $G.$

\begin{lemma}\llabel{lemma_Hlink_leq_Hbramble}
Let $\mathcal{H}$ be a proper minor-closed graph class.
For every graph $G$ it holds that $2\cdot\mathcal{H}\text{-}\mathsf{link}(G)\leq\mathcal{H}\text{-}\mathsf{sbn}(G).$
\end{lemma}

\begin{proof}
Let $G$ be a graph and $X$ be a $2k$-$\mathcal{H}$-linked set on $G$ for some positive integer $k.$
We define the following collection of subgraphs:
\begin{align*}
    \mathcal{B}_X\coloneqq \{ J_S \mid S\subseteq V(G)\text{ and }|S|\leq k \}.
\end{align*}
Moreover, we claim that $\mathcal{B}_X$ is a strict $\mathcal{H}$-bramble of order $k$ in $G.$
To see this we simply have to check the axioms of strict $\mathcal{H}$-brambles.

First notice that, by definition of $\mathcal{H}$-linked sets, $B\notin \mathcal{H}$ for all $B\in\mathcal{B}_X.$
Moreover, every such $B$ is connected.
This takes care of the second axiom.

For the first axiom let $J_{S_1},J_{S_2}\in\mathcal{B}_X$ be two distinct members of our collection.
Notice that $|S_1\cup S_2|\leq 2k.$
Thus, there also exists a component $J_{S_1\cup S_2}$ of $G-S_1-S_2$ such that $J_{S_1\cup S_2}\notin\mathcal{H}$ and $|V(J_{S_1\cup S_2} \cap X)|>\frac{2}{3}|X|.$
It follows that $V(J_{S_1\cup S_2}) \cap V(J_{S_i})\neq\emptyset$ for both $i\in[2].$
In fact, this means $J_{S_1\cup S_2}\subseteq J_{S_i}$ for both $i\in[2]$ and therefore $V(J_{S_1})\cap V(J_{S_2})\neq\emptyset$ as desired.

Finally, to see that $\mathcal{B}_X$ is of order $k$ as claimed simply notice that, by definition, for every set $S\subseteq V(G)$ of size at most $k$ there exists an element $J_S\in\mathcal{B}_X$ which fully avoids $S.$
\end{proof}

Finally, we complete the circle by showing that the absence of $k$-$\mathcal{H}$-linked sets implies that $\mathcal{H}\text{-}\mathsf{tw}$ is bounded.

\begin{lemma}\llabel{lemma_Htw_lew_Hlink}
    Let $\mathcal{H}$ be a proper minor-closed graph class.
    For every graph $G$ it holds that $\mathcal{H}\text{-}\mathsf{tw}(G)\leq 4\cdot \mathcal{H}\text{-}\mathsf{link}(G)+4.$
\end{lemma}

\begin{proof}
    Let $k\coloneqq \mathcal{H}\text{-}\mathsf{link}(G)+1,$ then $G$ does not contain a $k$-$\mathcal{H}$-linked set.

    We prove the following claim by induction on $|G|-|X|.$
    The assertion will then follow immediately.

    \paragraph{Inductive Claim.}
    Let $G$ be a graph with no $k$-$\mathcal{H}$-linked set.
    For every set $X\subseteq V(G)$ with $|X|\leq 3k+1$ there exist an $\mathcal{H}$-tree decomposition $(T,\beta)$ of width at most $4k$ and a leaf $d\in V(T)$ such that $X\subseteq \beta(d)$ and $|\beta(d)|\leq 4k.$
    \medskip

    Let $X\subseteq V(G)$ be any set of size at most $3k+1.$
    In case $|X|<3k+1$ we either have that $|G|<3k+1$ in which case we are immediately done, or may pick an arbitrary vertex $v\in V(G)\setminus X$ and apply the induction hypothesis to $X\cup \{ v\}.$

    Hence, we may assume $|X|=3k+1.$
    In the following, we define the desired $\mathcal{H}$-tree decomposition step by step.
    First let us introduce a node $d_0$ and set $\beta(d_0)\coloneqq X.$
    Our goal is to make $d_0$ into a leaf of the final decomposition.
    To meet the requirements of $\mathcal{H}$-tree decompositions this means we should introduce a node $d_1$ adjacent to $d_0$ and also set $\beta(d_1)\coloneqq X.$
    This way we know that $G[\beta(d_0)\setminus (\beta(d_0)\cap\beta(d_1))]\in\mathcal{H}.$
    For each component $G'$ of $G-X$ we now introduce a new vertex $t_{G'}$ adjacent to $d_1.$
    In the following, we describe how to extend this construction using our induction hypothesis.

    Towards this goal, let us fix $G'$ to be the subgraph of $G$ induced by a single component of $G-X$ together with $X.$
    Now, since $X$ is not $k$-$\mathcal{H}$-linked there exists a set $S\subseteq V(G')$ of size at most $k$ such that one of the following is true.
    \begin{description}
        \item[Case 1] There exists a component $J_S$ of $G'-S$ such that $|V(J_S)\cap X|>\frac{2}{3}|X|$ and $J_S\in \mathcal{H},$ or
        \item[Case 2] for every component $J$ of $G'-S$ it holds that $|V(J)\cap X|\leq \frac{2}{3}|X|\leq 2k+\frac{1}{3}\leq 2k.$
    \end{description}
    
    In either case we will set $\beta(t_{G'-X})\coloneqq X\cup S.$
    Notice that this means that $|\beta(t_{G'-X})|\leq 4k+1$ which still fits our bounds.

    \paragraph{Case 1.}
    For each component $J$ of $G'-S$ introduce a vertex $t_J$ adjacent to $t_{G'-X}$ and set $\beta(t_J)\coloneqq (V(J)\cap X)\cup S.$
    Notice that for the component $J_S$ as defined above this means $|\beta(t_{J_S})|\leq 4k+1$ and for each other component $J$ of $G'-S$ we have that $|V(J)\cap X|<\frac{1}{3}|X|\leq k$ implying $|\beta(t_J)|\leq 2k.$
    We then introduce a new node $d_{J_S}$ adjacent to $t_{J_S}$ and set $\beta(d_{J_S})\coloneqq V(J_S)\cup S.$
    By definition of $J_S$ we have that $G[\beta(d_{J_S})\setminus \beta(t_{J_S})]\in\mathcal{H}$ and thus we are fine leaving $d_{J_S}$ as a leaf of our decomposition.
    For every other component $J$ of $G'-S$ we may now apply the induction hypothesis with $X_J\coloneqq S\cup (X\cap V(J))$ to obtain an $\mathcal{H}$-tree decomposition $(T_J,\beta_J)$ of $G[V(J)\cup S]$ of width at most $4k$ such that there exists a leaf $t_{J,S}\in V(T_J)$ with $X_J\subseteq \beta(t_{J,S})$ and $|\beta(t_{J,S})|\leq 4k.$
    We extend out partially defined decomposition with the tree $T_j$ and the bags $\beta_J$ and introduce the edge $t_{J,S}t_J.$
    With this, we are done with the first case.

    \paragraph{Case 2.}
    For each component $J$ of $G'-S$ introduce a vertex $t_J$ adjacent to $t_{G'-X}$ and set $\beta(t_J)\coloneqq (V(J)\cap X)\cup S.$
    This time we may observe that $|\beta(t_J)|\leq k + \frac{2}{3}|X|\leq 3k.$
    For each $J,$ if there exists a vertex $v\in J-X$ we introduce a vertex $t'_J$ adjacent to $t_J,$ pick such a vertex $v$ arbitrarily and set $\beta(t'_J)\coloneqq \beta(t_J)\cup \{ v\}.$
    If there does not exist such a vertex $v$ we still introduce the node $t'_J$ but in this case we have that $V(J)\subseteq X$ and thus we may simply set $\beta(t'_J)\coloneqq \beta(t_J)$ and declare it a leaf.
    In this case, it follows that $G[\beta(t'_J)\setminus\beta(t_J)]\in\mathcal{H}$ and thus $t'_J$ may remain a leaf for the rest of the construction.

    Hence, from now on we may assume that we were able to choose the vertex $v.$
    By calling the induction hypothesis for $G[S\cup V(J)]$ and $X_J\coloneqq \beta(t'_J)$ and then merging our partial decomposition with the resulting decomposition for $G[S\cup V(J)]$ in the same way as we did in \textbf{Case 1}, we also complete this second case and thereby the proof by using \autoref{lemma_Htw_via_decompositions}.
\end{proof}

To conclude, we summarize the results of \autoref{lemma_bramble_lowerbound}, \autoref{lemma_Hlink_leq_Hbramble}, and \autoref{lemma_Htw_lew_Hlink}.

\begin{corollary}\llabel{cor_Htw_equivalences}
Let $\mathcal{H}$ be a proper minor-closed graph class.
Then, for every graph $G$ we have
\begin{align*}
    2\cdot \mathcal{H}\text{-}\mathsf{link}(G)-1 \leq \mathcal{H}\text{-}\mathsf{sbn}(G)-1 \leq \mathcal{H}\text{-}\mathsf{tw}(H) \leq 4\cdot \mathcal{H}\text{-}\mathsf{link}(G)+4.
\end{align*}
\end{corollary}
The above corollary implies that $\mathcal{H}\text{-}\mathsf{link},$ $\mathcal{H}\text{-}\mathsf{sbn}$ and $\mathcal{H}\text{-}\mathsf{link}$ are equivalent with linear gap function.

\subsection{The lower bound}\llabel{subsec_lowerbound_H_tw}
With the lemmas from the two previous subsections, we are now ready to obtain the first link between our walloids and the notion of $\mathcal{H}$-treewidth.
Indeed, we can prove an even stronger statement.
For this purpose, we need a special version of brambles.

Let $\mathcal{H}$ be a proper minor-closed class.
A strict $\mathcal{H}$-bramble $\mathcal{B}$ in a graph $G$ is said to be a \defi{$\nicefrac{1}{2}$-$\mathcal{H}$-bramble} if no vertex of $G$ appears in more than two members of $\mathcal{B}.$
Notice that, if $\mathcal{B}$ is a $\frac{1}{2}$-$\mathcal{H}$-bramble of order $k$ which is not of order $k+1,$ then $|\mathcal{B}|\leq 2k.$

\begin{lemma}\llabel{lemma_half_integral_Brambles}
There exists a function $f_{\ref{lemma_half_integral_Brambles}}\colon\mathbb{N}^2\to\mathbb{N}$ such that for every graph $Z,$ every embedding pair $\mathbf{p}\in \mathfrak{P}^{(Z)}$ and every positive integer $k,$ $\mathscr{W}_{f_{\ref{lemma_half_integral_Brambles}}(|Z|,k)}^{\mathbf{p}}$ contains a $\nicefrac{1}{2}$-$\excl(Z)$-bramble of order $k.$

Moreover, $f_{\ref{lemma_half_integral_Brambles}}(|Z|,k)\in \mathcal{O}_{|Z|}(k).$
\end{lemma}

\begin{figure}[h]
    \centering
    \scalebox{1.1}{
    \begin{tikzpicture}[scale=1]

        \pgfdeclarelayer{background}
        \pgfdeclarelayer{foreground}
            
        \pgfsetlayers{background,main,foreground}
            
        \begin{pgfonlayer}{main}
        \node (C) [v:ghost] {};

            \pgftext{\includegraphics[width=12cm]{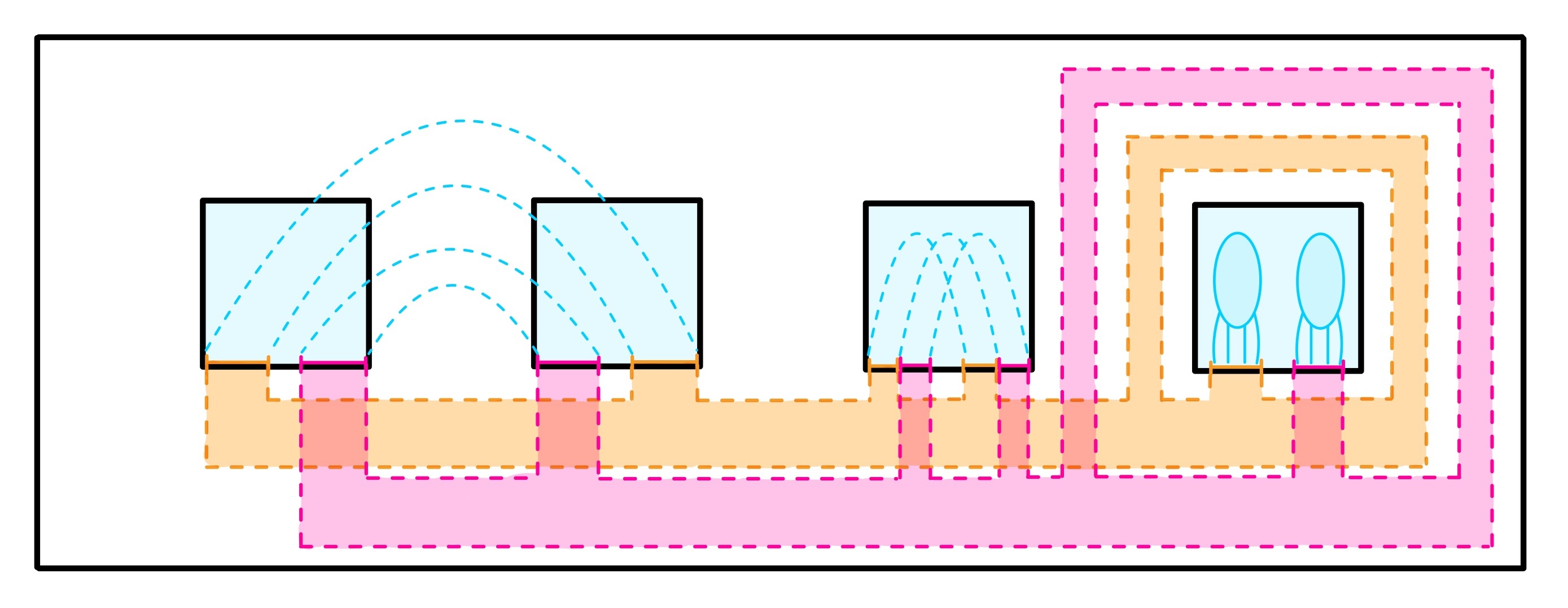}} at (C.center);

        \end{pgfonlayer}{main}
        
        \begin{pgfonlayer}{foreground}
        \end{pgfonlayer}{foreground}

        \begin{pgfonlayer}{background}
        \end{pgfonlayer}{background}
        
    \end{tikzpicture}
    }
    \caption{A sketch for constructing a $\nicefrac{1}{2}$-$\mathsf{excl}(Z)$-bramble in a $\mathbf{p}$-design of large order for some $\mathbf{p}\in\mathfrak{P}^{(Z)}.$}
    \llabel{fig_bramble_from_design}
\end{figure}

\begin{proof}
By \autoref{walloids_contain_designs} we know that for every $t\in\mathbb{N}$ we have that $\mathscr{D}_t^{\mathbf{p}}$ is a minor of $\mathscr{W}^{\mathbf{p}}_{5t}.$
Moreover, $\mathscr{D}^{\mathbf{p}}_{f_{\ref{designs_contain_graphs}}(|Z|)}$ contains $Z$ as a minor by \autoref{designs_contain_graphs}.

Now consider the design $\mathscr{D}^{\mathbf{p}}_{2k(f_{\ref{designs_contain_graphs}}(|Z|)+1)}.$
Similar to how walloids contain half-integral packings of themselves (recall \autoref{fig_half_integral_packing}), we may construct such a half-integral packing of extensions of $Z$ within $\mathscr{D}^{\mathbf{p}}_{2k(f_{\ref{designs_contain_graphs}}(|Z|)+1)}.$
Indeed, by the choice of the order we can fit a half-integral packing of $2k$ copies of extensions of $Z$ into $\mathscr{D}^{\mathbf{p}}_{2k(f_{\ref{designs_contain_graphs}}(|Z|)+1)}$ such that any two of these extensions intersect.
See \autoref{fig_bramble_from_design} for an illustration.
Let $\mathcal{B}$ be the collection of these extensions.
Since the packing is half-integral and $|\mathcal{B}|=2k,$ the order of the resulting bramble is at lest $k.$

The explicit construction of $\mathcal{B}$ is tedious, but it suffices to see that, given a fixed extension of $Z$ embedded in $\mathscr{D}^{\mathbf{p}}_{f_{\ref{designs_contain_graphs}}(|Z|)}$ we can take this extension and ``stretch'' and ``shift'' it slightly as depicted in \autoref{fig_bramble_from_design} to achieve the half-integral packing.
\end{proof}
 
As an immediate consequence of \autoref{lemma_half_integral_Brambles} and \autoref{lemma_bramble_lowerbound}, we obtain that for every proper minor-closed graph class $\mathcal{H}$ the parametric graphs in the family $\mathfrak{W}_{\mathcal{H}}$ act as witnesses for large $\mathcal{H}$-treewidth when found as a minor.
In particular, this immediately implies \autoref{thm_H_tw_lowerbound}.

\begin{corollary}\llabel{cor_walloid_lower_bound}
For every proper minor-closed graph class $\mathcal{H}$ there exists a function $f^{\mathcal{H}}_{\ref{cor_walloid_lower_bound}}\colon\mathbb{N}\to\mathbb{N}$ such that for every positive integer $k$ and every graph $G,$ if there exists $\mathscr{W}=\langle W_k \mid k\in\mathbb{N}\rangle\in\mathfrak{W}_{\mathcal{H}}$ such that $G$ contains $\mathscr{W}_{f^{\mathcal{H}}_{\ref{cor_walloid_lower_bound}}(k)}$ as a minor, then $\mathcal{H}\text{-}\mathsf{tw}(G)\geq k.$
\end{corollary}

In particular, this means that every strict $\mathcal{H}$-bramble of large enough order implies the existence of a $\frac{1}{2}$-$\mathcal{H}$-bramble of large order.

\begin{corollary}\llabel{cor_bramble_to_half_int_bramble}
For every proper minor-closed graph class $\mathcal{H}$ there exists a function $f^{\mathcal{H}}_{\ref{cor_bramble_to_half_int_bramble}}\colon\mathbb{N}\to\mathbb{N}$ such that for every positive integer $k$ and every graph $G,$ if $G$ contains a strict $\mathcal{H}$-bramble of order $f_{\ref{cor_bramble_to_half_int_bramble}}(k),$ then $G$ contains $\frac{1}{2}$-$\mathcal{H}$-bramble of order $k.$
\end{corollary}
 
\subsection{The upper bound}\llabel{subsec_upperbound_H_tw}

We are now ready to prove \autoref{main_grid_general}.
To do this we make use of \autoref{thm_local_structure} to show that any graph which contains a highly $\mathcal{H}$-linked set must also contain a member of $\mathfrak{W}_{\mathcal{H}}$ of large order.

In order to prove this statement we will need some additional notions from the theory of graph minors.

\paragraph{Of walls and well-linked sets.}

Let $G$ be a graph and $X\subseteq V(G)$ be a vertex set.
A set $S\subseteq V(G)$ is said to be a \defi{balanced separator} for $X$ if for every component $J$ of $G-S$ it holds that $|V(J)\cap X|\leq \frac{2}{3}|X|.$
We say that $X$ is \defi{$k$-well-linked} if there does not exist a balanced separator of size at most $k$ for $X.$

It is easy to see that well-linked sets form a special case of our $\mathcal{H}$-linked sets.
Indeed, they are a well-studied object acting as obstructions to treewidth.
Well-linked sets and walls both act as such obstructions and they relate to each other in a specific way which is usually encoded in the language of ``tangles''.
Since we do not have any need for the notion of tangles in this paper we provide a slightly different definition here which can easily be seen to be equivalent by a reader familiar with these concepts.
See Section 3 of \cite{thilikos2023excluding} for an in-depth discussion.

Let $k$ and $r$ be positive integers such that $k\geq r.$
Let $G$ be a graph $X$ a $k$-well-linked set in $G$ and $W$ an $r$-wall in $G.$

First notice that for every set $S\subseteq V(G)$ of order at most $r-1$ there exists a unique component $J_{W,S}$ of $G-S$ which contains an entire row and an entire column of $W.$
Similarly, for every $S\subseteq V(G)$ of order at most $k$ there exists a unique component $J_{X,S}$ such that $|V(J_{X,S})\cap X|>\frac{2}{3}|X|.$
We say that $X$ and $W$ \defi{agree} if for every $S\subseteq V(G)$ of order at most $r-1$ we have that $J_{W,S}\subseteq J_{X,S}.$

\begin{theorem}[\cite{thilikos2023excluding}]
    \llabel{thm_well_linked_to_wall}
Let $k\geq 3$ be an integer.
There exists a universal positive constant $c\in\mathbb{N}$ and an algorithm that, given a graph $G$ and a $(36ck^{20}+3)$-well-linked set $X\subseteq V(G),$ computes in time $\mathcal{O}_k(|G|^2|E(G)|\log(|G|))$ a $k$-wall $W$ such that $X$ and $W$ agree.
\end{theorem} 

With this, we may now prove the main tool towards proving \autoref{main_grid_general}.

\begin{lemma}\llabel{lemma_no_Hlinked_sets}
Let $\mathcal{H}$ be a proper minor-closed graph class and let $h\coloneqq h(\mathcal{H})$ the largest size of an obstruction of $\mathcal{H}.$
There exists a function $f_{\ref{lemma_no_Hlinked_sets}}\colon\mathbb{N}^2\to\mathbb{N}$ depending only on $h$ such that, for every graph $G$ and every positive integer $k$ one of the following statements holds
\begin{enumerate}
    \item there exists $\mathscr{W}=\langle \mathscr{W}_t \mid t\in\mathbb{N}\rangle\in\mathfrak{W}_{\mathcal{H}}$ such that $G$ contains a subdivision of $\mathscr{W}_{k},$ or
    \item $\mathcal{H}\text{-}\mathsf{link}(G)\leq f_{\ref{lemma_no_Hlinked_sets}}(h,k).$
\end{enumerate}
\end{lemma}

\begin{proof}
Let $c$ be the constant from \autoref{thm_well_linked_to_wall}.
We set $f_{\ref{lemma_no_Hlinked_sets}}(h,k)\coloneqq 36c\cdot f^1_{\ref{thm_local_structure}}(h,k)^{20}+3.$

Now suppose $G$ contains an $f_{\ref{lemma_no_Hlinked_sets}}(h,k)$-$\mathcal{H}$-linked set $X.$
Then, in particular, for every set $S\subseteq V(G)$ of order at most $f_{\ref{lemma_no_Hlinked_sets}}(h,k)$ there exists a unique component $J_{X,S}$ such that $|V(J_{X,S})\cap X|>\frac{2}{3}|X|.$
Hence, $X$ is also $f_{\ref{lemma_no_Hlinked_sets}}(h,k)$-well-linked.

It follows from \autoref{thm_well_linked_to_wall} that there exists a $f^1_{\ref{thm_local_structure}}(h,k)$-wall $W$ which agrees with $X.$
Now, by \autoref{thm_local_structure} one of to things hold.
Either 
\begin{enumerate}
    \item $G$ contains $K_{3k^2+kh}$ as a minor, or 
    \item there exists a set $A\subseteq V(G)$ of size at most $f^2_{\ref{thm_local_structure}}(h,k)<f^1_{\ref{thm_local_structure}}(h,k)<f_{\ref{lemma_no_Hlinked_sets}}(h,k)$ and a set $\mathcal{Z}$ of graphs of size at most $h$ such that the component $J_{X,A}$ is $Z$-minor-free for all graphs on at most $h$ vertices which do not belong to $\mathcal{Z}.$
\end{enumerate}
In the first case, $G$ must contain, as a minor, the walloid of order $k$ which has $k$ copies of $K_h$ in a single flower.
Clearly, this walloid contains a member of $\mathfrak{W}_{\mathcal{H}}$ of order $k$ as desired.

Hence, we may assume we are in the second outcome.
Since $J_{X,A}\notin\mathcal{H}$ and every obstruction of $\mathcal{H}$ has size at most $h,$ there must exist some graph $H\in\mathsf{obs}(\mathcal{H})\cap \mathcal{Z}.$
Therefore, the second outcome of \autoref{thm_local_structure} holds for $H$ which means that there exists $\mathscr{W}^H=\langle \mathscr{W}^H_t \mid t\in\mathbb{N}\rangle\in\mathfrak{W}_{\mathsf{excl}(H)}$ such that $G$ contains a subdivision $W_H$ of $\mathscr{W}^H_k.$
Moreover, the base cylinder of $W_H$ is a subgraph of $W.$
Since $\mathfrak{W}_{\mathsf{excl}(H)}\subseteq \mathfrak{W}_{\mathcal{H}}$ this is the first outcome of our claim.
Hence, if the first outcome does not hold it must be true that $\mathcal{H}\text{-}\mathsf{link}(G)\leq f_{\ref{lemma_no_Hlinked_sets}}(h,k).$
\end{proof}

Please notice that, with standard techniques for finding balanced separators, \autoref{thm_well_linked_to_wall}, and \autoref{thm_local_structure} the procedure described in the proof of \autoref{lemma_Htw_lew_Hlink} yields an algorithm that either finds one of the obstructions, and therefore a large half-integral $\mathcal{H}$-barrier, or an $\mathcal{H}$-tree decomposition of bounded width in time $\mathcal{O}_{h,k}(|G|^4\log|G|).$
Based on this decomposition we can then proceed as discussed in the proof of \autoref{obs_con_apex} to either find a bounded size modulator to $\mathcal{H}$ or a large integral $\mathcal{H}$-barrier.
The purpose of \autoref{univ_apex_all_small} is to make this more explicit.

We are now ready to give a proof of \autoref{main_grid_general}.

\begin{proof}[Proof of \autoref{main_grid_general}]
Let $f_{\ref{main_grid_general}}(h,k)\coloneqq 4\cdot f_{\ref{lemma_no_Hlinked_sets}}(h,k)+8.$
By \autoref{cor_Htw_equivalences} it follows that, if $\mathcal{H}\text{-}\mathsf{tw}(G)\geq f_{\ref{main_grid_general}}(h,k),$ then $\mathcal{H}\text{-}\mathsf{link}(G)\geq f_{\ref{lemma_no_Hlinked_sets}}(h,k)+1.$
This means that the second outcome of \autoref{lemma_no_Hlinked_sets} is impossible and thus, the first outcome must hold as desired. 
\end{proof}

\section{Universal obstructions for $\apex_{\Hcal}$ and  $\Hcal$-\td}
\llabel{univ_apex_all_small}

In this section, we give a proof that $\mathfrak{F}_{\Hcal}^\mathsf{min}$ is a universal obstruction for $\apex_{\Hcal}.$
For this, we proceed in two steps. First, in \autoref{univ_apex_con_discon}, we prove this result 
for the case where all obstructions of $\Hcal$ 
are connected, which is equivalent to the property 
that $\Hcal$ is closed under taking disjoint subgraphs.
Then, using this special case as a starting point, we prove the general result in \autoref{disc_case_obs}.
In \autoref{univ_td_all} we apply similar ideas as those of \autoref{disc_case_obs} in order to find an obstruction set for the parameter of \textsl{elimination distance} to $\Hcal,$ defined by 
Bulian and Dawar in \cite{BulianD16grap}. 
\medskip

The parameter  $\apex_{\Hcal}$ has been extensively studied 
for different instantiations of $\Hcal.$ Given a proper minor-closed 
class $\Hcal$ we define, for every $k\in\Nbbb,$  
\begin{eqnarray}
\Acal\Pcal_{k} &  = & \{G\mid \apex_{\Hcal}(G)≤k\}
\llabel{apex_obs_small}
\end{eqnarray}
In \cite{AdlerGK08comp} it was shown that for every $\Hcal$ 
there is a computable
function $f^{\sf ap}_{\Hcal} \colon \Nbbb\to\Nbbb$
such that for every $k\in\Nbbb,$ if $Z\in\obs(\Acal\Pcal_{k}),$ then $|Z|≤f^{\sf ap}_{\Hcal}(k)$ and this implied that the problem of checking whether $\apex_{\Hcal}(G)≤k$ is constructively in {\sf FPT} when parameterized by $k$ (see also \cite{FellowsL94onse}).
Recently, in \cite{SauST22apicesalg} an explicit bound of  $f^{\sf ap}_{\Hcal}(k)=2^{2^{2^{2^{\poly_{h}(k)}}}}$ was given.
Also, in \cite{SauST21kapiII}, an algorithm checking whether $\apex_{\mathcal{H}}(G)≤k$ was given, running in time $2^{\poly_{h}(k)}|G|^{2}.$

\subsection{Global apex pairs}
\llabel{univ_apex_con_discon}

Given a graph $Z$ we define the \defi{connectivization} of $Z,$ denoted by $\mathsf{conn}(Z),$ as the set of all edge-minimal connected graphs that contains $Z$ as a subgraph. 
Notice that every graph $\mathsf{conn}(Z)$ is the graph $Z$ plus a minimal number of edges that make it connected. Therefore, $|\mathsf{conn}(Z)|=2^{\Ocal(|Z|\cdot \log |Z|)}.$
Given a finite set $\Zcal$ of graphs, we define its \defi{minor-connectivization} as the set $\mathsf{conn}(\Zcal)=\mathsf{min}(\{\mathsf{conn}(Z)\mid Z\in \Zcal\})$ where, given some  set of graphs  $\Acal,$ we use $\mathsf{min}(\Acal)$ for the set of minor-minimal graphs in $\Acal.$
\medskip

Given a proper minor-closed class $\Hcal$ we define its \defi{connectivity closure} to be
\begin{align*}
\Hcal^{(c)}\coloneqq \{G\mid \text{every connected component of } G \text{ belongs to }\Hcal\}.
\end{align*}
The obstructions for $\Hcal^{(c)}$ were described by Bulian and Dawar \cite{BulianD17fixe} as follows.

\begin{proposition}\llabel{b_d_th_conn}
For every proper minor-closed class $\Hcal$ it holds that $\obs(\Hcal^{(c)})=\mathsf{conn}(\obs(\Hcal)).$
\end{proposition}

We define the parameter $\size\colon\gall\to\Nbbb$ such that $\size(G)=|G|.$
Based on this, by replacing $\tw$ with size $\size$ in \eqref{basic_scheme} we define the following parameter:
\begin{align*}
\Hcal\text{-}\size(G) & \coloneqq \min\{k\mid \text{there exists } X\subseteq V(G) \text{ s.\@ t.\@~} |X|\leq k\text{~and~} \text{$G-X\in \Hcal^{(c)}$}\}
\end{align*} 

Observe that $\Hcal\text{-}\size=\apex_{\Hcal^{(c)}}$ if and only if all obstructions of $\Hcal$ are connected.
However, in general $\Hcal\text{-}\size\preceq\apex_{\Hcal}.$ 
Let 
\begin{eqnarray}
\mathfrak{S}_{\Hcal} & \coloneqq  &\mathfrak{W}_{\Hcal}\cup\mathfrak{W}_{\Hcal^{(c)}}^{\varnothing}\llabel{size_obs_one}\\
\mathfrak{S}_{\Hcal}^\mathsf{min} & \coloneqq  & \mathsf{min}(\mathfrak{S}_{\Hcal} ).\nonumber
\end{eqnarray}
Next, we prove that $\mathfrak{S}_{\Hcal}$ is an obstructing set for $\Hcal\text{-}\size.$
We first prove the lower bound.

\begin{observation}
\llabel{easy_obs_gap_}
For every minor-closed class $\Hcal,$ it holds that $\p_{\mathfrak{S}_{\Hcal}}\preceq \Hcal\text{-}\size$ with linear gap function.
\end{observation}

\begin{proof}
Recall that $\p_{\mathfrak{S}_{\Hcal}}=\max\{\p_{\mathfrak{W}_{\Hcal}},\p_{\mathfrak{W}^{\varnothing}_{\Hcal^{(c)}}}\}.$
The inequality $\p_{\mathfrak{W}_{\Hcal^{(c)}}^{\varnothing}}\preceq \Hcal\text{-}\size,$ with linear gap function, follows from \autoref{b_d_th_conn} and the fact that for every $Z\in\mathsf{conn}(\obs(\Hcal)),$ $\Hcal\text{-}\size(k\cdot Z)=k.$
The other inequality $\p_{\mathfrak{W}_{\Hcal}}\preceq \Hcal\text{-}\tw,$ with linear gap function, holds because of \autoref{cor_walloid_lower_bound} and $\Hcal\text{-}\tw\leq \Hcal\text{-}\size.$
\end{proof}

Now we prove the upper bound, i.e., the direction $\Hcal\text{-}\size\preceq \p_{\mathfrak{S}_{\Hcal}}.$

\begin{lemma}
\llabel{obs_con_apex}
There exists a function $f_{\ref{obs_con_apex}}\colon\nn{2}{1}$ such that for proper every minor-closed class $\Hcal,$ where $h(\Hcal)=h,$ every graph $G,$ and every $k\in\Nbbb,$ either $\p_{\mathfrak{S}_{\Hcal}}(G)\geq k$ or $\Hcal\text{-}\size(G)\leq f_{\ref{obs_con_apex}}(h,k),$ where $f_{\ref{obs_con_apex}}(h,k)={2^{k^{\mathcal{O}_h(1)}}}.$
Moreover, there exists an algorithm that, given $\obs(\Hcal),$ $G,$ and $k,$ outputs a certificate of one of the two 
outcomes in time $2^{2^{k^{\mathcal{O}_h(1)}}}|G|^4\log |G|.$ 
\end{lemma}

\begin{proof}
Let $G$ be a graph where $\p_{\mathfrak{S}_{\Hcal}}(G)\leq k.$ 
This implies that $\p_{\mathfrak{W}_{\Hcal}}(G)\leq \p_{\mathfrak{S}_{\Hcal}}(G)\leq k$ and $\p_{\mathfrak{W}_{\Hcal^{(c)}}^{\varnothing}}(G)\leq \p_{\mathfrak{S}_{\Hcal}}(G)\leq k.$

As $\p_{\mathfrak{W}_{\Hcal}}(G)\leq k,$ due to \autoref{main_grid_general}, we have that $\Hcal$-$\tw(G)\leq f_{\ref{main_grid_general}}(h,k).$
Let $q\coloneqq f_{\ref{main_grid_general}}(h,k)$ and let $(T,\beta)$ be an $\mathcal{H}$-tree decomposition
of $G$ of width at most $q.$

Let $\mathsf{mp}(G)$ be the maximum size of a collection $\Ccal$ of pairwise disjoint subgraphs of $G$ such that for every $C\in\Ccal$ there is some graph in $\mathsf{conn}(\obs(\Hcal))$ which is a minor of $C.$
Observe that $\mathsf{mp}(G)\leq |\mathsf{conn}(\obs(\Hcal))|\cdot k,$ otherwise $G$ would contain a $\Hcal^{(c)}$-barrier of size larger than $k$ which would contradict the fact that $\p_{\mathfrak{W}_{\Hcal^{(c)}}^{\varnothing}}(G)\leq k.$

Let $p$ be some positive integer.
We claim that if $\mathsf{mp}(G)\in [p]$ then $\Hcal\text{-}\size(G)\leq |S| \leq (q+1)\cdot (2p-1).$
We prove this by induction on $p.$
Given an edge $tt'$ of $T$ where $T_{1}$ (resp. $T_{2}$) is the connected component of $T-e$ containing $t$ (resp. $t'$), we define 
\begin{eqnarray}
X_{t}\coloneqq \bigcup_{b\in V(T_{1})\setminus\{t\}}\beta(b) & \text{ and } & X_{t'}\coloneqq\bigcup_{b\in V(T_{2})\setminus\{t'\}}\beta(b)\llabel{all_use_sep}
\end{eqnarray}
If, for some edge $tt'\in E(T),$ it holds that $\mathsf{mp}(G[X_{t}\setminus X_{t'}])=p_{1} \geq 1$ and $\mathsf{mp}(G[X_{t'}\setminus X_{t}])=p_{2}\geq 1,$ then we have that $p_{1}+p_{1}\leq p$
and, from the induction hypothesis it follows that $\Hcal\text{-}\size(G)\leq |X_{t}\cap X_{t'}|+\mathsf{mp}(G[X_{t}\setminus X_{t'}])+\mathsf{mp}(G[X_{t'}\setminus X_{t}])\leq (q+1)+(q+1)\cdot (2p_1-1)+(q+1)\cdot (2p_2-1)\leq (q+1)\cdot (2p-1).$

If there exists some edge $tt'\in E(G)$ such that $\mathsf{mp}(G[X_{t}\setminus X_{t'}])=p_{1}=0$  and $\mathsf{mp}(G[X_{t'}\setminus X_{t}])=p_{2}=0,$ we obtain that every component of $G-(X_t\cap X_{t'})$ belongs to $\mathcal{H}$ which would complete our proof.

The only remaining case is where for every edge $tt'$ of $T,$ either $\mathsf{mp}(G[X_{t}\setminus X_{t'}])=p_{1}=0$ or $\mathsf{mp}(G[X_{t'}\setminus X_{t}])=p_{2}=0,$ but not both.

In this case, we orient each $tt'$ towards the component where the value of $\mathsf{mp}$ is positive.
This means that there is a node $t$ of $T$ which is a sink of the resulting orientation of $T.$
Thus, $G-X_{t}$ does not contain any graph of $\mathsf{conn}(\obs(\Hcal))$ as a minor.
Therefore, every component of $G-X_{t}$ belongs to $\mathcal{H}$ and $X_t$ is an $\Hcal^{(c)}$-modulator of $G$ of size $q+1 \leq (q+1)\cdot (2p-1)$ as desired.

Applying the claim above for $p=|\mathsf{conn}(\obs(\Hcal))|\cdot k,$ one may obtains that 
$\Hcal\text{-}\size(G)\leq 2(q+1)\cdot |\mathsf{conn}(\obs(\Hcal))|\cdot k.$
Recall that $|\mathsf{conn}(\obs(\Hcal))|=2^{\Ocal(h^2)}.$
Therefore the lemma holds for some $f_{\ref{obs_con_apex}}(h,k)={2^{k^{\mathcal{O}_h(1)}}}.$
\end{proof}

We just proved that $\p_{\mathfrak{S}_{\Hcal}}\sim \Hcal\text{-}\size$ with a single-exponential gap. This is a first step towards the main result of this section, which is that $\mathfrak{F}_{\Hcal}$ is an obstructing set for $\apex_{\Hcal}.$
This is proved in the next subsection.

\subsection{The disconnected case}
\llabel{disc_case_obs}
Given a graph $Z$ we use $\mathsf{cc}(Z)$ for the set of its components.
Given two graphs $Z_1$ and $Z_2,$ we denote by $Z_{1}+Z_{2}$ the disjoint union of $Z_{1}$ and $Z_{2}.$
Also, given two parametric graphs $\mathscr{G}^{i}=\lin{\mathscr{G}^{i}_{t}\mid t\in\Nbbb}, i\in[2],$ we write $\mathscr{G}^{1}+\mathscr{G}^{2}\coloneqq \lin{\mathscr{G}^{1}_{t}+\mathscr{G}^{2}_{t}\mid t\in\Nbbb}$ for their index-wise disjoint union.
Moreover, if  $\mathfrak{G}^{i},i\in[2]$ are families of parametric graphs, we define the parametric family 
\begin{align*}
\mathfrak{G}^{1}\otimes \mathfrak{G}^{2} \coloneqq \{\mathscr{G}^{1}+\mathscr{G}^{2}\mid (\mathscr{G}^{1},\mathscr{G}^{2})\in \mathfrak{G}^{1}\times \mathfrak{G}^{2}\}.
\end{align*}
Notice that the compositions defined above are commutative and they may be extended in the natural way to $Z_{1}+\cdots+Z_{r},$ to $\mathscr{G}^{1}+\cdots+\mathscr{G}^{r}$ and to 
$\mathfrak{G}^{1}\otimes\cdots\otimes \mathfrak{G}^{r}$ respectively.
\medskip

Let $\Hcal$ be a minor-closed class and let $Z\in\obs(\Hcal).$
Observe that $\mathfrak{F}^{(Z)} = \mathfrak{F}_{\excl(Z)}.$
Given the notation above, the definitions of $\mathfrak{F}_{\Hcal}$ and its related families in $\eqref{first_walloid_formula},$ $\eqref{second_walloid_formula},$ and $\eqref{third_walloid_formula}$ can be rewritten as follows.
\begin{eqnarray}
 \mathfrak{F}_{\excl(Z)}& \coloneqq  & \mathfrak{F}_{\excl(Z_1)}\otimes\cdots\otimes\mathfrak{F}_{\excl(Z_r)}, \text{~where $\mathsf{cc}(Z)=\{Z_{1}, \ldots, Z_{r}\}$}\llabel{disc_case_first}\\
 \mathfrak{F}_{\Hcal}^{\nicefrac{1}{2}} & \coloneqq & \big\{ \mathscr{W} \in \cupall\{ \mathfrak{F}_{\excl(Z)} \mid {Z\in\obs(\Hcal)}\}\mid \text{for every $\mathscr{W}' \in \mathfrak{F}_{\Hcal}^{\varnothing}$ it holds that $\mathscr{W}' \not\lesssim \mathscr{W}$} \big\}\\
 {\mathfrak{F}}_{\Hcal} & \coloneqq & \mathfrak{F}^{\nicefrac{1}{2}}_{\mathcal{H}} \cup \mathfrak{F}_{\Hcal}^{\varnothing} \llabel{disc_case_second}\\
 \mathfrak{F}_{\Hcal}^\mathsf{min}& \coloneqq & \mathsf{min}(\mathfrak{F}_{\Hcal})\llabel{disc_case_third}
\end{eqnarray}

\begin{observation}
\llabel{four_four_univ}
Let $G^1,G^2,$ $F^{1},$ and $F^{2}$ be graphs, then 
\begin{align*}
\barrier_{\excl(F^1+F^2)}(G_{1}+G_{2})  \geq \max\big\{\barrier_{\excl(F^1)}(G_{1}),\barrier_{\excl(F^2)}(G_{2}\big\}.
\end{align*}
\end{observation}

\begin{lemma}
\llabel{two_sum_univ}
Let $Z_{1}$ and $Z_{2}$ be two graphs and let $\mathfrak{F}^{1}$ and $\mathfrak{F}^{2}$ be obstructing sets for $\apex_{\excl(Z_{1})}$ and $\apex_{\excl(Z_{2})}$ respectively both with gap function $f.$
Then $\p_{\mathfrak{F}^{1}\otimes \mathfrak{F}^{2}}\preceq \apex_{\excl(Z_{1}+Z_{2})}.$
\end{lemma}

\begin{proof}
Let $k\in\Nbbb$ and let $\mathscr{F}_{k}^{i}\in\mathscr{F}^{i}\in\mathfrak{F}^{i}, i\in[2].$
If a graph $G$ contains some graph $\mathscr{F}_k^{1}+\mathscr{F}_k^{2}$ as a minor, then $\apex_{\Hcal}(G)\geq \apex_{\Hcal}(\mathscr{F}_k^{1}+\mathscr{F}_k^{2})\geq \barrier_{\Hcal}(\mathscr{F}_k^{1}+\mathscr{F}_k^{2}),$ which, by \autoref{four_four_univ}, is at least $ \max\{\barrier_{\Hcal^1}(\mathscr{F}_{k}^{1}),\barrier_{\Hcal^2}(\mathscr{F}_{k}^{2})\}\geq f(k).$
This immediately implies that $\p_{\mathfrak{F}^{1}\otimes \mathfrak{F}^{2}}\preceq \apex_{\excl(Z_{1}+Z_{2})}.$
\end{proof}

For every proper minor-closed class $\Hcal,$ we define the parameter $\p^{\Hcal}\coloneqq \max\{\apex_{\excl(Z)} \mid Z\in\obs(\Hcal)\}.$
One may immediately observe the following.

\begin{observation}
\llabel{many_to_one_obs}
For every proper minor-closed class $\Hcal,$ it holds that $\p^{\Hcal} \preceq \apex_{\Hcal} \preceq |\obs(\Hcal)| \cdot \p^{\Hcal}.$
\end{observation}

An immediate consequence of this observation is, that the parametric family $\mathfrak{F}_{\mathsf{excl}(Z)}$ indeed defines a parameter that acts as a lower bound to $\mathsf{apex}_{\mathsf{excl}(Z)}$ even if $Z$ is not connected.

\begin{corollary}
\llabel{added_rec_obs}
For every proper minor-closed class $\Hcal,$ it holds that $\p_{\mathfrak{F}_{\Hcal}}\preceq\apex_{\Hcal},$ with linear gap.
\end{corollary}

\begin{proof}
By applying \autoref{two_sum_univ} inductively to \eqref{disc_case_first}, while using \autoref{easy_obs_gap_} as the base case, we obtain that for every graph $Z,$ $\p_{\mathfrak{F}_{\excl(Z)}}\preceq \apex_{\excl(Z)}.$
The corollary holds due to the first inequality of \autoref{many_to_one_obs} and the fact that $\p_{\mathfrak{F}_{\Hcal}} =\max\{\p_{\mathfrak{F}_{\excl(Z)}}\mid Z\in \obs(\Hcal)\}.$
\end{proof}

The next lemma is an algorithmic restatement of \autoref{thm_apex_upper_bound} and we conclude this subsection with its proof.

\begin{lemma}
\llabel{apex_disc_obs} There exists a function $f_{\ref{apex_disc_obs}}\colon\nn{2}{1}$ such that for every proper minor-closed graph class $\Hcal,$ where $h=h(\Hcal),$ every graph $G,$ and every $k\in\Nbbb,$ either $\p_{\mathfrak{F}_{\Hcal}}(G)\geq k$ or $\apex_{\Hcal}(G) \leq f_{\ref{apex_disc_obs}}(h,k),$ where $f_{\ref{apex_disc_obs}}(h,k)={2^{k^{\mathcal{O}_h(1)}}}.$
Moreover, there exists an algorithm that, given $\obs(\Hcal),$ $G,$ and $k,$ outputs a certificate of one of the two outcomes in time $2^{2^{k^{\mathcal{O}_h(1)}}}|G|^4\log|G|.$
\end{lemma}
\begin{proof}
It suffices to prove that $\apex_{\Hcal}\preceq \p_{\mathfrak{F}_{\Hcal}}$ since the other direction is already given by \autoref{added_rec_obs}. 

Let $G$ be a graph.
As all graphs in $\obs(\Hcal)$ have at most $h$ vertices, we obtain that $|\obs(\Hcal)|= 2^{\Ocal(h^2)}.$
Due to the second inequality of \autoref{many_to_one_obs}, we have that $\apex_{\Hcal}(G)= 2^{\Ocal(h^2)}\cdot \p^{\Hcal}(G).$
This justifies us to focus on a single graph rather than the entire set $\mathsf{obs}(\mathcal{H}).$

Consider a graph $Z\in\mathsf{obs}(\mathcal{H}).$
In this context, we fix $\mathsf{cc}(Z)=\{Z_{1}, \ldots, Z_{r}\},$ $r\leq h.$
Given a set $I=\{i_{1},\ldots,i_{s}\}\subseteq[r]$ of integers, we write $Z_{I}\coloneqq \bigcup_{i\in I}Z_{i}$ and we write $\mathfrak{F}_{I}= \mathfrak{F}_{\excl(Z_{i_1})}\otimes\cdots\otimes\mathfrak{F}_{\excl(Z_{i_s})}$ for those parametric graphs generated from the selection of components of $Z$ induced by $I.$

Given some integer $p\in[r],$ a \defi{$p$-connectivization} of $Z$ is a graph $Z'$ obtained from $Z$ by selecting a set $I'\subseteq [r]$ with $|I'|=p$ and replacing the components $Z_i,$ $i\in I'$ by a graph from $\mathsf{conn}(\sum_{i\in I'}Z_i).$
We denote by $\mathcal{Z}_p$ the collection of all $p$-connectivizations of $Z.$

Observe that, by definition of $\mathfrak{F}_{\excl(Z)},$ it follows that for every $Z' \in \mathcal{Z}_{2},$ $\mathfrak{F}_{\excl(Z)} \lesssim \mathfrak{F}_{\excl(Z')}.$
Moreover there exists a function $c \coloneqq c(h)$ such that for every $Z' \in \mathcal{Z}_{2},$ $\lin{\mathscr{F}'_{t}}_{t \in \Nbbb} \in \mathfrak{F}_{\excl(Z')}$ there exists a $\lin{\mathscr{F}_{t}}_{t \in \Nbbb} \in \mathfrak{F}_{\excl(Z)}$ such that $\mathscr{F}_{k}$ is a minor of $\mathscr{F}'_{ck}.$

Finally, we set
\begin{align*}
g(1,h,k) &\coloneqq f_{\ref{obs_con_apex}}(h,k)\text{ and}\\
g(r,h,k) &\coloneqq h^2 \cdot g(r-1,h, ck)+ r^3k\cdot f_{\ref{obs_con_apex}}(h,k).
\end{align*}
The function $g(r,h,k)$ is a placeholder for the actual function $f_{\ref{apex_disc_obs}}$ we wish to define to prove the assertion.
However, at the moment it is more convenient to make the dependency on the number of components of $Z$ visible.
To obtain the result we wish for we may simply set $f_{\ref{apex_disc_obs}}(h,k)\coloneqq g(h,h,ck)$ as $h$ is always an upper bound on $r.$

Our proof proceeds by induction on $r.$
That is, we prove the following statement by induction on the number of components of $Z.$

\paragraph{Inductive statement.}
For every positive integer $k,$ every graph $Z$ on at most $h$ vertices with at most $r$ components, and every graph $G$ one of the following hold
\begin{enumerate}
    \item $G$ contains $\mathfrak{F}_k$ as a minor for some $\langle \mathfrak{F}_r\rangle_{r\in\mathbb{N}}\in \mathfrak{F}_{\mathsf{excl}(Z)},$ or
    \item there exists a set $S\subseteq V(G)$ of size at most $g(r,h,k)$ such that $G-S\in\mathsf{excl}(Z).$
\end{enumerate}

\paragraph{Base of the induction.}
The base follows in a straightforward way from \autoref{obs_con_apex}.
This is true since for the base we may consider the case $r=1$ where $g(r,h,k)=f_{\ref{obs_con_apex}}(h,k)$ and $Z$ is connected.
With $Z$ being connected there is no difference between the connectivizations of $Z$ and $Z$ itself and therefore the claim follows.

\paragraph{Inductive step.}
From here on, we may assume that $r \geq 2.$
As a first step, we consider the set $\mathcal{Z}_2$ of all $2$-connectivizations of $Z.$
Notice that every graph $Z' \in \mathcal{Z}_2$ has exactly $r-1$ components.
Moreover, $|\mathcal{Z}_2|\leq |Z|^2\leq h^2.$
For each $Z'\in\mathcal{Z}_2$ we may now call upon the induction hypothesis and obtain either
\begin{itemize}
    \item a set $S_{Z'}$ of size at most $g(r-1,h,ck)$ such that $G-S_{Z'}\in\mathsf{excl}(Z'),$ or
    \item $\mathscr{F}'_{ck}$ as a minor of $G$ for some $\langle \mathscr{F}'_t\rangle_{t\in\mathbb{N}}\in\mathfrak{F}_{\mathsf{excl}(Z')}.$
\end{itemize}
Recall that, for every $\lin{\mathscr{F}'_{t}}_{t \in \Nbbb} \in \mathfrak{F}_{\excl(Z')}$ there exists a $\lin{\mathscr{F}_{t}}_{t \in \Nbbb} \in \mathfrak{F}_{\excl(Z)}$ such that $\mathscr{F}_{k}$ is a minor of $\mathscr{F}'_{ck}.$
Thus, if for any $Z'\in\mathcal{Z}_2$ we obtain the second outcome, we are immediately done.
Hence, we may assume to obtain the set $S_{Z'}$ as above for each $Z'\in\mathcal{Z}.$
Let $S_1\coloneqq \bigcup_{Z'\in\mathcal{Z}_2}S_{Z'}$ and notice that $|S_1|\leq h^2 \cdot g(r-1,h, ck).$

Let $G_1\coloneqq G-S_1.$
Observe that for every minimal subgraph $H$ of $G$ such that $H$ contains $Z$ as a minor it must hold, that any component of $G_1$ may contain at most one component of $H.$
Otherwise, there would exist components $H_1$ and $H_2$ of $H$ such that there exists a path from $H_1$ to $H_2$ in $G_1$ which avoids all other components of $H$ in $G_1.$
Such a path, however, would witness that $G_1$ still contains a member of $\mathcal{Z}_2$ as a minor.

For every $i\in[r]$ let $\mathcal{K}_i$ be the set of all components of $G_1$ that contain the component $Z_i$ of $Z$ as a minor.
Moreover, for any set $\mathcal{K}$ of components of $G_1$ let $\overline{\mathcal{K}}$ be the set of all components of $G_1$ which do not belong to $\mathcal{K}.$
We first distinguish two cases, one of which is easily dealt with, and then we enter, as part of the second case, an iterative construction process that eventually leads to either our obstruction or the desired $\mathsf{excl}(Z)$-modulator.
The cases we distinguish are
\begin{description}
    \item [Case 1:] $|\mathcal{K}_i|\geq r\cdot k$ for every $i\in[r],$ and
    \item [Case 2:] there exists $i_1\in[r]$ such that $|\mathcal{K}_{i_1}|<r\cdot k.$
\end{description}
Notice that, in \textbf{Case 1}, we may greedily select $k$ components from $\mathcal{K}_i$ for every $i$ and thereby obtain $k\cdot Z$ as a minor of $G_1.$
Since $\langle t\cdot Z\rangle_{t\in\mathbb{N}}\in\mathfrak{F}_{\mathsf{excl}(Z)},$ this outcome would complete our proof and thus, we may assume to be in \textbf{Case 2}.
In the following, we will, at times, apply a similar counting argument.
To facilitate our writing, we will simply refer to this argument as the \textbf{few components argument}.

We initialize $I_1\coloneq \{ i_1\}$ and iteratively construct sets $I_q$ and $X_q$ as follows.
\paragraph{Target of our construction.}
For every $q \geq 1$ there exists a set $I_q=\{ i_1,\dots,i_q\}\subseteq I$ such that there exists a set $X_{q}$ of size at most $qr\cdot rk\cdot f_{\ref{obs_con_apex}}(h,k)$  and a set $\mathcal{Q}_q$ of at most $q\cdot r\cdot k$ components of $G_1$ with the following properties
\begin{enumerate}
    \item for every $i\in I_q$ and every component $K\in\mathcal{Q}_{q}$ either $K-X_q$ is $Z_i$-minor-free, or $K$ contains $\mathscr{F}^i_k$ as a minor for some $\langle \mathscr{F}^i_t \rangle_{t\in\mathbb{N}}\in\mathfrak{F}_{\mathsf{excl}(Z_i)},$
    \item for every $i\in I_q$ and every component $K\in\overline{\mathcal{Q}_q}$ it holds that $K$ is $Z_i$-minor-free, and
    \item for every minimal subgraph $H$ of $G_1-X_q$ which contains $Z$ as a minor and every $i\in I_q$ there exists a component $K_i\in\mathcal{Q}_{q}$ such that $H\cap K_i$ is minimal with respect to containing $Z_i$ as a minor.
    Moreover, if $i,j\in I_q$ are distinct, then so are $K_i$ and $K_j.$
\end{enumerate}
Once we complete this construction, we obtain that $I_r=I.$
During this construction, we stop whenever we reach a point where $G_1-X_q$ is $Z$-minor-free, or contains our desired obstruction.
For the second outcome of this stopping condition we will employ a small argument based on Hall's Theorem, this argument will finally also complete the case $q=r.$

We proceed with this construction by induction on $q$ and, as the base, show how to obtain the set $X_1.$

\paragraph{Base of our construction.}
We set $\mathcal{Q}_1\coloneqq \mathcal{K}_{i_1}.$
Notice that the second and third conditions immediately hold by the definition of $\mathcal{K}_{i_1}$ since no other component of $G_1$ can contain $Z_{i_1}$ as a minor and no $Z$-minor can realize two of its components in the same component of $G_1.$
For each of the at most $r\cdot k$ members of $\mathcal{K}_{i_1}$ we may apply \autoref{obs_con_apex} for $k$ and the graph $Z_{i_1}.$
Let $X_1$ be the union of all hitting sets, of size at most $f_{\ref{obs_con_apex}}(h,k)$ each, that are returned this way.
It follows that $I_1$ and $X_1$ also satisfy the first condition of our construction.

\paragraph{The inductive step of our construction.}
Let us now assume that for some $q\in[r-1]$ we have already constructed the sets $I_q,$ $\mathcal{Q}_q,$ and $X_q$ as above.

First, let us introduce the \textbf{matching argument}.
Let $B_q$ be the bipartite graph with $I_q$ as one of the two color classes and $\mathcal{Q}_{q}$ as the other color class.
There exists an edge $iK \in E(B_q)$ if and only if the component $K\in\mathcal{Q}_{q}$ still contains $Z_i$ as a minor after deleting $X_q.$
Notice that, by the second condition of our construction, this means that $K$ must contain $\mathscr{F}^i_k$ as a minor for some $\langle \mathscr{F}^i_t \rangle_{t\in\mathbb{N}}\in\mathfrak{F}_{\mathsf{excl}(Z_i)}.$
Now, if $B_q$ does not have a matching that covers the entire set $I_q,$ then, by Hall's Theorem, there exists a set $I'\subseteq I_q$ such that $|N_{B_q}(I')|<|I'|.$
Since every occurrence of $Z$ as a minor in $G_1-X_q$ must use $q$ distinct members of $\mathcal{Q}_{q},$ one for each $i\in I_q,$ to realize the components $Z_i,$ $i\in I_q,$ this implies that $G_1-X_q$ is $Z$-minor-free and we may terminate.
Hence, we may assume that $\cupall\mathcal{Q}_{q}$ contains $\mathscr{F}_k$ as a minor for some $\langle \mathscr{F}_t\rangle_{t\in\mathbb{N}}\in\mathfrak{F}_{\mathsf{excl(I_q)}}.$

Now suppose that $|\mathcal{K}_i\setminus\mathcal{Q}_q|\geq |I\setminus I_q|k$ for every $i\in I\setminus I_q.$
By the \textbf{few components argument} and the \textbf{matching argument} above this implies that $G_1$ contains $\mathscr{F}_k$ as a minor for some $\langle \mathscr{F}_t\rangle_{t\in\mathbb{N}}\in\mathfrak{F}_{\mathsf{excl}(Z)}.$
Hence, we may assume that there is some $i_{q+1}\in I\setminus I_q$ such that $|\mathcal{K}_{q+1}\setminus\mathcal{Q}_q|< |I\setminus I_q|k.$
Let $\mathcal{Q}_{q+1}\coloneqq \mathcal{Q}_q\cup \mathcal{K}_{i_{q+1}}.$

Next, we make use of \autoref{obs_con_apex}.
That is, we apply \autoref{obs_con_apex} for $Z_{i_{q+1}}$ and every member of $\mathcal{Q}_{q+1}.$
Let $X_{q+1}'$ be the union of all modulators, each of size at most $f_{\ref{obs_con_apex}}(h,k),$ which are returned by these applications.
Notice that 
\begin{align*}
    |X_{q+1}'| & \leq (q+1)\cdot rk\cdot f_{\ref{obs_con_apex}}(h,k).
\end{align*}
Finally, let $X_{q+1}\coloneqq X_q \cup X_{q+1}'$ and notice that
\begin{align*}
    |X_{q+1}| \leq |X_q| + (q+1)\cdot rk\cdot f_{\ref{obs_con_apex}}(h,k) \leq (q+1)r\cdot rk\cdot f_{\ref{obs_con_apex}}(h,k).
\end{align*}
The three sets $I_{q+1},$ $\mathcal{Q}_{q+1},$ and $X_{q+1}$ now satisfy the first condition of our construction by \autoref{obs_con_apex} and our construction of $X_{q+1},$ and the second and third condition are satisfied by our choice for $\mathcal{Q}_{q+1}.$

In particular, it follows from the \textbf{matching argument} that, if $\bigcup\mathcal{Q}_{q+1}$ does not contain $\sum_{j\in[q+1]}\mathscr{F}^j_k$ as a minor for some choice of $\langle \mathscr{F}^j_t\rangle_{t\in\mathbb{N}}\in\mathfrak{F}_{\mathsf{excl}(Z_j)}$ for each $j\in I_{q+1},$ then $G_1-X_{q+1}$ is $Z$-minor-free.
This last observation indeed proves that, once $I_r,$ $\mathcal{Q}_r,$ and $X_r$ are constructed, we are done.
\end{proof}

Concluding, \autoref{added_rec_obs} and \autoref{apex_disc_obs}  imply 
that $\mathfrak{F}_{\Hcal}^\mathsf{min}$ is a universal obstruction for $\apex_{\Hcal}.$

\subsection{Universal obstructions for $\Hcal$-$\td$}
\llabel{univ_td_all}

So far we  presented universal obstructions for the 
parameters $\Hcal$-$\tw$ and $\apex_{\Hcal},$ namely 
$\mathfrak{W}_{\Hcal}^\mathsf{min}$ and $\mathfrak{F}_{\Hcal}^\mathsf{min}$ respectively.
Notice that the knowledge of $\mathfrak{W}_{\Hcal}$  has been important 
for the derivation of $\mathfrak{F}_{\Hcal}$ as it permitted to assume 
that $\Hcal$-\tw is bounded and consider a bounded width 
tree-$\apex_{\Hcal}$ decomposition in the proof of \autoref{obs_con_apex}.
We next apply the same approach to the \defi{elimination distance} 
to  $\Hcal,$ also known as the \defi{$\Hcal$-tree-depth},  
that is recursively defined  as follows:
\begin{eqnarray}
   \td_\Hcal(G) & = &
   \begin{cases}
     0 & \text{if}\ G\in\Hcal, \\
     1+\min\{\ed_\Hcal(G\setminus \{v\})\mid v\in V(G)\} & \text{if}\ G\text{ is connected}, \\
		 \max\{\ed_\Hcal(H)\mid H \text{ is a connected component of}\ G\} & \text{otherwise.}
   \end{cases}\llabel{def_td_daw}
\end{eqnarray}
Another way to define $\td_\Hcal$ is to just replace in 
\eqref{basic_scheme} $\tw$  by $\td,$ that is the  parameter of \textsl{tree-depth}. 
The \defi{tree-depth} of a graph $G$   is
 defined as the minimum height of a rooted forest $F$ 
with the property that every edge of $G$ connects a pair of vertices that have an ancestor-descendant relationship to each other in $F.$
It is easy to observe that  $\td=\td_{\varnothing}$ (recall that $\varnothing=\excl(K_{1})$).
The parameter $\td_\Hcal$ was introduced by 
Bulian and Dawar~\cite{BulianD16grap,BulianD17fixe}
as a distance measure of a graph to some graph class $\Hcal.$
Notice that $\td_{\Hcal}$ remains the same parameter if we replace $\Hcal$ by $\Hcal^{(c)}.$

It can be easily seen that for every graph $G,$
 $\tw_{\Hcal}(G)≤\Hcal$-$\tw(G)≤\apex_{\Hcal}(G).$
For every $k,$ let 
\begin{eqnarray}
\Tcal\Dcal_{k}^{\Hcal} & = & \{G\mid \Hcal\text{-}\td(G)≤k\}.\llabel{td_obs_class_sm}
\end{eqnarray}
In \cite{BulianD17fixe}, Bulian and Dawar proved that there is a computable function 
 bounding the obstructions of  $\Tcal\Dcal_{k}^{\Hcal},$ i.e., for every $\Hcal,$ there is a computable function $f_{\Hcal}^{\sf td}:\nn{1}{1}$ where 
 for every $k\in \Nbbb,$ if $Z\in\obs(\Tcal\Dcal_{k}^{\Hcal}),$ then $|Z|≤f_{\Hcal}^{\sf td}(k).$ This implied that the problem of deciding whether $\Hcal\text{-}\td(G)≤k$ is constructively in {\sf FPT}, when parameterized by $k.$ This was improved in \cite{MorelleSST23fast} where a  time $2^{2^{2^{\poly_{h}(k)}}}n^2$ algorithm for this problem was given. Moreover, the results in \cite{MorelleSST23fast} imply an explicit bound $f_{\Hcal}^{\sf td}(k)=2^{2^{2^{2^{\poly_{h}(k)}}}}.$

In this section, we explain how to use an approach similar to the one of \autoref{univ_apex_con_discon} in order to find a universal obstruction for $\Hcal\text{-}\td.$

We start from the following key observation of Bulian and Dawar in~\cite{BulianD17fixe}.

\begin{lemma} 
For every minor-closed class $\Hcal,$ $\obs(\Hcal^{(c)})=\mathsf{conn}(\obs(\Hcal)).$
\end{lemma}

A \defi{2-rooted-graph} is a triple $\mathbf{Ζ}=(Ζ,v,u)$ where $Z$ is a connected graph and $v,u\in V(Ζ).$
We say that two 2-rooted-graphs $\mathbf{Ζ}=(Ζ,v,u)$ and $\mathbf{Ζ'}=(Ζ',v',u')$ are \defi{isomorphic}  
if there is an isomorphism $\sigma:V(Ζ)\to V(Ζ')$ such that $\sigma(v)=v'$ and $\sigma(u)=u'.$
We define $\mathcal{R}(Ζ)$ as the set of all pairwise non-isomorphic 2-rooted-graphs whose graph is $Ζ.$ An \defi{$\Hcal$-chain}  of size $t$ is a graph obtained 
if we first consider a sequence 
\begin{eqnarray}
\lin{(Ζ^1,v^1,u^1),\ldots,(Ζ^t,v^t,u^t)}\llabel{seq_two_root}
\end{eqnarray}
 of $2$-rooted graphs, where for each $i\in[t],$ there exists some $Z^i\in\obs(\Hcal)$ such that $(Z^{i},v^i,u^i)\in\Rcal(Z),$ 
and then add the edges in $\big\{\{u^{i},v^{i+1}\}\mid i\in[t-1]\big\}$ in the disjoint union of the graphs in \eqref{seq_two_root}. An $\Hcal$-chain where 
 $(Z^{i},v^i,u^i)$ is  isomorphic to some particular  $\mathbf{Ζ},$ then we call it  the  \defi{$\mathbf{Ζ}$-chain}.
For every 2-rooted graph $\mathbf{Ζ}=(Ζ,v,u)\in\mathcal{R}(Ζ),$ we define
the parametric graph $\mathscr{E}^{\mathbf{Ζ}}=\langle\mathscr{E}^{\mathbf{Ζ}}_{t}\rangle_{t\in\mathbb{N}}$
where $\mathscr{E}^{\mathbf{Ζ}}_{t}$ is the  $\mathbf{Z}$-chain of size $t.$

\begin{observation}
\llabel{one_obst_Htd}
There exists a function $f:\nn{1}{1}$ such that 
for  every minor-closed class $\Hcal$
and every $k\in\Nbbb,$  every $\Hcal$-chain  of size $t\cdot f(h)$
contains $\mathscr{E}^{\mathbf{Ζ}}_{t}$ as a minor 
for some $\mathbf{Ζ}\in\mathcal{R}(Ζ).$
\end{observation}

Let  $\Hcal$  be a minor-closed class and let $\Ocal_{\Hcal}=\obs(\Hcal^{(c)}).$
Then we define the following 
collections of parametric graphs.
\begin{eqnarray}
\mathfrak{T}_{\Hcal}&  \coloneqq &  \mathfrak{W}_{\Hcal}\cup \bigcup_{Z\in\Ocal_{\Hcal}}\big\{\mathscr{E}^{\mathbf{Ζ}}\mid \mathbf{Ζ}\in\mathcal{R}(Ζ)\big\} \llabel{obst_Htd_more}\\\mathfrak{T}_{\Hcal}^\mathsf{min}& \coloneqq &  \mathsf{min}(\mathfrak{T}_{\Hcal})\nonumber
\end{eqnarray}

The next observation follows immediately from \eqref{def_td_daw}
and the definition of $\mathscr{E}^{\mathbf{Ζ}}.$
\begin{observation}
\llabel{obst_Htd_lower}
Let $\Hcal$ be a minor-closed class and $Z\in\Ocal_{\Hcal}.$ Then for every  graphs $G,$ every 2-rooted-graph graph $\mathbf{Z}\in\mathcal{R}(Ζ),$ every 
$\mathscr{E}^{\mathbf{Ζ}}=\langle\mathscr{E}^{\mathbf{Ζ}}_{t}\rangle_{{t\in\mathbb{N}}},$ and every $t\in\Nbbb,$ 
it holds that 
$\td_{\Hcal}(\mathscr{E}^{\mathbf{Ζ}}_{t})=\Omega(\log(t)).$ 
\end{observation}

\begin{theorem}
Let $\Hcal$ be a minor-closed class. Then 
$\mathfrak{T}_{\Hcal}$ is a universal obstruction for $\Hcal$-$\td$ with single exponential gap.
\end{theorem}

\begin{proof}
Let $\mathfrak{E}_{\Hcal}\coloneqq \bigcup_{Z\in\Ocal_{\Hcal}}\big\{\mathscr{E}^{\mathbf{Ζ}}\mid \mathbf{Ζ}\in\mathcal{R}(Ζ)\big\}).$
We first prove that $\p_{\mathfrak{T}_{\Hcal}}\preceq \Hcal\text{-}\td.$ For this let
$k\in\Nbbb$ and let 
 $\mathscr{T}_{k}\in\mathscr{T}=\lin{\mathscr{T}_{k}}_{k\in\Nbbb}\in\mathfrak{T}_{\Hcal}.$
In the case where $\mathscr{T}\in\mathfrak{W}_{\Hcal}$
we know from \autoref{cor_walloid_lower_bound} that  ${\Hcal}\text{-}\td(\mathscr{T}_{t})≥{\Hcal}\text{-}\tw(\mathscr{T}_{t})=\Omega_{h}(k),$ where $h=h(\Hcal).$
In case where $\mathscr{T}\in  \mathfrak{E}_{\Hcal},$ it follows from \autoref{obst_Htd_lower} that 
${\Hcal}\text{-}\td(\mathscr{G}_{t})=\Omega(\log(k)).$ This proves  $\p_{\mathfrak{T}_{\Hcal}}\preceq \Hcal\text{-}\td.$

We proceed with the proof that $ \Hcal\text{-}\td\preceq \p_{\mathfrak{T}_{\Hcal}}.$
Let $k\in \Nbbb$ and let 
$G$ be a graph minor-excluding
$\mathscr{T}_{k}\in\mathscr{T}$ for every $\mathscr{T}\in\mathfrak{T}_{\Hcal}.$ W.l.o.g. we assume that $G$ is connected, otherwise we apply the proof to each of the components of $G.$
As $\mathfrak{T}_{\Hcal} \lesssim^{*} \mathfrak{W}_{\Hcal}$
we know that $\Hcal$-$\tw(G)≤q\coloneqq {2^{k^{\mathcal{O}_h(1)}}}.$
Let  $(T,\beta)$ be a $\mathcal{H}$-tree decomposition
of $G$ of width $≤q.$
We define the parameter  $\Hcal$-\textsf{lc}$(G)$ as the maximum $k$ for which 
$G$ contains an $\Hcal$-chain of size $k$ as a minor. From \autoref{one_obst_Htd}
we have that $\Hcal{\sf \text{-}lc}(G)$
is bounded by $\mathcal{O}_{h}(\p_{\mathfrak{E}_{\Hcal}}(G)) = \mathcal{O}_{h}(k).$

Therefore,  it is enough to prove that, for every pair of integers $k',q'\in\Nbbb,$
and every graph $G$ of $\mathcal{H}$-treewidth at most $q',$ if $\Hcal\mathsf{\text{-}lc}(G)\leq k',$ then $\Hcal\text{-}\td(G)\leq (q'+1)\cdot k'.$
For this, we use induction on $k'.$
For every edge $tt'$ of $T$ we define $X_{t}$ and $X_{t'}$ as in \eqref{all_use_sep}.

We first observe that if there is some edge $e=tt'$ of $T$ where $\Hcal\mathsf{\text{-}lc}(G[X_t\setminus X_{t'}])\leq k'-1$ and $\Hcal\mathsf{\text{-}lc}(G[X_{t'}\setminus X_{t}])\leq k'-1,$ then 
$$\Hcal\text{-}\td(G)\leq |X_{t}\cap X_{t'}|+\max\{\Hcal\text{-}\td(G[X_t\setminus X_{t'}]),\Hcal\text{-}\td(G[X_{t'}\setminus X_{t}])\}\leq k' + (q'+1) \cdot (k'-1) = (q'+1) \cdot k'.$$
On the other side if it holds that if $\Hcal\mathsf{\text{-}lc}(G[X_t\setminus X_{t'}])=\Hcal\mathsf{\text{-}lc}(G[X_{t'}\setminus X_{t}])=k',$ then the connected graph $G$ contains some subgraph $Z$ of  $G[X_{t}\setminus X_{t'}]$ and some subgraph $Z'$ of  $G[X_{t'}\setminus X_{t}]$ such that both $Z$ and $Z'$ are $\Hcal$-chains of size $k'$ and there is a path joining $Z$ and $Z'.$
This leads to a contradiction because two $\Hcal$-chains of length $k'$ connected by some path implies the existence 
of some $\Hcal$-chain of length $k'+1.$

The conclusion of the above is that for every $e=tt'$ exactly one of $G[X_{t}\setminus X_{t'}]$ and $G[X_{t'}\setminus X_{t}]$ contains a $\Hcal$-chain of length $k'.$
We orient every edge towards the direction where the $\Hcal$-chain of length $k'$ is located.
This implies that there is a node $t$ of $T$ which is a sink with respect to this orientation of $T.$

If $t$ is not a leaf of $T$ let $A\coloneqq \beta(t),$ otherwise let $A\coloneqq A_t$ where $A_t$ is the unique adhesion set of the leaf $t\in V(T).$
As a consequence, if $C$ is a connected component of $G-A,$ then $\Hcal\mathsf{\text{-}lc}(C)\leq k'-1,$ therefore, by the induction hypothesis, $\Hcal\text{-}\td(G)\leq (q'+1) \cdot (k'-1).$

Then, $$\Hcal\text{-}\td(G)\le |X_{t}|+\max\{\Hcal\text{-}\td(C)\mid C\in\mathsf{cc}(G-A)\}\leq q'+1 + (q'+1) \cdot (k'-1)= (q'+1)\cdot k'$$ which completes the proof of the theorem.
\end{proof}

\section{Consequences and open problems}\llabel{sec_conclusion}

In this paper we provided obstructing sets for the parameters $\Hcal$-$\tw,$ $\Hcal$-$\td,$ and $\apex_{\Hcal}.$
Interestingly, our obstructing set $\mathfrak{W}_{\Hcal}$ for $\Hcal$-$\tw$ is a common ingredient to all other obstructing sets constructed in this paper.
Our proofs are all essentially based on the fact that excluding the parametric graphs in $\mathfrak{W}_{\Hcal}$ allows for the assumption of the existence of a $\mathcal{H}$-tree decomposition of bounded width.
This, in turn, facilitates the proofs that determine the rest of the parametric graphs in the obstructing sets for $\Hcal$-$\td,$ and $\apex_{\Hcal}.$
In the case of $\apex_{\Hcal},$ the ``half-integral'' nature of the walloids in $\mathfrak{W}_{\Hcal}$ is the essential reason for the half-integral Erd\H{o}s-P{\'o}sa property of graph minors as asserted by Liu's theorem.
Moreover, this is exactly how we obtain our constructive proof of Thomas' conjecture and thereby completely delineate the borderline between integrality and half-integrality in Erd\H{o}s-P{\'o}sa dualities for minors. 

The algorithmic consequences have already been discussed in \autoref{sec_algo_consequences}. We discuss below further directions yet to be investigated.

\vspace{-0.1em}
\paragraph{Compositions of classes and parameter.}
Let $\Hcal$ be a minor-closed graph class and let $\p$ be a minor-monotone graph parameter.
Extending the idea of the definition in \eqref{basic_scheme}, we define the \defi{composition} of $\Hcal$ and $\p$ as the parameter $\Hcal$-$\p\colon\gall\to\Nbbb$ where
\begin{align}
\Hcal\text{-}\p(G) = \min\{k\mid~& \text{there exists } X\subseteq V(G)\text{ s.\@ t.\@~ }\p(\torso(G,X))\leq k\text{ and} \llabel{general_composition}\\
&\text{every component of $G-X$ belongs to $\Hcal$}\}\nonumber
\end{align} 
Notice that if $\p\preceq \p'$ then $\Hcal\text{-}\p\preceq \Hcal\text{-}\p'.$
From this point of view, $\Hcal\text{-}\tw\preceq\Hcal\text{-}\td\preceq\Hcal\text{-}\size$ is implied immediately by the fact that $\tw\preceq\td\preceq\size.$
These three parameters have been studied extensively from the algorithmic point of view in \cite{EibenGHK21,JansenK021verte,AgrawalKLPRSZ22delet,Agrawal2022Distance,Jansen20235Approx,inamdar2023fpt,MorelleSST23fast}.\medskip

Observe that  the edgeless graphs in $\lin{t\cdot K_{1}}_{t\in\Nbbb}$ is a universal obstruction of $\size$
and that the paths in $\lin{P_{t}}_{t\in\Nbbb}$ is a universal obstruction of $\td.$
According to \eqref{size_obs_one}, another way to define the collection $\mathfrak{S}_{\Hcal}$ is the following.

\begin{eqnarray}
\mathfrak{S}_{\Hcal}&  = &  \mathfrak{W}_{\Hcal}\cup \bigcup_{Z\in\Ocal_{\Hcal}}\big\{\lin{t\cdot Z}_{t\in\Nbbb}\big\}  \llabel{next_td_obs}
\end{eqnarray}
Let's discuss the similarities in the definitions of \eqref{obst_Htd_more} and \eqref{next_td_obs}. 
The first is that $\mathfrak{W}_{\Hcal}$ always appears as a common ``half-integral'' component.
Notice that, in \eqref{obst_Htd_more}, $\big\{\mathscr{E}^{\mathbf{Ζ}}\mid \mathbf{Ζ}\in\mathcal{R}(Ζ)\big\}$ expresses all possible ways of substituting the vertices an edges in the paths of the universal obstructions of $\td$ by replacing vertices and edges by an obstruction $Z$ of $\Hcal.$
Similarly, in \eqref{next_td_obs}, $\lin{t\cdot Z}_{t\in\Nbbb}$ substitutes the vertices of the edgeless graphs by the universal obstructions of $\size$ by copies of some obstruction $Z$ of $\Hcal.$

We believe that a similar composition mechanism should be able to construct universal obstructions for every parameter $\Hcal\text{-}\p,$  where $\tw\preceq\p\preceq\size$ and $\p\not\in\{\tw,\size\}.$
Assume that $\p$ has some \textsl{finite} universal obstruction $\mathfrak{Q}$ and that $\mathscr{Q}=\lin{\mathscr{Q}_{t}}_{t\in\Nbbb}$ is one of its parametric graphs.
Given some obstruction $Z$ of $\Hcal,$ we apply the following transformation to each $\mathscr{Q}_{t}$: pick three $2$-rooted graphs $(Z^v,v,x_v), (Z^x,x,x), (Z^{u},x_u,u)  \in \mathcal{R}(Ζ)$  and consider the 2-rooted graph $\overline{\mathbf{Z}}=(\overline{Z},v,u)$ where $\overline{Z}$ is obtained by the disjoint union of $Z^v,$ $Z^x,$ and $Z^{u}$ and the identification of $x,$ $x_{v},$ and $x_{u}$ to $x.$
We denote by $\overline{\Rcal}_{Z}$ the set containing all 2-rooted graphs $\overline{\mathbf{Z}}$ that can be defined in this way.
Next, we replace each edge $\{v,u\}$ of $\mathscr{Q}_{t}$ by $\overline{\mathbf{Z}}=(\overline{Z},v,u)$ in the natural way.
We denote by $\mathscr{Q}^{\overline{Z}}=\lin{\mathscr{Q}_{t}^{\overline{Z}}}_{t\in\Nbbb}$ the parametric graph obtained this way.
We believe that this procedure creates an obstructing set for $\Hcal\text{-}\p$ from $\mathfrak{W}_{\Hcal}$ together with a \textsl{finite} collection of parametric graphs whose structure is ``upper bounded'' by some $\mathscr{Q}^{\overline{\mathbf{Z}}},$ as above.
In particular, we set 
\begin{eqnarray*}
\mathfrak{G}_{\Hcal}&  \coloneqq &  \mathsf{min}(\mathfrak{W}_{\Hcal}\cup \bigcup_{Z\in\Ocal_{\Hcal}}\big\{\mathscr{Q}^{\overline{\mathbf{Z}}}\mid \overline{\mathbf{Ζ}}\in\overline{\mathcal{R}}(Ζ)\big\})   
\end{eqnarray*}
and we conjecture that if $\p$ has a finite universal obstruction $\mathfrak{Z}_{\Hcal},$ then 
$\mathfrak{Z}_{\Hcal}\lesssim^{*}\mathfrak{G}_{\Hcal}.$
We stress that this conjecture, if resolved, would indicate that $\mathfrak{W}_{\Hcal}$ is ``ubiquitous'' to the universal obstructions of $\Hcal\text{-}\p.$
However, it would not give a finite universal obstruction to $\Hcal\text{-}\p$ but rather an ``upper bound'' revealing some of the structural characteristics of an obstruction set.
An interesting question is whether the result of Thomas in \cite{Thomas1989wellquasi} on the BQO of graphs with respect to the minor relation may be useful for proving the finiteness of the obstructions of $\Hcal\text{-}\p$ for particular instantiations of $\p.$

\vspace{-0.2em}
\paragraph{Bounding the size of the universal obstructions.}
According to \eqref{half_int_obs}, $|\mathfrak{F}_{\Hcal}^{\nicefrac{1}{2}}|\in 2^{2^{\mathcal{O}(\ell(h^2))}},$ where $\ell(\cdot)$ is the linkage function for which the best known bound is the one in \cite{Kawarabayashi2010Shorter} which is at least triple exponential.
This bound propagates to all obstructing sets defined in this paper, i.\@ e.\@, $|\mathfrak{W}_{\Hcal}^\mathsf{min}|, |\mathfrak{T}_{\Hcal}^\mathsf{min}|, |\mathfrak{F}_{\Hcal}^\mathsf{min}|,$ and it would be nice to improve it.
Such an improvement would be useful for the actual construction of all embedding pairs that generate $\mathfrak{F}_{\Hcal}^{\nicefrac{1}{2}},$ especially when $\obs(\Hcal)$ becomes complicated.
We believe that much better upper bounds can be obtained at the price of a worse gap function.
This is because considering bigger walloids may provide us with enough space for proving that more complicated embedding pairs are ``absorbed'' by simpler ones which would allow us to discard some of the members of the obstructing set $\mathfrak{W}_{\mathcal{H}}$ for $\mathcal{H}$-treewidth.

Our conjecture is that the bound $2^{2^{\Ocal(\ell(h^{2}))}}$ can be considerably improved to one that is polynomial in $h.$
However, we do not believe that this may become possible without paying the price of a worse gap function in our parametric equivalences.

\vspace{-0.2em}
\paragraph{Only few obstructions are enough.}
Let $\Ocal=\lin{\Ocal_{k}}_{k\in\Nbbb}$ be some parametric family of antichains for the minor relation.
We say that $\Ocal$ \defi{defines} the parameter $\p$ if $\p(G)$ is the minimum $k$ such that $G$ excludes all graphs in $\Ocal_{k}$ as minors.
We would like to suggest the following interpretation of our obstructing sets (or better, the universal obstructions obtained from their minimizations) for the parameters $\Hcal$-\tw, $\Hcal$-\td, and $\apex_{\Hcal}$:
We first define 
\begin{eqnarray}
\Tcal\Wcal_{k}^{\Hcal} & = & \{G\mid \Hcal\text{-}\tw(G)\leq k\}, \llabel{tw_obs_class_sm}
\end{eqnarray}
and recall the analogous definitions of the graph classes $\Tcal\Dcal_{k}^{\Hcal}$ and $\Acal\Pcal_{k}^{\Hcal}$ given in \eqref{td_obs_class_sm} and \eqref{apex_obs_small} respectively.
Using the terminology above, it holds that  
\begin{eqnarray*}
\p_{\obs(\Tcal\Wcal_{k}^{\Hcal})}=\Hcal\text{-}\tw, & 
\p_{\obs(\Tcal\Dcal_{k}^{\Hcal})}=\Hcal\text{-}\td,\ \  \text{and}&
\p_{\obs(\Acal\Pcal_{k}^{\Hcal})}=\apex_\Hcal.
\end{eqnarray*}
Unfortunately, the sizes of the parametric families $\obs(\Tcal\Wcal_{k}^{\Hcal}),$ $\obs(\Tcal\Dcal_{k}^{\Hcal}),$ and $\obs(\Acal\Pcal_{k}^{\Hcal})$ grow fast as $k$ increases.
For instance, in the base case where $\Hcal=\varnothing,$ where we obtain the parameters $\tw$/$\td$/$\mathsf{vc},$\footnote{we use $\mathsf{vc}$ for the vertex cover number of a graph.} exponential lower bounds to the number of obstructions are known, see \cite{Ramachandramurthi95alow} for $\tw,$ \cite{DvorakGT12forbi} for $\td,$ and \cite{Dinneen97} for $\mathsf{vc}.$
Also, as we already mentioned in the beginnings of \autoref{univ_apex_all_small} and \autoref{univ_td_all} there are explicit upper bounds for the sizes of $\obs(\Tcal\Dcal_{k}^{\Hcal})$ and $\obs(\Acal\Pcal_{k}^{\Hcal})$ (to our knowledge no explicit bound to the size of the members of $\obs(\Tcal\Wcal_{k}^{\Hcal})$ is known). 
Given the current state of the art, it appears rather hopeless to expect that some precise characterization of these obstruction sets may be found.
Our obstructing sets suggest that this becomes possible if we drop our expectations and only ask for parametric equivalence. In fact, our results imply that, \textsl{no matter the value of $k$}, we may discard all but $2^{2^{\Ocal(\ell(h^{2}))}}$ of the members of $\obs(\Tcal\Wcal_{k}^{\Hcal})$/$\obs(\Tcal\Dcal_{k}^{\Hcal})$/$\obs(\Acal\Pcal_{k}^{\Hcal})$ and still have the following parametric families of antichains
\begin{eqnarray*}
\Ocal^{\Hcal\text{-}\tw}=\lin{\Ocal_{k}^{{\Hcal\text{-}\tw}}}_{k\in\Nbbb}, & 
\Ocal^{\Hcal\text{-}\td}=\lin{\Ocal_{k}^{{\Hcal\text{-}\td}}}_{k\in\Nbbb} \ \  \text{and}&
\Ocal^{\apex_\Hcal}=\lin{\Ocal_{k}^{\apex_\Hcal}}_{k\in\Nbbb},
\end{eqnarray*}
where the parameters defined by these sets are equivalent to $\Hcal\text{-}\tw,$ $\Hcal\text{-}\td,$ and $\apex_\Hcal.$
This means that ``only'' $2^{2^{\Ocal(\ell(h^{2}))}}$ obstructions are enough for determining the parametric behavior of $\Hcal\text{-}\mathsf{tw},$ $\Hcal\text{-}\mathsf{td},$ and $\apex_{\Hcal},$ for \textsl{any} choice of the class $\Hcal.$
An interesting question is whether by including more (but still not too many) obstructions in the obstruction sets above one can obtain better parametric (preferably with polynomial or even linear gaps) estimations for the corresponding parameters.
Is there an interpretation of the constant-factor approximation algorithms in \cite{Jansen20235Approx} via parametric graph
exclusion?

\vspace{-0.2em}
\paragraph{Strict and strong brambles.}
As we have seen in \autoref{cor_bramble_to_half_int_bramble}, for every minor-closed class $\Hcal,$ the existence of a strict $\Hcal$-bramble of order $f_{\ref{cor_bramble_to_half_int_bramble}}(h,k)$ in a graph implies the existence of a $\nicefrac{1}{2}$-$\Hcal$-bramble of order $k.$
The function $f_{\ref{cor_bramble_to_half_int_bramble}}(h,k)$ here again depends on our general gap function and is single exponential in $k.$
We would certainly like to improve this and we believe that it is indeed possible.
The reason is that the $\nicefrac{1}{2}$-$\Hcal$-brambles found by our proof are of very particular nature:
they stem directly from our walloids and their construction requires the full machinery of our proof and the supporting results from \cite{kawarabayashi2020quickly} and \cite{thilikos2023excluding}.
However, $\nicefrac{1}{2}$-$\mathcal{H}$-brambles may not only be certified by walloids and we cannot exclude the existence of better bounds.
For instance, when $\Hcal$ has some planar obstruction, that is when ${\Hcal}$-$\tw=\tw+\Theta_{h}(1),$ much better dependencies are known.
In particular, Hatzel, Komosa, Pilipczuk, and Sorge proved in \cite{HatzelKPS22} that the existence of a $\nicefrac{1}{2}$-bramble of order $k$ is implied by a bramble of order $\Omega(k^2\cdot \mathsf{polylog}(k))$ (see also \cite{GroheM09,KreutzerT10} for related results).
Whether these bounds can be extended for other choices of $\Hcal$ is an open question.
For this, even a polynomial gap for the case where $\Hcal$ is the class of planar graphs would be interesting.
Another direction is to pursue better dependencies by relaxing the half-integrality, for instance by asking for a $\nicefrac{1}{3}$-$\Hcal$-bramble or a $\nicefrac{1}{4}$-$\Hcal$-bramble where no vertex is used by more than $3$ or $4$ bramble elements.
Is there such a relaxation that may lead to a polynomial dependency for all $\Hcal$s?

\vspace{-0.2em}
\paragraph{Special families of $\Hcal$'s.}
Our work leaves the project of precisely determining $\mathfrak{W}_{\Hcal}^{\mathsf{min}}$ for particular instantiations of the class $\Hcal$ as an open problem.
Are there additional properties for $\Hcal$ that may facilitate such a description?
As a first step in this direction, in     \cite{PaulPTW24Delineating}, we focus on a particular family of $\Hcal$'s, namely those whose obstructions are all \textsl{Kuratowski connected} and which contain at least one obstruction that is a minor of a \textsl{shallow-vortex grid}.
Providing the definitions of these conditions escape the purposes of this paper (for the interested reader, the definitions can be found in \cite{robertson1995sachs} and \cite{thilikos2022killing}).
We wish to mention that when these conditions are imposed on $\Hcal,$ the obstructing sets are much simpler:
They are one or two walloids capturing surfaces (that is, these walloids do not have flower segments) and, using our terminology, are those which are generated by pairs $(\Sigma,\mathbf{B})$ where $\mathbf{B}=\emptyset,$ and a third type of walloid which is generated by pairs of the form $(\Sigma^{(0,0)},\mathbf{B}^*)$ where $\mathbf{B}^*=\{\langle Z,\langle y\rangle \rangle\},$ $Z\in\mathsf{obs}(\mathcal{H}),$ and $y$ is an arbitrary vertex of $Z.$
\medskip

\paragraph{Acknowledgements:}

The authors wish to thank the anonymous reviewers for their remarks and suggestions on  earlier versions of this paper. \medskip

The second and third authors wish to thank Professor \href{https://cgi.di.uoa.gr/~prondo/}{Panos Rondogiannis} for his pivotal influence, which contributed to the realization of this project.

\end{document}